\documentclass[english]{smfbook}
\numberwithin{equation}{chapter}
\usepackage{smf_mem-ns_156}
\smfvolume{156}
\smfyear{2018}
\ISBN{978-2-85629-877-9}	
\newISSN {0249-633X}{2275-3230}	 
\DOI{10.24033/msmf.464}

\let\mathcal\mathscr

\begin{document}
\frontmatter
\title[Irregular Hodge theory]
{Irregular Hodge theory}

\author[C.~Sabbah]{Claude Sabbah}
\address[C.~Sabbah]{CMLS, École polytechnique, CNRS, Université Paris-Saclay\\
F--91128 Palaiseau cedex\\
France}
\email{Claude.Sabbah@polytechnique.edu}
\urladdr{http://www.math.polytechnique.fr/perso/sabbah}
\thanks{This research was supported by the grant ANR-13-IS01-0001-01 of the Agence nationale de la recherche. The collaboration with Jeng-Daw Yu was partly supported by the MoST/CNRS grants 106-2911-I-002-539/TWN PRC 1632}


\begin{abstract}
We introduce the category of \emph{irregular mixed Hodge modules} consisting of  possibly irregular holonomic $D$-modules which can be endowed in a canonical way with a filtration, called the \emph{irregular Hodge filtration}. Mixed Hodge modules with their Hodge filtration naturally belong to this category, as well as their twist by the exponential of any meromorphic function. This category is stable by various standard functors, which produce many more filtered objects. The irregular Hodge filtration satisfies the $E_1$\nobreakdash-degeneration property with respect to any projective morphism. This generalizes some results previously obtained by H.\,Esnault, J.-D.\,Yu and the author. We also show that, modulo a condition on eigenvalues of monodromies, any rigid irreducible holonomic $D$-module on the complex projective line underlies an irregular pure Hodge module. In a chapter written jointly with Jeng-Daw~Yu, we make explicit the case of irregular mixed Hodge structures, for which we prove in particular a Thom-Sebastiani formula.
\end{abstract}

\alttitle{Théorie de Hodge irrégulière}
\altkeywords{D-module holonome, module de Hodge mixte, D-module avec structure de twisteur mixte, filtration de Hodge irrégulière, origine Hodge-exponentielle, formule de Thom-Sebastiani}

\begin{altabstract}
Nous introduisons la catégorie des \emph{modules de Hodge mixtes irréguliers} formée de $D$\nobreakdash-modules holonomes à singularités éventuellement irrégulières qui peuvent être munis de manière canonique d'une filtration, dite \emph{filtration Hodge irrégulière}. Les modules de Hodge mixtes avec leur filtration de Hodge sont naturellement des objets dans cette catégorie, de même que leur produit tensoriel avec l'exponentielle de toute fonction méromorphe. Cette catégorie est stable par plusieurs foncteurs standard, ce qui permet d'obtenir de nombreux exemples. La filtration de Hodge irrégulière satisfait à une propriété de dégénérescence en $E_1$ par rapport à un morphisme projectif. Ceci généralise des résultats précédemment obtenus par H.\,Esnault, J.-D.\,Yu et l'auteur. Nous montrons aussi que, modulo une condition sur les valeurs propres des monodromies, les $D$-modules holonomes irréductibles rigides sur la droite projective complexe sous-tendent des modules de Hodge purs irréguliers. Dans un chapitre écrit en collaboration avec Jeng-Daw~Yu, nous considérons le cas des structures de Hodge mixtes irrégulières, pour lequel nous montrons en particulier une formule de Thom-Sebastiani.
\end{altabstract}

\maketitle
\chapterspace{-4}
\tableofcontents

\mainmatter

\chapterspace{-3}
\chapter*{Introduction}
Let $X$ be a complex manifold. The category \index{$MTMX$@$\MTM(X)$}$\MTM(X)$ of \index{mixed twistor $\cD$-module}mixed twistor $\cD$-modules on $X$, introduced by T.\,Mochizuki \cite{Mochizuki11}, represents a vast generalization of that of \index{mixed Hodge module}mixed Hodge modules on $X$, introduced by M.\,Saito \cite{MSaito87} and denoted by \index{$MHMX$@$\MHM(X)$}$\MHM(X)$.\footnote{Throughout this article, we consider complex mixed Hodge module and $\MHM(X)$ stands for $\MHM(X,\CC)$, \cf Section \ref{subsec:MHM}.} An intermediate step to compare these categories is the category $\MTM^\intt(X)$ of integrable mixed twistor $\cD$-modules on $X$ (see below). There are natural functors
\[
\MHM(X)\mto\MTM^\intt(X)\mto\MTM(X),
\]
where the second functor forgets the integrable structure, while the first one is a natural fully faithful functor compatible with the standard functors on each category (\cf\cite[\S13.5]{Mochizuki11}). On the other hand, the category $\MTM^\intt(X)$ contains much more objects, since the holonomic $\cD$-module underlying a mixed Hodge module has regular singularities, while that underlying an object of $\MTM^\intt(X)$ can have arbitrary complicated irregular singularities.

One drawback of the category $\MTM^\intt(X)$ (or $\MTM(X)$) is that objects do not come with many numerical invariants, like Hodge numbers. Our aim is to introduce an intermediate category between $\MHM(X)$ and $\MTM^\intt(X)$, which is not as large as $\MTM^\intt(X)$, but so that the $\cD$-module associated with each object is equipped with a good filtration, called the \emph{irregular Hodge filtration}. Many interesting irregular holonomic $\cD$-modules underlie objects in this category, in particular those generated by $\exp\mero$, where~$\mero$ is any meromorphic function on~$X$. Some of the standard functors (pushforward by a projective morphism, pullback by a smooth morphism) extend to this category, that we call the category of (possibly) \index{mixed Hodge module!irregular --}\emph{irregular mixed Hodge modules} and that we denote by $\IrrMHM(X)$, and the filtered holonomic $\cD$-module associated with any object has a good behaviour with respect to these functors. The original idea of finding such a category is due to Deligne in 1984 \cite{Deligne8406}, and has waited the development of the theory of wild twistor $\cD$-modules \cite{Mochizuki08,Bibi06b} to start a new life in \cite{Bibi08}, before being considered from various points of view in \cite{Yu12,E-S-Y13,S-Y14,K-K-P14,Mochizuki15a}.

Let us emphasize some discrepancies between the usual Hodge filtration and the irregular Hodge filtration.
\begin{enumerate}
\item
A mixed Hodge module can naturally be seen as an object of $\IrrMHM(X)$, and in such a way its Hodge filtration is the irregular Hodge filtration. However, the properties of the irregular Hodge filtration of the $\cD$-module underlying an object of $\IrrMHM(X)$ are not enough to characterize this object as such, \ie are not constitutive of the definition of an object of $\IrrMHM(X)$, while the Hodge filtration is one of the fundamental objects used to define a Hodge module. In other words, the irregular Hodge filtration is only a byproduct of the definition of the category $\IrrMHM(X)$, and the main properties of this category are obtained from those of $\MTM^\intt(X)$.
\item
Recall that the behaviour of the Hodge filtration with respect to nearby cycles is one of the fundamental properties used to define the category $\MHM(X)$. In contrast, we do not exhibit any such property for the irregular Hodge filtration. It is reasonable to expect that a similar behaviour occurs with respect to a divisor along which the mixed twistor $\cD$-module is \emph{tame}, but we do not even have a conjectural statement in general.
\end{enumerate}

The \emph{rescaling operation} will be used to define the category $\IrrMHM(X)$. It has been defined in the framework of TERP structures in \cite{H-S06}. This rescaling operation is the main ingredient in order to define the category $\iMTM^\resc(X)$. However, before doing so, we need to modify the presentation of objects in $\MTM^\intt(X)$: while it is easy to define the rescaling of an integrable $\cR_\cX$-module (with $\cX:=X\times\CC_\hb$) by adding a new parameter $\theta\in\CC^*$ and changing $\hb$ to $\hb/\theta$, the pairing is not rescalable, since it is defined only on $X\times\bS$ ($\bS=\{|\hb|=1\}$), and the rescaling operation does not preserve~$\bS$. We are therefore led to give another description of an object in $\MTM^\intt(X)$, where the pairing $\iC$ is now $\iota$-sesquilinear ($\iota:\hb\mto-\hb$) and is defined on $\cX^\circ:=X\times\CC^*_\hb$ and not only on $X\times\bS$. The purpose of Chapter \ref{part:1}, which is essential for simply giving the definition of the rescaling operation, is to define this category $\iMTM^\intt(X)$ and to show the equivalence $\iMTM^\intt(X)\simeq\MTM^\intt(X)$. (Note that the integrability property is important and we cannot argue without it.)

\begin{remarque}[on the notation]
We usually denote an object of $\MTM$ as $(\cT,W_\bbullet)$, but we sometimes shorten the notation as $\cT$, when the context is clear.
\end{remarque}

\subsubsection*{The main theorems}
The category $\iMTM^\intt(X)$ comes equipped with standard functors which are compatible with those defined for $\MTM^\intt(X)$ in \cite{Mochizuki11}, as shown in Chapter \ref{part:1}:
\begin{itemize}
\item
The duality functor $\bD_X$.
\item
For a projective morphism $\map:X\to Y$ between smooth complex manifolds, the pushforward functors $\map_\dag^k=\cH^k\map_\dag$.
\item
For a smooth morphism $\mapsm:X\to Y$, the pullback functor $\mapsm^+$.
\item
The external product $\hbboxtimes$ (a bi-functor).
\item
Let $H$ be an effective divisor in $X$. One can define a localization functor and a dual localization functor, denoted respectively by $[*H]$ and $[!H]$, and in a short way by $[\star H]$ ($\star=*,!$), which is an endofunctor of $\MTM(X)$ and of $\MTM^\intt(X)$ (\cf\cite[\S3.4]{Bibi01c} for the definition of $[*H]$ on strictly specializable $\cR$-modules and \cite[Chap.\,3 \&\,\S11.2]{Mochizuki11} for the definition on $\cR$-triples and on $\MTM$). There are corresponding functors on $\iMTM^\intt(X)$.
\end{itemize}

Such functors also exist on mixed Hodge modules, and the natural fully faithful functor $\Phi_X:\MHM(X,\RR)\mto\MTM^\intt(X)$ (\cf\cite[Prop.\,13.5.4, 13.5.5]{Mochizuki11}) is compatible with them (\cf\cite[Prop.\,13.5.6]{Mochizuki11}).

\begin{itemize}
\item
Let $\mero$ be a meromorphic function on $X$ with pole divisor $P$. According to \cite{Mochizuki11} (\cf also \cite[Prop.\,3.3]{S-Y14}), there exists a localizable $\cR_\cX$-module $\cE^{\mero/\hb}$ extending $(\cO_\cU,\hb\rd+\rd\mero)$, where $U=X\moins P$, such that $\cE^{\mero/\hb}=\cE^{\mero/\hb}[*P]$ (\cf\S\ref{subsec:cTvarphihb} below for details). When equipped with the constant metric on $U$, it gives rise to an object $\cT^{\mero/\hb}$ of $\MTM^\intt(X)$ which satisfies $\cT^{\mero/\hb}[*P]=\cT^{\mero/\hb}$. For each object~$\cT$ in $\MTM^\intt(X)$, there is (\cf\cite[Prop.\,11.3.3]{Mochizuki11} and Section \ref{subsec:cTvarphihb} below) an object $\cT^{\mero/\hb}\otimes\cT$ in $\MTM^\intt(X)$ satisfying the following property: if $(\ccM,\nabla)$ is the holonomic $\cD$-module $\cM''/(\hb-1)\cM''$ associated with $\cT$, then $(\ccM(*P),\nabla+\rd\mero)$ is that associated with $\cT^{\mero/\hb}\otimes\cT$.
\end{itemize}

\skpt
\begin{theoreme}\label{th:mtmgresc}
\begin{enumerate}
\item\label{th:mtmgresc2}
The category $\IrrMHM(X)$ is an abelian full subcategory of $\iMTM^\intt(X)$ which contains (a subcategory equivalent to) $\MHM(X)$ as a full subcategory. As such, it is stable by direct summand in $\iMTM^\intt(X)$. The functors $\map_\dag^k$ ($\map:X\to Y$ projective) and $\mapsm^+$ ($\mapsm:X\to Y$ smooth) preserve the category $\IrrMHM$.

\item\label{th:mtmgresc3}
The functor $\cT^{\mero/\hb}\otimes\cbbullet$ sends the category $\MHM(X)$ to $\IrrMHM(X)$.
\end{enumerate}
\end{theoreme}

\begin{theoreme}\label{th:mainRF}
For each object $(\icT,W_\bbullet)=\bigl((\cM',\cM'',\iC),W_\bbullet\bigr)$ of $\IrrMHM(X)$, the associated $\cD_X$-module $\ccM:=\cM''/(\hb-1)\cM''$ is naturally endowed with a weight filtration $W_\bbullet\ccM$ by $\cD_X$-submodules and with a good filtration $F^\irr_\bbullet\ccM$ (called the \emph{irregular Hodge filtration}) indexed by $\RR$. The following properties are satisfied for such objects.
\begin{enumerate}
\item\textup{(Naturality)}\label{th:mainRF1}
If $(\icT,W_\bbullet)$ belongs to the essential image of $\MHM(X)$, then
\[
F^\irr_\bbullet\ccM=F_\bbullet\ccM.
\]
\item\textup{(Strictness)}\label{th:mainRF2}
Any morphism in $\IrrMHM(X)$ induces a morphism in $\Mod(\cD_X)$ which is bi-strictly compatible with the irregular Hodge filtrations and the weight filtrations.
\item\textup{(Smooth pullback)}\label{th:mainRF4}
For a smooth morphism $\mapsm:X\to Y$ we have
\[
F^\irr_\bbullet \mapsm^+\ccM=\mapsm^*F^\irr_\bbullet\ccM.
\]

\item\textup{(Degeneration at $E_1$ for the projective pushforward)}\label{th:mainRF6}
For a projective morphism $\map:X\to Y$ and for each $\alpha\in[0,1)$, the complex $\map_\dag R_{F^\irr_{\alpha+\bbullet}}\ccM$ is strict and for each $k\in\ZZ$,
\[
R_{F^\irr_{\alpha+\bbullet}}\map_\dag^k\ccM=\map_\dag^k R_{F^\irr_{\alpha+\bbullet}}\ccM.
\]
\end{enumerate}
\end{theoreme}

\skpt
\begin{remarques}
\begin{enumerate}
\item\label{rem:mainRF2}
See \cite[Def.\,1.2]{S-Y14} for the definition of a good filtration indexed by $\RR$, as well as for the notation $F_{\alpha+\bbullet}$ and the associated Rees construction $R_{F_{\alpha+\bbullet}}$, which then means a filtration indexed by~$\ZZ$.

\item
We do not claim that $\bD_X$, $\hbboxtimes$ and $[\star H]$ preserve $\IrrMHM(X)$ (see also Remark \ref{rem:mainRF}\eqref{rem:mainRF3} below), although it could be expected, but they preserve $\MHM$. In particular, we do not give a formula for the behaviour of the irregular Hodge filtration with respect to these functors. Special cases however have already been obtained by J.-D.\,Yu: the behaviour of the irregular Hodge filtration under duality has been considered in \cite[\S2]{Yu12} in the special case corresponding to $\ccE^\mero:=(\cO_X(*P),\rd+\rd\mero)$, and its behaviour with respect to the external product in the case of $\ccE^\mero\boxtimes\ccE^\psi$ has been considered in \cite{C-Y16}. See also Sections \ref{sec:KunnethTS} and \ref{sec:alternating} of Chapter \ref{chap:irregmhm}, written jointly with Jeng-Daw Yu, for complementary results.

\end{enumerate}
\end{remarques}

\begin{corollaire}[Irregular mixed Hodge modules of exponential-Hodge origin]\label{cor:expMHM}\index{mixed Hodge module!irregular -- of exponential origin}Let $\cT$ (\resp $\lambda:\cT_1\to\cT_2$) be a mixed Hodge module (\resp a morphism of mixed Hodge modules). Any object (\resp morphism) of $\iMTM^\intt$ obtained by applying to~$\cT$ (\resp $\lambda$) any finite sequence of functors~$\map_\dag^k$ ($\map$~projective), $\mapsm^+$ ($\mapsm$ smooth), $\bD$, $\hbboxtimes$, $[\star H]$ ($\star=*,!$) and $\cT^{\mero/\hb}\otimes\cbbullet$ ($\mero$~meromorphic), taken in any order, is an object (\resp a morphism) of $\IrrMHM$, hence is equipped with a canonical irregular Hodge filtration (\resp induces a strictly bi-filtered morphism between the corresponding bi-filtered holonomic $\cD$-modules).
\end{corollaire}

\begin{proof}[Idea of the proof]
We notice that, according to the commutation relations between these functors (\cf Section \ref{sec:compatibilities}), it is possible to obtain the same result in $\iMTM^\intt$ by applying first a sequence of functors of the type $\bD$, $\hbboxtimes$, $[\star H]$ or $\mapsm^+$ ($g$~smooth), then a \emph{single} twist $\cT^{\mero/\hb}\otimes\cbbullet$, and then a sequence of functors of the type $\map_\dag^k$. Starting from objects (\resp morphisms) in $\MHM$, the first sequence only produces objects (\resp morphisms) in $\MHM$. Applying then a single functor $\cT^{\mero/\hb}\otimes\cbbullet$ gives an object (\resp a morphism) in $\IrrMHM$, according to Theorem \ref{th:mtmgresc}\eqref{th:mtmgresc3}. Then for projective pushforwards we apply \ref{th:mtmgresc}\eqref{th:mtmgresc2}.
\end{proof}

The proof of the previous results will be given in Section \ref{sec:MTrM}.

\skpt
\begin{remarques}\label{rem:mainRF}
\begin{enumerate}
\item\label{rem:mainRF0}
We also say that an integrable mixed twistor $\cD$-module obtained as in Corollary \ref{cor:expMHM} is an \emph{integrable mixed twistor $\cD$-modules of exponential-Hodge origin} (\cf Definition \ref{def:expMHM}). The associated holonomic $\cD_X$-module comes then equipped with a canonical irregular Hodge filtration.

\item\label{rem:mainRF1}
When $\ccM$ underlies an \index{mixed Hodge module!exponentially twisted --}exponentially twisted mixed Hodge module, the behaviour with respect to \emph{a single} projective pushforward has already been obtained in \cite{S-Y14} and the particular case of a projection has been obtained in \cite{E-S-Y13}. We will check that in such a case, the definition of the irregular Hodge filtration we give here coincides with that of \loccit The approach of \loccit\ did not however allow iterating various functors of the kind $\map_\dag^k$ ($\map$ projective) and~$\mapsm^+$ ($\mapsm$ smooth) on a given exponentially twisted mixed Hodge module, while the present one gives a canonical irregular Hodge filtration on the result of such an iteration, and moreover shows that two such iterations starting from the same exponentially twisted mixed Hodge module lead to the same holonomic $\cD_X$-module, with the same irregular Hodge filtration.

\item\label{rem:mainRF3}(Nearby/vanishing cycles)
Theorem \ref{th:mainRF} says nothing about the behaviour of the irregular Hodge filtration with respect to moderate nearby/vanishing cycles along a function (\cf Section \ref{subsec:openq}). Let $\fun:X\to\CC$ be a holomorphic function and set $H=\fun^{-1}(0)$. If $\cT$ is \emph{tame} along $H$, one can expect that, for each $\alpha\in[0,1)$, the $R_F\cD_X$-module $R_{F^\irr_{\alpha+\bbullet}}\ccM$ is strictly specializable along $H$, so the behaviour of the irregular Hodge filtration along~$H$ is similar to the case of an ordinary Hodge filtration coming from $\MHM$. (See \cite[Th.\,6.1]{S-Y14}.)\enlargethispage{.5\baselineskip}%

On the other hand, if $\cT$ is wild along $H$, such a behaviour should not be expected, due to the presence of ``slopes'' (\cf\eg\cite[Def.\,3.5.1 \& 3.5.2]{C-G04}) between the irregular Hodge filtration and the Kashiwara-Malgrange filtration along $\fun$.

\item\label{rem:mainRF4}(Uniqueness for irreducible holonomic $\cD$-modules)
Assume that $X$ is a \emph{smooth projective variety} and let $\ccM$ be an \emph{irreducible} holonomic $\cD_X$-module. According to \cite[Th.\,1.4.4]{Mochizuki08}, there exists a unique (up~to isomorphism) polarizable pure twistor $\cD$-module $\cT=(\cM',\cM'',C_{\bS})$ of weight $0$ and polarization given by the identity on $\cM'$ and $\cM''$ such that $\Xi_{\DR}(\cT):=\cM''/(\hb-1)\cM''=\ccM$. By Remark \ref{rem:intuniquepure} below, $\cT$ admits at most one integrable structure (up to an obvious equivalence), hence there exists at most one (up to equivalence) object $\icT$ of the subcategory $\IrrMHM(X)$ of $\iMTM^\intt(X)$ giving rise to $\ccM$. The irregular Hodge filtration of $\ccM$, if defined, is obtained from this unique object, hence is intrinsically attached to $\ccM$ up to a shift by an arbitrary real number.

\item(The Lefschetz morphism)
We will show that the cup product by the Chern class of an ample line bundle induces a morphism in $\IrrMHM$. As a consequence, since the Hard Lefschetz theorem holds for pure objects in $\IrrMHM$ according to \cite{Mochizuki08}, one can conclude that the Hard Lefschetz isomorphism is strictly filtered with respect to the irregular Hodge filtration (\cf Corollary \ref{cor:lefschetz}).

\item(Real and rational structure)
Let $\kk$ be a subfield of $\RR$. The category $\MTM_\good^\intt(X,\kk)$ is defined in \cite[\S13.4.4]{Mochizuki11}. It consists of integrable mixed twistor $\cD$-modules with a good $\kk$-structure. First, these objects are endowed with a real structure $\kappa$ which is compatible with the action of $\hb^2\partial_\hb$, and which induces a real structure on the holonomic $\cR_\cX$-module $\cM''(\hbm)$. Second, this real structure is refined as a $\kk$-structure in the sense of \cite{Mochizuki10}. The equivalence $\MTM^\intt(X)\simeq\iMTM^\intt(X)$ enables us to define the category $\iMTM^\intt_\good(X,\kk)$, and this naturally leads to the category $\IrrMHM(X,\kk)$. Note that, in order to be completely compatible with the notion of a $\kk$-mixed Hodge structure, it would be natural to impose the condition of the existence of a polarization defined over $\kk$ on each graded module with respect to the weight filtration. We do not add this condition in order to be compatible with the definition in \cite[\S13.4.4]{Mochizuki11}. However, with such a condition, if $\kk=\RR$ and $X$ is a point, this would completely agree with the notion of tr-TERP structure of Hertling \cite{Hertling01}, and if $\kk=\QQ$ with that of non-commutative Hodge structure of \cite{K-K-P08}.\vspace*{-3pt}
\end{enumerate}
\end{remarques}

\subsubsection*{Irregular mixed Hodge structures}
In Chapter \ref{chap:irregmhm}, written jointly with Jeng-Daw Yu, we consider the simple case of irregular mixed Hodge modules when $X$ is reduced to a point, that we call irregular mixed Hodge structures. The presentation via the rescaling operation can be a little simplified. The category $\IrrMHS$ of \emph{irregular mixed Hodge structures} is endowed with a functor in the category of bi-filtered vector spaces $(H,F_\bbullet^\irr,W_\bbullet)$ and, as an application of Theorem \ref{th:mainRF}\eqref{th:mainRF2}, any morphism in $\IrrMHS$ gives rise to a strictly bi-filtered morphism. The main result in this chapter is the Thom-Sebastiani formula for the irregular Hodge filtration (Theorem \ref{th:main}), for which we give two different proofs, one much relying on the theory of mixed twistor $\cD$\nobreakdash-modules \cite{Mochizuki11}, and the other one depending more on results of \cite{S-Y14}. We give an application to the alternating product of the irregular Hodge filtrations. We also show that the category of exponential mixed Hodge modules of Kontsevich-Soibelman \cite{K-S10} is equipped with a natural functor to $\IrrMHS$.\vspace*{-2pt}\vskip0pt

\subsubsection*{Application to irreducible rigid $\cD$-modules on $\PP^1$}
Let us now consider an irreducible holonomic $\cD$-module $\ccM$ on the complex projective line $\PP^1$. At any of its singular points $x_o\in\PP^1$, the formalized module $\ccM_{\wh{x_o}}$ has a decomposition as the sum of a regular $\cD_{\PP^1,\wh{x_o}}$-module and an irregular one (\cf\eg\cite{Malgrange91}). The irregular one is a free $\cO_{\PP^1,\wh{x_o}}(*x_o)$-module with connection, which can be decomposed further as the direct sum of elementary $\cO_{\PP^1,\wh{x_o}}(*x_o)$-modules with connection, according to the Levelt-Turrittin theorem. For each summand as well as for the regular part there is a notion of formal monodromy, and it is meaningful to require that the eigenvalues of the formal monodromy of $\ccM_{\wh{x_o}}$ have absolute value equal to one. We also say that~$\ccM$ is \emph{locally formally unitary} if this holds at each singular point. On the other hand, given an irreducible holonomic $\cD$-module $\ccM$, restricting (algebraically) to any proper Zariski open set $j:U\hto\PP^1$ not containing any singularity of $\ccM$ produces a free $\cO_U$\nobreakdash-module $(\ccV,\nabla)$ with connection which is irreducible as such, and this \hbox{restriction} functor is an equivalence, a quasi-inverse functor being given by the minimal (or~\hbox{intermediate}) extension $j_{!*}$. We say that $\ccM$ is rigid if it satisfies Katz' criterion that its index of rigidity is equal to $2$, that~is,\vspace*{-.25\baselineskip}\enlargethispage{.6\baselineskip}%
\[
h^1\bigl(\PP^1,\DR(j_{!*}\ccEnd(\ccV,\nabla))\bigr)=0.\vspace*{-.25\baselineskip}
\]
For example (\cf\cite[Th.\,3.7.3]{Katz90}), irreducible hypergeometric $\cD_{\PP^1}$-modules of type $(n,m)$ with $n\neq m$ have an irregular singularity and are rigid.

\begin{theoreme}\label{th:rigidP1}
Assume that $\ccM$ is irreducible and rigid. Then the unique pure twistor $\cD$-module it underlies is integrable if and only if $\ccM$ is locally formally unitary. In~such a~case, this integrable pure twistor $\cD$-module is graded-rescalable and comes from a~unique object (up to equivalence) of $\IrrMHM(\PP^1)$. In particular, $\ccM$ can be equipped with a canonical irregular Hodge filtration (up to a shift).
\end{theoreme}

The case when $\ccM$ has regular singularities is well-known, according to \cite[Cor.\,8.1]{Simpson90} and \cite[Prop.\,1.13]{Deligne87}. These results assert that the corresponding local system underlies a unique polarized variation of complex Hodge structure (up to a shift), \cf\cite[\S2.4]{D-S12}. The proof of Theorem \ref{th:rigidP1} given in Section \ref{sec:rigid} relies on the existence of a variant, due to Deligne \cite{Deligne06b} and Arinkin \cite{Arinkin08}, of the Katz algorithm to reduce to a rank-one $(\ccV,\nabla)$. We give in Section \ref{subsec:hypergeom} an example of computation for the case of some confluent \index{hypergeometric}hypergeometric differential equation, originally due to A.\,Castaño Domínguez and C.\,Sevenheck \cite{CD-S17}.

\begin{remarque}
We do not know how to make precise the behaviour of the irregular Hodge data (\eg degrees and ranks of irregular Hodge bundles) along the Arinkin-Deligne algorithm, in a way similar to what is done in \cite{D-S12} for the regular case and the Katz algorithm (see an explicit computation for irreducible hypergeometric differential equations of type $(n,n)$ in \cite{Fedorov15}). We find the same problem as in Remark \ref{rem:mainRF}\eqref{rem:mainRF3}. Note also that, compared to Corollary \ref{cor:expMHM}, a new argument is needed here in order to treat the intermediate extension functor, regarded as taking the image of $[!H]$ into~$[*H]$. On the other hand, the Laplace transformation entering in the Arinkin-Deligne algorithm can be regarded as a sequence of functors considered in Corollary~\ref{cor:expMHM}.
\end{remarque}

\subsubsection*{Acknowledgements}
This work is a continuation of \cite{E-S-Y13,S-Y14}, and I thank Hélène Esnault and Jeng-Daw Yu for many discussions on this subject. Various enlight\-ening discussions with Jeng-Daw Yu led me to state the question of the irregular Hodge filtration in the present general framework, although some aspects are still missing. His comments on a preliminary version have also been very useful. Our discussions resulted in Chapter \ref{chap:irregmhm}, written jointly. The idea behind Corollary \ref{cor:expMHM} comes from discussions with Maxim Kontsevich, in particular related to his article \cite{Kontsevich09}. It is clear that this work could not have existed without the impressive results of Takuro Mochizuki about pure and mixed twistor $\cD$-modules. It is my pleasure to thank him for many interesting discussions and comments. In particular, he suggested a useful simplification in a previous definition of $\IrrMHM$. The idea of exploiting the rescaling procedure in these questions goes back to the work of Claus Hertling and Christian Sevenheck and I thank them for sharing their ideas with me. My interest on rigid irreducible holonomic $\cD_{\PP^1}$-modules owes much to Michael Dettweiler, whom I thank for that.\enlargethispage{-\baselineskip}%

\pagebreak[2]
All over this text, we will use the following notation and definition.

\begin{notation}\label{nota:general}
Let \index{$Cz$@$\CC_\hb$}$\CC_\hb$ be the complex line with coordinate $\hb$. We set \index{$Sb$@$\bS$}\hbox{$\bS\!=\!\{\hb\mid|\hb|\!=\!1\}$}, and we denote by \index{$OS$@$\cO_{\bS}$}$\cO_{\bS}$ the sheaf of real analytic functions on $\bS$. It can be regarded as the sheaf-theoretic restriction $i_{\bS}^{-1}\cO_{\CC_\hb}$ of the sheaf of holomorphic functions on $\CC_\hb$.

We denote by $\PP^1$ the completion of the complex line $\CC_\hb$. We will consider the following involutions:
\begin{itemize}
\item
\index{$Sigma$@$\sigma$}$\sigma$ is the anti-linear involution induced by $\hb\mto-1/\ov\hb$,
\item
\index{$Gamma$@$\gamma$}$\gamma$ is that induced by $\hb\mto1/\ov\hb$,
\item
\index{$IA$@$\iota$}$\iota$ is the holomorphic involution induced by $\hb\mto-\hb$, so that $\gamma=\sigma\circ\iota=\iota\circ\sigma$.
\end{itemize}
Note that $\sigma$ and $\gamma$ (but not $\iota$) exchange $0$ and $\infty$ on $\PP^1$. We also have $\gamma_{|\bS}=\id_{\bS}$, so that $\sigma$ and $\iota$ coincide on $\bS$.

Given a complex manifold $X$ of dimension $d_X$, we set \index{$XXC$@$\cX$}$\cX\!=\!X\times\CC_\hb$ and \index{$XXCcirc$@$\cX^\circ$}$\cX^\circ\!=\!X\times\CC_\hb^*$, and \index{$PI$@$\pi$}$\pi$ \resp \index{$PIC$@$\pi^\circ$}$\pi^\circ$ denote the corresponding projections to $X$.
\end{notation}

\begin{definition}
An $\cO_{\CC_\hb}$-module (\resp a $\CC[\hb]$-module, \resp a $\CC[\hb,\hbm]$-module) is called \index{strict}\emph{strict} if it is flat as such.
\end{definition}

\begin{notation}[for maps]\label{nota:maps}
We usually denote a projective morphism by \index{$F$@$\map$}$\map:X\to Y$ and a smooth morphism by \index{$G$@$\mapsm$}$\mapsm$. A holomorphic function \index{$H$@$\fun$}$\fun:X\to\CC$ has divisor \index{$HA$@$H$}$H=(\fun)$. The graph inclusion is denoted by \index{$IH$@$i_\fun$}\index{$XH$@$X_\fun$}$i_\fun:X\hto X\times\CC_t=:X_\fun$, and $\fun$ is the composition of $i_\fun$ and the function $t:X_\fun\to\CC_t$, the latter having divisor \index{$HAh$@$H_\fun$}$H_\fun:=X\times\{0\}$. The corresponding pushforward for a $\cD_X$-module $\ccM$ or an $\cR_\cX$-module $\cM$ is denoted by \index{$Mh$@$\ccM_\fun$}$\ccM_\fun:=i_{\fun+}\ccM$ or \index{$Mh$@$\cM_\fun$}$\cM_\fun:=i_{\fun+}\cM$.
\end{notation}

\begin{notation}[for $\Xi_{\DR}$]\label{nota:XiDR}
\index{$XXIDR$@$\Xi_{\DR}$|finindexnotations}We use the same notation for various functors.
\begin{itemize}
\item
If $\cT=(\cM',\cM'',C_{\bS})$ is an object of \index{$RTriplesX$@$\RTriples(X)$}$\RTriples(X)$ or if $\icT=(\cM',\cM'',\iC)$ is an object of $\RdiTriples(X)$, then we set
\[
\Xi_{\DR}(\cT)=\cM''/(\hb-1)\cM''\in\Mod(\cD_X),
\]
that we also denote by $\ccM$.
\item
If $(\cT,W_\bbullet)$ is a $W$-filtered object of $\RTriples(X)$, that is, an object of the category $\WRTriples(X)$, or if $(\icT,W_\bbullet)$ is an object of $\WRdiTriples(X)$, we set $\Xi_{\DR}(\cT,W_\bbullet)=(\Xi_{\DR}\cT,W_\bbullet\Xi_{\DR}\cT)$ and similarly for $\Xi_{\DR}(\icT,W_\bbullet)$, where the latter filtration is the filtration naturally induced by $W_\bbullet\cM''$.
\item
If $(\icT,W_\bbullet)$ is an object of $\IrrMHM(X)$, then we regard $\Xi_{\DR}(\icT,W_\bbullet)$ as also endowed with its irregular Hodge filtration, that is,
\[
\Xi_{\DR}(\icT,W_\bbullet)=(\ccM,F_\bbullet^\irr\ccM,W_\bbullet\ccM),
\]
and we regard $\Xi_{\DR}$ as a functor from $\IrrMHM(X)$ to the category of bi-filtered $\cD_X$-modules.
\end{itemize}
\end{notation}

\chapter[Integrable mixed twistor \texorpdfstring{$\cD$}{D}-modules]{An equivalent presentation of integrable~mixed twistor \texorpdfstring{$\cD$}{D}-modules}\label{part:1}\label{PART:1}

\section{A review of twistor structures in dimension zero}\label{sec:twistorstructures}

In this section, we review the notion of a pure polarized, or mixed, twistor structure, together with supplementary structures on it called real or good real structure and, for a subfield $\kk$ of $\RR$, good $\kk$-structure. These are the simplest examples of objects considered in this chapter. Various aspects of such structures, introduced by Simpson in \cite{Simpson97}, have also been considered with other names, \eg tr-TERP structures in the sense of Hertling \cite{Hertling01} and non-commutative $\kk$-Hodge structures as defined in \cite{K-K-P08} (\cf also \cite[\S3.8]{Bibi11}). We~first present these structures as living on a vector bundle on $\PP^1$, and then we use the presentation by $\RTriples$, introduced in \cite{Bibi01c}, in order to extend them in any dimension and make the link with the theory of mixed twistor structures \cite{Mochizuki11}. This presentation will also prove useful for expressing the rescaling property in Chapter~\ref{part:2}.

\subsection{Twistor structures as vector bundles on $\PP^1$}\label{sec:twistorvectorbundles}

A \index{twistor structure}\emph{twistor structure} is a vector bundle $\cT$ on~$\PP^1$. We then denote (\cf Notation \ref{nota:general})
\begin{equation}\label{eq:defdualsigmarc}
\begin{cases}
\bbullet\ \cT^\vee&\text{the dual bundle},\\
\bbullet\ \cT^*:=\sigma^*\ov\cT^\vee&\text{the \emph{Hermitian dual} bundle},\\
\bbullet\ \cT^\rc:=\gamma^*\ov\cT&\text{the \emph{conjugate} bundle},
\end{cases}
\end{equation}
and we define a \index{twistor structure!real --}\emph{real structure} on $\cT$ to be an isomorphism \index{$Kappa$@$\kappa$}$\kappa:\cT\to\cT^\rc$ such that, through the canonical isomorphism $(\cT^\rc)^\rc=\cT$, we have $\kappa^\rc=\kappa^{-1}$. Note that $\kappa^{*-1}$ (\resp $\kappa^{\vee-1}$) is then a real structure on $\cT^*$ (\resp on $\cT^\vee$).

\begin{remarque}
If we describe $\cT$ as the result of the gluing of vector bundles $\cM^0$ and~$\cM^\infty$ in the charts of $\PP^1$ centered at $0$ and $\infty$ respectively, then a real structure $\kappa$ identifies~$\cM^\infty$ with $\gamma^*\ov\cM{}^0$, so that the vector bundle $\cT$ on $\PP^1$ can be also described as resulting of the gluing of $\cM^0$ with $\gamma^*\ov\cM{}^0$ in a way invariant with respect to $\gamma^*\ov{\phantom{X}}$.
\end{remarque}

An \index{twistor structure!integrable --}\emph{integrable twistor structure} is a pair $(\cT,\nabla)$ consisting of a twistor structure endowed with a meromorphic connection $\nabla$ having a pole \emph{of order at most two} at $\hb=0$ and $\hb=\infty$, and no other pole. Then $(\cT,\nabla)^\vee,(\cT,\nabla)^*,(\cT,\nabla)^\rc$ are defined similarly, and a real structure on $(\cT,\nabla)$ is an isomorphism $\kappa:(\cT,\nabla)\isom(\cT,\nabla)^\rc$ such that $\kappa^\rc=\kappa^{-1}$.

\begin{exemple}[The Tate object]\label{exem:Tate}
\index{Tate object}We will use the convention and the notation of \cite[\S2.1.8]{Mochizuki11}, that we recall. For $\ell\in\ZZ$, we denote by \index{$Tl$@$\bT(\ell)$}$\bT(\ell)$ the bundle $\cO_{\PP^1}(-2\ell)$, that we endow with the meromorphic connection $\nabla$ having a simple pole at $0$ and~$\infty$, both with residue~$\ell$, and no other pole. We have natural identifications, that we denote by~$\id$:
\[
\iota^*\bT(\ell)=\bT(\ell),\quad\bT(-\ell)=\bT(\ell)^\vee,\quad\bT(-\ell)=\bT(\ell)^*,\quad\bT(\ell)=\bT(\ell)^\rc.
\]
However, we use the following identifications, compatible with tensor product:
\begin{starequation}\label{eq:defdualsigmarcTate}
\begin{cases}
\bbullet\ \id:\iota^*\bT(\ell)=\bT(\ell),\\
\bbullet\ d_{-2\ell}=\id:\bT(-\ell)=\bT(\ell)^\vee,\\
\bbullet\ s_{-2\ell}=(-1)^\ell\id:\bT(-\ell)\isom\bT(\ell)^*,\\
\bbullet\ \kappa_{-2\ell}=(-1)^\ell\id:\bT(\ell)\isom\bT(\ell)^\rc.
\end{cases}
\end{starequation}%
The local system $\bT(\ell)^\nabla_{\bS}$ is the constant local system $\CC_{\bS}$ endowed with the real structure $\kappa_{-2\ell}^\nabla$ induced by the multiplication by $(-1)^\ell$. It is natural to endow it with the rational structure $\bT(\ell)^\nabla_{\bS,\QQ}:=(\twopii)^\ell\QQ_{\bS}$, in such a way that
\[
\ker(\kappa_{-2\ell}^\nabla-\id)=:\bT(\ell)^\nabla_{\bS,\RR}=\RR\otimes_\QQ\bT(\ell)^\nabla_{\bS,\QQ}.
\]
\end{exemple}

Given a twistor structure $\cT$ (\resp an integrable twistor structure $(\cT,\nabla)$), an \index{Hermitian duality}\emph{Hermitian duality of weight~$w$} on $\cT$ (\resp on $(\cT,\nabla)$) is an isomorphism \index{$Sc$@$\cS$}$\cS:\cT\to\cT^*\otimes\bT(-w)$ (\resp compatible with $\nabla$) such that the composed isomorphism
\[
\cT\otimes\bT(w)\To{\id\otimes s_{-2w}}\cT\otimes\bT(-w)^*\To{\cS^*}\cT^*
\]
tensored by $\bT(-w)$ gives back $(-1)^w\cS$. In other words, $\cS$ is a pairing
\[
\cS:\cT\otimes_{\cO_{\PP^1}}\sigma^*\ov\cT=\cT\otimes_{\cO_{\PP^1}}\iota^*\cT^\rc\to\bT(-w),
\]
which satisfies, through the isomorphism $\kappa_{2w}^{-1}:\sigma^*\ov{\bT(-w)}=\bT(-w)^\rc\isom\bT(-w)$,
\[
\sigma^*\ov\cS=(-1)^w\cS,\quad\text{equivalently}\quad\iota^*\cS^\rc=(-1)^w\cS.
\]

\skpt
\begin{exemple}\label{exem:TU}
\begin{enumerate}
\item\label{exem:TU1}
We can regard $s_{-2\ell}$ as a morphism $\bT(\ell)\otimes\bT(-2\ell)\to\bT(\ell)^*$, hence as a morphism $\bT(\ell)\to\bT(\ell)^*\otimes\bT(2\ell)$, and as such it is a Hermitian duality of weight $-2\ell$ on~$\bT(\ell)$.
\item\label{exem:TU2}
We denote by \index{$Upq$@$\bU(p,q)$}$\bU(p,q)$ the bundle $\cO_{\PP^1}(p-q)$ that we equip with the connection having a simple pole with residue $-p$ at $0$ and $q$ at $\infty$, and no other pole. Its adjoint $\bU(q,p)=\bU(p,q)^*$ is $\cO_{\PP^1}(q-p)$ with connection. We have $\bT(\ell)=\bU(-\ell,\ell)$. The morphism
\[
\bU(p,q)\To{(-1)^p}\bU(p,q)^*\otimes\bT(q-p)
\]
is a Hermitian duality of weight $p-q$ on $\bU(p,q)$. We have $\bU(p,q)^\vee=\bU(-p,-q)$ and $\bU(p,q)^\rc=\bU(-q,-p)$. Note that, forgetting the connection, what distinguishes the various $\bU(p,q)$'s with $p-q$ fixed is the Hermitian duality.
\end{enumerate}
\end{exemple}

Given a Hermitian duality of weight $w$ on $\cT$, we obtain a Hermitian duality $\cS_0$ of weight $0$ on $\cT\otimes\cO_{\PP^1}(-w)$ by tensoring with $\bU(0,w)$:
\[
\cS_0:(\cT\otimes\bU(0,w))\to(\cT\otimes\bU(0,w))^*.
\]
We note that it induces an ordinary Hermitian pairing
\[
H^0(\PP^1,\cT\otimes\bU(0,w))\otimes_{\CC}\ov{H^0(\PP^1,\cT\otimes\bU(0,w))}\to\CC.
\]

\begin{definition}
A twistor structure $\cT$ (\resp an integrable twistor structure $(\cT,\nabla)$) is \index{twistor structure!pure --}\emph{pure of weight $w\in\ZZ$} if $\cT\otimes\cO_{\PP^1}(-w)$ is the trivial bundle. A \index{polarization}\emph{polarization} of $\cT$ (\resp of $(\cT,\nabla)$) is a Hermitian duality of weight~$w$ on $\cT$ (\resp on $(\cT,\nabla)$) such that the Hermitian form induced by $\cS_0$ on $H^0(\PP^1,\cT\otimes\bU(0,w))$ is positive definite.
\end{definition}

If $\cS$ is a polarization of $\cT$ of weight $w$, then $\cS^\rc$ (\resp $\cS^{*-1}$, \resp $\cS^{\vee-1}$) induces a polarization of $\cT^\rc$ (\resp $\cT^*$, \resp $\cT^\vee$) of weight $w$ (\resp $-w$).

\begin{definition}\label{def:realTS}
We say that a pure twistor structure $\cT$ (\resp an integrable pure twistor structure $(\cT,\nabla)$) is \index{twistor structure!polarizable pure --}\emph{polarizable} if it admits a polarization $\cS$ (\resp \index{twistor structure!integrable polarizable pure --}compatible with~$\nabla$).
\end{definition}

\skpt
\begin{remarque}\label{rem:realTS}
\begin{enumerate}
\item\label{rem:realTS1}
The category of polarizable (integrable) pure twistor structures of a given weight is abelian. A pure (integrable) twistor structure which is a direct summand of a polarizable one of the same weight is also polarizable.

\item\label{rem:realTS2}
The notion of \index{twistor structure!mixed --}\emph{mixed twistor structure} and \index{twistor structure!integrable mixed --}\emph{integrable mixed twistor structure} is defined as usual, and leads to an abelian category (we always assume that the pure objects $\gr_\ell^W\cT$ are polarizable). Note that the functors \eqref{eq:defdualsigmarc} preserve these categories and grading with respect to the weight filtration $W$ behaves as usual. In particular, $\gr_\ell^W(\cT^\rc)=(\gr_\ell^W\!\!\cT)^\rc$.

A \emph{real structure} is a morphism $\kappa:(\cT,W)\to(\cT,W)^\rc$ in the category of (inte\-grable) mixed twistor structure such that $\kappa^\rc\circ\kappa=\id$. This leads to the abelian categories of \index{twistor structure!real mixed --}\emph{real mixed twistor structures} and \index{twistor structure!real integrable mixed --}\emph{real integrable mixed twistor structures}.
\end{enumerate}
\end{remarque}

\subsection{Good real and $\kk$-structures}\label{subsec:realkstructMTS}
For a real integrable pure twistor structure $(\cT,\nabla,\kappa)$, the meromorphic flat bundle $\cT(*\{0,\infty\})$ corresponds, through the Riemann-Hilbert correspondence, to a \index{Stokes-filtered local system}\emph{Stokes-filtered local system} $(\cL,\cL_\bbullet^0,\cL_\bbullet^\infty)$: $\cL$ is a local system on $\CC_\hb^*$, that we also regard as a local system on the closed annulus $\wt\PP^1_\hb$ obtained as the real oriented blow-up space of $\PP^1$ at $0$ and $\infty$; $\cL_\bbullet^0$ is the \index{Stokes filtration}\emph{Stokes filtration} of the restriction of $\cL$ to the boundary circle above $0$, and $\cL_\bbullet^\infty$ is similarly defined (\cf \cite{Deligne78} and \cite{Malgrange91}).

The real structure~$\kappa$ induces an isomorphism $(\cL,\cL_\bbullet^0)\simeq\gamma^{-1}(\ov \cL,\ov\cL_\bbullet^\infty)$. We note that $\cL$ is completely determined by its restriction to $\bS$, on which $\gamma=\id$. Therefore, $\kappa$ induces a real structure $\kappa^\nabla:\cL\simeq\ov\cL$, \ie defines an $\RR$-local system $\cL_\RR$ such that $\cL=\CC\otimes_\RR\cL_\RR$.

It is simpler to imagine that $\cL$ is a local system on $\bS$, and that $\bS$ is identified both with the boundary over $0$ and with the boundary over $\infty$ of the oriented real blow-up of $\PP^1$ at $0$ and $\infty$. In such a way, we can replace $\gamma$ with $\id$, so that $\cL_\bbullet^0$ and $\cL_\bbullet^\infty$ are Stokes filtration on the same local system $\cL$, which are conjugate with respect to the real structure on $\cL$ induced by $\kappa^\nabla$. This suggests an enrichment of the notion of an $\RR$-structure by imposing that
\begin{itemize}
\item
the Stokes filtration $\cL^0_\bbullet$ is defined over $\RR$, and thus coincides with its complex conjugate $\cL^\infty_\bbullet$,
\end{itemize}

\begin{definition}[$\kk$-good integrable pure twistor structure]\label{def:goodrealintegrable}
Let $\kk$ be a subfield of~$\RR$. A~\emph{good $\kk$-structure} on an integrable twistor structure $(\cT,\nabla)$ consists of a real structure~$\kappa$ together with the data of a Stokes-filtered $\kk$-local system $(\cL_\kk,\cL_{\kk,\bbullet}^0)$ such that
\begin{enumerate}
\item\label{def:goodrealintegrable1}
$\RR\otimes_\kk\cL_\kk=\cL_\RR$ (the latter defined by $\kappa$),
\item\label{def:goodrealintegrable2}
$\CC\otimes_\kk(\cL_\kk,\cL_{\kk,\bbullet}^0)=(\cL,\cL_\bbullet^0)$.\end{enumerate}
\end{definition}

\begin{definition}[$\kk$-good integrable mixed twistor structure]\label{def:kgoodintMTS}
A \index{twistor structure!$\kk$-good integrable mixed --}good $\kk$-structure on an integrable mixed twistor structure $(\cT,\nabla,W)$ consists of a real structure $\kappa$ compatible with $W$ (Remark \ref{rem:realTS}\eqref{rem:realTS2}), together with the data of a $W$-filtered Stokes-filtered $\kk$-local system $(\cL_\kk,\cL_{\kk,\bbullet}^0,W)$ such that $(\RR\otimes_\kk\cL_\kk,W)=(\cL_\RR,W)$ and $\CC\otimes_\kk(\cL_\kk,\cL_{\kk,\bbullet}^0,W)=(\cL,\cL_\bbullet^0,W)$.
\end{definition}

\begin{proposition}\label{prop:TSTannakienne}
The categories of mixed twistor structures, \resp integrable mixed twistor structures, \resp $\kk$-good integrable mixed twistor structures are neutral Tannakian categories with fibre functor given by taking the fiber at $\hb=1$ (of the attached local systems in the integrable cases).
\end{proposition}

\begin{proof}
We will use the criterion of \cite[Prop.\,1.20]{D-M82}. The tensor product is obviously defined. The unit object $\bun$ is the trivial bundle $\cO_{\PP^1}$ with connection $\rd$ and trivial $W$-filtration jumping at zero, that is, $\bT(0)$. It is then clear that $\End(\bun)=\CC$ (or $\kk$). By construction, for any rank-one object $\cT$, the dual object $\cT^\vee$ satisfies $\cT\otimes\cT^\vee\simeq\bT(0)$.
\end{proof}

\subsection{Compatibility with a polarization}\label{subsec:compRpol}
We have given the definition of a real structure in order to be compatible with that of \cite{Mochizuki11} for general mixed twistor $\cD$-modules. However, in order to mimic the definition of a real (or rational) mixed Hodge structure, one should add a condition of compatibility with some polarization. As we will see, this gives back the notion of \index{TERP structure}TERP structure of \cite{Hertling01} or that of \index{Hodge structure!non-commutative --}non-commutative Hodge structure of \cite{K-K-P08}.

Let $\cT$ (\resp $(\cT,\nabla)$) be a pure (integrable) twistor structure of weight $w$. A real structure $\kappa$ on $\cT$ and a polarization $\cS$ on $\cT$ are said to be compatible (we also say that the polarization is real) if the following diagram commutes
\begin{equation}\label{eq:kappaS}
\begin{array}{c}
\xymatrix@C=1.2cm{
\cT\ar[r]^-\kappa\ar[d]_\cS&\cT^\rc\ar[d]^{\cS^\rc}\\
\cT^*\otimes\bT(-w)\ar[r]^-{\kappa^{*-1}}&\cT^{*\rc}\otimes\bT(-w)
}
\end{array}
\end{equation}
and a similar diagram in the integrable case. We say that a real (integrable) pure twistor structure $(\cT,\kappa)$ (\resp $(\cT,\nabla,\kappa)$) is \emph{polarizable} if it admits a real (integrable) polarization $\cS$.

Let $(\cT,\cS,\kappa)$ be a real polarized twistor structure of weight~$w$. From $\kappa$ and $\cS$ we obtain a nondegenerate $\cO_{\PP^1}$-linear pairing\index{$Qc$@$\cQ$}
\[
\cQ=\cS^\rc\circ\kappa=\kappa^{*-1}\circ\cS:\cT\to\iota^*\cT^\vee\otimes\bT(-w).
\]
This pairing ``commutes'' with $\kappa$, in the sense that $\cQ^\rc\circ\kappa=\kappa^*\circ\cQ$, and is $(-1)^w$\nobreakdash-$\iota$-symmetric in the sense that
\[
\cT\otimes\bT(w)\simeq\cT\otimes\iota^*\bT(-w)^\vee\To{\iota^*\cQ^\vee}\iota^*\cT^\vee,
\]
when tensored with $\bT(-w)$, gives back $(-1)^w\cQ$. In other words, $\cQ$ is a nondegenerate $(-1)^w$-$\iota$-symmetric pairing
\[
\cQ:\cT\otimes\iota^*\cT\to\bT(-w).
\]
Using the notation of Section \ref{subsec:realkstructMTS}, the pairing $\cQ$ induces a nondegenerate $\iota$\nobreakdash-sym\-met\-ric pairing
\[
\cQ^\nabla:(\cL,\cL_\bbullet^0,\cL_\bbullet^\infty)\otimes_\CC\iota^{-1}(\cL,\cL_\bbullet^0,\cL_\bbullet^\infty)\to\CC_{\CC^*_\hb},
\]
where $\CC_{\CC^*_\hb}$ is the constant sheaf on $\CC^*_\hb$ endowed with its trivial Stokes filtration at $0$ and at $\infty$. (Note that $\iota^{-1}\cL=\cL$, but the Stokes filtrations are rotated by $\iota^{-1}$.)

Therefore, a real polarized twistor structure of weight~$w$ can also be given as a triple $(\cT,\cQ,\kappa)$, noticing however that the positivity condition involves $\cS$. In the integrable case, $\cQ$ is compatible with $\nabla$. Moreover, the restriction of $\cQ^\nabla$ to $\cL_\RR\otimes_\RR\iota^{-1}\cL_\RR$ takes values in $(\twopii)^{-w}\RR$, due to the compatibility of $\cQ$ and $\kappa$. As in Section \ref{subsec:realkstructMTS}, this suggests to enrich the notion of a real polarized integrable pure twistor structure by adding the condition:
\begin{itemize}
\item
The pairing $\cQ^\nabla$ induces a nondegenerate pairing
\[
(\cL_\RR,\cL^0_{\RR,\bbullet})\otimes\iota^{-1}(\cL_\RR,\cL^0_{\RR,\bbullet})\to(\twopii)^{-w}\RR_{\bS},
\]
where the latter is equipped with its trivial Stokes filtration.
\end{itemize}
Then,
\begin{itemize}
\item
in Definition \ref{def:goodrealintegrable}, one adds the condition
\begin{enumerate}\setcounter{enumi}{2}
\item\label{def:goodrealintegrable3}
$\cQ^\nabla$ restricted to $(\cL_\kk,\cL_{\kk,\bbullet}^0)$ takes values in the constant sheaf $(\twopii)^{-w}\kk_{\bS}$ (with its trivial Stokes filtration),
\end{enumerate}
and in order to simply define the polarizability condition, one adds in Definition \ref{def:goodrealintegrable} the condition that there exists a polarization $\cQ$ satisfying \eqref{def:goodrealintegrable3} above,
\item
in Definition \ref{def:kgoodintMTS}, one adds the condition that each $\gr_\ell^W$ is a $\kk$-good integrable polarizable pure twistor structure of weight $\ell$, in the enriched sense of Definition~\ref{def:goodrealintegrable}, as defined above.
\end{itemize}

The analogue of Proposition \ref{prop:TSTannakienne} holds for this enriched setting.

\subsection{Twistor structures as triples}\label{subsec:twtriple}
The category \index{$RTriplespt$@$\RTriples(\pt)$}$\RTriples(\pt)$ gives another presentation of twistor structures, which is better suited to the extension to higher dimension and arbitrary singularities. We represent a vector bundle $\cT$ on $\PP^1$ by a \emph{triple} $\cT=(\cM',\cM'',C_{\bS})$, where $\cM',\cM''$ are vector bundles on $\CC_\hb$ and \index{$CS$@$C_{\bS}$}$C_{\bS}$ is a nondegenerate pairing $\cM'_{|\bS}\otimes_{\cO_{\bS}}\sigma^*\ov{\cM''}_{|\bS}\to\cO_{\bS}$. We regard the pairing as giving rise to a vector bundle by gluing $\cM^{\prime\vee}$ with $\sigma^*\ov{\cM''}$. Morphisms are pairs $(\lambda',\lambda'')$ compatible with the pairings, covariant with respect to~$\lambda''$ and contravariant with respect to $\lambda'$. The notation of \eqref{eq:defdualsigmarc} reads as follows:
\begin{equation}\label{eq:defdualsigmarctriples}
\begin{cases}
\bbullet\ \cT^\vee=(\cM^{\prime\vee},\cM^{\prime\prime\vee},C_{\bS}^{\vee-1}),\\
\bbullet\ \cT^*=(\cM'',\cM',C_{\bS}^*),\\
\bbullet\ \cT^\rc=\iota^*(\cT^\vee)^*=\iota^*(\cT^*)^\vee,
\end{cases}
\end{equation}
where we have set
\[
C_{\bS}^*(m'',\sigma^*\ov{m'}):=\sigma^*\ov{C_{\bS}(m',\sigma^*\ov{m''})},
\]
and we define $C_{\bS}^{\vee-1}$ so that, if $C_{\bS}$ is regarded as an isomorphism $\cM^{\prime\vee}_{|\bS}\isom\sigma^*\ov{\cM''_{|\bS}}$, then $C_{\bS}^{\vee-1}$ is the dual inverse isomorphism $\cM'_{|\bS}\isom\sigma^*\ov{\cM_{|\bS}^{\prime\prime\vee}}$. The category \index{$RTriplesptint$@$\RdTriples(\pt)$}$\RdTriples(\pt)$ of integrable triples is defined similarly, asking that the connection~$\nabla$ on $\cM',\cM''$ has a pole of order at most two at $\hb=0$ and no other pole, and~$C_{\bS}$ is compatible with the connections. Let us set $\cL'=\cM'(*0)^\nabla$ \resp $\cL''=\cM''(*0)^\nabla$, and let us use the notation $\cL$ for $\cT$ as in Section \ref{sec:twistorvectorbundles}. Then the correspondence is $C_{\bS}^\nabla:\cL^{\prime\vee}\isom\iota^{-1}\ov{\cL''}=\cL$.

\begin{remarque}\label{rem:algebraization}
Integrable triples have a more algebraic formulation: by using Deligne's meromorphic extension at infinity for $(\cM',\nabla)$ and $(\cM'',\nabla)$, we find free $\CC[\hb]$-modules with connection $(M',\nabla)$ and $(M'',\nabla)$ which give back the previous ones by tensoring with~$\cO_{\CC_\hb}$. Moreover, $\nabla$ on $M',M''$ has a regular singularity at infinity.
\end{remarque}

\skpt
\begin{exemple}[The Tate object]\label{exem:Tatetriple}
\begin{enumerate}
\item\label{exem:Tatetriple1}
The \index{$Tl$@$\bT(\ell)$}\index{Tate object}\emph{Tate object} (\cf \cite[\S2.1.8]{Mochizuki11}) is defined as the triple
\[
\bT(\ell):=(\hb^{-\ell}\cO_{\CC_\hb},\hb^\ell\cO_{\CC_\hb},C_0),
\]
where $C_0$ is induced from the natural sesquilinear pairing (\cf Notation  \ref{nota:general})
\[
C_0:\cO_{\bS}(*0)\otimes\sigma^*\ov{\cO_{\bS}(*0)}\to\cO_{\bS},\quad f'\otimes\sigma^*\ov{f''}\mto f'\cdot\sigma^*\ov{f''},
\]
and the connection $\nabla$ is induced from the standard differential $\rd$ on $\cO_{\bS}(*0)$. We keep the isomorphisms \eqref{eq:defdualsigmarcTate} as they are. We note that the local systems $\cL',\cL''$ for $\bT(\ell)$ are independent of $\ell$ and are canonically identified to $\CC_{\CC_\hb^*}$, and $C_{\bS}^\nabla$ is equal to $\id$. From the identification $\cL''=\iota^{-1}\ov\cL$ and the $\QQ$-structure $\cL_{\bS,\QQ}=(\twopii)^\ell\QQ_{\bS}$ we obtain
\[
\cL''_{\bS,\QQ}=(\twopii)^{-\ell}\QQ_{\bS}=\cL'_{\bS,\QQ}.
\]

\item\label{exem:Tatetriple2}
We use the triple \index{$Upq$@$\bU(p,q)$}$\bU(0,w):=(\cO_{\CC_\hb},\hb^w\cO_{\CC_\hb},C_0)$, whose associated vector bundle is $\cO_{\PP^1}(-w)$, to reduce to weight zero and define the polarizability, by means of the identification $\sigma^*\ov{\bU(0,w)}=\bU(-w,0)$ and the natural pairing
\[
\bU(0,w)\otimes\bU(-w,0)\To{\id}\bU(-w,w)=\bT(w),
\]
that we regard as a Hermitian duality of weight $-w$. This pairing is compatible with the connections induced by $\rd$.
\end{enumerate}
\end{exemple}

The full subcategory \index{$MTSC$@$\MTS(\CC)$}$\MTS(\CC)$ (\resp \index{$MTSCint$@$\MTS^\intt(\CC)$}$\MTS^\intt(\CC)$) of the category $\WRTriples(\pt)$ (\resp $\WRdTriples(\pt)$) of $W$-filtered triples consists of objects which give rise to mixed twistor structures (\resp integrable mixed twistor structures) by the gluing procedure above. The categories \index{$MTSR$@$\MTS(\RR)$}$\MTS(\RR)$, \index{$MTSRint$@$\MTS^\intt(\RR)$}$\MTS^\intt(\RR)$, \index{$MTSkint$@$\MTS_\good^\intt(\kk)$}$\MTS^\intt_\good(\kk)$, are defined accordingly. As already remarked, it is enough to define the $\kk$-structure on the Stokes-filtered local system $(\cL'',\cL''_\bbullet)$ attached to $(\cM'',\rd)$, since it coincides with that attached to~$(\cL',\cL'_\bbullet)$. Lastly, we note that there are forgetful faithful functors
\[
\MTS^\intt_\good(\kk)\mto\MTS^\intt(\RR)\mto\begin{matrix}\MTS(\RR)\\\text{or}\\[3pt]\MTS^\intt(\CC)\end{matrix}\mto\MTS(\CC).
\]

\begin{remarque}[On the notation]
In the next chapter, we will consider the notion of a mixed twistor $\cD$-module on a complex manifold, following \cite{Mochizuki11} and, with the notation introduced there, the previous chain of functors reads
\[
\MTM^\intt_\good(\pt,\kk)\mto\MTM^\intt(\pt,\RR)\mto\begin{matrix}\MTM(\pt,\RR)\\\text{or}\\[3pt]\MTM^\intt(\pt)\end{matrix}\mto\MTM(\pt).
\]
\end{remarque}

We obtain as in Proposition \ref{prop:TSTannakienne}:

\begin{proposition}\label{prop:MTSTannakienne}
The above categories are neutral Tannakian categories with fibre functor given by taking the fiber of $\cM''$ at $\hb=1$ (\resp that of the attached local system in the integrable cases), and the forgetful functors preserve these structures.\qed
\end{proposition}

\subsection{Integrable twistor structures as $\iota$-triples}\label{subsec:itriplesint}

For integrable triples $(\cM',\cM'',C_{\bS})$, the pairing $C_{\bS}$ is completely determined by its restriction \index{$CSnabla$@$C_{\bS}^\nabla$}$C_{\bS}^\nabla:\cL'_{\bS}\otimes_\CC\sigma^{-1}\ov{\cL''_{\bS}}\to\CC_{\bS}$. Since $\sigma_{|\bS}=\iota_{|\bS}$, giving $C_{\bS}^\nabla$ is equivalent to giving a nondegenerate pairing \index{$Cinabla$@$\iC^\nabla$}$\iC^\nabla:\cL'\otimes\iota^{-1}\ov{\cL''}\to\CC_{\CC_\hb^*}$ defined on the whole~$\CC^*_\hb$. We therefore define the category \index{$RTriplesptinti$@$\RdiTriples(\pt)$}$\RdiTriples(\pt)$ as the category consisting of triples $((\cM',\nabla),(\cM'',\nabla),\iC^\nabla)$, and we have an equivalence of categories
\[
\RdiTriples(\pt)\overset\sim\mto\RdTriples(\pt)
\]
obtained by sending $\iC^\nabla$ first to its restriction to $\bS$ and then to the associated $C_{\bS}$.

\skpt
\begin{remarque}\label{rem:iCnabla}
\begin{enumerate}
\item\label{rem:iCnabla1}
Giving $\iC^\nabla$ as above is equivalent to giving a sesquilinear pairing\index{$Ci$@$\iC$}
\[
\iC:\cM'_{|\CC^*_\hb}\otimes_\CC\ov{\iota^*\cM''_{|\CC^*_\hb}}\to\cA_{\CC^*_\hb},
\]
where $\cA_{\CC^*_\hb}$ denote the sheaf of real analytic functions on $\CC^*_\hb$, which satisfies
\[
\partial_z\,\iC(m',\ov{m''})=\iC(\nabla_{\partial_\hb} m',\ov{m''})\quad\text{and}\quad\partial_{\ov\hb}\,\iC(m',\ov{m''})=\iC(m',\ov{\nabla_{\partial_\hb} m''}).
\]

\item\label{rem:iCnabla2}
As in Remark \ref{rem:algebraization}, we can equivalently represent the objects of $\RdTriples(\pt)$ by triples $(M',M'',C^\nabla)$, where $M',M''$ are free $\CC[\hb]$-module with connection having a pole of order two at the origin and a regular singularity at infinity.
\end{enumerate}
\end{remarque}

Given the object $(M',M'',\iC^\nabla)$ of $\WRdiTriples(\pt)$, we can associate an object $(\cM',\cM'',C_{\bS})$ of $\WRdTriples(\pt)$, where
\[
C_{\bS}:\cM'_{|\bS}\otimes_{\cO_{\bS}}\sigma^*\ov{\cM''_{|\bS}}\to\cO_{\bS}
\]
is the sesquilinear pairing obtained by
\begin{itemize}
\item
extending $\iC^\nabla$ as a sesquilinear pairing
\[
\iC:\cM'_{|\CC^*_\hb}\otimes\iota^*\ov{\cM''_{|\CC^*_\hb}}\to\cA_{\CC^*_\hb}
\]
(\cf Remark \ref{rem:iCnabla}\eqref{rem:iCnabla1}),
\item
sheaf-theoretic restricting the latter to $\bS$,
\item
identifying $\iota^*$ with $\sigma^*$ on $\bS$ and regarding the restriction $\iC_{|\bS}$ as the desired pairing $C_{\bS}$.
\end{itemize}

\begin{exemple}[The Tate object]\label{exem:Tateitriple}
The \index{Tate object}Tate object is the triple\index{$Tli$@$\ibT(\ell)$}
\[
\ibT(\ell):=(\hb^{-\ell}\cO_{\CC_\hb},\hb^\ell\cO_{\CC_\hb},\iC_0),
\]
where $\iC_0$ is induced from the natural sesquilinear pairing
\[
\iC_0:\cO_{\CC_\hb^*}(*0)\otimes\iota^*\ov{\cO_{\CC_\hb^*}(*0)}\to\cA_{\CC_\hb^*},\quad f'\otimes\iota^*\ov{f''}\mto f'\cdot\iota^*\ov{f''},
\]
and the connection $\nabla$ is induced from the standard differential $\rd$ on $\cO_{\CC_\hb^*}(*0)$. One can define similarly $\ibU(p,q)$.
\end{exemple}

One can thus define the following abelian categories with their forgetful functors\index{$IMTSkint$@$\iMTS_\good^\intt(\kk)$}\index{$IMTSRint$@$\iMTS^\intt(\RR)$}\index{$IMTSCint$@$\iMTS^\intt(\CC)$}
\[
\iMTS^\intt_\good(\kk)\mto\iMTS^\intt(\RR)\mto\iMTS^\intt(\CC)
\]
and the functors in the corresponding $\MTS^\intt$ categories are equivalences of neutral Tannakian categories.

\begin{definition}\label{def:ncMHS}
We call $\iMTS^\intt(\CC)$ (\resp $\iMTS^\intt_\good(\kk)$) the category of \emph{non-com\-mut\-ative mixed Hodge structures} (\resp that of \emph{non-commutative mixed $\kk$-Hodge structures}), with the caveat that it does not include the compatibility with polarization as considered in Section \ref{subsec:compRpol}, hence is a little weaker, in the pure case, than the notions introduced in \cite{Hertling01} and \cite{K-K-P08}.
\end{definition}

\section{Presentation of the results in Chapter \ref{part:1}}\label{sec:introI}

We use the notations \ref{nota:general}--\ref{nota:XiDR}. Let $X$ be a complex manifold. We~denote by~\index{$RX$@$\cR_\cX$}$\cR_\cX$ the sheaf on $\cX$ of $\hb$\nobreakdash-differential operators (it~is locally written as $\cO_\cX\langle\partiall_{x_i}\rangle_{i=1,\dots,d_X}$, where $\partiall_{x_i}$ acts on~$\cO_\cX$ as $\hb\partial/\partial x_i$, \cf\hbox{\cite[\S1.1]{Bibi01c}}, \cite[\S14.1]{Mochizuki07}) and by $\cR_{X\times\bS}$ (\resp $\cR_{\cX^\circ}$) its sheaf-theoretic restric\-tion to $X\times\nobreak\bS$ (\resp to~$\cX^\circ$). We can regard $\cR_\cX$ as a subsheaf of $\cD_\cX$, by sending~$\partiall_{x_i}$ to $\hb\partial/\partial x_i$. Under this identification, $\cR_{\cX^\circ}$ is the sheaf $\cD_{\cX^\circ/\CC^*_\hb}$ of differen\-tial oper\-ators relative to the projection $\cX^\circ\to\CC^*_\hb$, since $\hb$ is invertible on $\cX^\circ$, and similarly $\cR_{X\times\bS}$ is the sheaf-theoretic restriction of the latter to $X\times\bS$. Given an $\cR_\cX$\nobreakdash-module~$\cM$, we denote by~$\cM_{\bS}$ its sheaf-theoretic restriction to $X\times\bS$ and by~\index{$Mccirc$@$\cM^\circ$}$\cM^\circ$ its restriction to~$\cX^\circ$.

We denote by \index{$RXint$@$\cR_\cX^\intt$}$\cR_\cX^\intt$ the sheaf $\cR_\cX\langle\hb^2\partial_\hb\rangle$ (with the standard commutation relations, so that it can be considered as a subsheaf of $\cD_\cX$ by sending $\partiall_{x_i}$ to $\hb\partial_{x_i}$). We will often identify $\cR_{\cX^\circ}^\intt$ with $\cD_{\cX^\circ}$, since $\hb$ is invertible on $\cX^\circ$, and $\cR_\cX^\intt[1/\hb]$ with $\cD_\cX[1/\hb]$. If we denote by $\cD_{X\times\bS}$ the sheaf of differential operators on $\cO_{X\times\bS}$ (sheaf-theoretic restriction $\cO_{X\times\CC^*_\hb\mid X\times\bS}$, equivalently sheaf of real-analytic functions on $X\times\bS$ which are holomorphic with respect to $X$), then we can also identify $\cR_{X\times\bS}^\intt$ with $\cD_{X\times\bS}$: If we write $\hb=\rme^{i\vartheta}$ on $\bS$, then we replace $\hb\partial_\hb$ with $-\sfi\partial_\vartheta$. In this article,\footnote{\label{footnote:nonintegrable}When considering possibly non-integrable triples, the convention for $\Db_{X\times\bS/\bS}$ is not the same, \cf\cite[\S0.5]{Bibi01c}.} we denote by \index{$DbXS$@$\Db_{X\times\bS/\bS}$}$\Db_{X\times\bS/\bS}$ the sheaf of relative distributions which are $C^k$ with respect to~$\bS$ for all $k\in\NN$. Then $\Db_{X\times\bS/\bS}$ is naturally a left $\cD_{X\times\bS}$-module (\cf Section \ref{subsec:OmegacontinuousDb}).

The category \index{$RTriplesXint$@$\RdTriples(X)$}$\RdTriples(X)$ is that consisting of triples $\cT=(\cM',\cM'',C_{\bS})$, where $\cM',\cM''$ are (left) $\cR_\cX^\intt$-modules and $C_{\bS}:\cM'_{\bS}\otimes_{\cO_{\bS}}\nobreak\sigma^*\ov{\cM''_{\bS}}\to\Db_{X\times\bS/\bS}$ is a $\cD_{X\times\bS}\otimes_{\cO_{\bS}}\sigma^*\ov{\cD_{X\times\bS}}$-linear morphism (see \cite[Chap.\,7]{Bibi01c}, \cite[\S2.1.5]{Mochizuki11} and Section \ref{subsec:remindersesqui} for more details). A~morphism $\lambda:\cT_1\to\cT_2$ is a pair of $\cR_\cX^\intt$-linear morphisms
\[
\lambda':\cM_2'\to\cM_1',\quad \lambda'':\cM''_1\to\cM''_2
\]
such that, for any local sections $m'_2,m''_1$,
\[
C_{1,\bS}(\lambda'(m'_2),\sigma^*\ov{m''_1})=C_{2,\bS}(m'_2,\sigma^*\ov{\lambda''(m''_1)}).
\]

\begin{remarque}[on the terminology]
We say that an \index{RXmodule@$\cR_\cX$-module!integrable --}$\cR_\cX$-module $\cM$ is \emph{integrable} (\cf\cite[Chap.\,7]{Bibi01c}) if it admits a compatible action of $\hb^2\partial_\hb$, that is, if the $\cR_\cX$-action can be extended to an $\cR_\cX^\intt$-action. However, such an extension is not unique: two actions of $\hb^2\partial_\hb$ are compatible if and only if they differ by an $\cR_\cX$-linear morphism $\cM\to\cM$. (See however Remark \ref{rem:intuniquepure}.) We will also use the terminology \emph{integrable} for an $\cR_\cX^\intt$-module, hoping that this will not produce any confusion when the context is clear. In other words, we will use the terminology \emph{integrable} either for the category $\Mod(\cR_\cX^\intt)$ or for its essential image in $\Mod(\cR_\cX)$ by the natural forgetful functor. We will use a similar terminology for the category $\RdTriples(X)$ and for its essential image in $\RTriples(X)$.
\end{remarque}

The category $\RdTriples(X)$ is abelian. This category, and the category \index{$WRTriplesXint$@$\WRdTriples(X)$}$\WRdTriples(X)$ of objects of $\RdTriples(X)$ equipped with a finite increasing filtration $W_\bbullet$, are endowed with various functors.

\skpt
\begin{definition}[of the functors]\label{def:functorsRTriples}
\begin{enumeratea}
\item\label{def:TateRTriples}
The Tate twist $\cT\otimes\bT(\ell)$ with $\ell\in\ZZ$ (\cf Example \ref{exem:Tatetriple}, \cite[\S1.6.a]{Bibi01c}, and also \cite[\S3.3]{Mochizuki07}, \cite[\S2.1.8]{Mochizuki11}), and the action of $\hb^2\partial_\hb$ is the natural one on the tensor product (\cf\cite[\S7.1.c]{Bibi01c}). It shifts by $-2\ell$ the $W$-filtration.
\item\label{def:adjunctionRTriples}
The adjunction $\cT^*=(\cM'',\cM',C_{\bS}^*)$ with $C_{\bS}^*(m'',\sigma^*\ov{m'}):=\sigma^*\ov{C_{\bS}(m',\sigma^*\ov{m''})}$.
\item\label{def:pushforwardRTriples}
The pushforward functors $\map_\dag^k$ by a morphism $\map:X\to Y$ between complex manifolds (\cf\cite[\S1.6.d\,\&\,\S7.1.d]{Bibi01c}).
\item\label{def:pullbackRTriples}
The pullback functor by a \emph{smooth} morphism $\mapsm:X\to Y$ (\cf \ref{def:functorsMTM}\eqref{enum:pullbackMTM} below for the definition).
\item
The external tensor product $\hbboxtimes$ (\cf Definition \ref{def:externalproduct} below for a similar definition).
\end{enumeratea}
\end{definition}

Given a holomorphic function $\fun:X\to\CC$ with associated divisor $H$ (\cf Notation \ref{nota:maps}), we have the notion of \index{RXmodule@$\cR_\cX$-module!strictly specializable --}strictly specializable $\cR_\cX$-module along $\fun=0$, obtained by considering strict specializability along $t=0$ of $\cM_\fun$ in terms of the existence of a Kashiwara-Malgrange $V$-filtration (\cf\cite[\S3.3]{Bibi01c}). The \index{localization!``stupid''}``stupid'' localization functor $\cM_\fun(*H_\fun)$ can then be refined (\hbox{in order} to keep coherence) as a \index{localization}localization functor denoted $[*H_\fun]$ (\cf\cite[Chap.\,3]{Mochizuki11}). There is also a \index{localization!dual}dual-localization functor $[!H_\fun]$. One then says that~$\cM$ is \index{RXmodule@$\cR_\cX$-module!localizable along $H$ --}\emph{localizable along $H$} if there exist $\cR_\cX$-modules \index{$McHstar$@$\cM[\star H]$}$\cM[\star H]$ (with $\star\!=\!*\text{ or }!$), such that \hbox{$\cM_\fun[\star H_\fun]=(\cM[\star H])_\fun$}. Integrability is preserved by these functors, and they can be extended to the subcategory of $\RdTriples(X)$ whose underlying $\cR$\nobreakdash-mod\-ules are strictly specializable \resp localizable (\cf\cite[Prop.\,3.2.1]{Mochizuki11}), and they preserve the subcategory $\RdTriples(X)$ (\cf\loccit and Proposition \ref{prop:speint} below).

The category \index{$MTMX$@$\MTM(X)$}$\MTM(X)$ \resp \index{$MTMXint$@$\MTM^\intt(X)$}$\MTM^\intt(X)$ of \index{mixed twistor $\cD$-module!integrable --}\emph{(integrable) mixed twistor $\cD$-modules} on $X$, as constructed in \cite{Mochizuki11}, is a full subcategory of the abelian category $\WRTriples(X)$ (\resp $\WRdTriples(X)$). It is obtained by adding an admissibility condition to the category \index{$MTWX$@$\MTW(X)$}$\MTW(X)$ (\resp \index{$MTWXint$@$\MTW^\intt(X)$}$\MTW^\intt(X)$), \cf\loccit

\begin{remarque}[Stability by direct summand]\label{rem:directsummand}
The category $\MTM^\intt(X)$ is stable by direct summand in $\WRdTriples(X)$. Indeed, the category of polarizable pure Hodge modules is stable by direct summand in $\RTriples(X)$, as follows from \cite[Prop.\,4.1.5, 4.2.5]{Bibi01c}, \cite{Bibi06b} and \cite[\S17.1]{Mochizuki08}. It is obvious to adapt these results by adding the integrability property. It follows that the category \index{$MTWX$@$\MTW(X)$}$\MTW(X)$ is stable by direct summand in \index{$WRTriplesX$@$\WRTriples(X)$}$\WRTriples(X)$ and similarly \index{$MTWXint$@$\MTW^\intt(X)$}$\MTW^\intt(X)$ is stable by direct summand in $\WRdTriples(X)$. It~remains to check that the admissible specializability property (\ie the existence of the relative monodromy filtration, \cf \cite[Def.\,4.4.4]{Mochizuki11}) is stable by direct summand. This follows from \cite[Cor.\,2.10]{S-Z85}, since the conditions in \cite[Lem.\,2.8]{S-Z85} are stable by direct summand.
\end{remarque}

The category $\MTM^\intt(X)$ is equipped with the functors corresponding to those in Definition \ref{def:functorsRTriples} in the following~way.

\skpt
\begin{definition}[Definition of the functors]\label{def:functorsMTM}
\begin{enumeratea}
\item\label{enum:TateMTM}
The Tate twist sends $\MTM^\intt(X)$ into itself, but shifts the weight filtration by~$-2\ell$.
\item\label{enum:pushforwardMTM}
We assume that $\map:X\to Y$ is \emph{projective}. The pushforward functors \hbox{$\map_\dag^k:\WRdTriples(X)\mto\WRdTriples(Y)$} (\cf \cite[\S1.6.2]{Bibi01c}) send the subcategory $\MTM^\intt(X)$ to $\MTM^\intt(Y)$ (\cf\cite[Prop.\,7.2.7\,\,\&\,\,7.2.12]{Mochizuki11}).

\item\label{enum:pullbackMTM}
We claim that the pullback functor
\[
\MTM^\intt(Y)\mto\MTM^\intt(X)
\]
by a \emph{smooth} morphism $\mapsm:X\to Y$ is well-defined.

Firstly, the pullback functor by a \emph{smooth} morphism $\mapsm:X\to Y$ is well-defined between the categories of $\cR^\intt$-triples on $Y$ and $X$. Indeed, the pullback on the $\cR^\intt$-module part of a triple is defined as for $\cD$-modules. For the sesquilinear pairing part, note that, due to smoothness of~$g$, integration on fibres of $g$ sends $C^\infty$-forms of maximal degree on $X$ with compact support to $C^\infty$-forms of maximal degree on~$Y$ with compact support. This enables one to define the pullback of a distribution on $Y\times\bS$, or of a $\bS$-relative distribution on $Y\times\bS$ (\cf\cite[\S0.5]{Bibi01c}). We get functors
\[
\mapsm^+:\WRdTriples(Y)\mto\WRdTriples(X)
\]
between the categories of integrable triples with finite filtration. Now, the full subcategory $\MTM^\intt(X)$ of $\WRdTriples(X)$ is defined by local properties, hence checking that~$\mapsm^+$ sends $\MTM^\intt(Y)$ to $\MTM^\intt(X)$ is a local question on $X$ and $Y$, so that we can assume that $\mapsm$ is a product. At this point we can use that the pullback functor can be expressed as an external tensor product, and the assertion follows from \cite[Prop.\,11.4.6]{Mochizuki11}. A similar arguments holds in the integrable case.
\item\label{enum:externalproductMTM}
For the external tensor product, \cf \cite[Prop.\,11.4.1]{Mochizuki11}.

\item\label{enum:dualityMTM}
The duality functor is constructed in \cite[Th.\,13.3.1]{Mochizuki11}.
\item\label{enum:locMTM}
Stability by the localization and dual-localization functors $[\star H]$ is shown in \cite[\S11.2]{Mochizuki11}.
\end{enumeratea}
\end{definition}

We introduce below (\cf Section \ref{subsec:stdfunctorsRdiT}) the category $\RdiTriples(X)$, by using the involution $\iota$ instead of $\sigma$. We will consider the sheaf $\Db_{\cX^\circ/\CC^*_\hb}$ of distributions on $\cX^\circ$ which are $C^\infty$ with respect to $\hb$ (\cf Section \ref{subsec:OmegacontinuousDb}). An object of \index{$RTriplesXinti$@$\RdiTriples(X)$}$\RdiTriples(X)$ is a triple $(\cM',\cM'',\iC)$, where $\cM',\cM''$ are as above and $\iC$ is a $\cD_{\cX^\circ}\otimes_\CC\iota^*\cD_{\ov\cX^\circ}$-linear pairing $\cM^{\prime\circ}\otimes_\CC\iota^*\ov{\cM^{\prime\prime\circ}}\to\Db_{\cX^\circ/\CC_\hb^*}$.

This category is equipped with functors $\map_\dag^k$ and $\mapsm^+$ as above. Since a distribution in $\Db_{\cX^\circ/\CC_\hb^*}$ restricts to a distribution in $\Db_{X\times\bS/\bS}$, and since $\sigma$ and $\iota$ coincide on~$\bS$, we have a natural functor $\RdiTriples(X)\mto\RdTriples(X)$ compatible with the above functors (\cf Section \ref{subsec:stdfunctorsRdiT}).

\begin{definition}[of $\iMTM^\intt(X)$]\label{def:iotamtm}
The category \index{$IMTMXint$@$\iMTM^\intt(X)$}$\iMTM^\intt(X)$ is the full subcategory of $\WRdiTriples(X)$ consisting of objects whose image by the natural functor
\[
\WRdiTriples(X)\mto\WRdTriples(X)
\]
is an object of $\MTM^\intt(X)$.
\end{definition}

Our main result in this chapter is the following theorem.

\begin{theoreme}[Equivalence between $\iMTM^\intt(X)$ and $\MTM^\intt(X)$]\label{th:equivalenceintegrable}\label{TH:EQUIVALENCEINTEGRABLE}
The restriction functor $\iMTM^\intt(X)\mto\MTM^\intt(X)$ is an equivalence, compatible with pushforward by a projective morphism, pullback by a smooth morphism, external tensor product, duality, and localization along a divisor $H$.
\end{theoreme}

Proving essential surjectivity amounts to showing the existence of a unique lifting of~$C_{\bS}$ to~$\iC$. The full faithfulness then means that any integrable morphism compatible with $C_{\bS}$ is also compatible with $\iC$. Note also that the duality functor on $\iMTM^\intt(X)$ still needs to be defined, as it is a~priori not defined on the whole category $\RdiTriples(X)$.

We will prove (\cf Section \ref{sec:pfequivalenceintegrable}) by induction on $d$ the following statements:
\begin{enumeratea}
\item\label{enum:equiva}
For each $\cT=((\cM',\cM'',C_{\bS}),W_\bbullet)\in\MTM^\intt(X)$ supported in dimension $\leq d$, there exists an integrable lifting $\iC$ of $C_{\bS}$ (which is unique, according to Lemma \ref{lem:iCSdec}).
\item\label{enum:equivb}
Each morphism $\lambda:\cT_1\to\cT_2$ in $\MTM^\intt(X)$ between objects of support $\leq d$ is compatible with $\iC$ constructed in \eqref{enum:equiva}.
\end{enumeratea}

By the uniqueness proved in Lemma \ref{lem:iCSdec}, it is enough to solve \eqref{enum:equiva} locally, and then~\eqref{enum:equivb} is also a local question. The case $d=0$ is given by Lemma \ref{lem:equivsmooth}.

The proof will use the local description of an object of $\MTM^\intt(X)$ with support of dimension $d$ by gluing, in a functorial way, objects supported in dimension~$<d$ with an object defined from a good admissible variation of mixed twistor structure (\cf\cite[Prop.11.1.1\,\&\,Eq.\,(11.2)]{Mochizuki11}).

\section{The category \texorpdfstring{$\iMTM^\intt(X)$}{iMTMX}}

In this section we introduce the notion of a $\iota$-sesquilinear pairing, which is not the one used in the definition of $\RdTriples(X)$. We will make precise the functors on the corresponding category $\RdiTriples(X)$.

\Subsection{Short reminder on strictly specializable \texorpdfstring{$\cR$}{R}-module along \texorpdfstring{$\fun=0$}{f0}}\label{subsec:shortreminder}

We continue using Notation \ref{nota:maps}. Let $\fun:X\to\CC$ be a holomorphic function and let $\cM$ be a coherent $\cR_\cX$-module or a coherent $\cR_\cX(*H)$-module which is \index{RXmodule@$\cR_\cX$-module!strictly specializable --}strictly specializable along $(\fun)$, \ie $\cM_\fun$ is so along $(t)$ (\cf\cite[\S3.4.a]{Bibi01c}), and let us denote by $V_\bbullet \cM_\fun$ the $V$-filtration, indexed by real numbers (notice however that the notation of indices in the present article is shifted by one with respect to the notation in \loccit, since we use the operator $t\partiall_t$ instead of $\partiall_tt$). This filtration is a~priori defined only locally with respect to $\hb$, but for $u=(a,\alpha)\in\RR\times\CC$, the nearby cycles $\psi_{\fun,u}\cM$, defined from the local $V$-filtrations (\cf\loccit and \cite[\S14.2.5]{Mochizuki07}) can be glued with respect to $\hb$. We say that $\cM$ is \index{RXmodule@$\cR_\cX$-module!strictly $\RR$-specializable --}\emph{strictly $\RR$-specializable} along $\fun=0$ if the indices~$u$ such that $\psi_{\fun,u}\cM\neq0$ belong to $\RR\times\{0\}$, so that we write this module $\psi_{\fun,a}\cM$ and, with respect to the notation in \loccit, the $V$-filtration is globally defined, as already noticed in \cite[Rem.\,3.3.6(2)]{Bibi01c}. The $V$-filtration satisfies $V_{\alpha+k}\cM_\fun=t^{-k}V_\alpha\cM_\fun$ for $\alpha\in[0,1)$ and $k\in-\NN$ (\resp $k\in\ZZ$ in the $(*H)$~case).

Let $\lambda:\cM_1\to\cM_2$ be a strictly specializable morphism between strictly $\RR$\nobreakdash-specializable $\cR_\cX$-modules or $\cR_\cX(*H)$-modules (\cf\loccit). Then $\lambda$ is $V$-strict, and $\ker\lambda$, $\image\lambda$, $\coker\lambda$ are strictly $\RR$-specializable with the expected behaviour of the $V$-filtration (\cf\cite[Lem.\,3.3.10]{Bibi01c}).

\subsection{Some properties of integrable \texorpdfstring{$\cR_\cX$}{RcX}-modules}

\begin{proposition}\label{prop:speint}
Let $\cM$ be an \index{RXmodule@$\cR_\cX$-module!integrable holonomic --}integrable holonomic $\cR_\cX$-module with characteristic variety contained in $\Lambda\times\CC_\hb$ for some Lagrangian subvariety $\Lambda\subset T^*X$ and let $\cM(*\hb):=\cR_\cX^\intt[1/\hb]\otimes_{\cR^\intt_\cX}\cM$ be the associated localized module, that we also regard as a coherent $\cD_{\cX/\CC_\hb}[1/\hb]$-module endowed with a compatible structure of $\cD_\cX[1/\hb]$-module.
\begin{enumerate}
\item\label{prop:speint1}
Then $\cM(*\hb)$ is a holonomic $\cD_\cX$-module and its characteristic variety as a $\cD_\cX[1/\hb]$-module is contained in $\Lambda\times(T^*_{\CC^*_\hb}\CC^*_\hb\cup T^*_\Sigma\CC^*_\hb)$ for some discrete subset $\Sigma\subset\CC_\hb$.
\item\label{prop:speint2}
If moreover $\cM$ is strictly specializable along some holomorphic function $\fun$, then
\begin{enumerate}
\item\label{prop:speint2a}
$\cM$ is strictly $\RR$-specializable along $\fun$; in particular, \cite[Lem.\,7.3.8]{Bibi01c} applies;
\item\label{prop:speint2c}
the $V$-filtration of $\cM_\fun^\circ:=\cM_{\fun|\cX_\fun^\circ}$ along $H_\fun$ as an $\cR_{\cX_\fun^\circ}$-module coincides with that as a $\cD_{\cX_\fun^\circ}$-module; in particular, for $a\in\RR$, $\psi_{\fun,a}\cM^\circ$ is equal to the nearby cycle module attached to the holonomic $\cD_{\cX^\circ}$\nobreakdash-module~$\cM^\circ$.
\end{enumerate}
\end{enumerate}
\end{proposition}

\skpt
\begin{proof}
\begin{enumerate}
\item
Note that $\cM(*\hb)$ is a coherent $\cD_{\cX/\CC_\hb}[1/\hb]$-module with characteristic variety contained in $\Lambda\times\CC^*_\hb\subset T^*(\cX^\circ/\CC^*_\hb)$. If $\cM$ is integrable, then $\cM(*\hb)$ is a $\cD_\cX[1/\hb]$\nobreakdash-module. As such it has the presentation
\begin{multline*}
\cD_\cX[1/\hb]\otimes_{\cD_{\cX/\CC_\hb}[1/\hb]}\cM(*\hb)\To{\partial_\hb\otimes\id-\id\otimes\partial_\hb}\cD_\cX[1/\hb]\otimes_{\cD_{\cX/\CC_\hb}[1/\hb]}\cM(*\hb)\\
\to\cM(*\hb)\to0,
\end{multline*}
where the first map can simply be written as
\begin{align*}
\cM(*\hb)[\partial_\hb]&\to\cM(*\hb)[\partial_\hb]\\
\sum_{k\geq0} m_k\partial_\hb^k&\mto\sum_{k\geq0}(m_{k-1}-\partial_\hb m_k)\partial_\hb^k,
\end{align*}
and induces an exact sequence
\[
0\to\cM(*\hb)[\partial_\hb]\to\cM(*\hb)[\partial_\hb]\to\cM(*\hb)\to0.
\]
Clearly, $\cM(*\hb)[\partial_\hb]$ is $\cD_\cX[1/\hb]$-subholonomic with characteristic variety contained in $\Lambda\times T^*\CC^*_\hb$. Then $\cM(*\hb)$ is $\cD_\cX[1/\hb]$-holonomic with characteristic variety contained in $\Lambda\times T^*\CC^*_\hb$. Since it is conic and Lagrangian in $T^*\cX^\circ$, it must be of the form given in the statement for some $\Sigma\subset\CC_\hb^*$, and we obtain the discreteness of $\Sigma$ in $\CC_\hb$ because we worked over the ring $\cD_\cX[1/\hb]$ (and not only over~$\cD_{\cX^\circ}$). It now follows from \cite[Th.\,1.3]{Kashiwara78} that $\cM(*\hb)$ is $\cD_\cX$-holonomic.

\item(See also \cite[Lem.\,7.3]{Mochizuki08b}.)
Due to the strict specializability of $\cM$, the statement in \eqref{prop:speint2a} holds as soon as it holds for $\cM^\circ$. We will therefore treat~\eqref{prop:speint2a} and~\eqref{prop:speint2c} for $\cM^\circ$. By \cite[Prop.\,7.3.1]{Bibi01c}, each $V^{(\hb_o)}_\beta\cM_\fun^\circ$ is stable by $\partial_\hb$ ($\beta\in\RR$), hence is $V_0\cD_{\cX\fun^\circ}$-coherent. On the one hand, the minimal polynomial of $-t\partial_t$ on $\gr_\beta^{V^{(\hb_o)}}\!\!\cM_\fun^\circ$ takes the form $\prod_{\gamma\in\CC}\bigl(s-\beta+(\gamma/\hb-\ov\gamma\hb)\bigr)^{\nu_\gamma}$ and on the other hand, since~$\cM_\fun^\circ$ is $\cD_{\cX_\fun^\circ}$-holonomic, it is of the form $\prod_{\delta\in\CC}(s-\delta)^{\nu_\delta}$, according to \cite{Kashiwara78}. As a consequence, the possible $\gamma$'s are zero and the possible $\delta$'s are equal to $\beta$, hence are real. Therefore, the filtration $V^{(\hb_o)}_\bbullet \cM_\fun^\circ$ satisfies the characteristic properties of the Kashiwara-Malgrange filtration of the holonomic $\cD_{\cX_\fun^\circ}$-module $\cM_\fun^\circ$, hence is equal to~it. The remaining assertions are then clear.\qedhere
\end{enumerate}
\end{proof}

\subsection{Sheaves of \texorpdfstring{$\Omega$-$C^\infty$}{Om} distributions}\label{subsec:OmegacontinuousDb}

(\Cf \cite[\S0.5]{Bibi01c} and \cite[\S2.1.3]{Mochizuki11} for this section.) Let $\Omega$ be a complex manifold equipped with a volume form $\rd\vol_\Omega$ (one can also consider a real-analytic manifold with a volume form, embedded into its complexification $\Omega^\CC$). Our main example will be the case of an open set of $\CC^*_\hb$ with $\rd\vol_\Omega=\itwopi(\rd\hb/\hb)\wedge(\rd\ov\hb/\ov\hb)$ (we also consider the case where $\Omega=\bS=\{|\hb|=1\}$ with its standard volume form).

\begin{notation}\label{nota:XovX}
If $\Omega$ is real-analytic, we define $\cO_{X\times\Omega}$ as the sheaf-theoretic restriction of $\cO_{X\times\Omega^\CC}$ to $X\times\Omega$, and $\cD_{X\times\Omega}$ as the subsheaf of $(\cD_{X\times\Omega^\CC})_{|X\times\Omega}$ generated by $\cO_{X\times\Omega},\Theta_X$ and partial derivatives with respect to coordinate in $\Omega$.

We denote by $\ov X$ the complex conjugate manifold with $\cO_{\ov X}:=\ov{\cO_X}$. We then set
\[
\cO_{X\times\Omega,\ov{X\times\Omega}}:=\cO_{X\times\Omega}\otimes_\CC\cO_{\ov{X\times\Omega}}\quad\cD_{X\times\Omega,\ov{X\times\Omega}}:=\cD_{X\times\Omega}\otimes_\CC\cD_{\ov{X\times\Omega}}.
\]
If $\Omega$ is real-analytic, we can replace $\ov{X\times\Omega}$ with $\ov X\times\Omega$.
\end{notation}

The sheaf of distributions \index{$DbXOmega$@$\Db_{X\times\Omega}$}$\Db_{X\times\Omega}$ is a left $\cD_{X\times\Omega,\ov{X\times\Omega}}$-module. We will consider the sheaf \index{$DbXOmegaOmega$@$\Db_{X\times\Omega/\Omega}$}\index{$DbXS$@$\Db_{X\times\bS/\bS}$}$\Db_{X\times\Omega/\Omega}$ of distributions on $X\times\Omega$ which are continuous with respect to~$\Omega$ (\cf \cite[\S0.5]{Bibi01c}). Recall that a section $u\in\Gamma(W,\Db_{X\times\Omega/\Omega})$ is a $C^\infty(\Omega)$-linear morphism from the space of relative forms of maximal degree $2d_X$ with compact support on $W\subset X\times\Omega$ to the space $C^0_\rc(\Omega)$, which is continuous with respect to the usual topologies on these spaces. Since a volume form is fixed on $\Omega$, we can regard $\Db_{X\times\Omega/\Omega}$ as a subsheaf of $\Db_{X\times\Omega}$ by the natural injective morphism $\Db_{X\times\Omega/\Omega}\to\Db_{X\times\Omega}$, defined by $u\mto\int_\Omega u(\cbbullet)\rd\vol_\Omega$, which is $C^\infty_{X\times\Omega}$-linear.

As is usual, we define the conjugate $\ov u$ of a relative distribution in $\Db_{X\times\Omega/\Omega}$ as $\ov u(\cbbullet):=\ov{u(\ov\cbbullet)}$. If $\Omega=\CC^*_\hb$, we also define $\iota^*u$ as $\iota^*u(\cbbullet):=\iota^*(u(\iota^*\cbbullet))$. We give a similar definition for $\sigma^*u$ for $u$ in $\Db_{X\times\bS/\bS}$ (defined similarly, \cf\cite[\S0.5]{Bibi01c}).

For $k\in\NN$, we say that $u\in\Db_{X\times\Omega/\Omega}(W)$ is $C^k$ with respect to $\Omega$ if all its derivatives up to order $k$ with respect to $\Omega$, when considered as elements of $\Db_{X\times\Omega}(W)$, belong to (the image of) $\Db_{X\times\Omega/\Omega}(W)$. For a vector field $\varsigma$ on $\Omega$ and a test relative form $\chi$ on $W$, we then have $(\varsigma u)(\chi)=-u(\varsigma(\chi))$ in $C^0_\rc(\Omega)$. This leads to the natural definition of \index{$DbXOmegaOmegainf$@$\Db^\infty_{X\times\Omega/\Omega}$}$\Db^\infty_{X\times\Omega/\Omega}(W)$, which is seen to be a $\Gamma(W,\cD_{X\times\Omega,\ov{X\times\Omega}})$-module.

\subsubsection*{Pullback with respect to $\Omega$}
Let $\nu:\Omega'\hto\Omega$ be the inclusion of a closed submanifold. For $W\subset X\times\Omega$, set $W'=W\cap(X\times\Omega')$. For $u\in\Db_{X\times\Omega/\Omega}(W)$ and $\eta$ a relative test form of maximal degree on $W$, $u(\eta)_{|\Omega'}$ only depends on $\eta_{|W'}$. For any relative test form $\eta'$ on $W'$, we choose $\eta$ with $\eta_{|\Omega'}=\eta'$ and we set $\nu^*u(\eta')=u(\eta)_{|\Omega'}$. Then $\nu^*\eta\in\Db_{X\times\Omega'/\Omega'}(W')$. This defines a restriction morphism $\nu^*:\nu^{-1}\Db_{X\times\Omega/\Omega}\to\Db_{X\times\Omega'/\Omega'}$. We will mainly consider the case of the inclusion $\nu:\bS\hto\CC^*_\hb$, for which we obtain the natural restriction morphism
\begin{equation}\label{eq:restrDb}
\nu^{-1}\Db_{\cX^\circ /\CC^*_\hb}\to\Db_{X\times\bS/\bS}.
\end{equation}

Let $q:\Omega':=\Omega\times\wt\Omega\to\Omega$ be the projection For $u\in\Gamma(W,\Db_{X\times\Omega/\Omega})$ the pullback $q^* u\in\Gamma(q^{-1}W,\Db_{X\times\Omega'/\Omega'})$ is defined as follows: for $\eta'(x,\hb,\wt\hb)\in\Gamma_\rc(q^{-1}W,\cE^{d_X,d_X}_{X\times\Omega'/\Omega'})$ and for $\wt\hb$ fixed, $u(\eta'(x,\hb,\wt\hb))$ belongs to $C^0_\rc(\Omega)$. When $\wt\hb$ varies, this defines an element of $C^0_\rc(\Omega\times\wt\Omega)$, that we denote by $q^*u(\eta')$.

For any $C^\infty$ map $\mu:\Omega'\to\Omega$, we decompose it as $\mu=q\circ\nu$ and we set $\mu^*u:=\nu^*(q^* u)\in\Db_{X\times\Omega'/\Omega'}(W')$, with $W':=\mu^{-1}(W)$. We denote by $T\mu$ the tangent map $\Theta^\infty_{\Omega'}\to\mu^*\Theta^\infty_\Omega$. The following lemma is straightforward.

\begin{lemme}\label{lem:mustaru}
If $u\in\Db^\infty_{X\times\Omega/\Omega}(W)$, then $\mu^*u\in\Db^\infty_{X\times\Omega'/\Omega'}(W')$, and for any vector field $\varsigma'$ on $\mu^{-1}W$, setting $T\mu(\varsigma')=\sum_i\varphi_i\otimes\varsigma_i$, with $\varphi_i\in C^\infty(W')$ and $\varsigma_i\in\Theta^\infty(W)$, we have $\varsigma'\mu^*u=\sum_i\varphi_i\mu^*(\varsigma_iu)$.\qed
\end{lemme}

\subsubsection*{Pullback with respect to $X$ by a smooth morphism}
Let $\mapsm:X\to Y$ be a \emph{smooth} morphism of complex manifolds, and set $d_X=\dim X$, $d_Y=\dim Y$. We denote by the same letter the map $\mapsm\times\id:X\times\Omega\to Y\times\Omega$ or $X\times\bS\to Y\times\bS$. Integration along the fibres of $\mapsm$ sends the space $\Gamma_\rc(X\times\Omega,\cE^{d_X,d_X}_{X\times\Omega/\Omega})$ of relative forms of maximal $X$-degree with compact support to the space $\Gamma_\rc(Y\times\Omega,\cE^{d_Y,d_Y}_{Y\times\Omega/\Omega})$. As a consequence, we can define the smooth pullback $\mapsm^*:\mapsm^{-1}\Db_{Y\times\Omega/\Omega}\to\Db_{X\times\Omega/\Omega}$ by the formula $(\mapsm^*u)(\chi)=u\bigl(\int_\fun\chi\bigr)$. Note that
\[
\cD_{X\times\Omega\to Y\times\Omega}:=\cO_{X\times\Omega}\otimes_{\mapsm^{-1}\cO_{Y\times\Omega}}\mapsm^{-1}\cD_{Y\times\Omega}
\]
is $\cD_{Y\times\Omega}$-flat. Let us set
\[
\mapsm^{++}\Db^\infty_{Y\times\Omega/\Omega}:=\cD_{(X\times\Omega,\ov{X\times\Omega})\to (Y\times\Omega,\ov{Y\times\Omega})}\underset{\mapsm^{-1}\cD_{Y\times\Omega,\ov{Y\times\Omega}}}\otimes\hspace*{-5mm}\mapsm^{-1}\Db^\infty_{Y\times\Omega/\Omega}.
\]
In other words,
\begin{equation}\label{eq:pullbackODb}
\mapsm^{++}\Db^\infty_{Y\times\Omega/\Omega}=\mapsm^*\Db^\infty_{Y\times\Omega/\Omega}:=\cO_{X\times\Omega,\ov{X\times\Omega}}\otimes_{\mapsm^{-1}\cO_{Y\times\Omega,\ov{Y\times\Omega}}}\mapsm^{-1}\Db^\infty_{Y\times\Omega/\Omega}
\end{equation}
as an $\cO_{X\times\Omega,\ov{X\times\Omega}}$-module, and the left action of $\cD_{X\times\Omega,\ov{X\times\Omega}}$ is the natural one. By extending the smooth pullback above, we obtain a $\cD_{X\times\Omega,\ov{X\times\Omega}}$-linear morphism
\begin{equation}\label{eq:pullbackDb}
\mapsm^{++}\Db^\infty_{Y\times\Omega/\Omega}\to\Db^\infty_{X\times\Omega/\Omega},
\end{equation}
which is compatible with the action of derivations with respect to $\Omega$.

\subsubsection*{Pushforward by a proper morphism}
Let $\map:X\to Y$ be a \emph{proper} morphism of complex manifolds. We will now use currents of maximal degree (with respect to $X$ or~$Y$) and use $\rd\vol_\Omega$ to switch from distributions to currents with respect to $\Omega$. Integration of currents defines a morphism \index{$CbXOmegaOmega$@$\gC^{d_X,d_X}_{X\times\Omega/\Omega}$}$\int_\map:\map_*\gC^{d_X,d_X}_{X\times\Omega/\Omega}\to\gC^{d_Y,d_Y}_{Y\times\Omega/\Omega}$, which preserves differentiability with respect to $\Omega$ and extends as a $\cD_{Y\times\Omega,\ov{Y\times\Omega}}$-linear morphism\index{$CbXOmegaOmegainf$@$\gC^{\infty,d_X,d_X}_{X\times\Omega/\Omega}$}
\begin{equation}
\map_{++}\gC^{\infty,d_X,d_X}_{X\times\Omega/\Omega}\to\gC^{\infty,d_Y,d_Y}_{Y\times\Omega/\Omega},
\end{equation}
induced by the right action of $\cD_{Y\times\Omega,\ov{Y\times\Omega}}$, where we have set
\begin{align*}
\map_{++}\gC^{\infty,d_X,d_X}_{X\times\Omega/\Omega}&:=\map_*\bigl(\gC^{\infty,d_X,d_X}_{X\times\Omega/\Omega}\otimes_{\cD_{X\times\Omega,\ov{X\times\Omega}}}\cD_{(X\times\Omega,\ov{X\times\Omega})\to (Y\times\Omega,\ov{Y\times\Omega})}\bigr)\\
&=\map_*\gC^{\infty,d_X,d_X}_{X\times\Omega/\Omega}\otimes_{\cO_{Y\times\Omega,\ov{Y\times\Omega}}}\cD_{Y\times\Omega,\ov{Y\times\Omega}}.
\end{align*}

\subsubsection*{External tensor product}
Let $u\in\Db_{X\times\Omega}(W)$ and $v\in\Db_{Y\times\Omega}(W')$. If $\chi(x,y,\hb,\hb')$ is an $(\Omega\times\Omega)$-relative test form on $W\times W'$, then $\langle v_{y,\hb'},\chi\rangle$ is an $\Omega$-relative test form on~$W$ and the external tensor product $u\boxtimes v$ is the distribution in $\Db_{X\times Y\times\Omega\times\Omega}(W{\times}W')$ defined by
\[
\langle u_{x,\hb}\boxtimes v_{y,\hb'},\chi(x,y,\hb,\hb')\rangle=\bigl\langle u_{x,\hb},\langle v_{y,\hb'},\chi(x,y,\hb,\hb')\rangle\bigr\rangle.
\]
If $u,v$ are $C^0$ with respect to $\Omega$, \ie are sections of $\Db_{X\times\Omega/\Omega}$ and $\Db_{Y\times\Omega/\Omega}$, then $u\boxtimes v$ is $C^0$ with respect to $\Omega\times\Omega$. We can thus restrict it to the diagonal $\Omega\hto \Omega\times\Omega$, and obtain a relative distribution denoted by $u\hbboxtimes v$ in $\Db_{X\times Y\times\Omega/\Omega}$. Differentiability with respect to $\Omega$ is preserved.

\subsubsection*{Localization}
Let $H$ be a hypersurface of $X$. We set
\[
\Db_{X\times\Omega/\Omega}(*H)=\cO_{X\times\Omega}(*(H\times\Omega))\otimes_{\cO_{X\times\Omega}}\Db_{X\times\Omega/\Omega}.
\]
Then $\Db_{X\times\Omega/\Omega}(*H)$ is naturally identified with the image of the natural morphism $\Db_{X\times\Omega/\Omega}\to j_*j^{-1}\Db_{X\times\Omega/\Omega}$, if $j:X\moins H\hto X$ denotes the open inclusion, whose kernel consists of relative distributions supported by $H\times\Omega$. We also denote this image by $\Db^\modH_{(X\moins H)\times\Omega/\Omega}$. A similar result holds for $\Db^\infty_{X\times\Omega/\Omega}$.

\subsection{Short reminder on $\sigma$-sesquilinear pairings}\label{subsec:remindersesqui}\index{sesquilinearpairingsigma@$\sigma$-sesquilinear pairing}
Given a section $P(x,\hb,\partiall_x)$ of~$\cR_{\cX|\bS}$, we denote by $\ov P$ its conjugate, which is an anti-holomorphic differential operator, and by $\sigma^*\ov P$ its pullback by $\sigma$, which is a section of $\cR_{\ov\cX|\bS}$. One can identify $\cR_{\cX|\bS}$ with $\cD_{X\times\bS/\bS}$ and $\cR_{\ov\cX|\bS}$ with $\cD_{\ov X\times\bS/\bS}$.

By an \emph{integrable $\sigma$-sesquilinear pairing} $C_{\bS}$ on a pair $(\cM',\cM'')$ of left $\cR^\intt_\cX$\nobreakdash-modules, we mean a $\CC$-linear morphism\index{$CS$@$C_{\bS}$}
\begin{equation}\label{eq:sesquisigma}
C_{\bS}:\cM'_{\bS}\otimes_{\CC}\sigma^*\ov{\cM''_{\bS}}\to\Db_{X\times\bS/\bS}
\end{equation}
such that, for sections $m'$ of $\cM'$ on $U\times\Omega\subset X\times\bS$ (\resp $m''$ of $\cM''$ on $U\times\sigma^{-1}(\Omega)$), we have, for any $P\in\Gamma(U\times\Omega,\cR_{\cX|\bS})$, the following equalities in $\Db_{X\times\bS/\bS}$:
\begin{equation}\label{eq:sesquisigmaP}
C_{\bS}(Pm',\sigma^*\ov{m''})=PC_{\bS}(m',\sigma^*\ov{m''}),\quad C_{\bS}(m',\sigma^*\ov{Pm''})=\sigma^*\ov{P}C_{\bS}(m',\sigma^*\ov{m''}),
\end{equation}
(in particular it descends to $\cM'_{\bS}\otimes_{\cO_{\bS}}\sigma^*\ov{\cM''_{\bS}}$) satisfying the integrability property:
\[
-\sfi\partial_\vartheta C_{\bS}(m',\sigma^*\ov{m''})=C_{\bS}(\hb\partial_\hb m',\sigma^*\ov{m''})- C_{\bS}(m',\sigma^*\ov{\hb\partial_\hb m''}).
\]
The latter equality is seen in $\Db_{X\times\bS}$. However, as seen in Section \ref{subsec:OmegacontinuousDb}, it means precisely that $C_{\bS}$ is $C^1$ with respect to $\bS$ and the equality holds in $\Db_{X\times\bS/\bS}$. It follows, by iterating the vector field $\partial_\vartheta$, that $C_{\bS}$ takes values in $\Db^\infty_{X\times\bS/\bS}$. As a consequence, \eqref{eq:sesquisigmaP} holds for any local section $P$ of $\cD_{X\times\bS}$.

\begin{convention}\label{conv:infty}
In order to simplify notation, we will simply use the notation $\Db_{X\times\Omega/\Omega}$ for $\Db^\infty_{X\times\Omega/\Omega}$, when $\Omega=\bS$ or $\Omega=\CC^*_\hb$.
\end{convention}

\begin{remarque}[The category $\RdTriples(X)$]\label{rem:RdTriples}
\index{$RTriplesXint$@$\RdTriples(X)$}The objects of $\RdTriples(X)$ are triples $\cT\!=\!(\cM',\cM'',C_{\bS})$ consisting of $\cR_\cX^\intt$-modules $\cM',\cM''$ and an integrable pairing between them, and morphisms consist of pairs of morphisms $(\lambda',\lambda'')$ of $\cR_\cX^\intt$-modules (with $\lambda''$ covariant and~$\lambda'$ contravariant) which are compatible with the pairings.\footnote{\label{foot:positiveconstant}Replacing $C_{\bS}$ with $cC_{\bS}$ with $c>0$ obviously leads to an isomorphic object.} Forgetting the integrable structure gives a functor to the category $\RTriples(X)$. Given an integrable triple $\cT$, any other action of $\hb^2\partial_\hb$ is obtained by replacing $\hb^2\partial_\hb:\cM'\to\cM'$ with $\hb^2\partial_\hb+\lambda'$, where $\lambda'$ is any $\cR_\cX$-linear morphism $\cM'\to\cM'$, and similarly for~$\cM''$ with $\lambda''$. The modified action is compatible with~$C_{\bS}$, \ie defines a new action on $\cT$, if and only if $\lambda',\lambda''$ satisfy
\[
\hbm C_{\bS}(\lambda'(m'),\sigma^*\ov{m''})=C_{\bS}(m',\sigma^*\ov{\hbm\lambda''(m'')})
\]
for any local sections $m'$ of $\cM'_{|\bS}$ and $m''$ of $\cM''_{|\bS}$. This equality means that $(\lambda',\lambda'')$ defines a morphism in $\RTriples(X)$:
\begin{starequation}\label{eq:RdTriples}
\lambda:\cT\to\cT\otimes\bT(-1).
\end{starequation}%
If $\cT$ is the trivial rank-one twistor structure on a point, that is, the vector bundle~$\cO_{\PP^1_\hb}$, an integrable structure is a connection with poles of order two at most at $\hb=0$ and $\hb=\infty$, and no other pole, that is, in the given trivialization, a polynomial of degree two in $\hb$. Starting from the triple $(\cO_{\CC_\hb}{\cdot}e,\cO_{\CC_\hb}{\cdot}e,C_{\bS})$ with $C_{\bS}(e,\sigma^*\ov{e})=1$ and the standard action of $\hb^2\partial_\hb$ on $\cO_{\CC_\hb}$, any other integrable structure is given by a pair of morphisms $\lambda',\lambda''$, which are thus multiplications by entire holomorphic functions on~$\CC_\hb$, denoted by $\lambda'(\hb),\lambda''(\hb)$. The compatibility condition reads
\begin{starstarequation}\label{eq:equivintegrable}
\lambda'(\hb)=-\hb^2\ov{\lambda''(-1/\ov\hb)},
\end{starstarequation}%
which implies (by considering the series expansion of these entire holomorphic functions) that $\lambda'(\hb),\lambda''(\hb)$ are polynomials of degree two and $\lambda'$ determines $\lambda''$.
\end{remarque}

\begin{definition}[Equivalent integrable structures]\label{def:equivintstruct}
We will say that two integrable structures on a given object $\cT$ of $\RTriples(X)$ are \emph{equivalent} if one is obtained from the other by tensoring by an integrable structure on the trivial rank-one twistor structure on a point, that is, if the corresponding morphism \eqref{eq:RdTriples} takes the form $(\lambda'(\hb)\id,\lambda''(\hb)\id)$, with $\lambda',\lambda''$ satisfying \eqref{eq:equivintegrable}.
\end{definition}

\begin{exemple}\label{exem:equivintstruct}
Set $\lambda'(\hb)=a+b\hb+c\hb^2$, so that $\lambda''(\hb)=-\ov c+\ov b\hb-\ov a\hb^2$. Then
\begin{multline}\tag{$\ref{exem:equivintstruct}*$}\label{eq:equivintstruct}
\bigl((\cM',\hb^2\partial_\hb+\lambda'(\hb)),(\cM'',\hb^2\partial_\hb+\lambda''(\hb)),C_{\bS}\bigr)\\
\simeq\bigl((\cM',\hb^2\partial_\hb+a+b\hb),(\cM'',\hb^2\partial_\hb-\ov c+\ov b\hb),\rme^{-c\hb-a/\hb}C_{\bS}\bigr).
\end{multline}
\end{exemple}

\begin{remarque}[Uniqueness of the integrable structure up to equivalence]\label{rem:intuniquepure}
Assume $X$ is a smooth complex projective variety and let $\ccM$ be an irreducible holonomic $\cD_X$-module. According to one of the main results in \cite{Mochizuki08}, there exists a unique (up to isomorphism) pure twistor $\cD$-module $\cT$ of weight $0$ giving rise to $\ccM$ by the functor $\Xi_{\DR}$ (Notation \ref{nota:XiDR}).

An integrable structure (if any) on $\cT$ is obtained from a given one through a morphism $\lambda:\cT\to\cT(-1)$, after Remark \ref{rem:RdTriples}. Set $\lambda=(\lambda',\lambda'')$. Then for any given $\hb_o\neq0$, $\lambda'_{|\hb_o}:\cM'_{\hb_o}\to\cM'_{\hb_o}$ takes the form $c'(\hb_o)\id$ since $\cM'_{\hb_o}$ is an irreducible $\cD_X$\nobreakdash-module. It follows that $\lambda'=c'(\hb)\id$ and similarly $\lambda''=c''(\hb)\id$, with $c'(\hb),c''(\hb)$ holomorphic, and defining an integrable structure on the trivial twistor structure of rank one on a point.

Hence, any irreducible holonomic $\cD_X$-module underlies at most one object of pure weight~$0$ (up to equivalence of integrable structures) in $\MTM^\intt(X)$.
\end{remarque}

\subsection{\texorpdfstring{$\iota$}{iota}-Sesquilinear pairing}
\index{sesquilinearpairingiota@$\iota$-sesquilinear pairing}We use Convention \ref{conv:infty}. By a $\iota$-sesquilinear pairing on a pair $(\cM',\cM'')$ of left $\cR^\intt_\cX$\nobreakdash-modules, we mean a $\CC$-linear morphism\index{$Ci$@$\iC$}
\begin{equation}\label{eq:rescsesqui}
\iC:\cM^{\prime\circ}\otimes_\CC\iota^*\ov{\cM^{\prime\prime\circ}}\to\Db_{\cX^\circ/\CC^*_\hb}
\end{equation}
such that, for sections $m'$ of $\cM'$ on $U\times\Omega$ (\resp $m''$ of $\cM''$ on $U\times\iota^{-1}(\Omega)$) with $\Omega\subset\CC_\hb^*$, we have, for any $P\in\Gamma(U\times\Omega,\cD_{\cX^\circ})$, the following equalities in $\Db_{\cX^\circ/\CC^*_\hb}$:
\begin{equation}\label{eq:iCP}
\iC(Pm',\iota^*\ov{m''})=P\,\iC(m',\iota^*\ov{m''}),\quad\iC(m',\iota^*\ov{Pm''})=\iota^*\ov{P}\,\iC(m',\iota^*\ov{m''}).
\end{equation}
In particular,
\begin{equation}\label{eq:iChbpartialhb}
\begin{split}
\iC(\hb\partial_\hb m',\iota^*\ov{m''})&=\hb\partial_\hb\, \iC(m',\iota^*\ov{m''}),\\
\iC(m',\iota^*\ov{\hb\partial_\hb m''})&=\ov\hb\partial_{\ov\hb}\,\iC(m',\iota^*\ov{m''}),
\end{split}
\end{equation}
where $\iota$ is missing on the last term because $\iota^*{\hb\partial_\hb}=\hb\partial_\hb$. Working in polar coordinates $\hb=\varrho \rme^{i\vartheta}$, this can be written as
\begin{equation}\label{eq:iCthetarho}
\begin{split}
-\sfi\partial_\vartheta\,\iC(m',\iota^*\ov{m''})&=\iC(\hb\partial_\hb m',\iota^*\ov{m''})-\iC(m',\iota^*\ov{\hb\partial_\hb m''}),\\
\varrho\partial_\varrho\,\iC(m',\iota^*\ov{m''})&=\iC(\hb\partial_\hb m',\iota^*\ov{m''})+\iC(m',\iota^*\ov{\hb\partial_\hb m''}).
\end{split}
\end{equation}

\begin{lemme}\label{lem:iCCS}
The restriction (in the sense of \eqref{eq:restrDb}) to $\bS$ of an integrable sesquilinear pairing \eqref{eq:rescsesqui} is an integrable sesquilinear pairing \eqref{eq:sesquisigma}.
\end{lemme}

\begin{proof}
Indeed, $\iota$ and $\sigma$ coincide when restricted to $\bS$, hence $\sigma^*(P_{\bS})=\iota^*(P_{\bS})$ for a differential operator.\end{proof}

\begin{remarque}\label{rem:iCint}
We note that $\iC$ is $\cD_{\cX^\circ,\ov\cX^\circ}$-linear, so if we regard $\iC$ as a $\cD_{\cX^\circ}$-linear morphism
\begin{starequation}\label{eq:iCint}
\cM^{\prime\circ}\to\cHom_{\cD_{\ov\cX^\circ}}(\iota^*\ov{\cM^{\prime\prime\circ}},\Db_{\cX^\circ/\CC^*_\hb}),
\end{starequation}%
we obtain by composition with the natural inclusion $\Db_{\cX^\circ/\CC^*_\hb}\hto\Db_{\cX^\circ}$ a $\cD_{\cX^\circ}$-linear morphism $\cM^{\prime\circ}\to\cHom_{\cD_{\ov\cX^\circ}}(\iota^*\ov{\cM^{\prime\prime\circ}},\Db_{\cX^\circ})$.
\end{remarque}

\begin{remarque}[Restriction to $\hb_o\in\CC^*_\hb$]\label{rem:iChbo}
For $\hb_o\in\CC^*_\hb$, we denote by $\cM^{\hb_o}$ the quotient of the $\cR_\cX$-module $\cM$ by the submodule $(\hb-\hb_o)\cM$, that we regard as a $\cD_X$-module by identifying $\cR_{\cX^\circ}^{\hb_o}$ with $\cD_X$. Then a $\iota$-sesquilinear pairing~$\iC$ can be restricted as a $\cD_X\otimes_\CC\cD_{\ov X}$-linear pairing $\iC^{\hb_o}:\cM^{\prime\hb_o}\otimes_\CC\ov{\cM^{\prime\prime-\hb_o}}\to\Db_X$.
\end{remarque}

\subsection{Standard functors on \texorpdfstring{$\RdiTriples(X)$}{RdiT}}\label{subsec:stdfunctorsRdiT}

The category \index{$RTriplesXinti$@$\RdiTriples(X)$}$\RdiTriples(X)$ is defined in a way very similar to $\RdTriples(X)$ by considering triples $(\cM',\cM'',\iC)$ instead of triples $(\cM',\cM'',C_{\bS})$ (\cf\cite[Chap.\,7]{Bibi01c}). By Lemma \ref{lem:iCCS}, we have a natural functor
\[
\RdiTriples(X)\mto\RdTriples(X)
\]
by restricting $\iC$ to $\bS$, and considering it as a $\sigma$-sesquilinear pairing. We now define functors on the category on the left-hand side. Let $\icT=(\cM',\cM'',\iC)$ be an object of $\RdiTriples(X)$.

\begin{exemple}[Changing the integrable structure]\label{exem:equivintstructT}
Let $\lambda'(\hb),\lambda''(\hb)$ be as in Example \ref{exem:equivintstruct}, and let $(\wt\cM',\wt\cM'',C_{\bS})$ denote the object $(\cM',\cM'',C_{\bS}$) of $\RTriples(X)$ with modified integrable structure $(\hb^2\partial_\hb+\lambda'(\hb),\hb^2\partial_\hb+\lambda''(\hb))$. Then the corresponding object of $\RdTriples(X)$ is $(\wt\cM',\wt\cM'',\iwtC)$ with
\[
\iwtC=|\hb|^{2b}\rme^{a(\ov\hb-1/\hb)+c(\hb-1/\ov\hb)}\iC.
\]
We also have
\begin{starequation}\label{eq:equivintstructT}
(\wt\cM',\wt\cM''\iwtC)\simeq\bigl((\cM',\hb^2\partial_\hb+a+b\hb),(\cM'',\hb^2\partial_\hb-\ov c+\ov b\hb),|\hb|^{2b}\rme^{-(a/\hb+c/\ov\hb)}\iC\bigr).
\end{starequation}%
\end{exemple}

\begin{definition}[Tate twist]
\index{Tate twist}For $\ell\in\ZZ$, we set $\icT(\ell):=\icT\otimes\ibT(\ell)$ (\cf Example \ref{exem:Tatetriple}).
\end{definition}

\begin{definition}[Adjunction]
\index{adjunction}The adjoint $\icT^*$ of $\icT$ is defined as $\icT^*\!=\!(\cM'',\cM',\iC^*)$ with $\iC^*(m'',\iota^*m'):=\iota^*\ov{\iC(m',\iota^*\ov{m''})}$. (Note that \eqref{eq:iCP} clearly holds for $\iC^*$.)
\end{definition}

\begin{definition}[Pullback by a smooth morphism]\label{def:pullbacksmooth}
Let $\mapsm:X\to Y$ be a smooth morphism and let $\icT$ be an object of $\RdiTriples(Y)$. Then $\mapsm^+\icT$ is defined as $(\mapsm^+\cM',\mapsm^+\cM'',\mapsm^+\iC)$, where $\mapsm^+\cM$ is the pullback as an $\cR_\cY^\intt$-module, and $\mapsm^+\iC$ is defined as follows. Firstly,~$\iC$ induces a $\cD_{\cX^\circ,\ov{\cX^\circ}}$-linear morphism
\[
\mapsm^+\cM^{\prime\circ}\underset\CC\otimes\iota^*\ov{\mapsm^+\cM^{\prime\prime\circ}}\ra\cD_{(\cX^\circ,\ov{\cX^\circ})\to(\cY^\circ,\ov{\cY^\circ})}\underset{g^{-1}\cO_{\cY^\circ,\ov{\cY^\circ}}}\otimes g^{-1}\Db_{\cY^\circ/\CC^*_\hb}
=g^{++}\Db_{\cY^\circ/\CC^*_\hb},
\]
and, by composing with \eqref{eq:pullbackDb}, we obtain an integrable sesquilinear pairing~$\mapsm^+\iC$.
\end{definition}

In order to get the compatibility with the behaviour of mixed Hodge modules, one modifies a little the pullback functor. Let us denote by $\cT_\cO(X)$ the object of $\RTriples(X)$ defined as $(\cM',\cM'',C_{\bS})\!=\!(\hb^{d_X}\cO_\cX,\cO_\cX,C_{\bS})$, with $C_{\bS}(\hb^{d_X},1)\!=\!\hb^{d_X}$. The polarization is adapted accordingly (\cf\cite[\S13.5.2.1]{Mochizuki11}). It is an object of $\RdTriples(X)$ if the $\hb^2\partial_\hb$-action on $\cM'$ is the action induced by that on~$\cO_\cX$. Lastly, one can define similarly the object $\icT_\cO(X)$ of $\RdiTriples(X)$ by defining~$\iC$ correspondingly. We have $\cT_\cO(X)\simeq\bU_X(d_X,0)$ (\resp $\icT_\cO(X)\simeq\ibU_X(d_X,0)$), \cf Examples \ref{exem:Tatetriple}\eqref{exem:Tatetriple2} and \ref{exem:Tateitriple}.

\begin{definition}[Normalized pullback by a smooth morphism]\label{def:normpullbacksmooth}
For $\mapsm:X\to Y$ smooth and let $\icT$ be an object of $\RdiTriples(Y)$. We define $\Tmapsm^+\icT$ as $(\hb^{d_{X/Y}}\mapsm^+\cM',\mapsm^+\cM'',\mapsm^+\iC)$. It is an object of $\RdiTriples(X)$.
\end{definition}

If $g:X\!=\!Y\!\times\!Z\!\to\!Y$ is a projection, then $\Tmapsm^+\icT\!=\!\icT\hbboxtimes\,\icT_\cO(Z)$ (\cf below for~$\hbboxtimes$).

\begin{definition}[Pushforward by a proper morphism]
Let $\map:X\to Y$ be a morphism between complex manifolds. The pushforward objects $\map_\dag^j\icT$ ($j\in\ZZ$) are defined as in \cite[\S1.6.d]{Bibi01c}. For $j\in\ZZ$ we obtain a functor $\map_\dag^j:\RdiTriples(X)\mto\RdiTriples(Y)$ (\cf\cite[Prop.\,7.1.4]{Bibi01c}).
\end{definition}

\begin{definition}[External tensor product]\label{def:externalproduct}
The external tensor product\index{$Boxtimeshb$@$\hbboxtimes$}
\[
\hbboxtimes:\RdiTriples(X)\times\RdiTriples(Y)\mto\RdiTriples(X\times Y)
\]
is defined as usual for the $\cR^\intt$-modules. Let $X,Y$ be complex manifolds and set $Z=X\times Y$. We set $\cO_{\cX\hbboxtimes\cY}:=\cO_\cX\boxtimes_{\cO_{\CC_\hb}}\cO_\cY$ and similarly for $\cR_{\cX\hbboxtimes\cY}$. Then~$\cO_\cZ$ (\resp $\cR_\cZ$) is flat over $\cO_{\cX\hbboxtimes\cY}$ (\resp $\cR_{\cX\hbboxtimes\cY}$). For $\cM\in\Mod(\cR_\cX)$ and $\cN\in\nobreak\Mod(\cR_\cY)$, the external tensor product $\cM\hbboxtimes\cN\in\Mod(\cR_{\cX\hbboxtimes\cY})$ can also be considered as being in $\Mod(\cR_\cZ)$ after tensoring with $\cR_\cZ$. If $\cM,\cN$ are endowed with a compatible action of $\hb^2\partial_\hb$, then so is $\cM\hbboxtimes\cN$ by using the Leibniz rule. For the sesquilinear pairing, one uses the relative external tensor product of relative distributions as defined in Section \ref{subsec:OmegacontinuousDb}. The integrability property is clearly preserved. In~order to avoid any derived external tensor product, it is natural to assume that $\cM,\cN$ are strict.
\end{definition}

\begin{remarque}[Specialization of a $\iota$-sesquilinear pairing]\label{rem:speiota}
The definition of the nearby cycles along a function $\fun:X\to\CC$ of a $\iota$-sesquilinear pairing between integrable strictly specializable $\cR_\cX$\nobreakdash-modules is obtained as for a $\sigma$-sesquilinear pairing (\cf\cite[\S3.6\,\&\,7.3.b]{Bibi01c} and \cite[\S22.10]{Mochizuki08}). The basic result \cite[Prop.\,3.6.4]{Bibi01c} (\cf also \cite[\S22.10]{Mochizuki08}) also applies to $\iota$-sesquilinear pairings in the integrable case and is compatible with the restriction considered in Lemma \ref{lem:iCCS}, and we will refer without any further ado to results concerning $\sigma$-sesquilinear pairings for application to $\iota$-sesquilinear pairings, as well as to the good behaviour with respect to restriction \ref{lem:iCCS}. We emphasize that the property for an object of $\RdiTriples(X)$ to be strictly specializable along $\fun$ is a condition on the $\cR^\intt_\cX$-modules of the object of $\RdiTriples(X)$, not on the $\iota$-sesquilinear pairing.
\end{remarque}

\pagebreak[2]
\begin{remarque}[Localization and dual localization of a $\iota$-sesquilinear pairing]\label{rem:lociota}
The results on localization also apply to $\iC$, in particular \cite[Prop.\,3.2.1]{Mochizuki11}. For a hypersurface $H$ of $X$, the \index{localization!``stupid''}``stupid'' localization functor $(*H)$ sends $\cR_\cX$\nobreakdash-modules to $\cR_\cX(*H)$-modules and $\iota$-sesquilinear pairings with values in $\Db_{\cX^\circ/\CC^*_\hb}$ to $\iota$\nobreakdash-sesqui\-li\-near pairings with values in $\Db_{(\cX^\circ\moins\cH^\circ)/\CC^*_\hb}^\modH=\Db_{\cX^\circ/\CC^*_\hb}(*H)$. This gives rise to the category $\RdiTriples(X,(*H))$.

Given now an \emph{effective divisor}~$H$ and $H$-strictly specializable $\cR^\intt_\cX$-modules, there is a notion of $H$-\emph{localizability} for these modules, to which one can apply \index{localization}\index{localization!dual}localization functors $[\star H]$ ($\star=*,!$). An integrable $\iota$-sesquilinear pairing is then automatically localizable, a statement which is proved exactly as in \cite[Prop.\,3.2.1]{Mochizuki11}, so the functors $[\star H]$ are defined on objects of $\RdiTriples(X)$ as soon as the $\cR^\intt_\cX$-module components are localizable. The full subcategory $\RdiTriples(X,[\star H])$ of $\RdiTriples(X)$ consists of objects $\icT$ satisfying $\icT\simeq\icT[\star H]$ ($\star=*,!$).
\end{remarque}

\vspace*{-5pt}
The following is now straightforward.\enlargethispage{\baselineskip}%

\begin{proposition}
The functors considered above are compatible with the restriction $\RdiTriples(X)\mto\RdTriples(X)$.\qed
\end{proposition}

\begin{remarque}[Beilinson's construction and gluing]
Beilinson's construction and gluing procedure along a function $\fun:X\to\CC$, as explained in \cite[\S4.2]{Mochizuki11} for the category $\RdTriples(X)$ can be extended in a straightforward way to $\RdiTriples(X)$, and the restriction functor $\RdiTriples(X)\mto\RdTriples(X)$ is compatible with this construction, as well as with the associated definition of vanishing cycles as done in \cite[\S4.2.3]{Mochizuki11}.
\end{remarque}

\begin{remarque}[Variations of mixed twistor structures]\label{rem:varMTS}
Let $H$ be a hypersurface in~$X$. The notion of variation of integrable mixed twistor structure on $(X,H)$ is defined in \cite[\S9.1.1]{Mochizuki11}, giving rise to the full subcategory $\MTS^\intt(X,H)$ of $\WRdTriples(X,(*H))$: the only condition is that, on $X\moins H$, the corresponding integrable object is a (smooth) variation of mixed twistor structure. A similar definition leads to $\iMTS^\intt(X,H)$ as a full subcategory of $\WRdiTriples(X,(*H))$, and the restriction functor\vspace*{-3pt}
\begin{starequation}\label{eq:varMTS}
\WRdiTriples(X,(*H))\mto\WRdTriples(X,(*H))
\end{starequation}%
sends $\iMTS^\intt(X,H)$ to $\MTS^\intt(X,H)$. Moreover, one can also define $\iMTS^\intt(X,H)$ as the full subcategory of $\WRdiTriples(X,(*H))$ whose objects have restriction \eqref{eq:varMTS} in the subcategory $\MTS^\intt(X,H)$ of $\WRdTriples(X,(*H))$.

On the other hand, the full subcategory $\MTS_\adm^\intt(X,H)$ of $\MTS^\intt(X,H)$ consisting of admissible variations of mixed twistor structures is defined in two steps. Firstly, if $H=D$ is a divisor with normal crossings, $\MTS_\adm^\intt(X,D)$ is defined in \cite[\S9.1.5]{Mochizuki11}. For $H$ general, the definition is given in \cite[\S11.3.2]{Mochizuki11} by pushing forward objects of $\MTS_\adm^\intt(X',D')$ by a suitable projective modification $(X',D')\to(X,H)$.

With the method already used, we define the category $\iMTS_\adm^\intt(X,H)$ as the full subcategory of $\WRdiTriples(X,(*H))$ whose objects restrict via \eqref{eq:varMTS} to an object of $\MTS_\adm^\intt(X,H)$. By definition, \eqref{eq:varMTS} induces a functor $\iMTS_\adm^\intt(X,H)\to\MTS_\adm^\intt(X,H)$.
\end{remarque}

\begin{remarque}[Compatibility with localization]\label{rem:localization}
Let $H$ be an effective divisor in $X$. Recall (\cf\cite[\S11.2.2]{Mochizuki11}) that the category \index{$MTMXintH$@$\MTM^\intt(X,[\star H])$}$\MTM^\intt(X,[\star H])$ is the full subcategory of $\MTM^\intt(X)$ consisting of objects satisfying $\cT[\star H]\simeq\cT$. We can define similarly \index{$IMTMXintH$@$\iMTM^\intt(X,[\star H])$}$\iMTM^\intt(X,[\star H])$ as a full subcategory of $\iMTM^\intt(X)$, the latter being given by Definition \ref{def:iotamtm}, and there is a restriction functor
\begin{starequation}\label{eq:localization}
\iMTM^\intt(X,[\star H])\mto\MTM^\intt(X,[\star H])\quad(\star=*,!).
\end{starequation}%
If we know the equivalence in Theorem \ref{th:equivalenceintegrable}, we deduce easily that the previous functor \eqref{eq:localization} is also an equivalence. However, the proof of the Theorem \ref{th:equivalenceintegrable} will go through Beilinson's construction, which essentially means proving first that the restriction functor \eqref{eq:localization} at the localized level is an equivalence.

On the other hand, tensoring with $\cR_\cX(*H)$ induces functors
\begin{align*}
\MTM^\intt(X,[*H])&\mto\WRdTriples(X,(*H)),\\
\tag*{and} \iMTM^\intt(X,[*H])&\mto\WRdiTriples(X,(*H)).
\end{align*}
Moreover, the first functor is fully faithful (\cf \cite[Lem.\,7.1.43\,\&\,Lem.\,11.2.4]{Mochizuki11}). For the sake of simplicity, we will denote by \index{$MTMXintHstupid$@$\MTM^\intt(X,(*H))$}$\MTM^\intt(X,(*H))$ its essential image, so that we have an equivalence
\[
\MTM^\intt(X,[*H])\simeq\MTM^\intt(X,(*H)).
\]
As a particular case of \cite[Prop.\,11.3.3]{Mochizuki11} we obtain that $\MTS_\adm^\intt(X,H)$, the category of admissible integrable variations of mixed twistor structures on $(X,H)$, is a full subcategory of $\MTM^\intt(X,(*H))$.

Once we know that \eqref{eq:localization} is an equivalence, we deduce that the functor
\begin{starstarequation}\label{eq:stupidloc}
\iMTM^\intt(X,[*H])\mto\WRdiTriples(X,(*H))
\end{starstarequation}%
is fully faithful, and we denote its essential image by \index{$IMTMXintHstupid$@$\iMTM^\intt(X,(*H))$}$\iMTM^\intt(X,(*H))$, which is therefore equivalent to $\MTM^\intt(X,(*H))$. We also conclude that $\iMTS_\adm^\intt(X,H)$ is a full subcategory of $\iMTM^\intt(X,(*H))$. Let us end this remark by noticing that, in the category $\iMTM^\intt(X,(*H))$, the weights produced by the poles along $H$ are irrelevant, that is, for an object $\icT$ in $\iMTM^\intt(X,[*H])$, we have $W_\bbullet[\icT(*H)]=[W_\bbullet\icT](*H)$.
\end{remarque}

\begin{remarque}[Compatibility with the restriction to $\hb=1$]\label{rem:compfunctorshb1}
Let us notice that the functors introduced on the category $\iMTM^\intt(X)$ are compatible with the similar functors on $\Mod_\hol(\cD_X)$ by the functor $\Xi_{\DR}:\icT\mto\cM''/(\hb-1)\cM''=\ccM$ (Notation \ref{nota:XiDR}). This is mostly obvious for the pushforward by a projective morphism, the pullback by a smooth morphism and the external tensor product, due to preservation of strictness by these functors. For nearby cycles, this is due to strict specializability and local unitarity (\cf\cite[Prop.\,3.3.14]{Bibi01c}), and this implies the property for $[\star H]$ (\cf also \cite[Cor.\,11.2.10]{Mochizuki11}). For duality, this is contained in \cite[Th.\,13.3.1]{Mochizuki11}.
\end{remarque}

\section{Proof of Theorem \ref{th:equivalenceintegrable}}\label{sec:pfequivalenceintegrable}

\subsection{The smooth case}\label{subsec:smoothcase}
The category of smooth integrable $\RTriples$ on $X$ is the full subcategory of $\RdTriples(X)$ whose objects $(\cM',\cM'',C_{\bS})$ are such that $\cM',\cM''$ are $\cO_\cX$-locally free of finite rank. We define in a similar way their $\iota$-version.

\begin{lemme}\label{lem:equivsmooth}
The natural functor $\RdiTriples^\smooth(X)\mto\RdTriples^\smooth(X)$ is an equivalence of categories.
\end{lemme}

This obviously implies Theorem \ref{th:equivalenceintegrable} for the full subcategory of smooth objects in $\iMTM^\intt(X)$.

\begin{proof}
We note that $\cM^{\prime\circ},\cM^{\prime\prime\circ}$ (restrictions of $\cM',\cM''$ to~$\cX^\circ:=X\times\CC_\hb^*$) are flat bundles, and giving $C_{\bS}$ is equivalent to giving its restriction $C_{\bS}^\nabla$ to the corresponding local systems on $X\times\bS$. The latter takes values in the constant sheaf, due to the integrability property. The lemma follows then from the property that the inclusion $X\times\bS\hto\cX^\circ$ is a homotopy equivalence, so the pairing
\[
C_{\bS}^\nabla:\cM^{\prime\nabla}_{\bS}\otimes_\CC\sigma^{-1}\ov{\cM^{\prime\prime\nabla}_{\bS}}\to\CC_{X\times\bS},
\]
that we can rewrite
\[
C_{\bS}^\nabla:\cM^{\prime\nabla}_{\bS}\otimes_\CC\iota^{-1}\ov{\cM^{\prime\prime\nabla}_{\bS}}\to\CC_{X\times\bS},
\]
determines in a unique way
\[
\iC^\nabla:\cM^{\prime\nabla}_{|\cX^\circ}\otimes_\CC\iota^{-1}\ov{\cM^{\prime\prime\nabla}_{|\cX^\circ}}\to\CC_{\cX^\circ},
\]
and then in a unique way an integrable $\iota$-sesquilinear pairing
\[
\iC:\cM^{\prime\circ}\otimes_\CC\iota^*\ov{\cM^{\prime\prime\circ}}\to\Db_{\cX^\circ/\CC^*_\hb}.
\]
This proves essential surjectivity. Full faithfulness is proved similarly.
\end{proof}

\subsection{The case of graded S-decomposable holonomic \texorpdfstring{$\cR_\cX$}{RcX}-modules}

Let us consider an object $(\cT,W)=((\cM',\cM'',C_{\bS}),W)$ of $\WRdTriples(X)$ such that, for each graded object $\gr_\ell^W\cT=(\gr_{-\ell}^W\cM',\gr_\ell^W\cM'',\gr_\ell^WC_{\bS})$ in the category $\RdTriples(X)$, the components $\gr_{-\ell}^W\cM',\gr_\ell^W\cM''$ are holonomic, integrable and strictly S-decomposable (\cf\cite[Def.\,3.5.1 \& Rem.\,7.3.5]{Bibi01c}).

\begin{lemme}\label{lem:iCSdec}
Given such an object $(\cT,W)$ in $\WRdTriples(X)$, there is at most one lifting $\iC$ of~$C_{\bS}$.
\end{lemme}

\begin{proof}
We can argue locally in the neighbourhood of any point $x_o\in X$. We argue by induction on the length of the filtration $W_\bbullet$. Let us assume first that $\cT$ is an object of $\RdTriples^\Sdec(X)$. Let $(Z_i)_{i\in I}$ be the family of strict components of $\cM',\cM''$ at $x_o$. Due to Remark \ref{rem:iCint}, the proof of \cite[Prop.\,3.5.8]{Bibi01c} can be applied to show that $\iC:\cM^{\prime\circ}_{Z_i}\otimes_\CC\iota^*\ov{\cM^{\prime\prime\circ}_{Z_j}}\to\Db_{\cX^\circ/\CC^*_\hb}$ is zero as soon as $Z_i\neq Z_j$. Let us now assume that $\cM',\cM''$ have the same strict support $Z$, which is closed irreducible in $(X,x_o)$. According to Kashiwara's equivalence \cite[Cor.\,3.3.12]{Bibi01c} and Section \ref{subsec:smoothcase} above, $C_{\bS}$ has at most one lift $\iC$ on some Zariski dense open set~$U$ of~$Z$. It is then enough to prove that, if some $\iota$-sesquilinear pairing $\iC$ between~$\cM^{\prime\circ}$ and $\cM^{\prime\prime\circ}$ vanishes on~$U$, then it is zero. We can assume that $Z\moins U$ is the zero set of some holomorphic function~$\fun$ on~$(X,x_o)$ and, up to changing $\cM',\cM''$ by their direct images by the graph inclusion~$i_\fun$, we can assume that $\fun$ is part of a coordinate system and we call it~$t$.

Since $\Db_{\cX^\circ/\CC^*_\hb}\subset\Db_{\cX^\circ}$, a local section $m'$ of $\cM'$ satisfies $\iC(m')=0$ (\cf Remark~\ref{rem:iCint}) if and only if $\iC(m')$ is zero in $\cHom_{\cD_{\ov\cX^\circ}}(\ov{\cM^{\prime\prime\circ}},\Db_{\cX^\circ})$. By our assumption, $\iC(m')$ is zero when restricted to $t\neq0$. For any local section $m''$ of~$\iota^*\cM''$, the distribution $\iC(m')(\ov{m''})$ is zero away from $t=0$ hence, up to shrinking the neighbourhood of $x_o$, is annihilated by some power of $t$. Applying this to a generating finite family of $\iota^*\cM''$ implies that $\iC(m')$ is annihilated by some power of $t$. At this point, we can argue exactly as in the end of the proof of \cite[Prop.\,3.5.8]{Bibi01c}.

We now assume that the filtration $W$ has arbitrary finite length and we argue by induction on its length. Assume for simplicity that $W_{\leq-1}=0$. Assume that~$\iC$ restricts to $C_{\bS}=0$ between $\cM'$ and $\cM''$. By the first part, the pairing $\iC$ between $W_0\cM'$ and $W_0\cM''$ is zero. Then the pairing $\iC$ between $W_0\cM'$ and $W_1\cM''$ comes from a pairing between $W_0\cM'$ and $\gr_1^W\!\!\cM''$, which is also zero by the first part. Continuing this~way, we conclude that the pairing induced by $\iC$ between $W_0\cM'$ and $\cM''$ or between~$\cM'$ and $W_0\cM''$ is zero, so we can replace $\cM',\cM''$ by $\cM'/W_0\cM',\cM''/W_0\cM''$ and conclude by induction that $\iC$ is zero.
\end{proof}

\begin{remarque}
As a consequence of Lemma \ref{lem:iCSdec}, the essential surjectivity statement in Theorem \ref{th:equivalenceintegrable} (hence the equivalence statement) is a local statement on $X$.
\end{remarque}

\subsection{The case of admissible mixed twistor structures}\label{subsec:proofadmissibleintegrable}
We start from a pair $(X,D)$, where $D$ is a divisor with normal crossings in a complex manifold $X$ (the setting is local near a point of $D$, so we can as well work with coordinates). Let $\cT=((\cM',\cM'',C_{\bS}),W_\bbullet)$ be an admissible variation of mixed twistor structure on~$(X,D)$ which is good wild (\cf\cite[\S9.1.2]{Mochizuki11}). The condition (Adm0) of \loccit ensures the local uniqueness of our construction of $\iC$, according to Lemma \ref{lem:iCSdec}, and we will not use it for another purpose. We will now focus on the condition (Adm1), asserting that~$\cT$ is a smooth good KMS-$\cR_{\cX(*D)}$-triple, and that the filtration $W_\bbullet$ is compatible with the corresponding structures. In particular, $C_{\bS}$ takes values in the sheaf \index{$DbXSmod$@$\Db_{X\times\bS/\bS}^\modD$}$\Db_{X\times\bS/\bS}^\modD$ of moderate distributions along $D$ (\cf\cite[\S2.1.4]{Mochizuki11}).

\begin{lemme}\label{lem:liftDmod}
There exists a $\iota$-sesquilinear pairing $\iC$ with values in $\Db_{\cX^\circ/\CC_\hb^*}^\modD$ lifting~$C_{\bS}$.
\end{lemme}

\begin{proof}
It is enough to prove that $\iC$, as given by Lemma \ref{lem:equivsmooth} on $\cX\moins\cD$, takes values in $\Db_{\cX^\circ/\CC_\hb^*}^\modD$ when applied to local sections of $\cM'_{|\cD^\circ},\cM''_{|\cD^\circ}$. The question is local on~$D$ and, by taking a local covering of $X$ ramified along the components of $D$, one can assume that the local set of exponential factors (irregular values) of $\cM',\cM''$ is non-ramified. Recall that the exponential factors take the form $\gota(\bmx)/\hb$ with $\gota\in\bmx^{-1}\CC[\bmx^{-1}]$ (where $\bmx$ is a local coordinate system on $X$ adapted to $D$), and they form a good family. We work locally at $(x_o,\hb_o)\in\cD^\circ$.

Let $\varpi:\wt\cX:=\wt\cX(D)\to\cX$ denote the real blow-up of $\cX$ along the components of $D\times\CC_\hb$ and let us denote by $\wt\cX^\circ$ its restriction above $\CC^*_\hb$. We have an open inclusion morphism $\wtj:\cX^\circ\moins\cD^\circ\hto\wt\cX^\circ$ whose complementary closed inclusion is denoted by $\wti:\partial\wt\cX^\circ\hto\wt\cX^\circ$. Let $\cA_{\wt\cX}$ denote the sheaf of $C^\infty$ functions on $\wt\cX$ which are holomorphic on $\wt\cX\moins\varpi^{-1}(D\times\CC_\hb)$ and let $\cA_{\wt\cX}^{\rdD}$ be the subsheaf of functions having rapid decay along $\varpi^{-1}(D\times\CC_\hb)$. It is enough to prove the proposition on $\wt\cX^\circ$ by considering
\[
\wt\cM^\circ:=\cA_{\wt\cX^\circ}\otimes_{\varpi^{-1}\cO_{\cX^\circ}}\cM^\circ\quad(\cM=\cM',\cM'').
\]

We now regard $\cM^\circ$ as a $\cD_{\cX^\circ}$-module, due to integrability. The pushforward $\cL:=\wtj_*(\cM^\circ_{|\cX^\circ\moins\cD^\circ})^\nabla$ of $(\cM^\circ_{|\cX^\circ\moins\cD^\circ})^\nabla$ to $\wt\cX^\circ$ is a local system on $\wt\cX^\circ$ whose \hbox{restriction} $\wti^{-1}\cL$ is equipped with a Stokes filtration $\cL_\bbullet$ which is non-ramified by our assumption and indexed by the good finite family $(\gota)$ (\cf\cite[Chap.\,4]{Mochizuki08}, \cite{Mochizuki10b} and \cite[\S9.e]{Bibi10}). We denote with an index $\bS$ the objects restricted to $\wt X\times\bS$.

Since $C_{\bS}$ is integrable, it induces a $\CC$-bilinear pairing on $(X\moins D)\times\bS$:
\[
C_{\bS}^\nabla:\wtj^{-1}\cL'_{\bS}\otimes_{\CC_{(X\moins D)\times\bS}}\sigma^{-1}\wtj^{-1}\ov{\cL''}_{\bS}\to\CC_{(X\moins D)\times\bS}.
\]

\begin{lemme}
The pairing $C_{\bS}^\nabla$ extends in a unique way as a pairing
\[
\cL'_{\bS}\otimes_{\CC_{\wt X\times\bS}}\sigma^{-1}\ov{\cL''}_{\bS}\to\CC_{\wt X\times\bS}
\]
and this extension induces a pairing of Stokes-filtered local systems
\[
(\wti^{-1}\cL'_{\bS},\cL'_{\bS,\bbullet})\otimes_{\CC_{\wt X\times\bS}}\sigma^{-1}(\wti^{-1}\ov{\cL''}_{\bS},\ov{\cL''}_{\bS,\bbullet})\to(\CC_{\partial\wt X\times\bS},\CC_{\partial\wt X\times\bS,\bbullet}),
\]
where the Stokes filtration $\CC_{\partial\wt X\times\bS,\bbullet}$ is the trivial one, \ie such that $\CC_{\partial\wt X\times\bS}=\CC_{\wt X\times\bS,\leq0}$ and for each open stratum $D_I^\circ$ of the natural stratification of $D$, denoting by $\wt D_I^\circ:=\varpi^{-1}(D_I^\circ)$, we have $\CC_{\wt D_I^\circ\times\bS,<0}=0$.
\end{lemme}

\begin{proof}
The first assertion is obvious since the inclusion $X\moins D\hto\wt X$ is a homotopy equivalence. The second assertion follows from \cite[Lem.\,5.5.1]{Mochizuki11}.
\end{proof}

\begin{corollaire}\label{cor:liftDmod}
There exists a unique $\CC$-bilinear pairing
\[
\iC^\nabla:(\cL',\cL''_\bbullet)\otimes_\CC\iota^{-1}(\ov{\cL''},\ov{\cL''}_\bbullet)\to(\CC_{\wt\cX^\circ},\CC_{\partial\wt\cX^\circ,\bbullet})
\]
whose restriction to $\wt X\times\bS$ is $C_{\bS}^\nabla$.
\end{corollaire}

\begin{proof}
The existence on $\cX^\circ\moins\cD^\circ$ is clear, since we can regard $C_{\bS}^\nabla$ as a morphism from the dual local system $\cL^{\prime\vee}_{\bS}$ to $\sigma^{-1}\ov{\cL''}_{\bS}=\iota^{-1}\ov{\cL''}_{\bS}$, and the inclusion
\[
(X\moins D)\times\bS\hto(X\moins D)\times\CC^*_\hb
\]
induces an isomorphism of fundamental groups $\pi_1$.

On $\partial\wt\cX^\circ$, we regard $C_{\bS}^\nabla$ as a morphism of good Stokes-filtered local systems, and we use the equivalence of categories \cite[Th.\,3.9]{Mochizuki10b} for the inclusion $\partial\wt X\times\bS\hto\partial\wt\cX^\circ$, which is easy in this case since the deformation of each exponential factor $\gota/\hb=\gota\mathrm{e}^{-\sfi\vartheta}$ is written as $\varrho\cdot\gota\mathrm{e}^{-\sfi\vartheta}$, where $\varrho$ varies in $\RR_+^*$, so not only goodness is preserved in this deformation, but also the order between the various pairs of exponential factors is strictly preserved.
\end{proof}

\oldsubsubsection*{End of the proof of Lemma \ref{lem:liftDmod}}
Let us fix a point $(\wt x_o,\hb_o)\in\partial\wt\cX^\circ$ above $(x_o,\hb_o)$. By the generalized Hukuhara-Turrittin theorem (\cf \cite[Chap.\,3]{Mochizuki11}, \cite[Th.\,12.5]{Bibi08}) there exists near $(\wt x_o,\hb_o)$ an isomorphism
\[
\wt\cM^\circ_{(\wt x_o,\hb_o)}\simeq\bigoplus_{\gota,\alphag}\bigl[\cA_{\wt\cX,(\wt x_o,\hb_o)}(*D)\otimes_{\varpi^{-1}\cO_{\cX,(x_o,\hb_o)}}(\cE^{\gota/\hb}\otimes\cF_{\gota,\alphag})\bigr],
\]
where $\gota$ varies in the finite good family above, $\cE^{\gota/\hb}\!:=\!\bigl(\cO_{\cX^\circ}(*D),\rd\!+\nobreak\rd(\gota/\hb)\bigr)$, \hbox{$\alphag\!\in\![0,1)^n$} and $\cF_{\gota,\alphag}$ is a free $\cO_{\cX^\circ}$-module with flat logarithmic connection whose residue along $D_i$ has eigenvalue~$\alpha_i$. There exists thus a horizontal basis \hbox{$\epsilong^{(\hb_o)}\!=\!(\epsilong^{(\hb_o)}_{\gota,\alphag})$} of $\cL$ on some neighbourhood of $(\wt x_o,\hb_o)$ in $\wt\cX^\circ$ and an $\cA_{\wt\cX,(\wt x_o,\hb_o)}(*D)$-basis $\bme^{(\hb_o)}=(\bme^{(\hb_o)}_{\gota,\alphag})$ of $\wt\cM^\circ_{(\wt x_o,\hb_o)}$ such that
\[
\bme^{(\hb_o)}_{\gota,\alphag}=\epsilong^{(\hb_o)}_{\gota,\alphag}\cdot\mathrm{e}^{-\gota/\hb}\bmx^{-\alphag\id+\rN},
\]
where $\rN=\rN_{\gota,\alphag}$ is nilpotent.

The compatibility of $\iC^\nabla$ with the Stokes filtration asserted by Corollary \ref{cor:liftDmod} is equivalent to the vanishing, for any pair $(\gota,\gotb)$ of exponential factors, of the entries of the matrix $\iC^\nabla(\epsilong^{\prime(\hb_o)}_{\gota,\alphag},\ov{\epsilong^{\prime\prime(-\hb_o)}_{\gotb,\betag}})$ as soon as $\reel(\gota/\hb-\gotb/\hb)<0$ on some nonempty set containing $(\wt x_o,\hb_o)$ in its closure (that we use $-\gotb/\hb$ is due to the involution $\iota$).

On the other hand, we will show that the entries of each matrix $\iC(\bme^{\prime(\hb_o)}_{\gota,\alphag},\ov{\bme^{\prime\prime(-\hb_o)}_{\gotb,\betag}})$ are $L^1_\loc$ and $C^\infty$ with respect to $\hb$. Such a matrix can be written~as
\[
\mathrm{e}^{\ov{\gotb/\hb}-\gota/\hb}\cdot\bmx^{-\alphag}\cdot\ov\bmx^{-\betag}\cdot\bmx^{\rN_{\gota,\alphag}^\mathrm{t}}\cdot\iC^\nabla(\epsilong^{\prime(\hb_o)}_{\gota,\alphag},\ov{\epsilong^{\prime\prime(-\hb_o)}_{\gotb,\betag}})\cdot\ov\bmx^{\rN_{\gotb,\betag}},
\]
where $\rN^\mathrm{t}$ is the transpose of $\rN$. By the compatibility of $\iC^\nabla$ with the Stokes filtration, this matrix is nonzero only if $\mathrm{e}^{\ov{\gotb/\hb}-\gota/\hb}$ is bounded near $(\wt x_o,\hb_o)$. Since the entries of $\iC^\nabla(\epsilong^{\prime(\hb_o)}_{\gota,\alphag},\ov{\epsilong^{\prime\prime(-\hb_o)}_{\gotb,\betag}})$ are constant in any case, the $L^1_\loc$ property follows from $\alpha_i,\beta_i\in\nobreak[0,1)$ for all $i=1,\dots,d_X$, and the $C^\infty$ behaviour with respect to $\hb$ follows from the same property for $\mathrm{e}^{\ov{\gotb/\hb}-\gota/\hb}$.
\end{proof}

\pagebreak[2]
\begin{remarque}[Full faithfulness]\label{rem:fullfaith}
In the present setting, full faithfulness is obvious. Indeed, given an integrable morphism $\lambda=(\lambda',\lambda''):\cT_1\to\cT_2$ in $\MTM^\intt(X)$ (that is, in $\WRdTriples(X)$), Lemma \ref{lem:equivsmooth} shows that, for local sections $m'_2,m''_1$ of $\cM'_2,\cM''_1$, the two moderate distributions $\iC(\lambda'(m'_2),\iota^*\ov{m''_1})$ and $\iC(m'_2,\iota^*\ov{\lambda''(m''_1)})$ coincide on $\cX^\circ\moins\cD^\circ$, hence they coincide, since $\Db_{\cX^\circ/\CC^*_\hb}^\modD$ is nothing but the image of the restriction morphism $\Db_{\cX^\circ/\CC^*_\hb}\to j_*\Db_{(\cX^\circ\moins\cD^\circ)/\CC^*_\hb}$.
\end{remarque}

\Subsection{End of the proof of equivalence in Theorem \ref{th:equivalenceintegrable}}\label{subsec:proofequivintegrable}

\subsubsection*{Equivalence in the localized case}

Let $\cT$ be an object of $\MTM^\intt(X)$ and let us denote by $Z$ its support. Let us fix $x_o\in Z\subset X$. Up to replacing $X$ with a neighbourhood of~$x_o$, there exists a hypersurface $H$ of $X$ and a projective modification $\modif:Z'\to Z$, that we regard also as a morphism $\modif:Z'\to X$, such that $Z'$ is smooth, $H':=\modif^{-1}(H)$ is a divisor with normal crossing in $Z'$, and $\modif:Z'\moins H'\to Z\moins H$ is an isomorphism, and an integrable admissible mixed variation of mixed twistor structure~$\cV$ on $(Z',H')$, such that $\cT(*H)=\modif_\dag^0\cV$ (\cf\cite[Prop.\,11.1.1]{Mochizuki11}). We can therefore apply Lemma \ref{lem:liftDmod} to $\cV$, together with the pushforward by $\modif^0_\dag$ (and~the compatibility between pushforward and restriction \ref{lem:iCCS}), to obtain the existence of~$\iC$ with values in $\Db_{\cX^\circ/\CC^*}^\modD$ lifting $C_{\bS}$ on $\cT(*H)$.

By the argument given in Remark \ref{rem:fullfaith}, we conclude that any morphism $\cT_1(*H)\to\cT_2(*H)$ in $\WRdTriples(X,(*H))$ is compatible with $\iC$.

\subsubsection*{Equivalence: the general case}
As mentioned in Remark \ref{rem:lociota}, the prolongations $[!H]$ and $[*H]$ of strictly specializable $\cR_{\cX(*H)}$-triples also apply, in a way compatible with restriction \ref{lem:iCCS}, to strictly specializable integrable $\cR_{\cX(*H)}$-$\iota$-triples (recall that strict specializability is a property of $\cM',\cM''$ only). It follows that the construction of the Beilinson functor and the corresponding exact sequences (\cf\cite[\S4.2]{Mochizuki11}) can be obtained in this category.

Let $\cT$ and $H$ be as in the localized case above, and assume that $H$ is defined by a local equation $h=0$. By \cite[\S11.1.4]{Mochizuki11} (which we refer to for the notation), we~can recover $\cT$ by a gluing procedure from the data of $\cT(*H)$ and of an object~$\cT'$ of $\MTM^\intt(X)$ supported on $H$, together with integrable morphisms $u,v$:
\[
\psi_h^{(1)}\cT(*H)\To{u}\cT'\To{v}\psi_h^{(0)}\cT(*H).
\]
We know that $\cT(*H)$ has a lifting in $\RdiTriples(X,*H)$, hence $\psi_h^{(1)}\cT(*H)$ and $\psi_h^{(0)}\cT(*H)$ have a natural lifting in $\RdiTriples(X)$. Moreover, by induction on the dimension of the support, so has $\cT'$, and $u,v$ are also lifted. It follows that $\mathrm{Glue}(\cT(*H),\cT',u,v)$ can also be lifted to $\RdiTriples(X)$ (\cf\cite[\S4.2.4]{Mochizuki11} for the definition of the object $\mathrm{Glue}(\cT(*H),\cT',u,v)$).

Since the $\mathrm{Glue}$ functor is also fully faithful, we obtain similarly the full faithfulness in Theorem \ref{th:equivalenceintegrable} by using the full faithfulness in the localized case above together with the induction hypothesis on the dimension of the support, to treat $(\cT',u,v)$.\qed

\subsection{Construction of the duality functor and compatibility by restriction}
In order to end the proof of the various compatibilities asserted in Theorem \ref{th:equivalenceintegrable}, it remains to construct the duality functor on $\iMTM^\intt(X)$ and to prove its compatibility with that of $\MTM^\intt(X)$ by the restriction~\ref{lem:iCCS}. We follow the construction in \cite[Th.\,13.3.1]{Mochizuki11}.

\skpt
\begin{definition}[Non-degenerate objects of $\RdiTriples(X)$]\label{def:nondeg}
Let $\icT=(\cM',\cM'',\iC)$ be an object of $\RdiTriples(X)$. We say that it is \emph{non-degenerate} if the following properties hold:
\begin{enumeratea}
\item\label{def:nondeg1}
Regarded as a $\cD_{\cX^\circ}$-$\iota$-sesquilinear pairing with values in $\Db_{\cX^\circ}$, $\iC$ is non-degenerate, that is, the induced $\cD_{\cX^\circ}$-linear morphism
\[
\cM^{\prime\circ}\to\cHom_{\ov\cD_{\cX^\circ}}(\iota^*\ov{\cM^{\prime\prime\circ}},\Db_{\cX^\circ})
\]
is an isomorphism.
\item\label{def:nondeg2}
For each $\hb_o\in\CC^*_\hb$, the induced $\cD_X$-sesquilinear pairing $\iC^{\hb_o}$ (\cf Remark \ref{rem:iChbo}) is non-degenerate, that is, the induced $\cD_X$-linear morphism
\[
\cM^{\prime\hb_o}\to\cHom_{\cD_{\ov X}}(\ov{\cM^{\prime\prime-\hb_o}},\Db_X)
\]
is an isomorphism.
\end{enumeratea}
\end{definition}

\begin{proposition}
Any object $\icT$ of $\iMTM^\intt(X)$ is non-degenerate.
\end{proposition}

\skpt
\begin{proof}[Sketch of proof]
\ref{def:nondeg}\eqref{def:nondeg1} We argue by induction as in Sections \ref{subsec:proofadmissibleintegrable} and \ref{subsec:proofequivintegrable}. We first consider the notion of non-degeneracy for $\cT(*H)$ where $\iC$ takes values in moderate distributions $\Db^\modH_{\cX^\circ\moins\cH^\circ}$. In this setting, a $\iota$-sesquilinear pairing is non-degenerate as soon as its restriction to $X\moins H$ is so. Therefore, the $\iota$-sesquilinear pairing obtained in the first part of Section \ref{subsec:proofequivintegrable} is non-degenerate. Arguing now as in the proof of \cite[Lem.\,12.2.1]{Mochizuki11}, one finds that $\icT[*H]$ and $\icT[!H]$ are non-degenerate. By applying \cite[\S12.2.4]{Mochizuki11} together with Remark \ref{rem:lociota}, and using the argument of the second part of Section \ref{subsec:proofequivintegrable}, one obtains the desired property.

\ref{def:nondeg}\eqref{def:nondeg2} The argument is similar, if we note that for any $\hb_o\in\CC^*_\hb$, the restriction to $\hb_o$ is compatible with all operations above, by strictness and the fact that, in the integrable case, any such $\hb_o$ is non-singular in the sense of \cite[\S09]{Bibi01c} (\cf\cite[Prop.\,3.3.14(3)]{Bibi01c}).
\end{proof}

Let $(\icT,W_\bbullet)$ be an object of $\iMTM^\intt(X)$. It follows from \cite[Th.\,13.3.1]{Mochizuki11} that, setting $\cM=\cM'\text{ or }\cM''$, we have
\[
\bD\cM:=\bR\cHom_{\cR_\cX}(\cM,\cR_\cX\otimes\hb^{d_X}\omega_\cX^{-1})[d_X]=\cH^0\bD\cM.
\]
Moreover, the dual of the restriction $\cT$ is well-defined and is integrable. By the arguments similar to those of Section \ref{subsec:proofequivintegrable}, one obtains in a way similar to \cite[Th.\,13.3.1]{Mochizuki11}:

\begin{proposition}
Let $(\icT,W_\bbullet)$ be an object of $\iMTM^\intt(X)$. Then there exists a unique integrable $\iota$-sesquilinear pairing $\bD\iC$ between $\bD\cM'$ and $\bD\cM''$, such that, for each $\hb_o\in\CC^*_\hb$, $(\bD\iC)^{\hb_o}=\bD(\iC^{\hb_o})$, where the latter is given by \ref{def:nondeg}\eqref{def:nondeg2} and \cite[Th.\,14.4.1]{Mochizuki11}. Moreover, $\bD\iC$ restricts to $\bD C_{\bS}$ as given by \cite[Th.\,13.3.1]{Mochizuki11}. Lastly, $\bD\iC$, regarded as a $\cD_{\cX^\circ}\otimes_\CC\ov\cD_{\cX^\circ}$-sesquilinear pairing \hbox{between}~$\bD\cM^{\prime\circ}$ and~$\iota^*\bD\cM^{\prime\prime\circ}$ with values in $\Db_{\cX^\circ}$, is equal to the dual of $\iC$ as given by \cite[Th.\,12.4.1]{Mochizuki11}.\qed
\end{proposition}

\section{Real and good \texorpdfstring{$\kk$}{k}-structures}\label{sec:realratMTM}
In this section, we review the notion of real and good $\kk$-rational structure (when $\kk$ is a subfield of $\RR$) for an object of $\MTM(X)$ and $\MTM^\intt(X)$, as defined in \cite[\S13.4]{Mochizuki11}. In view of the preceding results, the extension of such real and good $\kk$-rational structures to $\iMTM^\intt(X)$ is straightforward.

In a way similar to \eqref{eq:defdualsigmarctriples} we set, for an object $(\cT,W)=((\cM',\cM'',C_{\bS}),W)$ of $\MTM(X)$,\vspace*{-.8\baselineskip}
\begin{equation}\label{eq:defdualsigmarctriplesX}
\begin{cases}
\bbullet\ \cT^\vee=\bD\cT,\\
\bbullet\ \cT^*=(\cM'',\cM',C_{\bS}^*),\\
\bbullet\ \cT^\rc=\iota^*(\cT^\vee)^*=\iota^*(\cT^*)^\vee.
\end{cases}
\end{equation}
These three functors preserve the subcategories $\MTM^\intt(X)$ and $\iMTM^\intt(X)$, as recalled or explained above. Moreover, they behave in a natural way with respect to the weight filtration $W$. In particular, we have $\gr_\ell^W(\cT^\rc)=(\gr_\ell^W\cT)^\rc$.

We denote by $\bT_X(\ell)$ the pullback of $\bT(\ell)$ (\cf Examples \ref{exem:Tate} and \ref{exem:Tatetriple}) to $X$. Let us recall (\cf \cite[Def.\,3.1.2]{Bibi06b} and \cite[Def.\,17.1.9]{Mochizuki08}) that, for a pure object $\cT$ of weight $w$, a polarization~$\cS$ is a Hermitian duality of weight $w$, that is, a morphism $\cS:\cT\to\cT^*\otimes\bT_X(-w)$ such that $\cS^*=(-1)^w\cS$ (by using $\kappa_{-2w}$, \cf\eqref{eq:defdualsigmarcTate}) which satisfies the polarizability property after applying exponentially twisted nearby cycles along any germ of holomorphic function on $X$. The functors \eqref{eq:defdualsigmarctriplesX} preserve purity and also act in a natural way on polarizations (\cf the argument in \cite[\S13.3.8]{Mochizuki11}).

We can copy Definition \ref{def:realTS} and Remark \ref{rem:realTS}\eqref{rem:realTS2}, to obtain the definition of \index{$MTMXR$@$\MTM(X,\RR)$}$\MTM(X,\RR)$, \index{$MTMXRint$@$\MTM^\intt(X,\RR)$}$\MTM^\intt(X,\RR)$ and \index{$IMTMXRint$@$\iMTM^\intt(X,\RR)$}$\iMTM^\intt(X,\RR)$. Moreover, the natural functor $\iMTM^\intt(X,\RR)\mto\MTM^\intt(X,\RR)$ is an equivalence of categories.

For an object $((\cM',\cM'',C_{\bS}),\kappa)$ of $\MTM^\intt(X,\RR)$, the localization $\cM''(*\hb)$ of~$\cM''$ is a holonomic $\cD_\cX$-module (\cf Proposition \ref{prop:speint}\eqref{prop:speint1}) and $\kappa$ endows its restriction to~$\cX^\circ$ with an $\RR$-Betti structure, and this is compatible with projective pushforward and duality (\cf\cite[Lem.\,13.4.8 \&\ Prop.\,13.4.9]{Mochizuki11}). A \emph{good real structure} consists of the enrichment of $(\cT,\kappa)$ with a real perverse sheaf $\cP_\RR$ satisfying the following compatibility properties (similar to that of Definition \ref{def:goodrealintegrable}):
\begin{enumerate}
\item\label{def:goodrealintegrableX1}
$\CC\otimes_\RR\cP_\RR\simeq\pDR\cM''(*\hb)$ and the restriction of this isomorphism to $\cX^\circ$ is compatible with the real structures naturally existing on both sides,
\item\label{def:goodrealintegrableX2}
the pre-$\RR$-Betti structure $\cP_\RR$ of $\cM''(*\hb)$ is an $\RR$-Betti structure (\cf\cite[\S7.2.1]{Mochizuki10}), \ie is compatible with the Stokes structure of $\cM''(*\hb)$.\end{enumerate}

For a subfield $\kk$ of $\RR$, the notion of a good $\kk$-structure is defined similarly. The categories \index{$MTMXintkgood$@$\MTM^\intt_\good(X,\kk)$}$\MTM^\intt_\good(X,\kk)$ and \index{$IMTMXintkgood$@$\iMTM^\intt_\good(X,\kk)$}$\iMTM^\intt_\good(X,\kk)$ are then defined as in Definition \ref{def:kgoodintMTS}.

\begin{remarque}[Compatibility with a polarization]
As in Section \ref{subsec:compRpol}, it would be natural to add the compatibility of the real structure with the polarization, in the polarized case, and with some polarization in the polarizable case. Let $(\cT,\cS)$ be a pure polarized twistor $\cD$-module of some weight. We say that a real structure $\kappa:\cT\isom\cT^\rc$ is compatible with $\cS$ if the diagram \eqref{eq:kappaS} commutes. However, in order to use the corresponding categories, one would need to check that the standard functors behave nicely with respect to the polarization, that is, the statements of \cite[Prop.\,13.4.6]{Mochizuki11} also hold with this enhanced version of $\MTM(X,\RR)$, $\MTM^\intt(X,\RR)$, and thus $\iMTM^\intt(X,\RR)$. As in Section \ref{subsec:compRpol}, the notion of a good $\RR$-structure (or~$\kk$\nobreakdash-struc\-ture) should be enriched by the following property, in the pure polarized case, and then extended to the polarizable pure case, and then the mixed case:
\begin{enumerate}\setcounter{enumi}{2}
\item
the duality $\cQ=\cS^\rc\circ\kappa:\cT\to\iota^*\cT^\vee\otimes\bT_X(-w)$ induces an isomorphism of holonomic $\cD_{X\times\CC_\hb}$-modules with $\RR$-Betti structures:
\[
(\cM''(*\hb),\cP_\RR)\isom\iota^*(\cM''(*\hb),\cP_\RR)^\vee\otimes(\twopii)^{-w}\RR_{X\times\CC_\hb}.
\]
(Let us recall that morphisms in the category $\mathrm{Hol}(\cX,\RR)$ consist simply of pairs of compatible morphisms of $\cD$-modules and real perverse sheaves, \ie the category $\mathrm{Hol}(\cX,\RR)$ is full in the category of holonomic $\cD$-modules with pre-Betti structure, \cf Def.\,7.2.2 and below in \cite{Mochizuki10}.)
\end{enumerate}
\end{remarque}

\section{Exponential twist}\label{sec:exptwist}

\subsection{The object $\icT^{\mero/\hb}$ in $\iMTM^\intt_\good(X,[*P],\QQ)$}\label{subsec:cTvarphihb}

Let $\mero$ be a meromorphic function on $X$ with reduced pole divisor $P\subset X$. We~set $U=X\moins P$. Let $\cO_\cX(*P)$ be the sheaf of meromorphic functions on $\cX$ with poles contained in the divisor $P\times\CC_\hb$. The $\hb$-connection $\hb\rd+\rd\mero$ makes it a coherent left $\cR_\cX(*P)$-module. Moreover, the action of $\hb^2\partial_\hb$ defined by $\hb^2\partial_\hb\cdot1:=-\mero$ enhances this structure to an $\cR^\intt_\cX(*P)$-module structure. It is obtain by exponentially twisting the standard structure $(\cO_\cX(*P),\hb\rd,\hb^2\partial_\hb)$, \ie by replacing $\hb\rd$ with \hbox{$\rme^{-\mero/\hb}\circ\hb\rd\circ\rme^{\mero/\hb}$} and $\hb^2\partial_\hb$ with $\rme^{-\mero/\hb} \circ\hb^2\partial_\hb\circ \rme^{\mero/\hb}$. Our aim is to construct an object~\index{$Tciphihb$@$\icT^{\mero/\hb}$}$\icT^{\mero/\hb}$ in $\iMTM^\intt_\good(X,[*P],\QQ)$ such that $(\cO_\cX(*P),\hb\rd+\rd\mero,\hb^2\partial_\hb)$ underlies the associated localized object $\icT_*^{\mero/\hb}:=\icT^{\mero/\hb}(*P)$ in $\iMTM^\intt_\good(X,(*P),\QQ)$. (Recall that the \index{localization!``stupid''}``stupid'' localization functor $(*P)$ is not the appropriate one in $\MTM(X),\MTM^\intt(X),\iMTM^\intt(X),$ and has to be replaced with the functor \index{localization}$[*P]$.) This will make more precise the construction in \cite[\S3]{S-Y14}.

\subsubsection*{The smooth case}
We first construct the object in $\iMTS^\intt(X,P)$ (\cf Remark \ref{rem:varMTS}). We denote by \index{$Ecphihbstar$@$\cE_*^{\mero/\hb}$}$\cE_*^{\mero/\hb}$ the $\cR_\cX(*P)$-module $\cO_\cX(*P)$ equipped with the $\hb$\nobreakdash-connection $\hb\rd+\rd\mero$. It is also equipped with the $\hb^2\partial_\hb$-action as above. Let $(\cC^\infty_U,\ov\partial,\sfh)$ be the trivial bundle with its standard holomorphic structure, equipped with its standard metric for which $\sfh(1,1)=1$. Consider it as a harmonic Higgs bundle on $U$ with holomorphic Higgs field $\theta=\rd\mero$. The associated flat bundle is $(\cO_U,\rd+\rd\mero)$. We associate with it an integrable triple $\cT^{\mero/\hb}_U=(\cE_U^{\mero/\hb},\cE_U^{\mero/\hb},C_{\bS})$ where, for $\hb\in\bS$,
\[
C_{\bS}(1,\ov1)=\exp(\mero/\hb-\hb\ov\mero)
\]
(\cf\cite[\S2.2]{Bibi04}). It follows that $\cT_U^{\mero/\hb}$ is an integrable polarized variation of smooth twistor structure of weight~$0$, equivalently an integrable pure polarized smooth twistor $\cD$-module of weight $0$. The polarization $\cS$ is equal to $(\id,\id)$. The associated $\iota$\nobreakdash-sesquilinear pairing is given by\vspace*{-3pt}
\[
\iC(1,\ov1)=\exp(\mero/\hb-\ov\mero/\ov\hb).
\]
If we stay away from $P$, the object thus defined is isomorphic to the integrable triple (or $\iota$-triple) associated with $(\cO_\cU,\hb\rd,\hb^2\partial_\hb)$.

On the other hand, we note that the functions $\exp(\mero/\hb-\hb\ov\mero)$ on $U\times\bS$ and $\exp(\mero/\hb-\ov\mero/\ov\hb)$ on $U\times\CC_\hb^*$ have moderate growth, as well as all their derivatives, along $P\times\bS$ and $P\times\CC_\hb^*$. They define thus moderate distributions depending continuously on $\hb$, and therefore\index{$Tcphihbstar$@$\cT_*^{\mero/\hb}$}\vspace*{-3pt}
\[
\cT_*^{\mero/\hb}:=(\cE_*^{\mero/\hb},\cE_*^{\mero/\hb},\exp(\mero/\hb-\hb\ov\mero))
\]
is an object of $\MTS^\intt(X,P)$, while\index{$Tciphihbstar$@$\icT_*^{\mero/\hb}$}
\[
\icT_*^{\mero/\hb}:=(\cE_*^{\mero/\hb},\cE_*^{\mero/\hb},\exp(\mero/\hb-\ov\mero/\ov\hb))
\]
is an object of $\iMTS^\intt(X,P)$.

We have
\[
C_{\bS}^\vee(1,\ov1)=\exp(-\mero/\hb+\hb\ov\mero),\quad C_{\bS}^*(1,\ov1)=C_{\bS}(1,\ov1),
\]
so that
\begin{equation}\label{eq:isoT*P}
(\cT_*^{\mero/\hb})^\vee=\cT_*^{-\mero/\hb},\quad
(\cT_*^{\mero/\hb})^*=\cT_*^{\mero/\hb},\quad
(\cT_*^{\mero/\hb})^\rc=\cT_*^{\mero/\hb}
\end{equation}
in $\MTS^\intt(X,P)$, and the latter equality is interpreted as an integrable real structure $\kappa:\cT_*^{\mero/\hb}\isom(\cT_*^{\mero/\hb})^\rc$. A similar statement holds for $\icT_*^{\mero/\hb}$.

\subsubsection*{The object in $\iMTM^\intt(X,[*P])$}
It follows from \cite{Mochizuki11} (\cf\cite[Prop.\,3.3]{S-Y14}) that $\cE_*^{\mero/\hb}$ is localizable along $P$ and the coherent $\cR_\cX$-module \index{$Ecphihb$@$\cE^{\mero/\hb}$}$\cE^{\mero/\hb}\!:=\!(\cE_*^{\mero/\hb})[*P]$ is integrable and underlies an object \index{$Tcphihb$@$\cT^{\mero/\hb}$}$\cT^{\mero/\hb}$ of $\MTM^\intt(X)$ extending $\cT_*^{\mero/\hb}$. By construction, it satisfies $\cT^{\mero/\hb}=\cT^{\mero/\hb}[*P]$, \ie belongs to $\MTM^\intt(X,[*P])$.

The case when $\mero$ is a morphism $X\to\PP^1$ is simpler than the general case. Then $\cE^{\mero/\hb}=\cE_*^{\mero/\hb}$ and $\cT^{\mero/\hb}$ is the canonical prolongation of $\cT_*^{\mero/\hb}$ in the sense of \cite{Mochizuki08}, and is an object of $\MTM^\intt(X)$ which is pure of weight $0$ and polarized by $\cS=(\id,\id)$ (\cf\cite[Prop.\,3.3]{S-Y14}).

In general, however, $\cT^{\mero/\hb}$ can be mixed, but the graded objects of nonzero degree with respect to the weight filtration are supported on $P$. Let $\modif:X'\to X$ be a projective modification such that $\modif^{-1}(P)$ has normal crossings and $\modif^*\mero$ is good, \ie its pole divisor $P'$ does not intersect its zero divisor. Set also $\modif^{-1}(P)=P'\cup H'$. Then $\cT^{\mero'/\hb}[*H']$ is an admissible mixed twistor structure on $(X,\modif^{-1}(P))$, in the sense of \cite[\S9.1.4]{Mochizuki11}, and $\cT^{\mero/\hb}=\modif_\dag^0\cT^{\mero'/\hb}[*H']$ (we omit to mention here the weight filtration).

By using the $\iota$-version, we obtain an object~$\icT^{\mero/\hb}$ of $\iMTM^\intt(X,[*P])$ extending $\icT_*^{\mero/\hb}$ and corresponding to $\cT^{\mero/\hb}$. If $\mero$ induces a morphism $X\to\PP^1$, it is pure of weight $0$.

\subsubsection*{Real structure and good $\QQ$-structure}
Let us recall that the \index{localization!``stupid''}``stupid'' localization functor $\MTM^\intt(X,[*P])\mto\MTM^\intt(X,(*P))$ is fully faithful (in fact, an equivalence, by the very definition of the category $\MTM^\intt(X,(*P))$, \cf Remark \ref{rem:localization}), as well as $\iMTM^\intt(X,[*P])\mto\iMTM^\intt(X,(*P))$. Since $\MTM^\intt_\adm(X,P)$ is a full subcategory of $\MTM^\intt(X,(*P))$ (\cf Remark \ref{rem:localization}), it follows that the isomorphisms \eqref{eq:isoT*P} can be lifted in a unique way, giving rise to a real structure $\kappa:\cT^{\mero/\hb}\isom(\cT^{\mero/\hb})^\rc$, and similarly with $\icT^{\mero/\hb}$.

Let us now consider the $\QQ$-structure. We first note that $\cE^{\mero/\hb}(*\hb)=\cE_*^{\mero/\hb}(*\hb)$. This is a consequence of the following argument. Let $\fun$ be a local equation of $P$ and let us use Notation \ref{nota:maps}. For an integrable $\cR_{\cX_\fun}$\nobreakdash-mod\-ule~$\cM$ which is strictly specializable along $t$, we have $(V_0\cM_\fun)(*\hb)=V_0(\cM_\fun(*\hb))$, where the latter filtration is the $V$-filtration with respect to $t$ on the $\cD_{\cX_\fun}[1/\hb]$-module $\cM_\fun(*\hb)$. Now, it is standard that $\cM_\fun(*t)(*\hb)=\cD_{\cX_\fun}[1/\hb]\cdot V_0(\cM_\fun(*\hb))$, and we use that, by definition, $\cM_\fun[*t]=\cR_{\cX_\fun}\cdot(V_0\cM_\fun)\subset\cM_\fun(*t)$ (\cf\cite[\S3.1.2]{Mochizuki11}), to conclude that $\cM_\fun[*t](*\hb)=\cM_\fun(*t)(*\hb)$. Lastly, when $\cM$ is an integrable $\cR_\cX$-module which is strictly specializable and twistor-localizable along $P$, in the sense that $\cM[*P]$ is defined, we deduce that $\cM[*\fun](*\hb)=\cM_\fun(*\fun)(*\hb)$. We apply this to $\cM=\cE^{\mero/\hb}$ to obtain the desired assertion.

It follows that $\cE^{\mero/\hb}(*\hb)$ is a meromorphic flat bundle on $\cX$ with poles along $(P\times\CC_\hb)\cup(X\times\{0\})$. We can find a projective modification $\varpi:\cX'\to\cX$ which is an isomorphism away from the poles, such that $\cE^{\mero/\hb}(*\hb)$ is the pushforward of a good rank-one meromorphic flat bundle $\cM$ on $\cX'$. In the latter case, it is easy to check that $\cM$ has a natural $\QQ$-Betti structure extending the natural $\QQ$-Betti structure defined away from the poles. One obtain the $\QQ$-Betti structure on $\cE^{\mero/\hb}(*\hb)$ by pushing forward the latter by $\varpi$.

Note that we can assume that, away from $\hb=0$, $\varpi$ is isomorphic (at least locally on $X$) to $\modif\times\id_{\CC_\hb^*}$, where $\modif:X'\to X$ is as above. The compatibility with the $\RR$\nobreakdash-structure $\kappa$ (Condition \eqref{def:goodrealintegrableX1} in Section \ref{sec:realratMTM}) is then easily checked on $\cX^{\prime\circ}$, and then can be obtained for $\cE^{\mero/\hb}(*\hb)$ by pushforward by $\modif\times\id$.

\subsection{Exponential twist in $\iMTM^\intt(X)$}\label{subsec:exptwist}
Given an object $\cT$ of $\MTM^\intt(X)$, one can define the exponential twist $\cT_*^{\mero/\hb}\otimes\nobreak\cT$ which is an object of $\MTM^\intt(X,(*P))$ (\cf\cite[Prop.\,11.3.3]{Mochizuki11}, recalling that $\cT_*^{\mero/\hb}=\modif_\dag^0\cT_0$, with $\cT_0:=\cT_*^{\mero'/\hb}(*H')$). For $\star=*,!$, we then obtain the objects $(\cT_*^{\mero/\hb}\otimes\nobreak\cT)[\star P]$, which are objects of $\MTM^\intt(X,[\star P])$ (\cf\loccit). By a similar procedure, if~$\icT$ is an object of $\iMTM^\intt(X)$, one defines $\icT_*^{\mero/\hb}\otimes\nobreak\icT$, which is an object of $\iMTM^\intt(X,(*P))$, and then $(\icT_*^{\mero/\hb}\otimes\icT)[\star P]$ are objects of $\iMTM^\intt(X,[\star P])$. If $\icT=(\cM',\cM'',\iC)$, (we omit here the weight filtration), the first and second components of $\icT_*^{\mero/\hb}\otimes\icT$ are $(\cM'(*P),\nabla'+\rd\mero)$ and \hbox{$(\cM''(*P),\nabla''+\rd\mero)$}, where $\nabla',\nabla''$ are the $\hb$-connections, the action $\hb^2\partial_\hb$ is changed to $\hb^2\partial_\hb-\mero$ and the pairing with moderate growth $\iC$ is multiplied by \hbox{$\exp(\mero/\hb-\ov\mero/\ov\hb)$}. Also, the first and second component of $(\icT_*^{\mero/\hb}\otimes\icT)[*P]$ are $(\cM'[!P],\nabla'+\rd\mero)$ and $(\cM''[*P],\nabla''+\rd\mero)$, and a symmetric statement for $(\icT_*^{\mero/\hb}\otimes\icT)[!P]$, so that the latter is the adjoint of the former. In the following, we will use the notation\vspace*{-3pt}
\begin{starequation}\label{eq:notationTvarphiotimes}
\begin{split}
\icT^{\mero/\hb}\otimes\icT&:=(\icT^{\mero/\hb}_*\otimes\icT)[* P],\\\
\Gamma_{[!P]}(\icT^{\mero/\hb}\otimes\icT)&:=(\icT^{\mero/\hb}_*\otimes\icT)[!P].
\end{split}
\end{starequation}%

\begin{remarque}[Compatibility with the restriction to $\hb=1$]\label{rem:compexptwisthb1}
Similarly to Remark \ref{rem:compfunctorshb1}, and using the notation \eqref{eq:notationTvarphiotimes}, we have (\cf Notation \ref{nota:XiDR})\vspace*{-3pt}
\[
\Xi_{\DR}(\icT^{\mero/\hb}\otimes\cbbullet)=\ccE^\mero\otimes\cbbullet\quad\text{and}\quad \Xi_{\DR}\Gamma_{[!P]}(\icT^{\mero/\hb}\otimes\cbbullet)=(\ccE^\mero\otimes\cbbullet)(!P).\vspace*{-5pt}
\]
\end{remarque}

\begin{remarque}[Exponential twist in $\iMTM^\intt(X,\RR)$ and $\iMTM^\intt_\good(X,\kk)$]\label{rem:exptwistRgood}
If $\cT$ is an object of $\MTM^\intt(X,\RR)$ \resp $\MTM^\intt_\good(X,\kk)$, then so are $\cT^{\mero/\hb}\otimes\cT$ and $\Gamma_{[!P]}(\cT^{\mero/\hb}\otimes\cT)$, as follows by adapting the last point of \cite[Prop.\,13.4.6]{Mochizuki11}, according to the results in Section \ref{subsec:cTvarphihb}. A similar result holds for the corresponding object $\icT$ of $\iMTM^\intt(X,\RR)$ \resp $\iMTM^\intt_\good(X,\kk)$.\vspace*{-5pt}\enlargethispage{\baselineskip}%
\end{remarque}

\section{Applications of the compatibility between the functors}\label{sec:compatibilities}
In this section, we first consider compatibility properties between various functors on holonomic $\cD$-modules, and we introduce the notion of \emph{holonomic $\cD$-module of exponential-regular origin}. In a similar way, we introduce the notion of integrable mixed twistor-$\cD$-module of exponential-Hodge origin, by replacing the notion of regular holonomic $\cD$-module by that of mixed Hodge module, and working with the corresponding functors in the category of integrable mixed twistor $\cD$-modules.

\subsection{Holonomic \texorpdfstring{$\cD$}{D}-modules of exponential-regular origin}\label{subsec:compatibilitiesD}

Let $X$ be a complex manifold.

\refstepcounter{equation}
\subsubsection{Functors}\label{subsub:functors}
We have the following exact functors on the category $\Mod_\hol(\cD_X)$ of holonomic $\cD_X$-modules.

\begin{enumeratea}
\item\label{enum:Gamma}
The localization functor $(*H)$, that we will find convenient to denote by \index{$GammaH$@$\Gamma_{[*H]}$}$\Gamma_{[*H]}$, for a hypersurface $H$ in $X$: if $\cO_X(*H)$ denotes the sheaf of meromorphic functions with poles along $H$, then $\Gamma_{[*H]}\ccM=\ccM(*H)=\cO_X(*H)\otimes_{\cO_X}\nobreak\ccM$ with its natural structure of left $\cD_X$-module.
\item\label{enum:Gammac}
The dual localization functor $(!H)$, that we similarly denote by $\Gamma_{[!H]}$, which is adjoint by duality of $\Gamma_{[*H]}$ and can also be expressed without using the duality functor by using $V$-filtrations (\cf\cite[\S3.1.3]{Mochizuki11}, which also applies to the $\cD$\nobreakdash-module case).
\item\label{enum:duality}
the duality functor $\bD_X$,
\item\label{enum:pullback}
The pullback functor $\mapsm^+:\Mod_\hol(\cD_Y)\mto\Mod_\hol(\cD_X)$ by a \emph{smooth} morphism $\mapsm:X\to Y$.

\item\label{enum:localizedpullback}
Let $\modif:X'\to X$ be a projective modification for which there exists a hypersurface $D\subset X$ such that, setting $D'=\modif^{-1}(D)$, $e$ induces an isomorphism $X'\moins D'\isom X\moins D$. The pullback and pushforward of holonomic $\cD$-modules by $e$ induce inverse equivalences between the categories $\Mod_\hol(X,*D)$ and $\Mod_\hol(X',*D')$ of holonomic $\cD$-modules localized along $D,D'$. We denote them by \index{$Estarplus$@$\ste^+$}$\ste^+$ and \index{$Estardag$@$\ste^0_\dag$}$\ste^0_\dag$, by noticing that $\ste^k_\dag=0$ for $k\neq0$. We will use the functors $\ste^+\Gamma_{[*D]}:\Mod_\hol(X)\mto\Mod_\hol(X')$ and $\ste^0_\dag\Gamma_{[*D']}:\Mod_\hol(X')\mto\Mod_\hol(X)$.

\item\label{enum:externalproduct}
The external product \index{$BoxtimesD$@$\boxtimes_\cD$}$\boxtimes_\cD:\Mod_\hol(\cD_X)\times\Mod_\hol(\cD_{X'})\mto\Mod_\hol(\cD_{X\times X'})$ (a bi-functor exact on each side), defined by
\[
\ccM\boxtimes_\cD\ccN:=\cD_{X\times X'}\otimes_{(\cD_X\boxtimes\cD_{X'})}(\ccM\boxtimes\ccN).
\]

\item\label{enum:Ephi}
The exponential twist functor $\ccE^\mero\otimes_{\cO_X}\!\cbbullet$ for a nonzero meromorphic function $\mero\in\Gamma(X,\cO_X(*P))$ with reduced pole \hbox{divisor~$P$}, defined, for any left $\cD_X$-module~$\ccM$ regarded as an $\cO_X$-module with a flat connection $\nabla$, by the formula $\ccE^\mero\otimes_{\cO_X}\ccM:=(\ccM(*P),\nabla+\rd\mero)$. (In particular, $\ccE^\mero:=(\cO_X(*P),\rd+\rd\mero)$.)
\end{enumeratea}

Moreover, possibly non-exact, we have
\begin{enumeratea}\setcounter{enumi}{7}
\item\label{enum:pushforward}
the pushforward functors $\map_\dag^k:\Mod_\hol(\cD_X)\!\mto\!\Mod_\hol(\cD_Y)$ ($k\in\ZZ$) for a \emph{proper} morphism $\map:X\to Y$.
\end{enumeratea}

These functors, except for the exponential twist $\ccE^\mero\otimes_{\cO_X}\cbbullet$ by a non-trivial meromorphic function, are known to preserve the regularity property of holonomic $\cD$-modules. They satisfy the following commutation relations, where $X,Y$ are smooth.

\refstepcounter{equation}
\Subsubsection{Commutation relations}\label{subsub:commutation}
\begin{enumeratei}
\item\label{enum:rel1}
For the localization:
\begin{align*}
\Gamma_{[*H]}\circ\Gamma_{[*H']}&\simeq\Gamma_{[*H']}\circ\Gamma_{[*H]},\\
\Gamma_{[*H]}\circ\bD&\simeq\bD\circ\Gamma_{[!H]},\\
\ste^+\circ\Gamma_{[*H]}\circ(\ccE^\mero\otimes\cbbullet)&\simeq\ccE^{\modif^*\mero}\otimes\Gamma_{[*H']}\,\ste^+\cbbullet.
\end{align*}
\item\label{enum:rel2}
$(\ccE^\mero\otimes\cbbullet)\circ(\ccE^\psi\otimes\cbbullet)\simeq\Gamma_{[*(P\cup Q)]}\circ(\ccE^{\mero+\psi}\otimes\cbbullet) \simeq(\ccE^{\mero+\psi}\otimes\cbbullet) \circ\Gamma_{[*(P\cup Q)]}$.

\item\label{enum:rel4}
For $\map:X\to Y$ proper and $k\in\ZZ$,
\begin{align*}
\map_\dag^k\circ\Gamma_{[*\map^{-1}(H)]}&\simeq\Gamma_{[*H]}\circ \map_\dag^k,\\
\bD_Y\circ \map_\dag^k&\simeq \map_\dag^k\circ\bD_X,\\
\ccE^\mero\otimes(\map_\dag^k\bbullet)&\simeq \map_\dag^k\circ(\ccE^{\map^*\mero}\otimes\cbbullet),
\end{align*}
(for the first line, \cf \cite[Prop.\,3.6-4]{Mebkhout04}; for the second line, \cf \cite[Th.\,5.3.1]{Mebkhout87}, \cite{MSaito89b}, \cite[Prop.\,4.39]{Kashiwara03}; the third one easily follows from the first~one).

\item\label{enum:rel6}
For the external product:
\begin{align*}
\Gamma_{[*H'']}(\cbbullet\boxtimes_\cD\cbbullet)&\simeq\Gamma_{[*H]}\cbbullet\boxtimes_\cD\Gamma_{[*H']}\cbbullet,\quad H'':=(H\times X')\cup(X\times H'),\\
\ccE^{\mero(x)+\psi(x')}\otimes(\cbbullet\boxtimes_\cD\cbbullet)&\simeq (\ccE^{\mero(x)}\otimes\cbbullet)\boxtimes_\cD(\ccE^{\psi(x')}\otimes\cbbullet),\quad\mero\in\cO_X(*P),\,\psi\in\cO_{X'}(*P'),\\
\bD_{X\times X'}(\cbbullet\boxtimes_\cD\cbbullet)&\simeq \bD_X\cbbullet\boxtimes_\cD \bD_{X'}\cbbullet,\\
(\map\times\id_{X'})_\dag^k(\cbbullet\boxtimes_\cD\cbbullet)&\simeq \map_\dag^k\cbbullet\boxtimes_\cD\cbbullet\quad \map:X\ra Y\text{ proper},
\end{align*}
(the first two lines are obvious, the third line is treated in Appendix \ref{app:exernaldual}, and \cf\eg\cite[Prop.1.5.30]{H-T-T08} for the last line).

\item\label{enum:rel3}
For $g:Y'\to Y$ smooth:
\begin{align*}
\mapsm^+\circ\Gamma_{[*H]}&\simeq\Gamma_{[*g^{-1}(H)]}\circ \mapsm^+,\\
\mapsm^+\circ\bD_Y&\simeq\bD_X\circ \mapsm^+,\\
\mapsm^+\circ(\ccE^\mero\otimes\cbbullet)&\simeq \ccE^{\mapsm^*\mero}\otimes(\mapsm^+\bbullet),\\
(g\times\id_{X'})^+(\cbbullet\boxtimes_\cD\cbbullet)&\simeq \mapsm^+\cbbullet\boxtimes_\cD \cbbullet,\\[-20pt]
\mapsm^+\circ \map_\dag^k&\simeq \map_\dag^{\prime k}\circ g^{\prime+}\qquad
\begin{array}{r}
\text{$\map:X\to Y$ proper,}\\
\text{cartesian square:}
\end{array}\quad
\begin{array}{c}
\xymatrix{
X'\ar[d]_{\map'}\ar[r]^-{\mapsm'}\ar@{}[dr]|\square&X\ar[d]^\map\\
Y'\ar[r]_-{\mapsm}&Y
}
\end{array}
\end{align*}
(\cf \cite[Th.\,4.12]{Kashiwara03} for the second line, the first and third lines are easy, the fourth line follows from the flatness of $\cD_{X\times X'}$ over $\cD_X\boxtimes\cD_{X'}$, and the last line can be treated as in Proposition~\ref{prop:basechangeR}).

\item\label{enum:rel7}
For a projective modification $e:(X',D')\to(X,D)$ as in \eqref{enum:localizedpullback} and $H'=\modif^{-1}(H)$:\vspace*{-3pt}
\begin{align*}
\ste^+\Gamma_{[*D]}\Gamma_{[*H]}&\simeq\Gamma_{[*H']}\ste^+\Gamma_{[*D]},\\
\Gamma_{[*H]}\ste^0_\dag\Gamma_{[*D']}&\simeq\ste^0_\dag\Gamma_{[*D']}\Gamma_{[*H']},\\
\ste^+\circ(\ccE^\mero\otimes\cbbullet)\circ\Gamma_{[*D]}&\simeq(\ccE^{\mero\circ e}\otimes\cbbullet)\circ\ste^+\Gamma_{[*D]},\\
\ste^0_\dag\circ(\ccE^{\mero\circ e}\otimes\cbbullet)\circ\ste^+\Gamma_{[*D]}&\simeq(\ccE^\mero\otimes\cbbullet)\circ\Gamma_{[*D]}
\end{align*}
(this is mostly obvious, and we use $\ste^0_\dag\circ\ste^+\simeq\id$, $\ste^+\circ\ste^0_\dag\simeq\id$). Moreover, using the pushforward functor as in \eqref{enum:pushforward}, we have\vspace*{-3pt}
\[
\ste^k_\dag=\Gamma_{[*D]}\circ \modif^k_\dag\circ\Gamma_{[*D']}=\modif^k_\dag\circ\Gamma_{[*D']} \quad(=0\text{ if }k\neq0).
\]
\end{enumeratei}

From the commutation properties between $\bD$ and other functors we obtain the compatibility relations between $\Gamma_{[!H]}$ and $\mapsm^+$, $\map_\dag^k$, $\boxtimes_\cD$. In order to obtain the compatibility relation between $\Gamma_{[!H]}$ and $\ccE^\mero\otimes$, we will need the following proposition.

\begin{proposition}\label{prop:DEphi}
Let $\ccM$ be a regular holonomic $\cD_X$-module. Let $\mero$ be a meromorphic function on $X$ whose zero divisor and pole divisor do not intersect. Then the natural morphism $\bD_X(\ccE^{-\mero}\otimes\ccM)\to\ccE^\mero\otimes\bD_X\ccM$ is an isomorphism.
\end{proposition}

\begin{proof}
Let us denote by $P$ the reduced pole divisor of $\mero$. Note that, if $\ccI$ is an injective $\cD_X(*P)$-module, then $\ccE^\mero\otimes\ccI$ is also injective. Then, for a $\cD_X$-module $\ccM$, $\ccE^\mero\otimes\nobreak[(\bD_X\ccM)(*P)]=\ccE^\mero\otimes\nobreak\bD_X\ccM$ is naturally identified with the dual of $\ccE^{-\mero}\otimes\ccM$ in $\catD^+(\cD_X(*P))$, which is $[\bD_X(\ccE^{-\mero}\otimes\nobreak\ccM)](*P)$. For $\ccM$ holonomic, we thus have a natural morphism of $\cD_X$-modules\vspace*{-3pt}\enlargethispage{.5\baselineskip}%
\[
\bD_X(\ccE^{-\mero}\otimes\ccM)\to\bigr[\bD_X(\ccE^{-\mero}\otimes\ccM)\bigl](*P)=\ccE^\mero\otimes\bD_X\ccM.
\]
Proving the proposition amounts then to proving that the morphism above is an isomorphism, that is, any local equation of $P$ acts in an invertible way on $\bD_X(\ccE^{-\mero}\otimes\nobreak\ccM)$. This is thus a local question on $P$.

Since the zero divisor of $\mero$ does not intersect its pole divisor $P$, we can assume that there exists a holomorphic function $\psi$ such that $\mero=1/\psi$ and $P=\psi^{-1}(0)$. Let $i:X\hto X\times\CC$ be the inclusion of the graph of $\psi$. Then $i_+(\ccE^{-1/\psi}\otimes\nobreak\ccM)=\ccE^{-1/t}\otimes i_+\ccM$ (\cf\cite[Lem.\,4.2]{S-Y14}) and it is enough to prove that, for any regular holonomic $\cD_{X\times\CC}$-module~$\ccN$, $t$ acts in an invertible way on $\bD_{X\times\CC}(\ccE^{-1/t}\otimes\ccN)$. For that purpose, it is enough to check that the Kashiwara-Malgrange filtration of $\ccE^{-1/t}\otimes\ccN$ along $t=0$ is constant, since the same is then known to hold for the dual module. In other words, we are reduced to proving that $\ccE^{-1/t}\otimes\ccN$ is $\cD_{X\times\CC/\CC}\langle t\partial_t\rangle$-coherent.

We use the property that, for a regular holonomic $\cD_{X\times\CC}$-module $\ccN$, the terms of the Kashiwara-Malgrange filtration $V_\bbullet\ccN$ are $\cD_{X\times\CC/\CC}$-coherent. Then, noting that $\ccN[t^{-1}]=V_0\ccN[t^{-1}]$ and denoting by $\mathrm{e}^{-1/t}$ the generator $1$ of $(\cO_{X\times\CC}[t^{-1}],\rd-\rd(1/t))$, the equality $t\partial_t(\mathrm{e}^{-1/t}\otimes n)=t^{-1}(\mathrm{e}^{-1/t}\otimes n)+(\mathrm{e}^{-1/t}\otimes t\partial_t n)$ shows that\vspace*{-3pt}
\[
t^{-1}(\mathrm{e}^{-1/t}\otimes V_0\ccN)\subset t\partial_t(\mathrm{e}^{-1/t}\otimes V_0\ccN)+(\mathrm{e}^{-1/t}\otimes V_0\ccN)
\]
and, iterating the process, we find that $\ccE^{-1/t}\otimes\ccN=\sum_{k\geq0}(t\partial_t)^k(\mathrm{e}^{-1/t}\otimes V_0\ccN)$, hence the $\cD_{X\times\CC/\CC}\langle t\partial_t\rangle$-coherence.
\end{proof}

\begin{remarque}
The proposition is not true for any holonomic $\ccM$: for example, choose $\ccM=\ccE^\mero$, so that the left-hand side is equal to $\bD_X(\cO_X(*P))\simeq\cO_X(!P)$, while the right-hand side is, according to the lemma applied to $\ccM=\cO_X$, isomorphic to $\cO_X(*P)$.
\end{remarque}

\begin{corollaire}\label{cor:Evarphi!H}
If $\ccM$ is regular holonomic, and if the zero divisor and the pole divisor of~$\mero$ do not intersect, then for any hypersurface $H\subset X$ the natural morphisms\vspace*{-3pt}
\[
\Gamma_{[!H]}(\ccE^\mero\otimes\ccM)\from\Gamma_{[!H]}(\ccE^\mero\otimes\Gamma_{[!H]}\ccM)\to\ccE^\mero\otimes\Gamma_{[!H]}\ccM
\]
are isomorphisms, hence $\Gamma_{[!H]}(\ccE^\mero\otimes\ccM)\simeq\ccE^\mero\otimes\Gamma_{[!H]}\ccM$.
\end{corollaire}

\begin{proof}
Since $\Gamma_{[!H]}\ccM$ is known to be regular holonomic, it is enough to prove the dual statement, according to Proposition \ref{prop:DEphi}. The dual of the left morphism is
\[
\Gamma_{[*H]}(\ccE^{-\mero}\otimes\bD_X\ccM)\to\Gamma_{[*H]}(\ccE^{-\mero}\otimes\Gamma_{[*H]}\bD_X\ccM),
\]
and this is an isomorphism because of \ref{subsub:commutation}\eqref{enum:rel1} above (applied with $\modif=\id$). The dual of the right morphism~is
\[
\Gamma_{[*H]}(\ccE^{-\mero}\otimes\Gamma_{[*H]}\bD_X\ccM)\from\ccE^{-\mero}\otimes\Gamma_{[*H]}\bD_X\ccM,
\]
and we also use \ref{subsub:commutation}\eqref{enum:rel1} above to conclude.
\end{proof}

We now relax the assumption on the zero and pole divisors of $\mero$. Let us choose a projective modification $\modif:(X',\modif^{-1}(P))\to(X,P)$, where $P$ is the reduced pole divisor of~$\mero$, inducing an isomorphism $X'\moins\nobreak\modif^{-1}(P)\isom X\moins P$ and such that $\mero'=\modif^*(\mero)=\mero\circ\modif$ has non-intersecting zero and pole divisors. Note that the pole divisor $P'$ of $\mero'$ satisfies $P'\subset \modif^{-1}(P)$, and the inclusion can be strict.

For $\ccM$ holonomic we have, since $\ccE^\mero=\Gamma_{[*P]}\ccE^\mero$,
\[
\ccE^\mero\otimes\ccM=\ccE^\mero\otimes\Gamma_{[*P]}\ccM=\modif^0_\dag(\ccE^{\mero'}\otimes \modif^+\Gamma_{[*P]}\ccM).
\]
Therefore, applying the dual functor for holonomic $\cD$-modules, we get, if $\ccM$ is regular,
\begin{equation}\label{eq:EvarphieD}
\begin{split}
\bD_X(\ccE^\mero\otimes\ccM)&\simeq \modif^0_\dag\bD_{X'}(\ccE^{\mero'}\otimes \modif^+\Gamma_{[*P]}\ccM)\\
&\simeq \modif^0_\dag(\ccE^{-\mero'}\otimes\bD_{X'}\modif^+\Gamma_{[*P]}\ccM)\quad\text{(Proposition \ref{prop:DEphi})}.
\end{split}
\end{equation}

Let $H$ be any hypersurface of $X$ and set $H'=\modif^{-1}(H)$. Then we obtain similarly
\begin{equation}\label{eq:EvarphieH}
\Gamma_{[!H]}(\ccE^\mero\otimes\ccM)\simeq \modif^0_\dag(\ccE^{\mero'}\otimes\Gamma_{[!H']}\modif^+\Gamma_{[*P]}\ccM).
\end{equation}

\begin{corollaire}\label{cor:composition}
Let us consider the composition
\[
\Mod_\hr(\cD_{X_0})\Mto{\Psi_1}\Mod_\hol(\cD_{X_1})\Mto{\Psi_2}\cdots\Mto{\Psi_r}\Mod_\hol(\cD_{X_r})
\]
of a sequence of functors, each one being equal to one in \ref{subsub:functors}\eqref{enum:Gamma}--\eqref{enum:pushforward} (we start from a regular holonomic $\cD$-module). Then we obtain the same result, up to isomorphism, by applying functors of type \ref{subsub:functors}\eqref{enum:Gamma}--\eqref{enum:externalproduct} first (up to changing some functors $\Gamma_{[*H]}$ with $\modif^+\Gamma_{[*H]}$), then a~\emph{single} functor \ref{subsub:functors}\eqref{enum:Ephi}, and functors of type \ref{subsub:functors}\eqref{enum:pushforward} last.
\end{corollaire}

\begin{proof}
According to the commutation relations \ref{subsub:commutation}\eqref{enum:rel1}--\eqref{enum:rel6}, we can pass functors $\map_\dag^k$ at the end and functors $\mapsm^+$, $\boxtimes_\cD$ at the beginning. We are then left to reorder functors of the type $\ccE^\mero\otimes\cbbullet$, $\bD$ and $\Gamma_{[\star H]}$ ($\star=*,!$), starting from a regular holonomic $\cD$\nobreakdash-module, so that the functors $\Gamma_{[\star H]}$ and $\bD$ are applied first, and the functors $\ccE^\mero\otimes\cbbullet$ last. Let $\ccE^{\mero_1}\otimes\cbbullet$ be the first functor of this kind which is applied. It is thus applied to a regular holonomic $\cD$-module. If the next functor to be applied is
\begin{itemize}
\item
$\Gamma_{[*H]}$, we apply the third line of \ref{subsub:commutation}\eqref{enum:rel1},
\item
$\Gamma_{[!H]}$, we apply \eqref{eq:EvarphieH}
\item
$\bD$, we apply \eqref{eq:EvarphieD}.
\end{itemize}
The latter two operations introduce a new functor $\modif^0_\dag$, that we pass on the left up to the first pushforward functor. On the other hand, the holonomic $\cD$-module on the right of $\ccE^{\mero'_1}\otimes\cbbullet$ remains regular. After this procedure, if the the next functor is a pushforward functor, we stop, otherwise the next functor is $\ccE^{\mero_2}\otimes\cbbullet$ and we use \ref{subsub:commutation}\eqref{enum:rel2}. The number of functors $\ccE^\mero\otimes\cbbullet$ strictly decreases by one. We conclude by induction on the number of functors $\ccE^\mero\otimes\cbbullet$ in the sequence of functors.
\end{proof}

\begin{definition}[Holonomic $\cD$-modules of exponential-regular origin]
\index{holonomic $\cD$-module of exponential-regular origin}A holonomic $\cD_X$-module is said to have \emph{exponential-regular origin} if it is isomorphic to one obtained by applying a finite sequence $\Psi_1,\dots,\Psi_r$ of functors \ref{subsub:functors}\eqref{enum:Ephi} and~\eqref{enum:pushforward} to a regular holonomic $\cD_{X_0}$-module, so that $X=X_r$.
\end{definition}

\begin{corollaire}
The functors \ref{subsub:functors}\eqref{enum:Gamma}--\eqref{enum:externalproduct} (and, by definition, \ref{subsub:functors}\eqref{enum:Ephi} and~\eqref{enum:pushforward}) preserve holonomic $\cD$-modules of exponential-regular origin.\qed
\end{corollaire}

\subsection{Complex mixed Hodge modules as integrable mixed twistor modules}\label{subsec:MHM}

\subsubsection*{Real mixed Hodge modules and real integrable mixed twistor $\cD$-modules}

We denote by \index{$MHMXR$@$\MHM(X,\RR)$}$\MHM(X,\RR)$ the category of \index{mixed Hodge module!real --}real mixed Hodge modules of \cite{MSaito87}. In \cite[\S13.5]{Mochizuki11}, T.\,Mochizuki defines a fully faithful functor $\MHM(X,\RR)\mto\WRdTriples(X)$ whose essential image is a full subcategory of $\MTM^\intt(X)$. In fact this functor factorizes through $\WRdiTriples(X)$. Indeed, let $(\ccN,F_\bbullet\ccN,\ccF_\RR,W_\bbullet)$ be a real mixed Hodge module (here,~$\ccF_\RR$ is an $\RR$-perverse sheaf and we~do not write explicitly the isomorphism $\CC\otimes_\RR\ccF_\RR\isom\pDR\ccN$). The real structure induces a canonical $\cD_X\otimes_\CC\nobreak\cD_{\ov X}$-linear pairing $\ccC:\ccN^\vee\otimes_\CC\ov\ccN\to\Db_X$, where $\ccN^\vee$ is the $\cD_X$-module dual to~$\ccN$ (\cf\cite[Chap.\,12]{Mochizuki11}). The restriction to $\CC^*_\hb$ (in~the algebraic sense) of the \index{Rees module}Rees module \index{$RFM$@$R_F\ccM$}$R_F\ccN:=\bigoplus_pF_p\ccN\hb^p$ is~$\ccN[\hb,\hbm]$, and the pairing $\ccC$ extends in a unique way as a $\cD_X[\hb,\hbm]\otimes_\CC\ov{\cD_X[\hb,\hbm]}$-linear pairing\vspace*{-5pt}
\[
\iC:\ccN^\vee[\hb,\hbm]\otimes_\CC\iota^*\ov{\ccN[\hb,\hbm]}\to\Db_X[\hb^{\pm1},\ov\hb^{\pm1}].
\]
This pairing is obviously integrable. By taking analytification with respect to $\hb$ and~$\ov\hb$ we obtain an object of $\RdiTriples(X)$, and then, by restricting to $X\times\bS$, we get an object of $\RdTriples(X)$, which is shown to belong to the category $\MTM^\intt(X)$ in \cite[\S13.5]{Mochizuki11}. Therefore, the natural functor $\MHM(X,\RR)\mto\MTM^\intt(X)$ factorizes through $\iMTM^\intt(X)$. Note that one can define directly the $\sigma$-sesquilinear pairing $C_{\bS}$ by extending $\ccC$ in a $\sigma$-sesquilinear way, to obtain a pairing $C$ with values in $\Db_{\cX^\circ/\CC^*_\hb}^\an$, as in \cite{Mochizuki07}.

The previous functor $\MHM(X,\RR)\mto\MTM^\intt(X)$ can be enriched as a functor with values in $\MTM^\intt(X,\RR)$ by transporting the real structure, due to the compatibility with duality and adjunction. A similar statement holds with $\iMTM^\intt(X,\RR)$.

\subsubsection*{Goodness}
Let $\kk$ be a subfield of $\RR$. Then \cite[Prop.\,13.5.5]{Mochizuki11} shows that the previous functor induces a fully faithful functor \index{$MHMXk$@$\MHM(X,\kk)$}$\MHM(X,\kk)\mto\MTM^\intt_\good(X,\kk)$.

\subsubsection*{Complex mixed Hodge modules}
Let us define the category $\RdiTriples(X)^\alg$ consisting of triples $(R_F\ccN',R_F\ccN'',\iC)$, with $(\ccN',F_\bbullet\ccN')$ and $(\ccN'',F_\bbullet\ccN'')$ as $(\ccN,F_\bbullet\ccN)$ above, and similarly the category $\WRdiTriples(X)^\alg$. There is a natural analytification functor taking values in $\RdiTriples(X)$ \resp in $\WRdiTriples(X)$.

Let us recall (\cf Remark \ref{rem:directsummand}) that the abelian category $\MTM(X)$ is a full subcategory of the abelian category $\WRTriples(X)$, which is stable by direct summand in $\WRTriples(X)$ and that $\MTM^\intt(X)$ is a full subcategory of the abelian category $\WRdTriples(X)$, which is stable by direct summand in $\WRdTriples(X)$. The same property holds for $\iMTM^\intt(X)$ and $\WRdiTriples(X)$ respectively, according to Theorem \ref{th:equivalenceintegrable}.

\begin{definition}[Complex mixed Hodge modules]\label{def:MHMC}
We say that an object of the category $\WRdiTriples(X)^\alg$ is a \index{mixed Hodge module!complex --}\emph{complex mixed Hodge module} if it is a direct summand in $\WRdiTriples(X)^\alg$ of an object corresponding to a real mixed Hodge module. This defines the full subcategory \index{$MHMXC$@$\MHM(X,\CC)$}$\MHM(X)=\MHM(X,\CC)$ of $\WRdiTriples(X)^\alg$.
\end{definition}

\begin{remarque}\label{rem:MHMC}
If an object of $\WRdiTriples(X)$ is a direct summand in the category $\WRdiTriples(X)$ of an object of $\iMTM^\intt(X)$ (and therefore is an object of $\iMTM^\intt(X)$, \cf Remark \ref{rem:directsummand}) coming from a real mixed Hodge module, then it comes from a complex mixed Hodge module by analytification: the filtration is recovered by taking $\ker(\hb^2\partial_\hb-p\hb)$. In other words, one can equivalently define the category $\MHM(X,\CC)$ as a full subcategory of $\iMTM^\intt(X)$. It is stable by direct summand in $\WRdiTriples(X)$ (equivalently, in $\iMTM^\intt(X)$).
\end{remarque}

\subsection{Integrable mixed twistor $\cD$-modules of exponential-Hodge origin}\label{subsec:compatibilitiesT}

We have functors analogous to those of \ref{subsub:functors} at the level of $\iMTM^\intt$. For the exponential twist \ref{subsub:functors}\eqref{enum:Ephi}, we consider instead the twist $\icT^{\mero/\hb}\otimes\cbbullet$, and \hbox{$\Gamma_{[!P]}\circ(\icT^{\mero/\hb}\otimes\cbbullet)$}.

Let us recall (\cf\cite[Rem.\,7.2.9]{Mochizuki11}) that the functor $\Xi_{\DR}:\MTM(X)\mto\Mod_\hol(\cD_X)$ which associates with each object $\bigl((\cM',\cM'',C_{\bS}),W_\bbullet\bigr)$ the $\cD_X$-module $i_{\hb=1}^*\cM''=\ccM$ is faithful. The same result holds for $\Xi_{\DR}$ on $\MTM^\intt(X)$ and $\iMTM^\intt(X)$. Moreover, Remark \ref{rem:compfunctorshb1} shows that $\Xi_{\DR}$ behaves well with respect to the functors considered above. It follows that the existence of liftings to $\iMTM^\intt(X)$ of the morphisms in \ref{subsub:commutation}, \eqref{eq:EvarphieD} and \eqref{eq:EvarphieH} implies that these liftings are isomorphisms.

\begin{definition}[Integrable mixed twistor $\cD$-modules of exponential-Hodge origin]\label{def:expMHM}
An object (\resp a morphism) of $\iMTM^\intt$ is said to be of exponential-Hodge origin if it can be obtained from an object (\resp a morphism) of $\MHM$ by applying projective pushforward functors $\map_\dag^k$ (lifting \ref{subsub:functors}\eqref{enum:Ephi}) and exponential twists $\icT^{\mero/\hb}\otimes\cbbullet$ (lifting \ref{subsub:functors}\eqref{enum:pushforward}).
\end{definition}

\begin{proposition}\label{prop:expMHM}
The functors lifting \ref{subsub:functors}\eqref{enum:Gamma}--\eqref{enum:externalproduct} (and, by definition, \ref{subsub:functors}\eqref{enum:Ephi} and~\eqref{enum:pushforward}, hence also the functor $\Gamma_{[!P]}(\icT^{\mero/\hb}\otimes\cbbullet)$) preserve integrable mixed twistor $\cD$-modules of exponential-Hodge origin.
\end{proposition}

\begin{proof}
The proof consists in showing the existence of liftings of the morphisms in \ref{subsub:commutation}, \eqref{eq:EvarphieD} and \eqref{eq:EvarphieH}. The latter two morphisms are obtained in a way similar to the other ones, and we will not give details on their liftings.

\ref{subsub:commutation}\eqref{enum:rel1}. For the first line, we have natural morphisms of functors on $\iMTM^\intt$:
\[
\Gamma_{[*(H\cup H')]}\to\Gamma_{[*H']}\circ\Gamma_{[*(H\cup H')]}\to\Gamma_{[*H]}\circ\Gamma_{[*H']}\circ\Gamma_{[*(H\cup H')]}\from\Gamma_{[*H]}\circ\Gamma_{[*H']}
\]
and similar ones by exchanging $H$ and $H'$. Since they are isomorphisms after $\Xi_{\DR}$ (Notation \ref{nota:XiDR}), they are isomorphisms. The second line is \cite[Prop.\,13.3.5]{Mochizuki11}. For the third line, we have natural morphisms of functors
\[
\Gamma_{[*H]}\circ(\icT^{\mero/\hb}\otimes\cbbullet)\to\Gamma_{[*H]}\circ(\icT^{\mero/\hb}\otimes\cbbullet)\circ\Gamma_{[*H]}\from(\icT^{\mero/\hb}\otimes\cbbullet)\circ\Gamma_{[*H]},
\]
and they become isomorphisms after applying $\Xi_{\DR}$.

\ref{subsub:commutation}\eqref{enum:rel2}. By construction, for any object $\icT$ of $\iMTM^\intt(X)$, there exists a natural morphism $\icT\to\Gamma_{[*H]}\icT$. On the other hand, working with $\cR(*(P\cup Q))$-triples, we have
\[
\big[(\icT_*^{\mero/\hb}\otimes\cbbullet)\circ(\icT_*^{\psi/\hb}\otimes\cbbullet)\big](*(P\cup Q))\simeq(\icT_*^{(\mero+\psi)/\hb}\otimes\cbbullet).
\]
We therefore have a natural morphism of functors
\[
\Gamma_{[*(P\cup Q)]}\circ(\icT^{\mero/\hb}\otimes\cbbullet)\circ(\icT^{\psi/\hb}\otimes\cbbullet)\to(\icT^{(\mero+\psi)/\hb}\otimes\cbbullet),
\]
which is an isomorphism since it is such after applying $\Xi_{\DR}$.

\ref{subsub:commutation}\eqref{enum:rel4}. The first line is given by \cite[Prop.\,11.2.7]{Mochizuki11}, the second line by \cite[Cor.\,13.3.3]{Mochizuki11} and the third line by \cite[Lem.\,11.3.4]{Mochizuki11}.

\ref{subsub:commutation}\eqref{enum:rel6}. The first two lines are given by\cite[Lem.\,11.4.15]{Mochizuki11}, the third line by \cite[Prop.\,13.3.9]{Mochizuki11} and the last line by \cite[Lem.\,11.4.14]{Mochizuki11}.

\ref{subsub:commutation}\eqref{enum:rel3}. The construction of the morphisms are standard and left to the reader (one can locally interpret $\mapsm^+$ as an external tensor product).
\end{proof}

\chapterspace{-2}
\chapter{Irregular mixed Hodge modules}\label{part:2}\label{PART:2}

\section{Introduction to Chapter \ref{PART:2}}

We continue using Notation \ref{nota:general}--\ref{nota:XiDR}. Let us indicate the main lines of the definition, given in this chapter, of the category $\IrrMHM(X)$ of irregular mixed Hodge modules. The basic procedure is to consider objects of $\iMTM^\intt(X)$ having a good behaviour with respect to rescaling the variable $\hb$, obtained by replacing $\hb$ with $\theta\hb$, $\theta\in\CC^*$.

The rescaling operation is defined for any integrable $\iota$-triple $\icT=(\cM',\cM'',\iC)$, and it is a strong condition that, starting from an object of $\iMTM^\intt(X)$, for any $\theta\in\CC^*$ the $\theta$-rescaled object remains an integrable mixed twistor $\cD$-module. Even stronger is the condition that the rescaled object extends as a mixed twistor $\cD$-module on the affine line with coordinate $\tau=1/\theta$. In the pure case, and when $X$ is a point, this is a condition similar to the notion of ``Sabbah orbit'' as defined in \cite{H-S06}. This leads to the category $\iMTM^\resc(X)$, the objects of which can then be endowed with an irregular Hodge filtration.

However, supplementary properties are needed:
\begin{itemize}
\item
The first one is a partial regularity property along $\tau=0$, which imposes a ``tame'' behaviour of the rescaled objects when $\theta\to\infty$ (in the case of a ``Sabbah orbit'', see \cite[Th.\,7.3]{H-S06}).
\item
The second one is a grading property, which ensures a good behaviour of the irregular Hodge filtration with respect to various operations.
\end{itemize}
Adding the first condition produces the subcategory $\iMTM^\resc(X)$, and adding moreover the second one actually produces the category $\IrrMHM(X)$. One can also define the category $\IrrMHM(X,\kk)$ by adding the data of a $\kk$-perverse sheaf on $X$ so that the underlying object is an object in $\iMTM^\intt_\good(X,\kk)$.

The case when $X$ is a point will be described more explicitly in Chapter \ref{chap:irregmhm}. This leads to the category of irregular mixed Hodge structures, whose definition uses nevertheless the category of integrable mixed twistor $\cD$-modules on the line $\CC_\tau$. However, the proofs of the main results for $\IrrMHM(\pt)$ are not much simpler than those for $\IrrMHM(X)$, and will thus only be given in the general case.

\begin{notation}[for Chapter \ref{part:2}]\label{nota:chap2}
We keep the notation introduced in Section \ref{sec:introI}. We consider~$\PP^1$ equipped with two coordinate charts $\CC_\theta,\CC_\tau$ with $\tau=1/\theta$ on the intersection. We~set \index{$XTheta$@$\thetaX$}$\thetaX=X\times\CC^*_\theta$, \index{$XTau$@$\tauX$}$\tauX=X\times\CC_\tau$ (where $\CC_\theta,\CC_\tau$ have their analytic topology) and \index{$XTau0$@$\tauX_0$}$\tauX_0=X\times\{\tau=0\}$. We will not consider the behaviour near $\theta=0$. We will consider the projection $p:\tauX\to X$ and we denote by the same letter its restriction to $\thetaX$. We also set \index{$XXCtheta$@$\thetacX$}$\thetacX\simeq\thetaX\times\CC_\hb=\cX\times\CC^*_\theta$ and \index{$XXCtau$@$\taucX$}$\taucX:=\tauX\times\CC_\hb\simeq\cX\times\CC_\tau$. An object living on $\taucX$ (\resp $\thetacX$) will be denoted with a left exponent $\tau$ (\resp $\theta$). For example, the restriction to $\thetacX$ of an object defined on $\taucX$ will simply adquire a left exponent $\theta$ instead of $\tau$.\enlargethispage{.5\baselineskip}%

Let us consider the map \index{$Mu$@$\mu$}$\mu:\thetacX\to\cX$ defined by $\mu:(x,\theta,\hb)\mto(x,\theta\hb)$. We will decompose~$\mu$ as $\mu=q\circ\nu^{-1}$ by introducing the morphism $\nu:(x,\tau,\hb)\mto(x,\tau,\zeta)=(x,\tau,\tau\hb)$, which induces an isomorphism $\thetacX\isom\thetacX$, and the projection $q:(x,\tau,\hb)\mto(x,\hb)$. For $\tau_o\in\CC^*_\tau$, we denote by \hbox{$\mu_{\tau_o}:\cX\to\cX$} the isomorphism $(x,\hb)\mto(x,\hb/\tau_o)$. In particular, $\mu_1=\id$.

We will also use the following notation:
\begin{itemize}
\item
\index{$IHZ1$@$i_{\hb=1}$}$i_{\hb=1}:X\hto\cX$ is the inclusion induced by $\{1\}\hto\CC_\hb$,
\item
\index{$IHT1$@$i_{\tau=1}$}$i_{\tau=1}:\cX\hto\thetacX$ is the inclusion induced by $\{1\}\hto\CC^*_\tau$,
\item
more generally, the inclusion $i_{\tau=\tau_o}:\cX\hto\thetacX$ is induced by \hbox{$\{\tau_o\}\hto\CC^*_\tau$},
\item
\index{$IHTZ$@$i_{\tau=\hb}$}$i_{\tau=\hb}:\cX^\circ\hto\thetacX$ is the inclusion induced by $\CC_\hb^*\hto\CC^*_\tau\times\CC_\hb$, $\hb\mto(\hb,\hb)$.
\end{itemize}
Lastly, we recall that $\pi:\cX=X\times\CC_\hb\to X$ denotes the projection, and $\pi^\circ$ its restriction $\cX^\circ:=X\times\CC^*_\hb\to X$. We have $\mu\circ i_{\tau=\hb}=i_{\hb=1}\circ \pi^\circ:\cX^\circ\to\cX$, so that we can identify $i_{\tau=\hb}^*\cO_{\thetacX}$ with $\cO_{\cX^\circ}$ by sending $\map(x,\tau,\hb)$ to $\map(x,\hb,\hb)$. Similarly, we identify $i_{\tau=\hb}^*\cO_{\taucX}(*\taucX_0)$ with $\cO_\cX[1/\hb]$.\vspace*{-5pt}
\end{notation}

\section{Rescaling an integrable \texorpdfstring{$\cR$}{R}-triple}\label{sec:rescaling}

\subsection{Rescaling a coherent \texorpdfstring{$\cO_\cX$}{OX}-module}\label{subsec:rescaling}

We define the functor $\mu^\cbbullet:\Mod(\CC_\cX)\mto\Mod(\CC_{\taucX})$ \resp $\mu^\star:\Mod(\cO_\cX)\mto\Mod(\cO_{\taucX})$ as the composition $\mu^\cbbullet=\nu_*\circ q^{-1}$ \resp $\mu^\star=\nu_*\circ q^*$. We note that $\nu_*$ induces an equivalence $\Mod(\cO_{\taucX}(*\taucX_0))\simeq\Mod(\cO_{\taucX}(*\taucX_0))$. We then define the functor $\mu^*:\Mod(\cO_\cX)\mto\Mod(\cO_{\taucX}(*\taucX_0))$ as the composition
\begin{equation}\label{eq:mustar}
\nu_*\circ(\cO_{\taucX}(*\taucX_0)\otimes\cbbullet)\circ q^*=(\cO_{\taucX}(*\taucX_0)\otimes\cbbullet)\circ\nu_*\circ q^*,
\end{equation}
and we can replace $\nu_*$ with $(\nu^{-1})^*$ on the left-hand side.

Note that $\cO_{\taucX}(*\taucX_0)$ is a flat $q^{-1}\cO_\cX$-module, so $\nu_*\cO_{\taucX}(*\taucX_0)=\cO_{\taucX}(*\taucX_0)$ is a flat $\mu^\cbbullet\cO_\cX$-module. For an $\cO_\cX$-module $\cF$, we have
\[
\mu^*\cF=\cO_{\taucX}(*\taucX_0)\otimes_{\mu^\cbbullet\cO_\cX}\mu^\cbbullet\cF.
\]

If $\ccF$ is an $\cO_X$-module, we set $\pi^{\circ*}\ccF:=\cO_{\cX}[1/\hb]\otimes_{\pi^{-1}\cO_X}\ccF$, and if~$\cF$ is an $\cO_\cX$-module, we have an isomorphism $i_{\tau=\hb}^*\mu^*\cF\simeq\pi^{\circ*}i_{\hb=1}^*\cF$ given by (with a slight abuse of notation)\vspace*{-.3\baselineskip}
\[
f(x,\tau,\hb)\otimes m\bmod(\tau-\hb)\mto f(x,\tau,\tau)\otimes m\bmod(\tau-\hb)=f(x,\hb,\hb)\otimes(m\bmod(1-\hb)).
\]

\begin{lemme}\label{lem:basicrescale}\index{rescaling!of a coherent $\cO_X$-module}
With the above notation,
\begin{enumerate}
\item\label{lem:basicrescale1}
$\mu^*$ is an exact functor $\Mod(\cO_\cX)\mto\Mod(\cO_{\taucX}(*\taucX_0))$ and for each $\tau_o\in\CC^*_\tau$,
\[
i_{\tau=\tau_o}^*\circ\mu^*=\bL i_{\tau=\tau_o}^*\circ\mu^*=\mu_{\tau_o}^*.
\]
\end{enumerate}
Let $\cF$ be an $\cO_\cX$-module and set $\ccF=i_{\hb=1}^*\cF$.
\begin{enumerate}\refstepcounter{enumi}
\item\label{lem:basicrescale3}
If $\cF$ is strict, then so is any $\mu_{\tau_o}^*\cF$, and $\mu^*\cF$ is $\cO_{\CC_\tau\times\CC_\hb}[1/\tau]$-flat. If $\lambda:\cF\to\cG$ is a strict morphism between strict objects of $\Mod(\cO_\cX)$, then $\mu^*\lambda$ and each $\mu_{\tau_o}^*\lambda$ are strict.
\item\label{lem:basicrescale5}
If $\cF$ is strict, then $\pi^{\circ*}\ccF=i_{\tau=\hb}^*\mu^*\cF=\bL i_{\tau=\hb}^*\mu^*\cF$ and if $\lambda:\cF\to\cG$ is a strict morphism between strict objects, then $i_{\tau=\hb}^*\mu^*\lambda=\pi^{\circ*}i_{\hb=1}^*\lambda$ is strict and $i_{\tau=\hb}^*\mu^*$ commutes with $\ker,\image,\coker$ for $\lambda$.
\item\label{lem:basicrescale6}
For $i=1,2$, let $\cF_i$ be an $\cO_{\cX_i}$-module which is strict. Then $\tau_1-\tau_2$ acts injectively on $\mu_1^*\cF_1\hbboxtimes\nobreak\mu_2^*\cF_2$ and we have $i_{\tau_1=\tau_2}^*(\mu_1^*\cF_1\hbboxtimes\mu_2^*\cF_2)=\mu^*(\cF_1\hbboxtimes\cF_2)$.
\end{enumerate}
\end{lemme}

\begin{proof}
The desired properties \eqref{lem:basicrescale1}--\eqref{lem:basicrescale3} are clear if we replace $\mu^*$ by the pullback by the projection~$q^*$ composed with the localization functor. They remain true after the isomorphism $\nu_*$. For \eqref{lem:basicrescale5} we use the equality $\bL i_{\tau=\hb}^*\circ\mu^*=\pi^{\circ*}\circ\bL i_{\hb=1}^*$ and the fact that, on strict objects, $\bL i_{\hb=1}^*=i_{\hb=1}^*$, and similarly on strict morphisms between strict objects. For \eqref{lem:basicrescale6}, we note that $\mu^*\cF=\bL\mu^*\cF$ by the flatness property above, and if $\cF_i$ are strict, then so are $\mu_i^*\cF_i$ ($i=1,2$), so that the external tensor products involved are strict and are also the derived ones. Since the desired equality holds at the derived level, we conclude that $\bL i_{\tau_1=\tau_2}^*(\mu_1^*\cF_1\hbboxtimes\nobreak\mu_2^*\cF_2)=\mu^*(\cF_1\hbboxtimes\cF_2)$, which is the desired statement.
\end{proof}

\subsection{Rescaling an \texorpdfstring{$\cR^\intt_\cX$}{RX}-module}
Let $\cM$ be a left $\cR^\intt_\cX$-module. Then the $\cO_{\taucX}(*\taucX_0)$-module $\mu^*\cM$ can be endowed with a natural action of $\cR_{\taucX}^\intt(*\tauX_0)$ in the following way: we set
\begin{equation}\label{eq:partiall}
\begin{aligned}
\hb(1\otimes m)=\hb\otimes m&=\tau\otimes\hb m=\tau(1\otimes\hb m),\\
\partiall_{x_i}(1\otimes m)&=\tau(1\otimes\partiall_{x_i}m)=\tau\otimes\partiall_{x_i}m,\\
\partiall_\tau(1\otimes m)&=-1\otimes\hb^2\partial_\hb m,\\
\hb^2\partial_\hb(1\otimes m)&=\tau(1\otimes\hb^2\partial_\hb m)=\tau\otimes\hb^2\partial_\hb m
\end{aligned}
\end{equation}
(the last equality also holds if we replace everywhere $\hb^2\partial_\hb$ with $\partial_\hb\hb^2$), and extend this action in the usual way. We denote this $\cR_{\taucX}^\intt(*\tauX_0)$-module by \index{$Mctau$@$\taucM$}$\taucM$, that we call the \index{rescaling!of an $\cR^\intt_\cX$-module}\emph{rescaling} of $\cM$. Similarly, any morphism $\lambda:\cM_1\to\cM_2$ can be rescaled as a morphism $\taulambda:\taucM_1\to\taucM_2$.

We also regard $\taucM$ as obtained by the composition of the pullback functor~$q^+$ of $\cR^\intt$-modules, the localization functor, and the sheaf-theoretic functor $\nu_*$ with the change of the action given by \eqref{eq:partiall}, which can also be interpreted as a pushforward~$\nu_+$ for $\cR_{\taucX}^\intt(*\tauX_0)$-modules.

The sheaf $\cR^\intt_{\taucX}(*\tauX_0)$ is coherent, equipped with the increasing filtration by locally free $\cO_{\taucX}$-submodules of finite rank obtained by bounding the order of operators together with the order of the pole along~$\tauX_0$. Accordingly, we have the notion of coherent (\resp good) $\cR^\intt_{\taucX}(*\tauX_0)$\nobreakdash-module, and similarly with the sheaf $\cR^\intt_{\taucX/\CC_\tau}(*\tauX_0)$ of relative operators.

On the other hand, recall that we set $\ccM=\Xi_{\DR}\cM$ (\cf Notation \ref{nota:XiDR}). If $\cM$ is $\cR_\cX$-coherent (\resp good), then $\ccM$ is $\cD_X$-coherent (\resp good).

\begin{lemme}\label{lem:relcoh}
Let $\cM$ be an $\cR^\intt_\cX$-module.
\begin{enumerate}
\item\label{lem:relcoh1}
For each $\tau_o\neq0$ we have a functorial identification $i_{\tau=\tau_o}^+\taucM=\mu_{\tau_o}^+\cM$ as $\cR_\cX$\nobreakdash-modules.
\item\label{lem:relcoh2}
If moreover $\cM$ is $\cR_\cX$-coherent (\resp $\cR_\cX$-good), then $\taucM$ is $\cR_{\taucX/\CC_\tau}(*\tauX_0)$-coherent (\resp $\cR_{\taucX/\CC_\tau}(*\tauX_0)$-good), hence $\cR_{\taucX}(*\tauX_0)$-coherent (\resp $\cR_{\taucX}(*\tauX_0)$-good).
\item\label{lem:relcoh4}
If moreover $\cM$ is strict, then for each $\tau_o\neq0$, $\taucM$ is \index{RXmodule@$\cR_\cX$-module!strictly $\RR$-specializable --!regular --}strictly specializable and regular along $\tau=\tau_o$ (\cf \cite[\S3.1.d]{Bibi01c}) and the corresponding $V$-filtration is given by $V_k\taucM=(\tau-\tau_o)^{\max(-k,0)}\taucM$ ($k\in\ZZ$).
\item\label{lem:relcoh3}
By the identification $i_{\tau=\hb}^*\cR_{\taucX/\CC_\tau}(*\tauX_0)=\cR_\cX[1/\hb]$ we have a functorial identification $i_{\tau=\hb}^*\taucM=\pi^{\circ*}\ccM$ as $\cR_\cX[1/\hb]$-modules, where on the right-hand side, $\partiall_{x_i}$ acts as $\hb\partial_{x_i}$. It is coherent as such if~$\cM$ is $\cR_\cX$-coherent. Moreover, the operator $\hb^2\partial_\hb+\tau\partiall_\tau$ sends $(\tau-\hb)\taucM$ into itself, and induces a $\pi^{-1}\cD_X$-linear endomorphism of $i_{\tau=\hb}^*\taucM$ which corresponds to the natural action of $\hb^2\partial_\hb$ on $\pi^{\circ*}\ccM$.
\end{enumerate}
\end{lemme}

\begin{proof}
Let us prove \eqref{lem:relcoh4}. Let us set $\tau'=\tau-\tau_o$. According to \eqref{eq:partiall} read in the variable~$\tau'$, the filtration $U_k\taucM$ defined by the right-hand side in \eqref{lem:relcoh4} satisfies all the properties characterizing the $V$-filtration since each $U_k\taucM$ is $\cR_{\taucX/\CC_\tau}(*\tauX_0)$-coherent by~\eqref{lem:relcoh2} and each $\gr_k^U\taucM$ is strict, by the strictness of $\cM$. We conclude that it is the $V$\nobreakdash-filtration of $\taucM$ along $\tau'=0$.
\end{proof}

\begin{remarque}[Side changing, integrability and rescaling]
We denote by $\omega_\cX$ the sheaf $\hb^{-d_X}\Omega^{d_X}_{\cX/\CC_\hb}$, with $d_X=\dim X$. It is a right $\cR_\cX$-module, that we enhance with a right $\cR^\intt_\cX$-module structure as follows: it is naturally equipped with a left action of $\hb^2\partial_\hb$ (even of a left action of $\hb\partial_\hb$); we transform it into a right action by transposition, by setting $\omega\cdot(\hb^2\partial_\hb):=-(\partial_\hb\hb^2)\omega=-\hb^2\partial_\hb\omega-2\hb\omega$. Then the side-changing functors for a left (\resp right) $\cR^\intt_\cX$-module $\cM$ (\resp $\cN$) are simply given by $\cM\mto\omega_\cX\otimes_{\cO_\cX}\cM$ and $\cN\mto\cHom_{\cO_\cX}(\omega_\cX,\cN)$ with the usual rules.

For a right $\cR^\intt_\cX$-module $\cM$, we set $\taucM=\mu^*\cM$ with right action given~by
\begin{align*}
(m\otimes1)\partiall_{x_i}&=(m\partiall_{x_i}\otimes1)\tau=m\partiall_{x_i}\otimes\tau,\\
(m\otimes1)\partiall_\tau&=-m\partial_\hb\hb^2\otimes1,\\
(m\otimes1)\hb^2\partial_\hb&=(m\hb^2\partial_\hb\otimes1)\tau=m\hb^2\partial_\hb\otimes\tau.
\end{align*}
(The last equality also holds if we replace everywhere $\hb^2\partial_\hb$ with $\partial_\hb\hb^2$.) We have ${}^\tau\!\omega_\cX=\omega_{\taucX/\CC_\tau}(*\taucX_0)$ with the right action specified by the formulas above. Then~${}^\tau\!\omega_\cX$ serves for the side-changing functors on $\taucX$. We end by noticing the compatibility for a left module $\cM$ and a right module $\cN$, in a functorial way in $\cM$ and~$\cN$:
\begin{align*}
{}^\tau\!(\omega_\cX\otimes_{\cO_\cX}\cM)&\simeq {}^\tau\!\omega_\cX\otimes_{\cO_{\taucX}(*\taucX_0)}\taucM,\\
{}^\tau\!\!\!\cHom_{\cO_\cX}(\omega_\cX,\cN)&\simeq\cHom_{\cO_{\taucX}(*\taucX_0)}({}^\tau\!\omega_\cX,\taucN).
\end{align*}
By definition, the equalities hold as $\cO_{\taucX}(*\taucX_0)$-modules, and the coincidence of the $\cR^\intt_{\taucX}(*\tauX_0)$\nobreakdash-actions is easily checked.
\end{remarque}

\subsection{Rescaling an integrable \texorpdfstring{$\iota$}{iota}-sesquilinear pairing \texorpdfstring{$\iC$}{iC}}\label{subsec:rescalingpairing}
Although the rescaling of an $\cR^\intt_\cX$-module can be defined meromorphically with \hbox{respect} to $\tau=0$, the rescaling of an integrable $\iota$-sesquilinear pairing is only \hbox{defined} on~$\thetacX$ in general. The question of its extendability as a moderate distribution along~$\taucX_0^\circ$ will be considered in Section \ref{sec:wellresc}.

Recall that we have defined in Section \ref{subsec:OmegacontinuousDb} the pullback with respect to $\mu$, that~is, the operation
\[
\mu^*:\mu^{-1}\Db_{\cX^\circ/\CC^*_\hb}\to\Db_{\thetacX^\circ/\CC^*_\theta\times\CC^*_\hb}\quad(\cX^\circ:=X\times\CC^*_\hb,\;\thetacX^\circ:=\thetaX\times\CC^*_\hb)
\]
attached to the map $\mu:\thetacX^\circ\to \cX^\circ$, that we write as $q\circ\nu^{-1}$. Note that, concerning the involution $\iota$, we have $\iota^+\thetacM={}^\theta\!(\iota^+\cM)$, according to the first line of \eqref{eq:partiall}. Let $\iC$ be an integrable pairing \eqref{eq:rescsesqui}. We define\index{$Citheta$@$\thetaiC$}\index{rescaling!of an integrable $\iota$-sesquilinear pairing}
\[
\thetaiC:\thetacM'_{|\thetacX^\circ}\otimes_\CC\iota^+\ov{\thetacM''}_{|\thetacX^\circ}\to\Db_{\thetacX^\circ/\CC^*_\hb}
\]
by $\cO_{\thetacX^\circ}\otimes_\CC\ov\cO_{\thetacX^\circ}$-linearity from the formula
\[
\thetaiC\bigl(1\otimes m',\iota^*(1\otimes\ov{m''})\bigr)=\mu^*\iC(m',\iota^*\ov{m''}),
\]
and by composing with the natural inclusion $\Db_{\thetacX^\circ/\CC^*_\theta\times\CC^*_\hb}\hto\Db_{\thetacX^\circ/\CC^*_\hb}$ (integrating with respect to $\theta$). One checks, by using Lemma \ref{lem:mustaru} and the relations \eqref{eq:partiall}, that~$\thetaiC$ is $\cD_{\thetacX^\circ,\ov\thetacX^\circ}$-linear. In such a way we get an object $\thetacT=(\thetacM',\thetacM'',\thetaiC)$ of $\RdiTriples(\thetaX)$.

\begin{lemme}\label{lem:thetavarphi}
Let $\lambda$ be a morphism in $\RdiTriples(X)$. Then $\thetalambda$ is a morphism in $\RdiTriples(\thetaX)$.
\end{lemme}

\begin{proof}
It amounts to proving that $\thetalambda$ is compatible with $\thetaiC_1,\thetaiC_2$, which is straightforward.
\end{proof}

On the other hand, since $\thetaiC$ takes values in $\Db_{\thetacX^\circ/\CC^*_\theta\times\CC^*_\hb}$, one can restrict it to $\theta=1$ and obtain a pairing $\cM'_{|\cX^\circ}\otimes_\CC\iota^+\ov{\cM''}_{|\cX^\circ}\to\Db_{\cX^\circ/\CC^*_\hb}$.

\begin{lemme}\label{lem:retrictioniC}
The restriction of $\thetaiC$ to $\theta=1$ is equal to $\iC$.
\end{lemme}

\begin{proof}
This immediate from the definition of $\mu^*$ as given in Section \ref{subsec:OmegacontinuousDb}.
\end{proof}

\begin{definition}[Rescaling]\label{def:rescalingRtriple}
The \index{rescaling functor}\emph{rescaling functor}
\[
\RdiTriples(X)\mto\RdiTriples(\thetaX)
\]
is defined by\index{$Tcitheta$@$\thetacT $}
\[
\icT=(\cM',\cM'',\iC)\mto\thetacT=(\thetacM',\thetacM'',\thetaiC).
\]
\end{definition}

As a consequence of Lemma \ref{lem:retrictioniC}, the functor ``restriction to $\theta=1$'' is a quasi-inverse to the rescaling functor.

\subsection{Behaviour of the rescaling with respect to standard functors}\label{subsec:behaviourtheta}

\skpt
\begin{lemme}[Adjunction, Tate twist by an integer, and rescaling]\label{lem:rescalingtateadj}
\begin{enumerate}
\item\label{lem:rescalingtateadj1}
Rescaling is compatible with adjunction, \ie $(\thetacT)^*={}^\theta\!(\icT^*)$.
\item\label{lem:rescalingtateadj2}
Rescaling is compatible with the Tate twist by an integer $k$, \ie
\[
\thetacT(k)={}^\theta\!(\icT(k)),\quad \forall k\in\ZZ.
\]
\end{enumerate}
\end{lemme}

\begin{proof}
The first assertion is obvious. For the second one, we use the formula
\[
\icT(k)=(\hb^{-k}\cM',\hb^k\cM'',\iC)
\]
(\cf\cite[\S2.1.8]{Mochizuki11}), and the identification ${}^\theta\!(\hb^k\cM)=\hb^k\theta^k\thetacM=\hb^k\thetacM$.
\end{proof}

\begin{proposition}[Pushforward and rescaling]\label{prop:rescalingimdir}
Let $\map:X\to Y$ be a proper morphism between complex manifolds, denote by \index{$Ftau$@$\taumap$}$\taumap:\tauX\to \tauY$ the induced morphism and let~$\thetamap$ its restriction to $\thetaX$. Let~$\icT$ be an object of $\RdiTriples(X)$ with good $\cR^\intt_\cX$-components $\cM',\cM''$ both denoted by $\cM$. Then the natural functorial morphisms
\begin{starequation}\label{eq:rescalingimdir1}
{}^\tau\!(\map_\dag^k\cM)\to\taumap_\dag^k\taucM,\quad
{}^\theta\!(\map_\dag^k\icT)\to\thetamap_\dag^k\thetacT
\end{starequation}%
are isomorphisms for all $k$. If $\cM$ is a good $\cR^\intt_\cX$-module and $\map_\dag^k\cM$ is strict for all $k$, then
\begin{starstarequation}\label{eq:rescalingimdir2}
i_{\tau=\hb}^*\taumap_\dag^k\taucM\simeq\pi^{\circ*}\map_\dag^k\ccM.
\end{starstarequation}%
\end{proposition}

\begin{proof}
For \eqref{eq:rescalingimdir1}, we use the decomposition \eqref{eq:mustar} and its analogue for $\cR^\intt$-modules. The compatibility of $\map_\dag^k$ with $q^+$ amounts to the compatibility with the external tensor product with the trivial integrable triple on $\CC_\tau$. The compatibility with localization is then clear, as well as that with $\nu_*$, which amounts to changing the actions as in \eqref{eq:partiall}. The argument for sesquilinear pairings is similar. Let us check \eqref{eq:rescalingimdir2}:
\begin{align*}
i_{\tau=\hb}^*\taumap_\dag^k\taucM&=i_{\tau=\hb}^*{}^\tau\!(\map_\dag^k\cM)\\
&=\pi^{\circ*}(i_{\hb=1}^*\map_\dag^k\cM)\quad\text{by \ref{lem:relcoh}\eqref{lem:relcoh3}}\\
&=\pi^{\circ*}\map_\dag^k\ccM\quad\text{by strictness of $\map_\dag^k\cM$ for all $k$.}\qedhere
\end{align*}
\end{proof}

\pagebreak[2]
The following proposition is mostly obvious.\vspace*{-3pt}

\begin{proposition}[Smooth pullback and rescaling]\label{prop:rescalingpullback}
Let $\mapsm:X\to Y$ be a smooth morphism, let $\taumapsm:\tauX\to \tauY$ be the induced (smooth) morphism, and let $\thetamapsm$ be its restriction to $\thetaX$. Let~$\icT$ be an object of $\RdiTriples(Y)$ and $\cM=\cM',\cM''$. Then there is a natural functorial isomorphism\vspace*{-3pt}
\begin{starequation}\label{eq:rescalingpullback1}
{}^\tau\!(\mapsm^+\cM)\simeq\taumapsm^+\taucM.
,\quad{}^\theta\!(\mapsm^+\icT)\simeq\thetamapsm^+\thetacT.
\end{starequation}%
Moreover, if $\cM$ is an $\cR^\intt_\cY$-module\vspace*{-3pt}
\begin{starstarequation}\label{eq:rescalingpullback2}
i_{\tau=\hb}^*\taumapsm^+\taucM\simeq\pi^{\circ*}\mapsm^+\ccM.
\end{starstarequation}%
\end{proposition}

By using Lemma \ref{lem:basicrescale}\eqref{lem:basicrescale6}, we also obtain the following result.\vspace*{-3pt}

\begin{proposition}[External product and rescaling]\label{prop:rescalingexttens}
For $i=1,2$, let $\icT_i$ be an object of $\RdiTriples(X_i)$ with strict $\cR^\intt_{\cX_i}$\nobreakdash-components. Then $\tau_1-\tau_2$ acts injectively on $\,{}^{\tau_1}\mkern-9mu \cM_1\hbboxtimes\nobreak{}^{\tau_2}\mkern-9mu \cM_2$ ($\cM_i=\cM'_i$ \resp $\cM''_i$), and there is a natural bi-functorial isomorphism\vspace*{-3pt}
\begin{starequation}\label{eq:rescalinghbboxtimes1}
{}^\tau\!(\cM_1\hbboxtimes\cM_2)\simeq i^+_{\tau_1=\tau_2}({}^{\tau_1}\mkern-9mu\cM_1\hbboxtimes{}^{\tau_2}\mkern-9mu\cM_2),\quad{}^\theta\!(\icT_1\hbboxtimes\icT_2)\simeq i^+_{\theta_1=\theta_2}({}^{\theta_1}\mkern-9mu\icT_1\hbboxtimes{}^{\theta_2}\mkern-9mu\icT_2).\vspace*{-.2\baselineskip}
\end{starequation}%
Moreover,\vspace*{-.3\baselineskip}%
\begin{starstarequation}\label{eq:rescalinghbboxtimes2}
i_{\tau=\hb}^*{}^\tau\!(\cM_1\hbboxtimes\cM_2)=\pi^{\circ*}(\ccM_1\boxtimes\ccM_2).
\end{starstarequation}%
\end{proposition}

\subsubsection*{Specialization, localization an maximalization along a hypersurface}
Let $\fun:X\to\CC$ be a holomorphic function and let $\cM$ be an $\cR_\cX^\intt$-module.

\skpt
\begin{proposition}[Strict specializability and rescaling]\label{prop:Vtheta}
\begin{enumerate}
\item\label{prop:Vtheta1}
Assume that $\cM$ is strictly specializable along $\{\fun=0\}$. Then $\taucM$ is so along $\fun\circ p=0$ and, for any $\beta\in\RR$, we have\vspace*{-3pt}\enlargethispage{\baselineskip}%
\begin{starequation}\label{eq:thetapsi}
\psi_{\fun\circ p,\beta}\taucM={}^\tau\!(\psi_{\fun,\beta}\cM).
\end{starequation}%
Moreover, the nilpotent endomorphism ${}^\tau\rN$ on the left-hand side is induced by $\tau\otimes\rN$ on the right-hand side.

\item\label{prop:Vtheta2}
If moreover $\fun$ is a projection $X=X_0\times\CC_t\to\CC_t$, the $V$-filtrations are related~by\vspace*{-3pt}
\begin{starstarequation}\label{eq:thetaV}
V_\beta\taucM={}^\tau\!(V_\beta\cM).
\end{starstarequation}%

\item\label{prop:Vtheta3}
If $\iC$ is a $\iota$-sesquilinear pairing between strictly specializable $\cR_\cX^\intt$-modules, then, under the isomorphism \eqref{eq:thetapsi}, we have $\psi_{\fun\circ p,\beta}(\thetaiC)={}^\theta\!(\psi_{\fun,\beta}\iC)$.
\end{enumerate}
\end{proposition}

\begin{proof}
Let $i_\fun:X\hto X\times\CC_t$ denote the graph inclusion. Note that, by Proposition \ref{prop:rescalingimdir}, we have ${}^\tau\!(\cM_\fun)={}^\tau\!i_{\fun+}\taucM$, so we can assume, for the proof, that $\fun$ is a projection, as \eqref{prop:Vtheta2} that we start proving.

Recall that, due to integrability, the roots of the Bernstein polynomials involved with $\cM$ are of the form $\beta\hb$ with $\beta\in\RR$ (\cf Proposition \ref{prop:speint}\eqref{prop:speint2a}). If $m$ is a local section of $\cM$, then $(t\partiall_t+\beta\hb)(1\otimes m)=\tau(1\otimes(t\partiall_t+\beta\hb)m)$. The filtration $U_\bbullet\taucM$ defined by the right-hand side of \eqref{eq:thetaV} is a good $V$-filtration (easy), and $(t\partiall_t+\beta\hb)$ is nilpotent on $\gr_\beta^U\taucM$. Moreover, by flatness of $\cO_{\taucX}(*\taucX_0)$ over $\cO_\cX$, $\gr_\beta^U\taucM={}^\tau\!(\gr_\beta^V\cM)$. We conclude that $U_\bbullet\taucM$ satisfies the properties characterizing the $V$\nobreakdash-filtration along $\fun=0$, so it is equal to it, and $\taucM$ is strictly specializable along $\fun\circ p=0$. Now, \eqref{prop:Vtheta1} is clear.

For \eqref{prop:Vtheta3}, we keep the setting as in \eqref{prop:Vtheta2}. Following the definition (\cf\cite[(3.6.10)]{Bibi01c}), we first note that, if $\chi(t)$ is a test function near the origin in $\CC_t$ such that $\chi\equiv1$ near $t=0$, $\mu^*$ of the meromorphic (with respect to the new variable $s$) distribution\vspace*{-3pt}
\[
\langle|t|^{2s}\iC(m',\ov{m''}),\cbbullet\wedge\chi(t)\itwopi\rd t\wedge\rd\ov t\rangle
\]
is the meromorphic distribution
\[
\langle|t|^{2s}(\mu^*\iC)(m'\otimes1,\ov{m''\otimes1}),\cbbullet\wedge\chi(t)\itwopi\rd t\wedge\rd\ov t\rangle.
\]
Since the poles of the former meromorphic distribution are independent of $\hb$ (they are of the form $\beta\hb/\hb=\beta$), they also are the poles of the latter, and the polar coefficients, among which the residue, correspond by $\mu^*$. This enables us to conclude that $\psi_{\fun\circ p,\beta}(\thetaiC)={}^\theta\!(\psi_{\fun,\beta}\iC)$.
\end{proof}

Let $H$ (\resp $\tauH$) be the divisor of $\fun$ (\resp $\fun\circ p$, \ie $\tauH=H\times\CC_\tau$). If $\cM$ is an integrable $\cR_\cX$-module, then its ``stupid'' localization $\cM(*H)$ is so. Recall that, if $\cM(*H)$ is strictly specializable along $H$, we say that it is localizable along $H$ if the modules $\cM[*H]$ and $\cM[!H]$ exist (\cf\cite[\S3.3]{Mochizuki11}) and that they exist if $\fun$ is a projection.

\skpt
\begin{proposition}[Localization and rescaling]\label{prop:starHtheta}
\begin{enumerate}
\item\label{prop:starHtheta1}
Assume that $\cM$ is strictly specializable and localizable along $H$. Then so is $\taucM$ along $\tauH$ and we have
\[
\taucM[\star\tauH]={}^\tau\!(\cM[\star H])\qquad(\star=*,!).
\]
\item\label{prop:starHtheta2}
If $\iC$ is a $\iota$-sesquilinear pairing between strictly specializable and localizable $\cR_\cX^\intt$-modules, then $\thetaiC[\star\thetaH]={}^\theta\!(\iC[\star H])$.
\end{enumerate}
\end{proposition}

\begin{proof}
For \eqref{prop:starHtheta1}, we first note that $\taucM(*\tauH)={}^\tau\!(\cM(*H))$ is obvious. If $\fun$ is a projection, then the result follows from Proposition \ref{prop:Vtheta}, according to the definition of $[\star H]$. It follows that, for an arbitrary $\fun$, ${}^\tau\!(\cM[\star H])$ satisfies the properties expected for $\taucM[\star\tauH]$, according to Proposition \ref{prop:rescalingimdir} applied to the graph inclusion of $\fun$.

For \eqref{prop:starHtheta2}, we note that, according to \eqref{eq:thetaV} and the proof of \cite[Prop.\,3.2.1]{Mochizuki11}, it is enough to check the equality $\thetaiC(*\thetaH)={}^\theta\!(\iC(*H))$ for the ``stupidly'' localized $\iota$-sesquilinear forms. This amounts to showing that if $\iC$ takes values in $\Db_{(X\moins H)\times\CC_\hb^*/\CC_\hb^*}^{\modH}$, then $\mu^*\iC$ takes values in $\Db_{(X\moins H)\times\CC_\tau\times\CC_\hb^*/\CC_\tau\times\CC_\hb^*}^{\rmod\thetaH}$. This is obvious since $\Db_{(X\moins H)\times\CC_\hb^*/\CC_\hb^*}^{\modH}=\Db_{X\times\CC_\hb^*/\CC_\hb^*}(*H)$.
\end{proof}

\begin{remarque}[Maximalization and Beilinson's gluing construction]\label{rem:thetagluing}
In a way similar to \cite[\S4.1.7]{Mochizuki11}, we find that if, together with the assumptions of Proposition \ref{prop:starHtheta}, $\cM$ is \index{maximalization}maximalizable along $H$, then so is $\taucM$ along $\tauH$, and the \index{Beilinson's functor}Beilinson functor is compatible with rescaling, that is, \index{$XXIB$@$\Xi_\fun$}${}^\tau\!(\Xi_\fun\cM)\simeq\Xi_{\fun\circ p}\taucM$. As a consequence, any object $\taucM$ can be recovered by a gluing construction from the data $(\taucM[*H],\taucN,\mu^*\can,\mu^*\var)$ with $\cN\simeq\phi_{\fun,1}\cM$ supported on $H$.
\end{remarque}

\subsubsection*{Exponential twist by a meromorphic function}
Let now $\mero$ be a meromorphic function on $X$ with pole divisor $P$. For an integrable $\cR_\cX$-module $\cM$ we denote by $\cE^{\mero/\hb}_*\otimes\cM$ the $\cO_\cX$-module $\cM(*P)$ equipped with the twisted structure of $\cR_\cX^\intt$-module defined locally by the formulas
\[
\partiall_{x_i}(m\otimes\refh)=\Bigl(\partiall_{x_i}+\frac{\partial \mero}{\partial x_i}\Bigr)m\otimes\refh,\quad\hb^2\partial_\hb(m\otimes\refh)=(\hb^2\partial_\hb-\mero)m\otimes\refh.
\]

\begin{proposition}[Exponential twist and rescaling]\label{prop:expftheta}
We have the following equality of $\cR_{\tauX}^\intt$-modules (we also regard $\mero$ as a meromorphic function on $\tauX$):
\[
{}^\tau\!(\cE^{\mero/\hb}\otimes\cM)=\cE^{\tau \mero/\hb}\otimes\taucM.
\]
Moreover, if $\cE^{\mero/\hb}\otimes\cM$ is localizable along $P$ then $\cE^{\tau \mero/\hb}\otimes\taucM$ is so along ${}^\tau\!P$ and we have
\[
{}^\tau\!(\cE^{\mero/\hb}\otimes\cM[\star P])=(\cE^{\tau \mero/\hb}\otimes\taucM)[\star{}^\tau\!P]\qquad(\star=*,!).
\]
\end{proposition}

\begin{proof}
The first part is a direct check by using Formulas \eqref{eq:partiall}. The second part follows then from Proposition \ref{prop:starHtheta}.
\end{proof}

For further reference, we collect all compatibility results.

\begin{proposition}\label{cor:imdirresc}
The functor $\RdiTriples(X)\mto\RdiTriples(\thetaX)$ defined by
\[
(\cM',\cM'',\iC)\mto(\thetacM',\thetacM'',\thetaiC)
\]
is compatible with proper pushforward and smooth pullback (\cf Section \ref{subsec:stdfunctorsRdiT}). When restric\-ted to strictly specializable and localizable objects, the functor is compatible with nearby cycles, localization $[\star H]$ ($\star=*,!$) and exponential twist. When restricted furthermore to maximalizable objects, it is compatible with Beilinson's functor $\Xi_\fun$, to vanishing cycles $\phi_{\fun,1}$ and to the corresponding gluing construction.
\end{proposition}

\begin{proof}
The first part has already been proved. Let us emphasize the compatibility with exponential twist. It follows from the property that the function $\exp(\tau \mero/\hb-\nobreak\ov{\tau \mero/\hb})$ is a multiplier on $\Db_{\thetacX^\circ/\CC^*_\hb}$ since it has moderate growth along $P\times\CC^*_\theta\times\CC^*_\hb$ as well as all its derivatives. Note that the minus sign is due to the involution $\iota$.
\end{proof}

\section{Well-rescalable integrable \texorpdfstring{$\cR$}{R}-triples}\label{sec:wellresc}

\subsection{Well-rescalable integrable \texorpdfstring{$\cR$}{R}-modules}\label{subsec:wellrescRmod}

We will use the results on strict specializability and regularity for the projection \hbox{$\tauX\to\CC_\tau$}, as recalled in Section \ref{subsec:shortreminder}. We will also use that $i_{\tau=\hb}^*\cR^\intt_{\taucX}(*\tauX_0)=\cR^\intt_\cX[\hbm]$.

\begin{definition}[Well-rescalable $\cR^\intt_\cX$-modules]\label{def:Rrescdata}
Let $\cM$ be a good $\cR^\intt_\cX$\nobreakdash-module. We say that it is \emph{well-rescalable} if the good $\cR^\intt_{\taucX}(*\tauX_0)$-module $\taucM$ is \emph{strictly $\RR$\nobreakdash-specializable and regular along $\tauX_0$} (\cf \cite[\S3.1.d]{Bibi01c}).
\end{definition}

\skpt
\begin{remarques}\label{rem:Rrescdata}
\begin{enumerate}
\item\label{rem:Rrescdata0a}
By a well-rescalable $\cR^\intt_\cX$-module, we will always mean a well-rescalable \emph{good} $\cR^\intt_\cX$-module.
\item\label{rem:Rrescdata0b}
It follows from Lemma \ref{lem:basicrescale}\eqref{lem:basicrescale1} that $i_{\tau=1}^*\taucM=\bL i_{\tau=1}^*\taucM$.

\item(\emph{$\cR_\cX$-coherence})\label{rem:Rrescdata3b}
By the regularity assumption in Definition \ref{def:Rrescdata}, since we have $i_{\tau=\hb}^*\cR_{\taucX/\CC_\tau}\simeq\cR_\cX$, the module $i_{\tau=\hb}^*\tauV_\beta\taucM$ is $\cR_\cX$\nobreakdash-coherent. This is the statement where the regularity condition in \ref{def:Rrescdata} plays a major role.

\item(\emph{The $\hb$-adic filtration and the action of} $\hb^2\partial_\hb+\tau\partiall_\tau$)\label{rem:Rrescdata4}
Let us consider the $\hb$-adic filtration
\[
\hb^k\cO_\cX\otimes_{\pi^{-1}\cO_X}\pi^{-1}\ccM\subset\pi^{\circ*}\ccM.
\]
The corresponding graded module $\gr\bigl(\pi^{\circ*}\ccM\bigr)$ is $\cO_X[\hb,\hbm]\otimes_{\cO_X}\nobreak\ccM$, with $\gr^p\bigl(\pi^{\circ*}\ccM\bigr)\simeq\hb^p\otimes\ccM$, and we have (analytification with respect to $\hb$)
\[
\pi^{\circ*}\ccM=\cO_\cX[\hbm]\otimes_{\pi^{-1}\cO_X[\hb,\hbm]}\gr\bigl(\pi^{\circ*}\ccM\bigr).
\]
With the identification of Lemma \ref{lem:relcoh}\eqref{lem:relcoh3}, the above isomorphism becomes
\begin{starequation}\label{eq:adic}
i_{\tau=\hb}^*\taucM\simeq\cO_\cX[\hbm]\otimes_{\pi^{-1}\cO_X[\hb,\hbm]}\gr\bigl(i_{\tau=\hb}^*\taucM\bigr),
\end{starequation}%
with $\hb^p\otimes\ccM\simeq\gr^p\bigl(i_{\tau=\hb}^*\taucM\bigr)$.

On the other hand, we can identify $\gr^p\bigl(i_{\tau=\hb}^*\taucM\bigr)$ with a sub-object of $i_{\tau=\hb}^*\taucM$. Indeed,
the natural action of $\hb^2\partial_\hb+\tau\partiall_\tau$ on $i_{\tau=\hb}^*\taucM$ corresponds to that of $\hb^2\partial_\hb$ on $\pi^{\circ*}\ccM$, according to Lemma \ref{lem:relcoh}\eqref{lem:relcoh3}. Therefore, since
\[
\hb^p\otimes\pi^{-1}\ccM=\ker\big(\hb^2\partial_\hb-p\hb\text{ acting on }\pi^{\circ*}\ccM\big),
\]
we also have $\gr^p\bigl(i_{\tau=\hb}^*\taucM\bigr)\simeq\ker(\hb^2\partial_\hb+\tau\partiall_\tau-p\hb)$, where the latter operator acts on $i_{\tau=\hb}^*\taucM$. If we write $\gr\bigl(i_{\tau=\hb}^*\taucM\bigr)=\bigoplus_p\gr^p\bigl(i_{\tau=\hb}^*\taucM\bigr)\cdot\hb^p$, we thus obtain
\begin{starstarequation}\label{eq:adicsub}
i_{\tau=\hb}^*\taucM=\cO_\cX[\hbm]\otimes_{\pi^{-1}\cO_X[\hb,\hbm]}\Bigl(\bigoplus_p\ker(\hb^2\partial_\hb+\tau\partiall_\tau-p\hb)\cdot\hb^p\Bigr).
\end{starstarequation}%
\end{enumerate}
\end{remarques}

\begin{lemme}\label{lem:Vrescaling}
Assume that the good $\cR^\intt_\cX$-module $\cM$ is well-rescalable.
\begin{enumerate}
\item\label{lem:Vrescaling1}
For each $\beta\in\RR$, we have $(\tau-\hb)\tauV_\beta\taucM=(\tau-\hb)\taucM\cap\tauV_\beta\taucM$ and thus a natural inclusion
\begin{starequation}\label{eq:Vrescaling}
i_{\tau=\hb}^*\tauV_\beta\taucM\hto i_{\tau=\hb}^*\taucM.
\end{starequation}%
\item\label{lem:Vrescaling3}
For each $\alpha\in[0,1)$, the morphism $\chi$ induces an inclusion
\[
i_{\tau=\hb}^*\tauV_\alpha\taucM\hto\pi^{\circ*}\ccM,
\]
which is an isomorphism when restricted to $\hb\neq0$.
\item\label{lem:Vrescaling2}
For each $\alpha\in[0,1)$, $i_{\tau=\hb}^*\tauV_\alpha\taucM$ is strict.
\end{enumerate}
\end{lemme}

It follows from the last statement of \ref{lem:Vrescaling}\eqref{lem:Vrescaling3} that $\Xi_{\DR}(i_{\tau=\hb}^*\tauV_\alpha\taucM)\isom\ccM$ (\cf Notation \ref{nota:XiDR}).

\pagebreak[2]
\skpt
\begin{proof}
\begin{enumerate}
\item (\Cf\cite[Proof of Prop.\,3.1.2]{E-S-Y13}.) The result is clear away from $\{\tau=\hb=0\}$ since $\tauV_\beta\taucM=\taucM$ away from $\{\tau=0\}$. We therefore work at a point $(x,0,0)$. Let $m$ be a local section of $\tauV_\gamma\taucM$ such that $(\tau-\hb)m$ is a local section of $\tauV_\beta\taucM$. If $\gamma>\beta$ and the class of $m$ in $\gr_\gamma^{\tauV}\taucM$ is nonzero, then the class of $(\tau-\hb)m$ in $\gr_\gamma^{\tauV}\taucM$ is zero, but this is the class of $-\hb m$. By strictness of $\gr_\gamma^{\tauV}\taucM$ we get a contradiction.

\item
It suffices to compose \eqref{eq:Vrescaling} with the identification of Lemma \ref{lem:relcoh}\eqref{lem:relcoh3}.

When restricting this morphism to $\cX^\circ$, we can also restrict $\taucM$ to $\thetacX$, and the last assertion follows from the identification $\tauV_\alpha\taucM_{|\thetacX}=\taucM_{|\thetacX}$.

\item
The strictness of $i_{\tau=\hb}^*\tauV_\alpha\taucM$ immediately follows from the inclusion in \eqref{lem:Vrescaling3}.
\qedhere
\end{enumerate}
\end{proof}

\begin{definition}[The irregular Hodge filtration]\label{def:irrHF}
\index{irregular Hodge filtration}
Let $\cM$ be a well-rescalable good $\cR^\intt_\cX$\nobreakdash-module. For every $\alpha\in[0,1)$, the $\hb$-adic filtration $\hb^k\cO_\cX\otimes_{\pi^{-1}\cO_X}\pi^{-1}\ccM$ induces on $i_{\tau=\hb}^*\tauV_\alpha\taucM$ a filtration whose graded module $\gr\bigl(i_{\tau=\hb}^*\tauV_\alpha\taucM\bigr)$ is a graded $R_F\cD_X$-submodule of $\gr\bigl(\pi^{\circ*}\ccM\bigr)=\cO_X[\hb,\hbm]\otimes_{\cO_X}\ccM$, hence strict and corresponding thus to a uniquely defined good filtration \index{$Fzzirr$@$F^\irr_\bbullet$}$F^\irr_{\alpha+\bbullet}\ccM$ indexed by $\ZZ$, in the sense that$\index{$RFirr$@$R_{F_{\alpha+\bbullet}^\irr}\ccM$}$
\[
\gr\bigl(i_{\tau=\hb}^*\tauV_\alpha\taucM\bigr)=R_{F_{\alpha+\bbullet}^\irr}\ccM
\]
as graded $R_F\cD_X$-module.
\end{definition}

\begin{remarque}[Filtration indexed by $\RR$]
We can regard the family $\!(F^\irr_{\alpha+p}\ccM)_{\alpha\in[0,1),p\in\ZZ}$ as a family of coherent $\cO_X$-submodules of $\ccM$ indexed by $\RR$ (or by $A+\ZZ$ for some finite subset $A\subset[0,1)$). We claim that it is a \emph{filtration indexed by $\RR$}, that is, $F^\irr_\beta\ccM\subset F^\irr_\gamma\ccM$ for $\beta\leq\gamma\in\RR$. As a matter of fact, multiplication by $\tau^{-p}$ induces an isomorphism of graded modules
\[
\gr\bigl(i_{\tau=\hb}^*\tauV_\alpha\taucM\bigr)\isom\gr\bigl(i_{\tau=\hb}^*\tauV_{\alpha+p}\taucM\bigr)[-p]
\]
and thus identifies $F^\irr_{\alpha+p}\ccM$ with the term of degree zero in $\gr\bigl(i_{\tau=\hb}^*\tauV_{\alpha+p}\taucM\bigr)$. The inclusion $\tauV_\beta\taucM\subset\tauV_\gamma\taucM$ induces the inclusion between the degree zero terms of $\gr\bigl(i_{\tau=\hb}^*\tauV_\beta\taucM\bigr)$ and $\gr\bigl(i_{\tau=\hb}^*\tauV_\gamma\taucM\bigr)$, hence the desired inclusion.
\end{remarque}

\begin{notation}\label{nota:tauvarphialpha}
Let $\lambda:\cM_1\to\cM_2$ be a morphism between well-rescalable $\cR^\intt_\cX$-modules. For $\alpha\in[0,1)$, we denote by $\taulambda_\alpha$ the induced morphism $\tauV_\alpha\taucM_1\to\tauV_\alpha\taucM_2$ and by $\gr(i_{\tau=\hb}^*\taulambda_\alpha)$ the corresponding graded morphism (with respect to the induced $\hb$\nobreakdash-adic filtrations).
\end{notation}

\begin{remarque}\label{rem:tauvarphialpha}
Recall (\cf \cite[Prop.\,7.3.1]{Bibi01c}) that $\tauV_\alpha\taucM$ is preserved by $\hb^2\partial_\hb$ and, of course, by $\tau\partiall_\tau$. It follows that $\tauV_\alpha\taucM$ is stable by $\hb^2\partial_\hb+\tau\partiall_\tau$. As a consequence, the $\cO_X$-module
\[
F^{\prime\irr}_{\alpha+p}\ccM:=\ker\big[(\hb^2\partial_\hb+\tau\partiall_\tau-p\hb):i_{\tau=\hb}^*\tauV_\alpha\taucM\to i_{\tau=\hb}^*\tauV_\alpha\taucM\big]
\]
is identified by means of \eqref{eq:adicsub} with a submodule of $\ccM$ (identified with $\hb^p\otimes\ccM$), since it is isomorphic to $(\hb^p\otimes\ccM)\cap(i_{\tau=\hb}^*\tauV_\alpha\taucM)$.
\end{remarque}

\skpt
\begin{definition}[Graded well-rescalable $\cR^\intt_\cX$-modules and morphisms]\label{def:gRrescdata}
\begin{enumerate}
\item\label{def:gRrescdata1}
Let $\cM$ be a well-rescalable $\cR^\intt_\cX$-module. We say that it is \index{well-rescalable $\cR^\intt_\cX$-module!graded --}\emph{graded} if, through the isomorphism \eqref{eq:adic}, we have
\begin{starequation}\label{eq:twrescaled}
\forall\alpha\in[0,1),\quad\cO_\cX\otimes_{\pi^{-1}\cO_X[\hb]}\pi^{-1}R_{F_{\alpha+\bbullet}^\irr}\ccM=i_{\tau=\hb}^*\tauV_\alpha\taucM.
\end{starequation}%
\item\label{def:gRrescdata2}
Let $\lambda:\cM_1\to\cM_2$ be a morphism between graded well-rescalable $\cR^\intt_\cX$-modules. We say that $\lambda$ is \emph{graded} if, for each $\alpha\in[0,1)$, with respect to the identification \eqref{eq:twrescaled} we have (\cf Notation \ref{nota:tauvarphialpha})
\begin{starstarequation}\label{eq:twrescaledmorphism}
i_{\tau=\hb}^*\taulambda_\alpha=\id\otimes\gr(i_{\tau=\hb}^*\taulambda_\alpha).
\end{starstarequation}%
\end{enumerate}
\end{definition}

\skpt
\begin{lemme}\label{lem:lambdagraded}
\begin{enumerate}
\item\label{lem:lambdagraded1}
If $\cM$ is a well-rescalable $\cR^\intt_\cX$-module, it is graded if and only if \eqref{eq:twrescaled} holds with $R_{F_{\alpha+\bbullet}^{\prime\irr}}\ccM$ (\cf Remark \ref{rem:tauvarphialpha}), and then we have $F_{\alpha+\bbullet}^{\prime\irr}\ccM=F_{\alpha+\bbullet}^\irr\ccM$.
\item\label{lem:lambdagraded2}
Any morphism $\lambda$ between graded well-rescalable $\cR^\intt_\cX$-modules $\cM_1,\cM_2$ is graded.
\item\label{lem:lambdagraded3}
Let $\cM_1,\cM_2$ be well-rescalable $\cR^\intt_\cX$-modules. If their direct sum is graded, so is each of them.
\end{enumerate}
\end{lemme}

\skpt
\begin{proof}
\begin{enumerate}
\item
If \eqref{eq:twrescaled} holds, then $\ker(\hb^2\partial_\hb-p\hb)$ acting on the left-hand side of \eqref{eq:twrescaled} is equal to $F_{\alpha+p}^\irr\ccM$. Conversely, if \eqref{eq:twrescaled} holds with $R_{F_{\alpha+\bbullet}^{\prime\irr}}\ccM$, then grading the left-hand side of the corresponding \eqref{eq:twrescaled} with respect to the $\hb$-adic filtration gives $R_{F_{\alpha+\bbullet}^{\prime\irr}}\ccM$, hence $R_{F_{\alpha+\bbullet}^{\prime\irr}}\ccM=R_{F_{\alpha+\bbullet}^\irr}\ccM$.
\item
Since $\taulambda$ commutes with $\hb^2\partial_\hb+\tau\partiall_\tau$, it follows that, if we regard $\pi^{-1}R_{F_{\alpha+\bbullet}^{\prime\irr}}\ccM$ as being naturally included in $\cO_\cX\otimes_{\pi^{-1}\cO_X[\hb]}\pi^{-1}R_{F_{\alpha+\bbullet}^{\prime\irr}}\ccM$, then $\taulambda$ sends $\pi^{-1}R_{F_{\alpha+\bbullet}^{\prime\irr}}\ccM_1$ to $\pi^{-1}R_{F_{\alpha+\bbullet}^{\prime\irr}}\ccM_2$. The desired assertion is thus a consequence of \eqref{lem:lambdagraded1}.
\item
Set $\ccM=\ccM_1\oplus\ccM_2$. We have $R_{F_{\alpha+\bbullet}^{\prime\irr}}\ccM=R_{F_{\alpha+\bbullet}^{\prime\irr}}\ccM_1\oplus R_{F_{\alpha+\bbullet}^{\prime\irr}}\ccM_2$ and $\tauV_\alpha\taucM=\tauV_\alpha\taucM_1\oplus\tauV_\alpha\taucM_2$. The natural morphism\vspace*{-.2\baselineskip}%
\[
\cO_\cX\otimes_{\pi^{-1}\cO_X[\hb]}\pi^{-1}R_{F_{\alpha+\bbullet}^{\prime\irr}}\ccM\to i_{\tau=\hb}^*\tauV_\alpha\taucM
\]
is thus the direct sum of the corresponding morphisms for $\ccM_1$ and $\ccM_2$. It is an isomorphism if and only if each of them is such.\qedhere
\end{enumerate}
\end{proof}

\skpt
\begin{remarques}\label{rem:filteredmorphism}
\begin{enumerate}
\item\label{rem:filteredmorphism1}
Property \eqref{eq:twrescaledmorphism} implies that the morphism \hbox{$i_{\hb=1}^*\lambda:\ccM\to\ccM'$} is filtered with respect to $F^\irr_{\alpha+\bbullet}$ for each $\alpha\in[0,1)$. We will give below sufficient conditions in order that this filtered morphism is strict.

\item\label{rem:filteredmorphism2}
By the assumption of strict $\RR$-specializability, the $\cR_\cX$-module $\gr^{\tauV}_\alpha\!(\taucM):=\tauV_\alpha\taucM/\tauV_{<\alpha}\taucM$ ($\alpha\in[0,1)$) is coherent and strict. It is a priori not equal to the analytification of the Rees module of a coherent filtered $\cD_X$-module. However, the associated Higgs sheaf $\gr^{\tauV}_\alpha\!\!\taucM/\hb\gr^{\tauV}_\alpha\!\!\taucM$ is identified with the analytification of\vspace*{-.2\baselineskip}\enlargethispage{.5\baselineskip}%
\[
R_{F^\irr_{\alpha+\bbullet}}\ccM/R_{F^\irr_{<\alpha+\bbullet}}\ccM=\bigoplus_{p\in\ZZ}\gr_{\alpha+p}^{F^\irr}\ccM.
\]
In particular, the support of this Higgs sheaf is homogeneous and the union of the supports when $\alpha$ varies in $[0,1)$ is the characteristic variety of $\ccM$. If moreover $\ccM$ is holonomic, then the previous identification can be extended to characteristic cycles.
\end{enumerate}
\end{remarques}

\begin{lemme}[A criterion for isomorphic irregular Hodge filtrations]\label{lem:critgradedrescaled}
Let $\cM_i$ ($i=1,2$) be isomorphic $\cR^\intt_\cX$-modules. Assume that $\cM_i$ are strict and well-rescalable. Then $(\ccM_i,F^\irr_{\alpha+\bbullet}\ccM_i)$ ($i=1,2$) are filtered-isomorphic.
\end{lemme}

\begin{proof}
Let $\lambda:\cM_1\isom\cM_2$ be an isomorphism. Because of the uniqueness of the $\tauV$-filtration for strictly $\RR$-specializable $\cR_\cX$\nobreakdash-modules, $\taulambda$ induces an isomorphism $\taulambda_\alpha:\tauV_\alpha\taucM_1\isom\tauV_\alpha\taucM_2$ for each $\alpha\in[0,1)$. We thus have a commutative diagram
\[
\xymatrix@C=1.5cm@R=1cm{
i_{\tau=\hb}^*\tauV_\alpha\taucM_1\ar[r]^-{i_{\tau=\hb}^*\taulambda_\alpha}_-\sim\ar@{^{ (}->}[d]_{\eqref{eq:Vrescaling}}& i_{\tau=\hb}^*\tauV_\alpha\taucM_2\ar@{^{ (}->}[d]^{\eqref{eq:Vrescaling}}\\
i_{\tau=\hb}^*\taucM_1\ar[r]^-{i_{\tau=\hb}^*\taulambda}_-\sim& i_{\tau=\hb}^*\taucM_2
}
\]
Since~$\lambda:=i_{\tau=1}^*\taulambda$ is strict, the morphism $i_{\hb=1}^*\lambda$ is an isomorphism, and so is the pullback $\pi^{\circ*}i_{\hb=1}^*\lambda$, which is moreover obviously strict with respect to the $\hb$-adic filtration, hence $i_{\tau=\hb}^*\taulambda$ is also strict with respect to the $\hb$-adic filtration. Since the $\hb$-adic filtration on the upper terms of the diagram is the filtration induced by that on the lower terms through the inclusions \eqref{eq:Vrescaling}, we conclude that $i_{\tau=\hb}^*\taulambda_\alpha$ is a strict isomorphism with respect to the induced $\hb$-adic filtration, so $\gr(i_{\tau=\hb}^*\taulambda_\alpha)$ is an isomorphism.
\end{proof}

\skpt
\begin{lemme}[Strictness]\label{lem:strictness}
\begin{enumerate}
\item\label{lem:strictness1}
Let $\cM$ be a strict well-rescalable $\cR^\intt_\cX$-module. Then, for any germ of holomorphic function $q(\tau,\hb)$, multiplication by $q(\tau,\hb)$ is injective on $\taucM$ (hence on $\tauV_\beta\taucM$ for all $\beta\in\RR$). In particular, $\taucM$ is strict and $(\tau-\hb)$ is injective on $\taucM$.

\item\label{lem:strictness2}
Let $\lambda$ be a strict morphism between strict well-rescalable $\cR^\intt_\cX$-modules $\cM_1,\cM_2$. Then, $\taulambda$ is strict.
\end{enumerate}
\end{lemme}

\skpt
\begin{proof}
\begin{enumerate}
\item[\eqref{lem:strictness1}]
By Lemma \ref{lem:basicrescale}\eqref{lem:basicrescale3}, $\taucM_{|\thetacX}$ is $\cO_{\CC^*_\tau\times\CC_\hb}$-flat. The kernel of multiplication by $q(\tau,\hb)$ is thus supported on $\tau=0$, hence each of its local sections is annihilated by some power of $\tau$. Since multiplication by $\tau$ is invertible on~$\taucM$, this kernel is zero.

\item[\eqref{lem:strictness2}]
Recall that the functor $^\tau(\cbbullet)$ is exact, hence $\ker\taulambda={}^\tau(\ker\lambda)$, etc. By using the faithful flatness of $\cO_{\thetacX}$ over $\mu^{-1}\cO_\cX$, we conclude that $(\ker\taulambda)_{|\thetacX}$ and $(\coker\taulambda)_{|\thetacX}$ are strict. Since~$\tau$ is bijective on $\ker\taulambda$ and $\coker\taulambda$, we conclude as in~\eqref{lem:strictness1} that both are strict.\qedhere
\end{enumerate}
\end{proof}

\begin{proposition}\label{prop:criterionabelian}
Let $\lambda:\cM_1\to\cM_2$ be a morphism between well-rescalable $\cR^\intt_\cX$-modules. Assume that
\begin{enumeratea}
\item\label{prop:criterionabeliana}
$\cM_1$, $\cM_2$ and $\lambda$ are strict,
\item\label{prop:criterionabelianb}
$\taulambda$ is strictly specializable along $\tau=0$.
\end{enumeratea}
\noindent Then, setting $\cK:=\ker\lambda$, $\cI:=\image\lambda$, $\cC:=\coker\lambda$, and $\ccK:=i_{\hb=1}^*\cK$ etc., we have
\begin{enumerate}
\item\label{prop:criterionabelian2}
$\ccK=\ker i_{\hb=1}^*\lambda$ etc.,
\item\label{prop:criterionabelian1}
$i_{\tau=\hb}^*\taucK=\ker i_{\tau=\hb}^*\taulambda$ etc.,
\item\label{prop:criterionabelian1b}
$\cK,\cI,\cC$ are strict well-rescalable $\cR^\intt_\cX$-modules.
\item\label{prop:criterionabelian3}
for every $\alpha\in[0,1)$, we have $i_{\tau=\hb}^*\tauV_\alpha\taucK=\ker(i_{\tau=\hb}^*\taulambda_\alpha)$ etc.
\end{enumerate}
\end{proposition}

\begin{proof}
For \eqref{prop:criterionabelian2}, the strictness of $\lambda$ implies the equalities $\ccK=\ker i_{\hb=1}^*\lambda$, etc.

For \eqref{prop:criterionabelian1} we reduce the question to $\pi^{\circ*}i_{\hb=1}^*\lambda$, and the assertion follows from the second part of \eqref{prop:criterionabelian2} together with the flatness of $\cO_\cX[\hbm]$ over $\pi^{-1}\cO_X$.

By Assumption~\eqref{prop:criterionabelianb}, we have $\tauV_\alpha\taucK=\ker\taulambda_\alpha$ etc., according to \cite[Lem.\,3.3.10]{Bibi01c}, so $\cK,\cI,\cC$ are well-rescalable. Moreover, since~$\lambda$ is strict, its kernel, image and cokernel are strict, which concludes \eqref{prop:criterionabelian1b}.

By the strict specializability of $\taulambda$, we have exact sequences
\[
0\ra\tauV_\alpha\taucK\ra\tauV_\alpha\taucM_1\ra\tauV_\alpha\taucI\ra0 \quad\text{and}\quad 0\ra\tauV_\alpha\taucI\ra\tauV_\alpha\taucM_2\ra\tauV_\alpha\taucC\ra0.
\]
Since $\tau-\hb$ is injective on each such module (by arguing as in the proof of Lemma \ref{lem:strictness}\eqref{lem:strictness1}), we obtain exact sequences for the corresponding modules $i_{\tau=\hb}^*\tauV_\alpha$, hence~\eqref{prop:criterionabelian3}.
\end{proof}

\begin{proposition}\label{prop:criterionabelianfiltered}
Together with Assumptions \eqref{prop:criterionabeliana} and \eqref{prop:criterionabelianb} in Proposition \ref{prop:criterionabelian} (from which we keep the notation), assume moreover that
\begin{enumeratea}\setcounter{enumi}{2}
\item\label{prop:criterionabelianc}
$\cM_1,\cM_2$ are graded.
\end{enumeratea}
\noindent Then\vspace*{-3pt}
\begin{enumerate}
\item\label{prop:criterionabelian4}
$\cK$, $\cI$, $\cC$ are graded,
\item\label{prop:criterionabelian5}
for each $\alpha\in[0,1)$, we have $R_{F^\irr_{\alpha+\bbullet}}\ccK=\ker\gr(i_{\tau=\hb}^*\taulambda_\alpha)$, etc.,
\item\label{prop:criterionabelian6}
$i_{\hb=1}^*\lambda:(\ccM_1,F^\irr_{\alpha+\bbullet}\ccM_1)\!\to\!(\ccM_2,F^\irr_{\alpha+\bbullet}\ccM_2)$ is strict, and we have $\ker i_{\hb=1}^*\lambda\!=\!(\ccK,F^\irr_{\alpha+\bbullet}\ccK)$,~etc.
\end{enumerate}
\end{proposition}

\begin{proof}
Since $\lambda$ is graded (Lemma \ref{lem:lambdagraded}), we have for each $\alpha\in[0,1)$,
\[
i_{\tau=\hb}^*\tauV_\alpha\taucK=\ker(i_{\tau=\hb}^*\taulambda_\alpha)=\cO_\cX\otimes_{\pi^{-1}\cO_X[\hb]}\pi^{-1}\ker\gr(i_{\tau=\hb}^*\taulambda_\alpha),
\]
and similarly for $\image$ and $\coker$, so that $\cK,\cI,\cC$ are graded, hence \eqref{prop:criterionabelian4}.

By faithful flatness of $\cO_\cX$ over $\pi^{-1}\cO_X[\hb]$, we conclude that $\ker\gr(i_{\tau=\hb}^*\taulambda_\alpha)$, etc., are strict, hence correspond to irregular Hodge filtrations on $\ccK,\ccI,\ccC$, so that \eqref{prop:criterionabelian5} holds. Moreover, $\gr(i_{\tau=\hb}^*\taulambda_\alpha)$ is strict, hence \eqref{prop:criterionabelian6}.
\end{proof}

\subsection{Example 1: the case where \texorpdfstring{$\dim X=0$}{dimXO}}\label{subsec:example1}
In this subsection, we assume that $X$ is reduced to a point. We have $\cR_\cX=\cO_{\CC_\hb}$ and $\cR^\intt_\cX=\cO_{\CC_\hb}\langle\hb^2\partial_\hb\rangle$. We assume that $\cM$ is a coherent $\cO_{\CC_\hb}$-module which is integrable and \emph{strict}, that is, $\cO_{\CC_\hb}$-locally free of finite rank. We denote by $\ccM=\cM/(\hb-1)\cM$ its fibre at $\hb=1$. This is a finite-dimensional vector space.

The \index{Deligne's meromorphic extension}Deligne meromorphic extension of $\cM$ at $\hb=\infty$, together with GAGA, produces a free $\CC[\hb]$-module $M$ such that $\cO_{\CC_\hb}\otimes_{\CC[\hb]}M=\cM$ and the connection induced by $\hb^2\partial_\hb$ has a regular singularity at infinity.\footnote{The notation $G_0$ is used for such an $M$ in \cite{Bibi96bb,Bibi05,H-S06, Bibi08}, \cf also Sections \ref{subsec:Cmhm} and \ref{sec:alternating}.} We have $M\subset M[\hbm]:=\CC[\hb,\hbm]\otimes_{\CC[\hb]}M$. Let $V_\bbullet(M[\hbm])$ be the $V$-filtration of $M[\hbm]$ at $\hb=\infty$. We~assume for simplicity that the indices are real (\ie the eigenvalues of the monodromy on~$\CC_\hb^*$ have absolute value equal to one). Each $V_\beta(M[\hbm])$ is a free $\CC[\hbm]$-module and $\CC[\hb,\hbm]\otimes_{\CC[\hbm]}V_\beta(M[\hbm])=M[\hbm]$, and we have $(-\hb\partial_\hb+\beta)^{\nu_\beta}V_\beta(M[\hbm])\subset V_{<\beta}(M[\hbm])$ for some $\nu_\beta\in\NN$ with $\nu_{\beta+k}=\nu_\beta$ for $k\in\ZZ$.

\begin{proposition}\label{prop:Xpoint}
The strict module $\cM$ is well-rescalable and graded, and the corresponding irregular Hodge filtration $F^\irr_{\alpha+\bbullet}\ccM$ ($\alpha\in\nobreak[0,1)$) is the filtration naturally induced by $V_{\alpha+\bbullet}(M[\hbm])$ on the $\CC$-vector space
\[
\ccM=M/(\hb-1)M=M[\hbm]/(\hb-1)M[\hbm].
\]
\end{proposition}

\begin{remarque}\label{rem:irregHspectrum}
The \index{jumping index}jumping indices of the irregular Hodge filtration, together with the dimension of the corresponding graded spaces, form the \index{spectrum}spectrum of $\cM$ at infinity, as defined in \cite[\S1]{Bibi96bb}.
\end{remarque}

\begin{proof}
We will work in the algebraic framework with respect to $\tau$ and $\hb$. Below we will simply set $V_\beta=V_\beta(M[\hbm])$. By definition, $\tauM=\CC[\tau,\taum]\otimes_\CC M$, where the $\CC[\tau,\taum,\hb]$-structure is given by $\hb(1\otimes m)=\tau\otimes\hb m$, $\tau^k(1\otimes m)=\tau^k\otimes m$, and the differential structure by~\eqref{eq:partiall}. Hence, as a $\CC[\tau,\taum]$-module, we have
\begin{equation}\label{eq:Mtau}
\tauM=\bigoplus_{k\in\ZZ}\tau^k\otimes M,
\end{equation}
and the $\hb$-action is given by $\hb\cdot\sum_k\tau^k\otimes m_k=\sum_k\tau^k\otimes\hb m_{k-1}$. It follows that~$\tauM$ is $\CC[\tau,\taum,\hb]$-free (because $M$ is $\CC[\hb]$-free), and $\tauM/(\tau-1)\tauM=M$. It also follows that
\[
\tauM/(\tau-\nobreak\hb)\tauM=\bigoplus_k\tau^k\otimes(M/(\hb-1)M),
\]
and after identifying the action of $\tau$ with that of $\hb$ on this quotient, we find the isomorphism:
\begin{equation}\label{eq:identifM}
\tauM/(\tau-\hb)\tauM=\CC[\hb,\hbm]\otimes_\CC\ccM.
\end{equation}

We now show that $\tauM$ is strictly specializable and regular along $\tau=0$ by making explicit its $V$-filtration. We set
\[
\tauU_\beta\tauM=\bigoplus_k\tau^k\otimes(V_{\beta+k}\cap M).
\]
Then, clearly, $\tauU_\bbullet\tauM$ is an increasing filtration of $\tauM$ by $\CC[\tau,\hb]$-submodules, and we have $\tauU_{\beta-\ell}\tauM=\tau^\ell \tauU_\beta\tauM$ for each $\beta\in\RR$ and $\ell\in\ZZ$. Since \hbox{$(-\hb\partial_\hb+\beta)^{\nu_\beta} V_\beta\subset V_{<\beta}$}, we have $\hb^{\nu_\beta}(-\hb\partial_\hb+\beta)^{\nu_\beta} (V_\beta\cap M)\subset V_{<\beta+\nu_\beta}\cap M$, from what we conclude that \hbox{$(\tau\partiall_\tau+\beta\hb)^{\nu_\beta}\tauU_\beta\tauM\subset \tauU_{<\beta}\tauM$} since, according to \eqref{eq:partiall}, we have
\[
(\tau\partiall_\tau+\beta\hb)^{\nu_\beta}\bigl[\tau^k\otimes(V_{\beta+k}\cap M)\bigr]=\tau^{k+\nu_\beta}\otimes\hb^{\nu_\beta}\bigl(-\hb\partial_\hb+(\beta+k)\bigr)^{\nu_\beta}(V_{\beta+k}\cap M).
\]

Recall (\cf\cite[\S1.a]{Bibi08}) that, for $\beta\in\RR$, $V_{\beta+k}\cap M$ is a finite dimensional vector space and there exist $k_0,k_1$ such that $V_{\beta+k}\cap M=0$ for $k<k_0$ and, for $k\geq k_1$,
\begin{align*}
V_{\beta+k}\cap M&=(V_{\beta+k}\cap\hb M)+(V_{\beta+k-1}\cap M)=\hb(V_{\beta+k-1}\cap M)+(V_{\beta+k-1}\cap M)\\
&=\sum_{j=0}^{k-k_1}\hb^j(V_{\beta+k_1}\cap M).
\end{align*}
Hence, for $k\geq k_1$,
\[
\tau^k\otimes(V_{\beta+k}\cap M)=\biggl(\sum_{j=0}^{k-k_1}\CC\tau^{k-k_1-j}\hb^j\biggr)\tau^{k_1}\otimes(V_{\beta+k_1}\cap M).
\]
We conclude that $\tauU_\beta\tauM=\bigoplus_{k=k_0}^{k_1-1}\tau^k\otimes(V_{\beta+k}\cap M)+\CC[\tau,\hb](1\otimes V_{\beta+k_1}\cap M)$, hence the $\CC[\tau,\hb]$ finiteness.

It remains to show the strictness of $\tauU_\beta\tauM/\tauU_{<\beta}\tauM$. We have
\[
(\hb-\hb_o)\sum_k\tau^k\otimes m_k=\sum_k\tau^k\otimes(\hb m_{k-1}-\hb_om_k),
\]
and the $\hb$-torsion freeness amounts to showing the injectivity of
\[
\hb:\frac{V_{\beta+k-1}\cap M}{V_{<\beta+k-1}\cap M}\to\frac{V_{\beta+k}\cap M}{V_{<\beta+k}\cap M}.
\]
This is clear, since for $m\in V_{\beta+k-1}$, $\hb m\in V_{<\beta+k}$ is equivalent to $m\in V_{<\beta+k-1}$.

We finally conclude that $\tauM$ is strictly specializable along $\tau=0$ and $\tauV_\beta\tauM=\tauU_\beta\tauM$. We will now compute $\tauV_\alpha\tauM/(\tau-\hb)\tauV_\alpha\tauM$, for $\alpha\in[0,1)$. Although we have given a general proof that $(\tau-\hb)\tauV_\alpha\tauM=(\tau-\hb)\tauM\cap\tauV_\alpha\tauM$ (\cf Lemma \ref{lem:Vrescaling}\eqref{lem:Vrescaling1}), we will show it in the present situation. This equality amounts to
\[
V_{\alpha+k}\cap(1-\hb)M=(1-\hb)(V_{\alpha+k-1}\cap M).
\]
Let $m\in M$ such that $(1-\hb)m\in V_{\alpha+k}$. Let $\gamma$ be such that $m\in V_\gamma$. If $\gamma\geq\alpha+k$, we have $\hb m\in V_\gamma$, hence $m\in V_{\gamma-1}$ and by decreasing induction we conclude that $m\in V_{\alpha+k-1}$, as wanted.

Using the identification as in \eqref{eq:identifM}, we find
\begin{equation}\label{eq:gradingpoint}
\frac{\tauV_\alpha\tauM}{(\tau-\hb)\tauV_\alpha\tauM}=\bigoplus_k\hb^k\,\frac{V_{\alpha+k}\cap M}{(1-\hb)(V_{\alpha+k-1}\cap M)}=\bigoplus_k\hb^k\,\frac{V_{\alpha+k}\cap M}{V_{\alpha+k}\cap(1-\hb)M},
\end{equation}
so $F^\irr_{\alpha+k}\ccM$ is the filtration induced by $V_{\alpha+k}$ on $\ccM$. The grading property is obviously satisfied.
\end{proof}

\skpt
\begin{definition}\label{def:dualtauM}
\begin{enumerate}
\item
The dual $(\tauM)^\vee$ is $\Hom_{\CC[\tau,\taum,\hb]}(\tauM,\CC[\tau,\taum,\hb])$ equipped with the dual action of $\partiall_\tau$ and $\hb^2\partial_\hb$.
\item
The tensor product $\tauM_1\otimes_{\CC[\tau,\taum,\hb]}\tauM_2$ is endowed with the usual tensor product action of $\partiall_\tau$ and $\hb^2\partial_\hb$.
\end{enumerate}
\end{definition}

The following is straightforward.

\begin{lemme}\label{lem:dualtauM}
We have natural isomorphisms
$$
(\tauM)^\vee\simeq{}^\tau\!(M^\vee),\qquad
\tauM_1\otimes_{\CC[\tau,\taum,\hb]}\tauM_2\simeq{}^\tau\!(M_1\otimes_{\CC[\hb]}M_2).\eqno\qed
$$
\end{lemme}

\subsection{Example 2: the case of a filtered \texorpdfstring{$\cD_X$}{DX}-module}\label{subsec:example2}
Let $\ccN$ be a coherent $\cD_X$-module equipped with a good filtration $F_\bbullet\ccN$. We set $N=R_F\ccN$ as a graded $R_F\cD_X$-module and we will work in the algebraic framework with respect to $\tau$ and $\hb$.

\skpt
\begin{proposition}\label{prop:filtered}
\begin{enumerate}
\item\label{prop:filtered1}
The module $N$ is a strict and graded well-rescalable $R_F\cD_X$-module, and the corresponding irregular Hodge filtration $F^\irr_{\bbullet}\ccN$ is the filtration $F_\bbullet\ccN$ (in particular, the only interesting $\alpha$ is $\alpha=0$).

\item\label{prop:filtered2}
Let $\lambda:R_F\ccN_1\to R_F\ccN_2$ be a graded morphism. Then $\taulambda$ is strictly specializable along $\tau=0$ if and only if the filtered morphism $i_{\hb=1}^*\lambda:(\ccN_1,F_\bbullet\ccN_1)\to(\ccN_2,F_\bbullet\ccN_2)$ is strict.
\end{enumerate}
\end{proposition}

\skpt
\begin{proof}
\eqref{prop:filtered1} We set $\tauN=\cO_X[\tau,\taum,\hb]\otimes_{\cO_X[\hb]}N$ with the $R_F\cD_X[\tau,\taum]\langle\partiall_\tau\rangle$-structure \hbox{described} by \eqref{eq:partiall}. We have, as an $\cO_X[\tau,\taum,\hb]$-module,
\begin{equation}\label{eq:Ntau}
\begin{split}
\tauN=\bigoplus_{k\in\ZZ}\bigoplus_j(\tau^k\otimes\hb^jF_j\ccN)&=\bigoplus_{k,j}(\tau^{k-j}\hb^j\otimes F_j\ccN)\\
&=\CC[\tau,\taum]\otimes_\CC\Bigl(\bigoplus_j(\hb^j\otimes F_j\ccN)\Bigr),
\end{split}
\end{equation}
and the $\partiall_{x_i},\partiall_\tau$ action is given by
\begin{align*}
\partiall_{x_i}(\hb^j\otimes F_j\ccN)&=(\hb^{j+1}\otimes\partial_{x_i}F_j\ccN)\subset(\hb^{j+1}\otimes F_{j+1}\ccN)\\
\partiall_\tau\hb^j(1\otimes F_j\ccN)&=0.
\end{align*}
In other words,
\begin{equation}\label{eq:taufiltN}
\tauN=\CC[\tau,\taum,\hb]\hbboxtimes R_FN,
\end{equation}
where $\CC[\tau,\taum,\hb]$ has its standard $R_F\cD_{\CC_\tau}$-module structure. Note that $\tauN[\hbm]=\CC[\tau,\taum,\hb,\hbm]\boxtimes\ccN$. We also deduce from \eqref{eq:taufiltN} that $\tauN/(\tau-\nobreak\hb)\tauN=\CC[\hb,\hbm]\otimes_\CC\ccN$.

We easily conclude that $\tauV_k \tauN=\tau^{-k}\CC[\tau,\hb]\hbboxtimes R_F\ccN$ ($k\in\ZZ$). The regularity is then clear. We also have $\gr_{\alpha+k}^{\tauV}\tauN=0$ for $\alpha\in(0,1)$ and any $k\in\ZZ$, and $\gr_k^{\tauV}\tauN\simeq\nobreak R_F\ccN$ for each $k\in\ZZ$, hence the strict specializability. Therefore, only $\alpha=0$ is to be considered for the irregular Hodge filtration. Hence, all properties of Definition \ref{def:Rrescdata} are satisfied. For the grading property, we have
\[
\tauV_0 \tauN/(\tau-\hb)\tauV_0 \tauN=R_F\ccN\subset\ccN[\hb,\hbm],
\]
and $\tauV_0 \tauN/(\tau-\hb)\tauV_0 \tauN$ is nothing but $\gr\bigl(\tauV_0 \tauN/(\tau-\hb)\tauV_0 \tauN\bigr)$ with its grading forgotten.

\eqref{prop:filtered2}
We note that, due to the identifications above, $\gr_0^{\tauV}\taulambda$ is nothing but \hbox{$\lambda:R_F\ccN_1\to R_F\ccN_2$}, showing that $\taulambda$ is strictly specializable along $\tau=0$ if and only if $\lambda$ is strict, that is, $i_{\hb=1}^*\lambda$ is strictly filtered. Similarly, $i_{\tau=\hb}^*\taulambda$ is identified with $\lambda$ viewed as acting on $\cO_X[\hb,\hbm]\otimes_{\cO_X[\hb]}R_F\ccN=\pi^{\circ*}\ccN$, that is, $\pi^{\circ*}i_{\hb=1}^*\lambda$.
\end{proof}

\subsection{Well-rescaled \texorpdfstring{$\cR$}{R}-triples}\label{subsec:wellrescalingT}
Let $\icT=(\cM',\cM'',\iC)$ be an object of $\RdiTriples(X)$. According to the constructions in Section \ref{sec:rescaling}, one can rescale it as an object $(\taucM',\taucM'',\thetaiC)$, where $\thetaiC$ is a sesquilinear pairing
\[
\thetaiC:\thetacM'_{|\thetacX^\circ}\otimes_\CC\iota^+\ov{\thetacM''}_{|\thetacX^\circ}\to\Db_{\thetacX^\circ/\CC^*_\hb}.
\]
Let $j:\thetaX\hto\tauX$ denotes the open inclusion. For $\cM=\cM',\cM''$, we have $\taucM\subset j_*\thetacM$, hence $j_*\thetaiC$ induces a sesquilinear pairing
\[
j_*\thetaiC:\taucM'_{|\thetacX^\circ}\otimes_\CC\iota^+\ov{\taucM''}_{|\thetacX^\circ}\to j_*\Db_{\thetacX^\circ/\CC^*_\hb}.
\]

\begin{definition}
We say that $\icT=(\cM',\cM'',\iC)$ is \emph{well-rescalable} if
\begin{enumerate}
\item
$\cM',\cM''$ are well-rescalable in the sense of Definition \ref{def:Rrescdata},
\item
$j_*\thetaiC:\taucM'_{|\thetacX^\circ}\otimes_\CC\iota^+\ov{\taucM''}_{|\thetacX^\circ}\to j_*\Db_{\thetacX^\circ/\CC^*_\hb}$ takes values in the subsheaf of moderate distributions along $\tauX_0$ and defines a moderate sesquilinear pairing
\[
\tauiC:\taucM'_{|\taucX^\circ}\otimes_\CC\iota^+\ov{\taucM''}_{|\taucX^\circ}\to\Db_{\taucX^\circ/\CC^*_\hb}(*\tauX_0)
\]
\end{enumerate}

We say that $\icT=(\cM',\cM'',\iC)$ is graded well-rescalable if moreover $\taucM',\taucM''$ are graded.
\end{definition}

In this way we have defined two full subcategories \index{$RTriplesXrescgr$@$\RgscTriples(X)$}$\RgscTriples(X)$ and \index{$RTriplesXresc$@$\RscTriples(X)$}$\RscTriples(X)$ of $\RdiTriples(X)$. It is easy to check that $\RscTriples(X)$ is stable by direct summand in $\RdiTriples(X)$ and, according to Lemma \ref{lem:lambdagraded}\eqref{lem:lambdagraded3},
\begin{equation}\label{eq:dirsumWRd}
\text{$\RgscTriples(X)$ is stable by direct summand in $\RscTriples(X)$.}
\end{equation}

Let us denote by $\RdiTriples(\tauX,(*\tauX_0))$ the category consisting of localized integrable $\iota$-triples on $(\tauX,(*\tauX_0))$. The subcategories factorize as follows:
\begin{equation}\label{eq:functorresc}
\begin{array}{c}
\xymatrix@R=1mm@C=1.5mm{%
\RgscTriples(X)\ar@{|->}[dd]\\ \\
\RscTriples(X)\ar@{|->}[r]& \RdiTriples(\tauX,(*\tauX_0))\ar@{|->}[rrr]^-{\tau=1}&&&\RdiTriples(X)\\
\taucT^\resc\ar@{|->}[r]&\taucT\!=\!(\taucM',\taucM'',\tauiC)\ar@{|->}[rrr]&&&\icT\!=\!(\cM',\cM'',\iC).
}
\end{array}
\end{equation}

\skpt
\begin{lemme}[Adjunction and Tate twist by an integer]\label{lem:strongrescalingtateadj}
The categories $\RgscTriples(X)$ and $\RscTriples(X)$ are stable by the adjunction functor and the Tate twist functor by an integer in a way compatible with those existing on $\RdiTriples(\tauX,(*\tauX_0))$.

Moreover, for $k\in\ZZ$, if $F^\irr_{\alpha+\bbullet}\ccM$ is the irregular Hodge filtration associated with $\icT^\resc$, then $F^\irr_{\alpha-k+\bbullet}\ccM$ is that attached to $\icT^\resc(k)$.
\end{lemme}

\begin{proof}
The case of adjunction is obvious, as well as that of the Tate twist by an integer, according to Lemma \ref{lem:rescalingtateadj}\eqref{lem:rescalingtateadj2}, since $\taucT(k)=(\hb^{-k}\taucM',\hb^k\taucM'',\tauiC)$.

The Rees module for the irregular Hodge filtration attached to $\icT^\resc(k)$ is then equal to $\hb^kR_{F^\irr_{\alpha+\bbullet}}\ccM$, hence the last statement.
\end{proof}

\begin{remarque}[Changing the integrable structure]\label{rem:changingintstruct}
Let us change the integrable structure of $\cT$ as in Examples \ref{exem:equivintstruct} and \ref{exem:equivintstructT} and let us use the variant \eqref{eq:equivintstructT}. Let us assume furthermore that $b$ is real and that $c=-\ov a$. This corresponds to twisting~$\cT$ by a pure polarized integrable twistor structure of weight zero. Let us denote by $\wt\cM',\wt\cM''$ the $\cR_\cX$-modules $\cM',\cM''$ endowed with this new action of $\hb^2\partial_\hb$, that is, \eg $\hb^2\partial_\hb\cdot m'=(\hb^2\partial_\hb+a+b\hb)m'$, etc. The action of $\tau\partiall_\tau$ on $\tauwtcM'$ is by $\tau\partiall_\tau-a\tau-b\hb$ and that of $\hb^2\partiall_\hb$ is by $\hb^2\partiall_\hb+a\tau+b\hb$. If~$\taucM'$ is strictly $\RR$-specializable and regular along $\tau=0$, so is $\tauwtcM'$, whose $\tauV$-filtration is the $\tauV$-filtration of~$\taucM'$ shifted by~$b$. A similar statement holds for $\tauwtcM''$. Lastly, $\thetaiC$ is changed to $\thetaiwtC:=|\hb/\tau|^{2b}\rme^{(\ov{a\tau/\hb}-a \tau/\hb)}\cdot\thetaiC$, which still has moderate growth along $\tauX_0$ if $\thetaiC$ has so. As a consequence, twisting $\cT$ by a pure polarized integrable twistor structure of rank one and weight zero induces an endofunctor of the category $\RgscTriples(X)$.
\end{remarque}

\section{Irregular mixed Hodge modules}\label{sec:MTrM}

\subsection{General construction}

We will now use the theory of mixed twistor $\cD$-modules, as developed in \cite{Mochizuki11}. However, we will use its ``$\iota$-version'' considered in Chapter \ref{part:1}. Let $\iMTM^\intt(\tauX)$ be the category of integrable mixed $\iota$-twistor $\cD$-modules on $\tauX$. It is a full subcategory of the category $\WRdiTriples(\tauX)$ of integrable $W$-filtered $\cR$-triples on~$\tauX$. We~will only consider the full subcategory $\iMTM^\intt(\tauX,[*\tauX_0])$ of \emph{localized} integrable mixed twistor modules $(\taucT,W_\bbullet)$, such that $\taucT=\taucT[*\tauX_0]$ (\cf Remark \ref{rem:localization}). Recall that we have a ``stupid'' localization functor \eqref{eq:stupidloc}
\begin{equation}\label{eq:stupidloctau}
\iMTM^\intt(\tauX,[*\tauX_0])\mto\WRdiTriples(\tauX,(*\tauX_0))
\end{equation}
which sends $W_\ell\,\taucT$ to $W_\ell[\taucT(*\tauX_0)]:=[W_\ell\,\taucT](*\tauX_0)$. By \cite[Lem.\,11.2.4]{Mochizuki11}, it~is fully faithful. We denote by $\iMTM^\intt(\tauX,(*\tauX_0))$ its essential image. Therefore, the ``stupid'' localization functor $\iMTM^\intt(\tauX,[*\tauX_0])\mto\iMTM^\intt(\tauX,(*\tauX_0))$ is an equivalence (\cf Remark \ref{rem:localization}). An object $(\taucT,W_\bbullet)$ in $\iMTM^\intt(\tauX,(*\tauX_0))$ comes thus from a unique (up to isomorphism) object in $\iMTM^\intt(\tauX,[*\tauX_0])$. We will not make the distinction between both.

\begin{remarque}[Condition near $\tau=0$]\label{rem:tau=0}
The mixed twistor condition that we want to impose to an object of $\WRdiTriples(\tauX,(*\tauX_0))$ is important only in some neighbourhood of $|\tau|\leq1$ (for example, the strict specializability along $\tauX_0$ or the restriction to $\tau=1$), and the behaviour for $|\tau|$ large is not relevant. In the following, $\iMTM^\intt(\tauX,(*\tauX_0))$ should be taken with this wider interpretation. This corresponds to the notion of a ``Sabbah orbit'' in \cite{H-S06}, when $X$ is reduced to a point and when we only consider polarized objects which are pure of weight $0$. On the other hand, the notion of nilpotent orbit in \loccit, which controls the behaviour when $|\tau|\to\infty$, is not relevant here.
\end{remarque}

\begin{definition}[Mixed twistor-rescaled $\cD$-modules]\label{def:iMTMresc}
The category \index{$IMTMresc$@$\iMTM^\resc(X)$}$\iMTM^\resc(X)$ of \index{mixed twistor-rescaled $\cD$-module}mixed twistor-rescaled $\cD$-modules is the full subcategory of $\WRscTriples(X)$ whose objects are filtered objects $(\icT,W_\bbullet)$ of the category $\RscTriples(X)$ such that the rescaled object $(\taucT,W_\bbullet)$ is an object of $\iMTM^\intt(\tauX,(*\tauX_0))$, which is a full subcategory of $\WRdiTriples(\tauX,(*\tauX_0))$.
\end{definition}

\begin{remarque}
Since $(\icT,W_\bbullet)$ is recovered from $(\taucT,W_\bbullet)$ by restricting to $\tau=1$, we conclude that $(\icT,W_\bbullet)$ is an object of $\iMTM^\intt(X)$
\end{remarque}

\skpt
\begin{definition}[Irregular mixed Hodge modules]\label{def:IrrMHMX}
\begin{enumerate}
\item\label{def:IrrMHMX1}
The category of \index{mixed Hodge module!irregular --}irregular mixed Hodge modules \index{$IrrMHM$@$\IrrMHM(X)$}$\IrrMHM(X)$ is defined similarly as a full subcategory of the category $\WRgscTriples(X)$, and consists of objects in $\iMTM^\resc(X)$ which are \emph{graded}.
\item
If $\kk$ is a subfield of $\RR$, the objects in the category \index{$IrrMHMk$@$\IrrMHM(X,\kk)$}$\IrrMHM(X,\kk)$ are pairs $((\icT,W_\bbullet),(\cP_\kk,W_\bbullet))$ in $\iMTM^\intt_\good(X,\kk)$ (\cf Section~\ref{sec:realratMTM}) such that $(\icT,W_\bbullet)$ is an object in $\IrrMHM(X)$.
\end{enumerate}
\end{definition}

\skpt
\begin{proposition}\label{prop:MTMresc}
\begin{enumerate}

\item\label{prop:MTMresc2}
The category $\iMTM^\resc(X)$ is an abelian full subcategory of $\iMTM^\intt(X)$, which is stable by direct summand in $\iMTM^\intt(X)$.

\item\label{prop:MTMresc3}
The category $\IrrMHM(X)$ is an abelian full subcategory of $\iMTM^\resc(X)$, which is stable by direct summand in $\iMTM^\resc(X)$, hence in $\iMTM^\intt(X)$. Each morphism~$\lambda$ induces a strictly bifiltered morphism
\[
i_{\hb=1}^*\lambda:(\ccM_1,F^\irr_\bbullet\ccM_1,W_\bbullet\ccM_1)\to(\ccM_2,F^\irr_\bbullet\ccM_2,W_\bbullet\ccM_2).
\]
\item\label{prop:MTMresc4}
The category $\IrrMHM(X,\kk)$ is an abelian full subcategory of $\iMTM^\intt_\good(X,\kk)$ stable by direct summand.
\end{enumerate}
\end{proposition}

\begin{proof}
Abelianity follows from Propositions \ref{prop:criterionabelian} and \ref{prop:criterionabelianfiltered}, together with the property that \ref{prop:criterionabelian}\eqref{prop:criterionabelianb} is fulfilled in $\iMTM^\intt(\tauX,(*\tauX_0))$ and that every morphism between graded well-rescalable objects is graded.

The only point remaining to be proved is the bistrictness of $i_{\hb=1}^*\lambda$. We~note that, since $\lambda$ is a morphism in $\iMTM^\intt(X)$, the induced morphism $i_{\hb=1}^*\lambda$ is strictly filtered with respect to $W_\bbullet$. On the other hand, Proposition \ref{prop:criterionabelianfiltered} implies that it is strictly filtered with respect to $F^\irr_\bbullet$. For each $\ell$, $F^\irr_\bbullet(\gr_\ell^W\!\!\ccM)$ is the filtration naturally induced by $F^\irr_\bbullet\ccM$, as follows from Proposition \ref{prop:criterionabelianfiltered}\eqref{prop:criterionabelian6}, hence the desired assertion.
\end{proof}

\subsection{Examples of objects of \texorpdfstring{$\IrrMHM(\mathrm{pt})$}{iMTMgresc}}\label{subsec:imts}

\subsubsection*{Pure polarized integrable twistor structures of weight zero and rank one}
Let us consider such an object (also called a polarizable non-commutative Hodge structure of weight zero) as given by the trivial bundle $\cO_{\PP^1}$ equipped with the connection $\nabla=\rd+(a/\hb+b-\ov a\hb)\rd\hb/\hb$ with $b$ real. We can also present it as in Example \ref{exem:equivintstruct} with $\cM'=\cM''=\cO_{\CC_\hb}$ with $c=-\ov a$, and $C_{\bS}=\id$, and work with the isomorphic object \eqref{eq:equivintstruct}. We can interpret this object as the result of the change of the integrable structure of $((\cO_{\CC_\hb},\hb^2\partial_\hb),(\cO_{\CC_\hb},\hb^2\partial_\hb),C_{\bS}=\id)$. In such a way, $\cM'=\cO_{\CC_\hb}\otimes_{\CC[\hb]}M'$, where $M'=\CC[\hb]$ is defined as in Example \ref{subsec:example1} from Deligne's meromorphic extension at $\hb=\infty$, and similarly for $\cM''$, so that $\cM',\cM''$ are well-rescalable and graded, according to Proposition \ref{prop:Xpoint}. As in Remark \ref{rem:changingintstruct}, we find $\tauiwtC=|\hb/\tau|^{2b}\rme^{(\ov{a\tau/\hb}-a \tau/\hb)}$. Then the restriction of this object to $\tau=\tau_o\neq0$ remains a pure polarized integrable twistor structure of weight zero, since this operation reduces to changing $a$ into $a\tau_o$ (recall that multiplication by the positive constant $|\tau_o|^{-2b}$ leads to an isomorphic object, \cf Footnote \ref{foot:positiveconstant} on Page \pageref{foot:positiveconstant}). We conclude:

\begin{proposition}
Any polarized non-commutative Hodge structure of rank one induces an object of $\IrrMHM(\mathrm{pt})$.\qed
\end{proposition}

\begin{corollaire}\label{cor:grescequivint}
Let $\icT$ be an object of $\IrrMHM(X)$. Then any object of $\iMTM^\intt(X)$ which has an equivalent integrable structure, and so defines the same object in $\MTM(X)$ (\cf Remark \ref{rem:RdTriples}), gives also rise to an object in $\IrrMHM(X)$.
\end{corollaire}

\begin{proof}
Such an object is obtained from $\icT$ by tensoring with a rank-one integrable polarizable twistor structure $\icT_1$ which is pure of weight zero, according to Remark \ref{rem:changingintstruct}. As noticed in this remark, the graded-rescaling property is preserved. We then have ${}^\tau\!(\icT\otimes\icT_1)\simeq\taucT\otimes\taucT_1$, which is an object of $\iMTM^\intt(\tauX,(*\tauX_0))$.
\end{proof}

\begin{remarque}
By changing the integrable structure in such a way, one can shift the irregular Hodge filtration of $\ccM$ by an arbitrary real number.
\end{remarque}

\subsubsection*{Pure polarized integrable twistor structures of weight zero of any rank}
Such an object $\cT=(\cM',\cM'',C_{\bS})$ is described as follows (\cf\cite[\S7.2.b]{Bibi01c}). There is a complex vector space $H$ equipped with a positive definite Hermitian pairing $h:H\otimes_\CC\ov H\to\CC$, such that $\cM'=\cM''=H\otimes_\CC\cO_{\CC_\hb}=:\cM$ and the pairing
\[
C_{\bS}:(H\otimes\cO_{\bS})\otimes(\ov H\otimes\cO_{\bS})\to\cO_{\bS}
\]
restricted to $(H\otimes1)\otimes(\ov H\otimes1)$ is equal to $h$. The connection is given by the formula (for $\cM$)
\begin{equation}\label{eq:UQU}
(\hb^2\partial_\hb)_{|H\otimes1}=U\otimes1-Q\otimes\hb-U^*\otimes\hb^2,
\end{equation}
where $U$ and $Q$ are endomorphisms of $H$, $Q$ is self-adjoint with respect to $h$ and $U^*$ is the $h$-adjoint of $U$. On the other hand, the formula for $\iC$ and that for the Deligne meromorphic extension of $\cM$ is not explicit, in contrast with the case of rank one.

\begin{proposition}\label{prop:sorbits}
Let $\icT$ be a pure polarized integrable twistor structure of weight zero. Assume moreover that $\thetacT$ is a variation of pure polarized integrable twistor structure of weight zero on $\CC_\theta^*$. Then $\icT$ is an object of $\IrrMHM(\mathrm{pt})$.
\end{proposition}

\begin{remarque}
In \cite{H-S06}, $\cT$ is called a ``Sabbah orbit'' when the condition in the proposition is fulfilled for $|\theta|\gg0$.
\end{remarque}

\begin{proof}
The graded well-rescaling property of $\cM:=\cM',\cM''$ follows from Proposition \ref{prop:Xpoint}. We need to check that $\thetaiC$ has moderate growth when restricted to $\taucM',\taucM''$ and that the resulting object $\taucT$ is in $\iMTM^\intt(\tauX,(*\tauX_0))$.

For this point, we notice as in \cite[Lem.\,4.6]{H-S06} that the harmonic bundle corresponding to the variation $\thetacT$ is tame at $\tau=0$. It follows from \cite{Simpson90,Mochizuki07} that it extends as polarizable integrable twistor $\cD$-module $\wt\taucT=(\wt{\taucM'},\wt{\taucM''},\wt\tauiC)$ on $\CC_\tau$, which is pure of weight zero. The localized object $\wt\taucT[*\{\tau=\nobreak0\}]$ is thus an object of $\iMTM^\intt(\CC_\tau)$. It is now enough to identify its $\cR$-components with $\taucM',\taucM''$.

We have already remarked in the proof of Proposition \ref{prop:Xpoint} that $\tauM$ is $\CC[\hb,\tau,\taum]$-free. On the other hand, $\wt\taucM$ is $\cO_{\CC_\tau\times\CC_\hb}[1/\tau]$-free, as follows from \cite[Lem.\,15.6]{Mochizuki07}, after the definition 15.1.1 in \loccit, which identifies $^{\scriptscriptstyle\square}\cE$ with our $\wt\taucM$. See also \cite[Cor.\,4.15]{Mochizuki02} and \cite[Cor.\,5.3.1(1)]{Bibi01c}. So, both $\taucM$ and $\wt\taucM$ are $\cO_{\CC_\tau\times\CC_\hb}[1/\tau]$-free.

By construction, $\taucM$ and $\wt\taucM$ coincide on the open set $\CC^*_\tau\times\CC_\hb$. On the other hand, they also coincide on $\{0\}\times\CC^*_\hb$: indeed, on $\CC_\tau\times\CC^*_\hb$, they are both identified with the Deligne meromorphic extension along the divisor $\tau=0$ of the flat bundle with connection $\taucM_{|\CC^*_\tau\times\CC^*_\hb}=\wt\taucM_{|\CC^*_\tau\times\CC^*_\hb}$. Let us set $\CC^{2*}_{\tau,\hb}:=\CC_\tau\times\CC_\hb\moins\{(0,0)\}$. By Hartogs, the induced isomorphism $\taucM_{|\CC^{2*}_{\tau,\hb}}\isom\wt\taucM_{|\CC^{2*}_{\tau,\hb}}$ extends as an isomorphism on~$\CC^2_{\tau,\hb}$, according to the freeness of~$\taucM$ and $\wt\taucM$.
\end{proof}

\begin{caveat}
Let $U,Q$ be as in \eqref{eq:UQU}. Since $Q$ is self-adjoint, it is semi-simple with real eigenvalues, and defines a direct sum decomposition of $H$ indexed by real numbers (the eigenvalues of $Q$). We can regard this decomposition as a complex Hodge structure. However, the irregular Hodge filtration is not, in general, the filtration associated with this decomposition. This occurs if $U=0$ but, for $U\neq0$, there may be a difference between the jumping indices of the irregular Hodge filtration (called the \emph{spectrum at infinity} in \cite[\S7.1]{Bibi11}) and the jumping indices of the $Q$-decomposition (called the \emph{new supersymmetric index} in \cite[\S7.3]{Bibi11}).

In higher dimensions, if $X$ is a quasi-projective variety, a variation of integrable polarizable pure twistor structure of weight $0$ on $X$ which is tame at infinity defines, via $Q$ and by the rigidity theorem \cite[Th.\,6.2]{H-S08}, a variation of polarized complex Hodge structure, but the irregular Hodge filtration, if defined, does not necessarily fit with the Hodge decomposition. This occurs however if the constant endomorphism~$U$ of the corresponding Higgs bundle is zero.
\end{caveat}

\subsubsection*{Fourier-Laplace transforms of mixed Hodge modules}
We now assume that \hbox{$X=\PP^1$}. Let $\cT$ be an object of $\MTM^\intt(\PP^1,[*\infty])$ coming from a mixed Hodge module. Its \index{mixed Hodge module!Fourier-Laplace transform of a --}Fourier-Laplace transform is an object of $\MTM^\intt(\Afu_\tau)$, according to \cite{Mochizuki11} (\cf Section \ref{subsec:exptwist} for the exponential twist). Moreover, it is tame at $\tau=0$, as follows from \cite[Prop.\,4.1]{Bibi05b}. Comparing the formulas in Section \ref{subsec:example1} and \cite[Lem.\,2.1]{Bibi05}, we identify this Fourier-Laplace transform with the rescaling of its restriction at $\tau=1$. Note that this restriction is nothing but the pushforward by the constant map $\PP^1\to\mathrm{pt}$ of the twisted object $\cT_*^{t/\hb}\otimes\cT$, where $t$ is the affine coordinate on $\Afu=\PP^1\moins\{\infty\}$. We therefore obtain the following result.

\begin{proposition}\label{prop:imts}
With the above notation, the fibre at $\tau=1$ of the Fourier-Laplace transform of a mixed Hodge module on $\PP^1$ is an object of $\IrrMHM(\mathrm{pt})$.\qed
\end{proposition}

\Subsection{Pushforward by a projective morphism}\label{subsec:pushforwardresc}

\begin{theoreme}\label{th:pushforwardresc}
Let $\map:X\to Y$ be a projective morphism and $\taumap:\tauX\to\tauY$ the induced map.
\begin{enumerate}
\item\label{th:pushforwardresc1}
For each $(\icT,W_\bbullet)\in\iMTM^\resc(X)$ and each $k\in\ZZ$, the pushforward $\map^k_\dag(\icT,W_\bbullet)$ is an object of $\iMTM^\resc(Y)$ and for each $\alpha\in[0,1)$ we have, for $\cM=\cM',\cM''$,
\begin{starequation}\label{eq:pushforwardresc1}
i_{\tau=\hb}^*\tauV_\alpha(\taumap^k_\dag(\taucM))\simeq \map^k_\dag\bigl(i_{\tau=\hb}^*\tauV_\alpha\taucM\bigr).
\end{starequation}%
\item\label{th:pushforwardresc2}
Moreover, if $(\icT,W_\bbullet)\in\IrrMHM(X)$, then $\map^k_\dag(\icT,W_\bbullet)\in\IrrMHM(Y)$ and, for each $\alpha\in[0,1)$, the complex $\map_\dag R_{F^\irr_{\alpha+\bbullet}}\ccM$ is strict and
\begin{starstarequation}\label{eq:pushforwardresc2}
R_{F^\irr_{\alpha+\bbullet}}\map^k_\dag(\ccM)\simeq \map^k_\dag(R_{F^\irr_{\alpha+\bbullet}}\ccM).
\end{starstarequation}%
\end{enumerate}
\end{theoreme}

\begin{lemme}
For an object $\icT=(\cM',\cM'',\iC)$ of $\RscTriples(X)$, if $\map_\dag^k\psi_\tau\cM$ is strict for all $k$ ($\cM=\cM',\cM''$), then $\map_\dag^k\icT$ belongs to $\RscTriples(Y)$.
\end{lemme}

\begin{proof}
For the $\cR^\intt$-module part, the condition implies that strict $\RR$-specializability is preserved by proper pushforward, according to \cite[Th.\,3.3.15]{Bibi01c}, and the regularity is also preserved, according to \cite[\S3.1.d]{Bibi01c}. On the other hand, if $\thetaiC$ is moderate along $\tau=0$, then the formula of \cite[\S1.6.d]{Bibi01c} for $\map_\dag^k\thetaiC$ implies that the latter is also moderate along $\tau=0$.
\end{proof}

\pagebreak[2]
\skpt
\begin{proof}[Proof of Theorem \ref{th:pushforwardresc}]
\begin{enumerate}
\item[\eqref{th:pushforwardresc1}]
Let us recall that for objects of $\iMTM^\intt(\tauX,(*X_0))$ and $\map$ projective, $\taumap^k_\dag$ is strict with respect to the $\tauV$-filtration, and $\taumap^k_\dag$ is a functor $\iMTM^\intt(\tauX,(*X_0))\mto\iMTM^\intt(\tauY,(*Y_0))$. Therefore, the above lemma together with Proposition \ref{prop:rescalingimdir} imply that the functor $\map_\dag^k:\iMTM^\intt(X)\mto\iMTM^\intt(Y)$ sends $\iMTM^\resc(X)$ to $\iMTM^\resc(Y)$. The second part of the assertion is obtained as in the proof of \cite[Th.\,4.1]{S-Y14}.

\item[\eqref{th:pushforwardresc2}]
Arguing as in the proof of Proposition \ref{prop:basechangeO} we have
\[
\cO_\cY\otimes_{\cO_Y[\hb]}\map^k_\dag(R_{F^\irr_{\alpha+\bbullet}}\ccM)\simeq \map^k_\dag(\cO_\cX\otimes_{\cO_X[\hb]}R_{F^\irr_{\alpha+\bbullet}}\ccM),
\]
and the grading property of $\icT$ gives, according to \eqref{th:pushforwardresc1},
\begin{equation}\label{eq:fdagRF}
\cO_\cY\otimes_{\cO_Y[\hb]}\map^k_\dag(R_{F^\irr_{\alpha+\bbullet}}\ccM)\simeq \map^k_\dag\bigl(i_{\tau=\hb}^*\tauV_\alpha\taucM\bigr)\simeq i_{\tau=\hb}^*\tauV_\alpha(\taumap^k_\dag(\taucM)).
\end{equation}
As a consequence, by Lemma \ref{lem:Vrescaling}\eqref{lem:Vrescaling2} applied to $\taumap^k_\dag(\taucM)$, $i_{\tau=\hb}^*\tauV_\alpha(\taumap^k_\dag(\taucM))$ is strict, hence so is $\cO_\cY\otimes_{\cO_Y[\hb]}\map^k_\dag(R_{F^\irr_{\alpha+\bbullet}}\ccM)$, and therefore also $\map^k_\dag(R_{F^\irr_{\alpha+\bbullet}}\ccM)$ by faithful flatness of $\cO_\cY$ over $\cO_Y[\hb]$. In other words, the spectral sequence of the filtered complex $\taumap_\dag(F^\irr_{\alpha+\bbullet}\ccM)$ degenerates at $E_1$, a property which amounts to \eqref{eq:pushforwardresc2}. Now, \eqref{eq:fdagRF} reads
\[
\cO_\cY\otimes_{\cO_Y[\hb]}R_{F^\irr_{\alpha+\bbullet}}\map^k_\dag(\ccM)\simeq i_{\tau=\hb}^*\tauV_\alpha(\taumap^k_\dag(\taucM)),
\]
hence the gradedness of $\map^k_\dag(\cM)$.
\qedhere
\end{enumerate}
\end{proof}

\subsection{The Lefschetz morphism and the Hard Lefschetz theorem}
We keep the setting and notation of Section \ref{subsec:pushforwardresc}. Let $c\in H^2(X,\CC)$ be a real $(1,1)$-class. It defines a morphism $\cL_c:\map^k_\dag(\icT)\to \map^{k+2}_\dag(\icT)(1)$ in $\iMTM^\intt(Y)$ (\cf\cite[\S1.6.e]{Bibi01c}).

\skpt
\begin{corollaire}\label{cor:lefschetz}
\begin{enumerate}
\item\label{cor:lefschetz1}
If $\icT$ is an object of $\IrrMHM(X)$, then $\cL_c$ induces a bi-strict morphism for each $\alpha\in[0,1)$ and each $k\in\ZZ$:
\[
(\map^k_\dag(\ccM),F_{\alpha+\bbullet}^\irr \map^k_\dag(\ccM),W_\bbullet \map^k_\dag(\ccM))\to(\map^{k+2}_\dag(\ccM),F_{\alpha-1+\bbullet}^\irr \map^{k+2}_\dag(\ccM),W_{\bbullet+2} \map^{k+2}_\dag(\ccM)).
\]

\item\label{cor:lefschetz2}
If moreover $\icT$ is pure and $c$ is the first Chern class of an ample line bundle then, for each $k\geq0$, $\cL_c^k$ induces a strict isomorphism
\[
(\map^{-k}_\dag(\ccM),F_{\alpha+\bbullet}^\irr \map^{-k}_\dag(\ccM))\isom(\map^k_\dag(\ccM),F_{\alpha-k+\bbullet}^\irr \map^k_\dag(\ccM)).
\]
\end{enumerate}
\end{corollaire}

\begin{proof}
We note that \eqref{cor:lefschetz2} follows from \eqref{cor:lefschetz1}, since the Hard Lefschetz theorem holds for pure objects, according to \cite{Mochizuki08}. Then \eqref{cor:lefschetz1} is a consequence of Proposition \ref{prop:MTMresc}\eqref{prop:MTMresc3}.
\end{proof}

\begin{exemple}[also obtained in a different way by J.-D. Yu]
Let $\varphi:X\to\PP^1$ be a projective morphism with pole divisor $P=\varphi^{-1}(\infty)$. We set $U=X\moins P$. Then we can apply \ref{cor:lefschetz}\eqref{cor:lefschetz2} to $\icT^{\varphi/\hb}$ (\cf Section \ref{subsec:expotwistgresc} below) and $\map$ is the constant map, according to \cite{Mochizuki08}, since $\cT^{\varphi/\hb}$ is pure (\cf Section \ref{subsec:cTvarphihb}). We hence get a strict isomorphism for an ample class $c$ and $\beta\in\QQ$, by using the decreasing version of the irregular Hodge filtration $F_\irr^{-\alpha-p}:=F^\irr_{\alpha+p}$:
\[
c^k:F_\irr^\beta H^{d_X-k}(U,\rd+\rd\varphi)\isom F_\irr^{\beta+k}H^{d_X+k}(U,\rd+\rd\varphi),\quad\forall\beta\in\QQ.
\]
\end{exemple}

\subsection{Pullback by a smooth morphism}
Let $\mapsm:X\to Y$ be a smooth morphism. We denote by $\taumapsm:\tauX\to\tauY$ the induced map. The smooth pullback $\taumapsm^+:\iMTM^\intt(\tauY)\!\mto\!\iMTM^\intt(\tauX)$ is defined as in \ref{def:functorsMTM}\eqref{enum:pullbackMTM}.

\begin{proposition}\label{prop:pullbackresc}
Let $\mapsm:X\to Y$ be a smooth morphism.
\begin{enumerate}
\item
The smooth pullback $\mapsm^+$ preserves the well-rescaling property.
\item
Assume that $(\icT,W_\bbullet)$ is an object of $\IrrMHM(\tauY)$. Then $\mapsm^+(\icT,W_\bbullet)$ is an object of $\IrrMHM(\tauY)$ and we have for each $\alpha\in[0,1)$:
\begin{starequation}\label{eq:pullbackresc}
R_{F^\irr_{\alpha+\bbullet}}\mapsm^+\ccM\simeq \mapsm^+R_{F^\irr_{\alpha+\bbullet}}\ccM.
\end{starequation}%
\end{enumerate}
\end{proposition}

\begin{proof}
Assume that $\cM$ is well-rescalable. Since $\cO_{\taucX}$ is $\cO_{\taucY}$-flat, the functor $\taumapsm^*$ is exact and the filtration $\taumapsm^*\tauV_\bbullet(\taucM)$ satisfies all the characteristic properties of the $\tauV$\nobreakdash-filtration of $\taumapsm^+\taucM$ (in particular, and importantly, strictness of $\taumapsm^*\gr_a^{\tauV}\!\!\taucM$). Hence, $\taumapsm^+\taucM$ is strictly $\RR$-specializable along $\tau=0$ and we have $\tauV_\bbullet(\taumapsm^+\taucM)=\taumapsm^+\tauV_\bbullet(\taucM)$. Moreover, the regularity property is easily seen to be satisfied, so \ref{def:Rrescdata} holds for $\mapsm^+\taucM$.

Moreover, we have in an obvious way for the $\cO_\cX$-module structure:
\[
i_{\tau=\hb}^*\taumapsm^*\tauV_\bbullet(\taucM)=\mapsm^*i_{\tau=\hb}^*\tauV_\bbullet(\taucM)\subset \mapsm^*i_{\tau=\hb}^*\taucM=i_{\tau=\hb}^*\taumapsm^*\taucM,
\]
and this equality extends to the $\cR_\cX$-module structure in an obvious way.

Assume that a sesquilinear pairing $\thetaiC$ has moderate growth along $\taucX_0$. This means that, when restricted to $\taucM',\taucM''$ viewed as contained in $j_*\thetacM',j_*\thetacM''$, it can be extended as a section of $\Db_{\taucX^\circ/\CC_\hb^*}$. Then, its pullback by $\taumapsm$, as defined by \eqref{eq:pullbackDb}, satisfies the similar property.

Both properties show that the functor $\mapsm^+:\iMTM^\intt(Y)\mto\iMTM^\intt(X)$ sends $\iMTM^\resc(Y)$ to $\iMTM^\resc(X)$.

Note that, since $\mapsm^+$ preserves strictness, we know that the right-hand side of \eqref{eq:pullbackresc} takes the form $R_F\mapsm^+\ccM$ for some filtration $F_\bbullet \mapsm^+\ccM$. It is then enough to prove the gradedness of $\mapsm^+\icT$, and the proof will identify the filtration $F_\bbullet \mapsm^+\ccM$ with $F^\irr_{\alpha+\bbullet}\mapsm^+\ccM$.

We can now argue as in the case of the pushforward:
\begin{align*}
\cO_\cX\otimes_{\cO_X[\hb]}\mapsm^+R_{F^\irr_{\alpha+\bbullet}}\ccM&=\mapsm^+\bigl(\cO_\cY\otimes_{\cO_Y[\hb]}R_{F^\irr_{\alpha+\bbullet}}\ccM\bigr)\\
&\simeq \mapsm^+\bigl(i_{\tau=\hb}^*\tauV_\alpha(\taucM)\bigr)\quad\text{(gradedness of $\icT$)}\\
&\simeq i_{\tau=\hb}^*\tauV_\alpha(\taumapsm^+\taucM)\quad\text{(previous argument)}.\qedhere
\end{align*}
\end{proof}

\subsection{Mixed Hodge modules as objects of \texorpdfstring{$\IrrMHM(X)$}{iMTMgresc}}
We keep the notation of Section \ref{subsec:MHM}.

\begin{proposition}\label{prop:MHMresc}
The subcategory $\MHM(X,\CC)$ of $\iMTM^\intt(X)$ is contained in $\IrrMHM(X)$.
\end{proposition}

\begin{proof}
Proposition \ref{prop:filtered} shows the graded well-rescaling property of the $\cR^\intt$-components of an object of $\MHM(X,\CC)$. Moreover, Formula \eqref{eq:taufiltN} identifies the rescaled object with the localized pullback by the functor $p^+:\iMTM^\intt(X)\mto\iMTM^\intt(\tauX,*(\tauX_0))$ associated with the projection $p:\tauX\to X$. Hence the properties of Definitions \ref{def:iMTMresc} and \ref{def:IrrMHMX}\eqref{def:IrrMHMX1} are satisfied.
\end{proof}

\subsection{Exponential twist}\label{subsec:expotwistgresc}
In this subsection, we prove Theorem \ref{th:mtmgresc}\eqref{th:mtmgresc3}. Recall that $P$ denotes the pole divisor of the meromorphic function $\varphi$. Let $\icT_*^{\varphi/\hb}$ be the object of $\iMTM^\intt(X,(*P))$ defined in Section \ref{subsec:cTvarphihb}, with respect to the meromorphic function $\varphi$ on $X$. We consider similarly $\icT_*^{\tau\varphi/\hb}$ with respect to the meromorphic function $\tau\varphi$ on $\tauX$, which is an object of $\iMTM^\intt(\tauX,(*\tauP))$. Then~$\icT_*^{\tau\varphi/\hb}(*\tauX_0)$ is an object of $\iMTM^\intt(\tauX,(*(\tauP\cup\tauX_0)))$. This is also the rescaled object of $\icT_*^{\varphi/\hb}$. We aim at proving the graded well-rescaling property of its $\cR^\intt$-components, and a similar property for the twist with a mixed Hodge module.

Let $\icT$ be the object of $\iMTM^\intt(X)$ associated with a mixed Hodge module. It induces an object $\taucT$ of $\iMTM^\intt(\tauX,(*\tauX_0))$.

We claim that \index{mixed Hodge module!exponentially twisted --}$\bigl(\icT_*^{\tau\varphi/\hb}(*\tauX_0)\otimes\taucT\bigr)[*\tauP]$, which is an object of the category $\iMTM^\intt(\tauX,(*\tauX_0))$ according to \cite[Prop.\,11.3.3]{Mochizuki11}, satisfies the graded well-rescaling property. This will show that the exponentially twisted mixed Hodge module $\icT_*^{\varphi/\hb}\otimes\icT$ belongs to $\IrrMHM(X)$.

Concerning Definition \ref{def:Rrescdata}, the only new point to be checked is the regularity along $\tau=0$. If $\varphi$ is the inverse of a coordinate function $t$, then this is contained in \cite[Prop.\,4.1(ii)]{Bibi08}. In general, one uses the construction of \cite[\S2.a]{S-Y14} to reduce to this case, by using that the regularity property is preserved under the pushforward by a projective morphism $X'\to X$ (if the zero set and the pole set $P$ of $\varphi$ do not intersect, this simply amounts to considering, in the neighbourhood of any point of~$P$, the graph of the holomorphic function $1/\varphi$).

Lastly, the grading condition is given by \cite[Prop.\,5.5\,\&\,Rem.\,5.6]{S-Y14}.\qed

\subsection{Open questions}\label{subsec:openq}
Although the category $\iMTM(X)$ is endowed with a nearby cycle functor, a localization and dual localization functor, and a vanishing cycle functor, we do not know whether $\IrrMHM(X)$ is preserved by these functors. Even if so, there remains the problem of computing the behaviour of the irregular Hodge filtration with respect to these functors, as indicated in Remark \ref{rem:mainRF}\eqref{rem:mainRF3}.

\section{Rigid irreducible holonomic \texorpdfstring{$\cD_{\PP^1}$}{DP1}-modules}\label{sec:rigid}
In this section we give a proof of Theorem \ref{th:rigidP1}. Our base variety $X$ is now equal to the Riemann sphere $\PP^1$ (we use its analytic topology). Recall (\cf Remark \ref{rem:intuniquepure}) that, to any irreducible holonomic $\cD_{\PP^1}$-module corresponds exactly one object $\cT$ of $\MTM(\PP^1)$ which is pure of weight zero, up to an obvious ambiguity, and there is at most one integrable structure up to equivalence on this object. Our first goal is to give a precise criterion for the existence of such an integrable structure, at least in the rigid case.

We will reformulate in Proposition \ref{prop:rigidalgo} the existence of a variant of the Katz algo\-rithm, due to Arinkin \cite{Arinkin08} and Deligne \cite{Deligne06b}, for possibly irregular irreducible holonomic $\cD_{\PP^1}$\nobreakdash-modules, which terminates in rank one if and only if the rigidity property holds. In such a case, the rigidity property holds at each step of the algorithm, and the rank never increases all along the process. The algorithm consists in performing suitably chosen operations: twist by a rank-one meromorphic connection, changing the point at infinity, and Laplace transformation. Due to the first operation, one can \hbox{assume} without loss of generality that the terminating rank-one holonomic $\cD_{\PP^1}$\nobreakdash-module is $(\cO_{\PP^1},\rd)$. That rigidity is preserved by the Laplace transformation is a result due to Bloch-Esnault \cite{B-E04}.

\begin{proposition}\label{prop:rigidalgo}
Let $\ccM$ be a rigid irreducible holonomic $\cD_{\PP^1}$-module. There exist
\begin{enumeratea}
\item\label{prop:rigidalgoa}
a smooth projective complex variety $X$ and a normal crossing divisor $D\subset X$, together with a subdivisor $H\subset D$,
\item\label{prop:rigidalgob}
a projective morphism $\map:X\to\PP^1$,
\item\label{prop:rigidalgoc}
a meromorphic function $\mero$ on $X$ with poles contained in $D$ and whose pole and zero divisors do not intersect,
\item\label{prop:rigidalgod}
a locally free rank-one $\cO_X(*D)$-module $\ccN=\ccN_\regg$ with a regular singular meromorphic connection $\nabla$,
\end{enumeratea}
\noindent
such that, denoting by $\cT_\ccM$ the polarizable pure twistor $\cD$-module of weight $0$ associated with $\ccM$, and by $\cT_\ccN$ the mixed twistor $\cD$-module associated with $\ccN$,
\begin{enumerate}
\item\label{prop:rigidalgo0}
if $\ccM$ is locally formally unitary, the local system $\ccN^\nabla$ is unitary,
\item\label{prop:rigidalgo2}
$\cT_\ccM$ is the image of the natural morphism
\[
\map^0_\dag(\cT_*^{\mero/\hb}\otimes\Gamma_{[!H]}\cT_\ccN)\to \map^0_\dag(\cT_*^{\mero/\hb}\otimes\Gamma_{[*H]}\cT_\ccN).
\]
\end{enumerate}
\end{proposition}

Let us make precise that the minimal (or intermediate) extension $\ccN_{\min}$ of $\ccN$ is an irreducible regular holonomic $\cD_X$-module and $\ccN=\ccN_{\min}(*D)$. There is a unique polarizable pure twistor $\cD$-module $\cT_{\ccN_{\min}}$ of weight $0$ associated with $\ccN_{\min}$, and by definition $\cT_\ccN$ is the mixed twistor $\cD$-module $\cT_{\ccN_{\min}}[*D]$. In particular, \hbox{$\Gamma_{[*H]}\cT_\ccN=\cT_\ccN$}, but we keep the first notation to have in mind the analogous notation $\Gamma_{[!H]}\cT_\ccN$. Moreover, if $\ccN^\nabla$ is locally unitary, then $\cT_{\ccN_{\min}}$ is a pure complex Hodge module of weight~$0$, and $\cT_\ccN$ is a complex mixed Hodge module.

\begin{corollaire}\label{cor:rigidalgo}
Let $\ccM$ be a rigid irreducible holonomic $\cD_{\PP^1}$-module, with associated pure twistor $\cD$-module $\cT_\ccM$. Then $\cT_\ccM$ is integrable if and only if $\ccM$ is locally formally unitary.
\end{corollaire}

\begin{proof}
If $\cT_\ccM$ is integrable, then the associated $\ccM$ is locally formally unitary (this does not use irreducibility nor rigidity): one applies Proposition \ref{prop:speint}\eqref{prop:speint2a} to $\cT$ as well as to any triple obtained from $\cT$ after ramification and exponential twist, which remains integrable.

On the other hand, if $\ccM$ is irreducible, rigid and locally unitary, then $\ccN^\nabla$ is unitary, according to \ref{prop:rigidalgo}\eqref{prop:rigidalgo0}, so $\cT_\ccN$ is integrable, and then so is $\cT_\ccM$, according to \ref{prop:rigidalgo}\eqref{prop:rigidalgo2}.
\end{proof}

\begin{proof}[Proof of Theorem \ref{th:rigidP1}]
Since Corollary \ref{cor:rigidalgo} gives the first part of the theorem, it remains to prove the second part. The uniqueness up to equivalence follows from Corollary \ref{cor:grescequivint}. Let us use the notation and results of Proposition \ref{prop:rigidalgo}. Since $\cT_\ccN$ is a complex mixed Hodge module, so does $\Gamma_{[\star H]}\cT_\ccN$ ($\star=*,!$). Also, the natural morphism $\Gamma_{[!H]}\icT_\ccN\to\Gamma_{[*H]}\icT_\ccN$ is a morphism of mixed Hodge modules. From Theorem \ref{th:mtmgresc}\eqref{th:mtmgresc3} we conclude that
\[
\bigl(\icT^{\mero/\hb}\otimes\Gamma_{[!H]}\icT_\ccN\bigr)^\resc\to\bigl(\icT^{\mero/\hb}\otimes\Gamma_{[*H]}\icT_\ccN\bigr)^\resc
\]
is a morphism in $\IrrMHM(X)$, and from Theorem \ref{th:mtmgresc}\eqref{th:mtmgresc2} we obtain that
\[
\map^0_\dag\Bigl[\bigl(\icT^{\mero/\hb}\otimes\Gamma_{[!H]}\icT_\ccN\bigr)^\resc\Bigr]\to \map^0_\dag\Bigl[\bigl(\icT^{\mero/\hb}\otimes\Gamma_{[*H]}\icT_\ccN\bigr)^\resc\Bigr]
\]
is a morphism in $\IrrMHM(\PP^1)$. Its image in the abelian category $\IrrMHM(\PP^1)$ restricts through $i_{\tau=1}^*$ to the corresponding image in $\iMTM^\intt(\PP^1)$, which is nothing but $\cT_\ccM$. This gives the desired result.
\end{proof}

\begin{remarque}[The use of the graded-rescaling property]
Note that both source and target of the morphism
\[
\map^0_\dag(\ccE^\mero\otimes\Gamma_{[!H]}\ccN)\to \map^0_\dag(\ccE^\mero\otimes\Gamma_{[*H]}\ccN)
\]
carry an irregular Hodge filtration, obtained in a strict way by pushforward by $\map^0_\dag$ of that of $\ccE^\mero\otimes\Gamma_{[\star H]}\ccN$, according to Theorem \ref{th:mainRF}\eqref{th:mainRF6}, and the graded-rescaling property implies that this morphism is strictly filtered with respect to it, so $R_{F^\irr_{\alpha+\bbullet}}\ccM$ is the strict image of the corresponding Rees morphism.
\end{remarque}

\begin{proof}[Proof of Proposition \ref{prop:rigidalgo}]
Let $\ccM$ be rigid irreducible and let us assume that there exist $X,D,H$, $\map:X\to\PP^1$, $\mero$ and $(\ccN,\nabla)$ as in \ref{prop:rigidalgo}\eqref{prop:rigidalgoa}\nobreakdash--\eqref{prop:rigidalgod} for $\ccM$, satisfying \ref{prop:rigidalgo}\eqref{prop:rigidalgo2}. We will prove the following properties.
\begin{enumeratei}
\item\label{enum:rigidalgo1}
If $\ccL$ is a rank-one meromorphic connection on $\PP^1$ with poles along $\Sigma\subset\PP^1$, then Data \ref{prop:rigidalgo}\eqref{prop:rigidalgoa}--\eqref{prop:rigidalgod} satisfying \ref{prop:rigidalgo}\eqref{prop:rigidalgo2} exist for $\Gamma_{[*\Sigma]}(\ccM\otimes\ccL)$ and for the image $\ccM'$ of $\Gamma_{[!\Sigma]}(\ccM\otimes\ccL)\to\Gamma_{[*\Sigma]}(\ccM\otimes\ccL)$.

Moreover, if $\ccN^\nabla$ is unitary and $\ccL$ is locally formally unitary, then so is $\ccM'$ and $(\ccN',\nabla)$ can be chosen unitary.

\item\label{enum:rigidalgo2}
Let $\infty$ be a point of $\PP^1$, chosen as the point at infinity. By \eqref{enum:rigidalgo1}, we can assume that $\ccM=\Gamma_{[*\infty]}\ccM$. Let $\ccM'$ be the Laplace transform of $\ccM$ with respect to this point. Then Data \ref{prop:rigidalgo}\eqref{prop:rigidalgoa}--\eqref{prop:rigidalgod} satisfying \ref{prop:rigidalgo}\eqref{prop:rigidalgo2} exist for $\ccM'$.

Moreover, if $\ccN^\nabla$ is unitary, $\ccN'^\nabla$ can be chosen to be unitary.
\end{enumeratei}

These two properties allow us to conclude the proof of the proposition, since any rigid holonomic $\cD_{\PP^1}$-module can be obtained by applying a sequence of \eqref{enum:rigidalgo1} and \eqref{enum:rigidalgo2} to $(\cO_{\PP^1},\rd)$, according to the Arinkin-Deligne algorithm, and moreover, if $\ccM$ is locally formally unitary, then the rank-one connections $\ccL_\regg$ chosen at each step are locally unitary.

We start by making explicit some admissible modifications of Data \ref{prop:rigidalgo}\eqref{prop:rigidalgoa}--\eqref{prop:rigidalgod}. Let $\modif':X'\to X$ be a projective modification which induces an isomorphism $X'\moins\nobreak \modif^{\prime-1}(D)\isom X\moins D$, and set $D'=\modif^{\prime-1}(D)$, $H'=\modif^{\prime-1}(H)$, $\map'=\map\circ \modif'$, $\mero'=\mero\circ \modif'$, $\ccN'=\ste^{\prime+}\ccN$. Then \ref{prop:rigidalgo}\eqref{prop:rigidalgo2} and possibly \ref{prop:rigidalgo}\eqref{prop:rigidalgo0} hold for $\ccM$ with these data. Indeed, we have $\ccE^\mero\otimes\ccN\simeq\modif^{\prime0}_\dag(\ccE^{\mero'}\otimes\ccN')$ and, after Corollary \ref{cor:Evarphi!H},
\[
\ccE^\mero\otimes\Gamma_{[!H]}\ccN=\Gamma_{[!H]}(\ccE^\mero\otimes\ccN)=\Gamma_{[!H]}\modif^{\prime0}_\dag(\ccE^{\mero'}\otimes\ccN')=\modif^{\prime0}_\dag(\ccE^{\mero'}\otimes\Gamma_{[!H']}\ccN'),
\]
and similarly $\modif^{\prime k}_\dag(\ccE^{\mero'}\otimes\Gamma_{[!H']}\ccN')=0$, so $f^{\prime0}_\dag(\ccE^{\mero'}\otimes\Gamma_{[!H']}\ccN')\simeq \map^0_\dag(\ccE^\mero\otimes\Gamma_{[!H]}\ccN)$. A~similar equality obviously holds with $\Gamma_{[*H']}$, and these isomorphisms can be lifted at the level of $\MTM$.

Let us show \eqref{enum:rigidalgo1}. There exists a meromorphic function $\psi$ on $\PP^1$ and a rank-one meromorphic connection $\ccL_\regg$ with regular singularities, such that $\ccL=\ccE^\psi\otimes\ccL_\regg$. We can write $\ccL=(\cO_{\PP^1}(*\Sigma),\rd+\rd\psi+\omega)$, where $\Sigma$ is the pole divisor of $\ccL$ and~$\omega$ is a one-form with at most simple poles at $\Sigma$. Moreover, $\ccL$ is locally formally unitary if and only if $\ccL_\regg$ is unitary, \ie the residues of $\omega$ at $\Sigma$ are real. Set $D_1=\map^{-1}(\Sigma)$. By a suitable change of data as above, we can assume that the pole and zero divisors of~\hbox{$\mero+\psi\circ\map$} do not intersect, and that $D\cup D_1$ is a normal crossing divisor. We~claim that, with this assumption, the similar data with $D'=D\cup D_1$, $H'=H\cup D_1$, $\mero'=\mero+\psi\circ f$, $\ccN'=(\Gamma_{[*D_1]}\ccN)\otimes(\Gamma_{[*D]}\map^+\ccL_\regg)$ is a suitable set of data for $\ccM'$.

Firstly, one can check that $\ccL\otimes \map_\dag^0(\ccE^\mero\otimes\Gamma_{[\star H]}\ccN)\simeq \map_\dag^0\bigl(\map^+\ccL\otimes (\ccE^\mero\otimes\Gamma_{[\star H]}\ccN)\bigr)$. Then one can notice that
\[
\map^+(\ccE^\psi\otimes\ccL_\regg)\otimes(\ccE^\mero\otimes\Gamma_{[!H]}\ccN)=\ccE^{\mero'}\otimes (\map^+\ccL_\regg)\otimes\Gamma_{[!H]}\ccN\simeq \ccE^{\mero'}\otimes\Gamma_{[!H]}(\map^+\ccL_\regg\otimes\ccN),
\]
since the second isomorphism is mostly obvious for rank-one regular meromorphic connections and normal crossing divisors. Moreover, from the first two lines of \ref{subsub:commutation}\eqref{enum:rel4} one deduces that $\Gamma_{[!\Sigma]}\circ \map_\dag^0\simeq \map_\dag^0\circ\Gamma_{[!D_1]}$. It follows that, for $\star=*,!$, $\Gamma_{[\star\Sigma]}(\ccL\otimes\ccM)$ is the image of the morphism
\[
\map^0_\dag\Bigl(\ccE^{\mero'}\otimes\Gamma_{[\star D_1]}\Gamma_{[!H]}(\map^+\ccL_\regg\otimes\ccN)\Bigr)\to \map^0_\dag\Bigl(\ccE^{\mero'}\otimes\Gamma_{[\star D_1]}\Gamma_{[*H]}(\map^+\ccL_\regg\otimes\ccN)\Bigr).
\]
As a consequence, $\ccM'$ is the image of
\[
\map^0_\dag\Bigl(\ccE^{\mero'}\otimes\Gamma_{[!D_1]}\Gamma_{[!H]}(\map^+\ccL_\regg\otimes\ccN)\Bigr)\to \map^0_\dag\Bigl(\ccE^{\mero'}\otimes\Gamma_{[*D_1]}\Gamma_{[*H]}(\map^+\ccL_\regg\otimes\ccN)\Bigr),
\]
and since $\Gamma_{[\star D_1]}\Gamma_{[\star H]}=\Gamma_{[\star H']}$ ($\star=*,!$), we obtained the desired assertion.

Let us now show \eqref{enum:rigidalgo2}. Recall that the Laplace transformation is an exact functor $\Mod_\hol(\PP^1,*\infty)\mto\Mod_\hol(\PP^1,*\infty)$, so $\ccM'$ is the image of the Laplace transform of $\map^0_\dag(\ccE^\mero\otimes\Gamma_{[!H]}\ccN)$ to that of $\map^0_\dag(\ccE^\mero\otimes\Gamma_{[*H]}\ccN)$. We denote by $t$ the variable on $\PP^1_t$ and we introduce a new $\PP^1$ with variable $\tau$, and we denote by $p,q:\PP^1_t\times\PP^1_\tau\to\PP^1_t,\PP^1_\tau$ the first and second projections. For any holonomic $\cD_{\PP^1_t}$-module $M$, its Laplace transform is given by the formula $q^0_\dag(\ccE^{-t\tau}\otimes\Gamma_{[*(\PP^1_t\times\infty)]}p^+\Gamma_{[*\infty]}M)$ (and $q_\dag^k(\cdots)=0$ for $k\neq0$). We apply this formula to $M=\map^0_\dag(\ccE^\mero\otimes\Gamma_{[\star H]}\ccN)$ ($\star=*,!$).

Since $\ccM=\Gamma_{[*\infty]}\ccM$, we can assume that $D$ contains $\map^{-1}(\infty)$ but $H$ does not contain any component of it, and we can forget $\Gamma_{[*\infty]}$ in the previous formula. We denote by $p_X,q_X$ the projections $X\times\PP^1_\tau\to X,\PP^1_\tau$, so that $q\circ(f\times\id_{\PP^1_\tau})=q_X$. For $\star=*,!$, we~have
\begin{align*}
q^0_\dag\Bigl(\ccE^{-t\tau}\otimes &\Gamma_{[*(\PP^1_t\times\infty)]}p^+\map^0_\dag(\ccE^\mero\otimes\Gamma_{[\star H]}\ccN)\Bigr)\\
&\simeq q^0_\dag\Bigl(\ccE^{-t\tau}\otimes\Gamma_{[*(\PP^1_t\times\infty)]}(f\times\id)^0_\dag p_X^+(\ccE^\mero\otimes\Gamma_{[\star H]}\ccN)\Bigr)\quad(\text{after }\ref{subsub:commutation}\eqref{enum:rel3})\\
&\simeq q^0_\dag\Bigl(\ccE^{-t\tau}\otimes (f\times\id)^0_\dag \Gamma_{[*(X\times\infty)]}p_X^+(\ccE^\mero\otimes\Gamma_{[\star H]}\ccN)\Bigr)\quad(\text{after }\ref{subsub:commutation}\eqref{enum:rel4})\\
&\simeq q^0_\dag\Bigl((f\times\id)^0_\dag \Gamma_{[*(X\times\infty)]}(\ccE^{p^*\mero-\tau\map}\otimes p_X^+\Gamma_{[\star H]}\ccN)\Bigr)\quad(\text{after }\ref{subsub:commutation}\eqref{enum:rel4})\\
&\simeq q^0_{X\dag}\Bigl(\ccE^{p^*\mero-\tau\map}\otimes\Gamma_{[*(X\times\infty)]}p_X^+\Gamma_{[\star H]}\ccN\Bigr)\quad(q^k_\dag(\cdots)=0\text{ for }k\neq0).
\end{align*}
Set $D_1=(D\times\PP^1_\tau)\cup(X\times\infty)$, $H_1=H\times\PP^1_\tau$ and $\ccN_1=\Gamma_{[*(X\times\infty)]}p_X^+\ccN$. Then the latter expression can also be written as
\[
q^0_{X\dag}\bigl(\ccE^{p^*\mero-\tau\map}\otimes\Gamma_{[\star H_1]}\ccN_1\bigr).
\]
Let $\modif':X'\to X\times\PP^1_\tau$ be a projective modification which induces an isomorphism above the complement of $D_1$, whose pullback is denoted by~$D'$, such that the pole and zero divisors of $\mero':=(p\circ \modif')^*\mero-\modif^{\prime*}(\tau\map)$ do not intersect, and set~$H'=\modif^{\prime-1}(H_1)$ and $\ccN'=\ste^{\prime+}\ccN_1$. Then
\begin{align*}
q^0_{X\dag}\bigl(\ccE^{p^*\mero-\tau\map}\otimes&\Gamma_{[\star H_1]}\ccN_1\bigr)\\
&\simeq q^0_{X\dag}\modif^{\prime0}_\dag\bigl(\ccE^{\mero'}\otimes\Gamma_{[\star H']}\ccN'\bigr)\quad(\text{after \ref{subsub:commutation}\eqref{enum:rel7} and Cor.\,\ref{cor:Evarphi!H} for $\star=!$})\\
&\simeq (q_X\circ \modif')^0_\dag\bigl(\ccE^{\mero'}\otimes\Gamma_{[\star H']}\ccN'\bigr)\quad(\modif^{\prime k}_\dag(\cdots)=0\text{ for }k\neq0).
\end{align*}
Set now $f'=q_X\circ \modif'$. Then, with respect to the data $X',D',H',f'$ and $\mero',\ccN'$, \ref{prop:rigidalgo}\eqref{prop:rigidalgo2} and possibly \ref{prop:rigidalgo}\eqref{prop:rigidalgo0} hold for~$\ccM'$.
\end{proof}

\specialmakechapterhead{Irregular mixed Hodge structures}{\texorpdfstring{Jeng-Daw Yu\protect\ftmark}{JDY}}\label{chap:irregmhm}\label{CHAP:IRREGMHM}
\markboth{\MakeUppercase{\chaptername\ \thechapter. Irregular mixed Hodge structures}}{\MakeUppercase{\chaptername\ \thechapter. Irregular mixed Hodge structures}}

\bgroup
\makeatletter
\def\@makefnmark{(*)\enspace}
\makeatother
\footnotetext{%
\texorpdfstring{Jeng-Daw Yu: Department of Mathematics, National Taiwan University, Taipei 10617, Taiwan.\\ \emph{E\nobreakdash-mail~:} \url{jdyu@ntu.edu.tw}\quad$\cbbullet$\quad\emph{Url~:} \url{http://homepage.ntu.edu.tw/~jdyu/}\\
This work was partially supported by the TIMS and the MoST of Taiwan.}{JDYaddress}}
\egroup

\section{Introduction to Chapter \ref{chap:irregmhm}}

Mirror symmetry has suggested various extensions of the notion of a mixed Hodge structure, in order to make it compatible with the cohomological isomorphisms occurring in some of the variants of mirror symmetry. These are
\begin{itemize}
\item
the \index{Hodge structure!semi-infinite --}semi-infinite Hodge structures \cite{Barannikov01},
\item
the \index{TERP structure}tr-TERP structures \cite{Hertling01}
\item
the \index{Hodge structure!non-commutative --}non-commutative Hodge structures \cite{K-K-P08},
\item
the \index{mixed Hodge structure!exponential --}exponential mixed Hodge structures \cite{K-S10},
\item
the \index{twistor structure!integrable mixed --}integrable mixed twistor structures \cite{Simpson97,Mochizuki11}.
\end{itemize}

The purpose of this chapter is to make more explicit the category $\IrrMHS=\IrrMHM(\pt)$ of irregular mixed Hodge structures, which is ``in between'' that of exponential mixed Hodge structures ($\EMHS$) and that of non-commutative mixed Hodge structures ($\iMTS^\intt(\CC)$, Definition \ref{def:ncMHS}). The category $\IrrMHS$ is endowed with a functor to the category of bi-filtered vector spaces, one filtration being induced by the weight filtration, and the other one is called the \emph{irregular Hodge filtration}. The interest of $\IrrMHS$ is that morphisms in this category induce \emph{strictly} bi-filtered morphisms of bi-filtered vector spaces. On the other hand, this functor, when restricted to the category of exponential mixed Hodge modules, was already considered by Deligne in \cite{Deligne8406} in some special cases.

As explained in Section \ref{sec:twistorstructures}, the notion of non-commutative mixed Hodge structure is that of an integrable mixed twistor structure (a variant of the notion introduced by Simpson \cite{Simpson97}). It is equivalent to the notion of integrable mixed twistor $\cD$-module over a point, which enables an extension to dimensions $\geq1$ and the use of functors in the category $\MTM^\intt$ of integrable mixed twistor $\cD$-modules \cite{Mochizuki11}. Similarly, $\IrrMHS$ is the category of irregular mixed Hodge modules on a point, by using the category $\IrrMHM$ introduced in Chapter \ref{part:2}.

The category $\EMHS$ is known to be a neutral Tannakian category by using additive convolution as tensor product (\cf\cite[\S4.2]{K-S10}). We will endow $\IrrMHS$ with such a structure (recall that $\MTS^\intt(\CC)$ is endowed with such a structure, \cf Proposition \ref{prop:TSTannakienne}), and we will show that the functors
\[
\EMHS\mto \IrrMHS\mto\iMTS^\intt(\CC)
\]
are compatible with these structures. Moreover, the functor mentioned above from $\IrrMHS(\CC)$ to that of bi-filtered vector spaces is also compatible with tensor products. This is translated by Künneth formulas for $\IrrMHS$, and leads to Thom-Sebastiani formulas when applied to suitable objects of $\EMHS$.

For any subfield $\kk$ of $\RR$, the category $\EMHS(\kk)$ is defined, as well as the category $\iMTS^\intt_\good(\kk)$. We will define accordingly the category $\IrrMHS(\kk)$ in such a way that the above functors are defined over $\kk$.

\section{Connections with a pole of order two}\label{sec:poleordertwo}

In this section, we define the functor ``irregular Hodge filtration'' for meromorphic connections with a pole of order two, and prove its good behaviour with respect to tensor product and duality. However, the category of free $\CC[\hb]$-modules with a connection having a pole of order two at the origin, a regular singularity at infinity, and no other pole, is not abelian, and more structure will be added in Section \ref{sec:irregHS} in order to develop an irregular Hodge theory.

\subsection{Harder-Narasimhan filtration and irregular Hodge filtration}\label{subsec:HNconnection}

Let $\PP^1$ be covered by two affine charts $\Afu_\hb,\Afu_v$ with $v=1/\hb$ on the intersection of the two charts. We set $0=\{v=0\}$ and $\infty=\{\hb=0\}$. Let $\cM$ be a free $\cO_{\Afu_\hb}$-module of finite rank, equipped with a connection
\[
\nabla:\cM\to\Omega_{\Afu_\hb}^1(2\cdot[\infty])\otimes\cM,
\]
having a pole of order two at $\hb=0$ and no other pole. We also denote by $\cM_*$ the Deligne rational extension at $v=0$, that we regard as a $\cO_{\PP^1}(*0)$-module, endowed with a rational connection having a pole of order two at $\hb=0$ and a regular singularity at $v=0$. We will assume in this section that \emph{the eigenvalues of the monodromy of the connection have absolute value equal to one} (this assumption is automatically satisfied for the irregular mixed Hodge structures defined in Section \ref{subsec:irregHS}, \cf Lemma \ref{lem:realeigenvalues}). For each $\beta\in\RR$, we also denote by $\cM_\beta$ the locally free $\cO_{\PP^1}$-module (Deligne extension) such that the connection has a simple pole at $v=0$ and all eigenvalues of its residue belong to $[-\beta,-\beta+1)$. It satisfies $\cM_\beta(k\cdot[0])=\cM_{\beta+k}\subset\nobreak\cM_*$ for every $k\in\ZZ$. We regard $(\cM_\beta)_{\beta\in\RR}$ as an increasing filtration of $\cM_*$, each term being equipped with the induced rational connection
\[
\nabla:\cM_\beta\to\Omega^1_{\PP^1}([0]+2\cdot[\infty])\otimes\cM_\beta.
\]
The jumping indices form a finite set modulo $\ZZ$, and for each jumping index, we denote by ${<}\beta$ its predecessor. Then $v\nabla_{\partial_v}+\beta$ is nilpotent on $\cM_\beta/\cM_{<\beta}$, and the index of nilpotence is bounded by a fixed integer independent of $\beta$, due to the finiteness above. When we wish to fix the index in $[0,1)$, we will denote it by $\alpha$ instead of $\beta$.

We will denote by \index{$HNp$@$\HN^p$}$\HN^p(\cM_\beta)$ the \index{Harder-Narasimhan filtration}\emph{Harder-Narasimhan filtration} of $\cM_\beta$. We will now work with the modules of sections on the various affine charts. So we set
\[
M=\Gamma(\Afu_\hb,\cM)=\Gamma(\PP^1,\cM_*),
\]
which a free $\CC[\hb]$-module of finite rank, endowed with an action of $\hb^2\partial_\hb$, which is regular at $v=0$ and we assume that the eigenvalues of the monodromy around $v=0$ have absolute value equal to one. Then $M[\hbm]$ is a free $\CC[\hb,\hbm]=\CC[v,v^{-1}]$-module and we have the free $\CC[v]$-modules $V_\beta M[\hbm]$ ($\beta\in\RR$), that is, the Deligne extension as above. We simply denote it by $V_\beta$. It satisfies $v^kV_\beta=V_{\beta-k}$ for every $k\in\ZZ$.

The algebraic bundle~$\cM_\beta$ on $\PP^1$ is obtained by the gluing data
\[
\begin{array}{ccccc}
V_\beta:=V_\beta M[\hbm]&\subset&M[\hbm]&\supset&M\\[5pt]
\Afu_v&\supset&\mathbb{G}_\mathrm{m}&\subset&\Afu_\hb
\end{array}
\]
We then have
\[
\cM_\beta/\cM_{<\beta}=V_\beta/V_{<\beta}=:\gr_\beta^VM[\hbm].
\]

For any $\beta\in\RR$, we consider the intersection $V_\beta\cap M$ in $M[\hbm]$. This is a finite-dimensional vector space, which is nothing but $\Gamma(\PP^1,\cM_\beta)$. We have a natural morphism
\[
\cO_{\PP^1}\otimes_\CC\Gamma(\PP^1,\cM_\beta)\to\cM_\beta
\]
whose image is $\HN^0(\cM_\beta)$: this is seen by using the Birkhoff-Grothendieck decomposition of $\cM_\beta$ to reduce to the rank-one case. It follows that
\begin{align*}
\Gamma\big(\Afu_v,\HN^0(\cM_\beta)\big)&\simeq\CC[v]\cdot(V_\beta\cap M)\subset \Gamma(\Afu_v,\cM_\beta)=V_\beta,\\
\Gamma\big(\Afu_\hb,\HN^0(\cM_\beta)\big)&\simeq\CC[\hb]\cdot(V_\beta\cap M)\subset \Gamma(\Afu_\hb,\cM_\beta)=M.
\end{align*}
Then, for every $p\in\ZZ$,
\[
\HN^p(\cM_\beta)\simeq\HN^0(\cM_\beta(-p))(p)\simeq \HN^0(\cM_{\beta-p})(p),
\]
so
\begin{align*}
\Gamma\big(\Afu_v,\HN^p(\cM_\beta)\big)&\simeq v^{-p}\CC[v]\cdot(V_{\beta-p}\cap M)\subset v^{-p}\Gamma\big(\Afu_v,\cM_{\beta-p}\big)=v^{-p}V_{\beta-p}=V_\beta,\\
\Gamma\big(\Afu_\hb,\HN^p(\cM_\beta)\big)&\simeq\CC[\hb]\cdot(V_{\beta-p}\cap M)\subset\Gamma\big(\Afu_\hb,\cM_{\beta-p}\big)=M.
\end{align*}
Notice that $\HN^p(\cM_\beta)$ is a filtration of $\cM_\beta$ by sub-bundles. For $\beta\in\RR$, let us set $\alpha=\beta-[\beta]\in[0,1)$ and $p=-[\beta]\in\ZZ$. The description above shows that\index{$FzzHN$@$F^\sHN_\bbullet$}
\begin{equation}\label{eq:FM*}
F^\sHN_\beta M:=\Gamma\big(\Afu_\hb,\HN^p(\cM_\alpha)\big)=\CC[\hb]\cdot(V_\beta\cap M),
\end{equation}
is an increasing filtration of $M$ by sub-bundles discretely indexed by $\RR$ (\ie each $F_\beta^\sHN M/F_{<\beta}^\sHN M$ is a free $\CC[\hb]$-module). We will also regard it as decreasing by setting $F_\sHN^{-\beta}M=F^\sHN_\beta M$. It is an easy remark that this filtration satisfies the Griffiths transversality property (\cf\cite[Rem.\,6.3]{S-Y14}). The restriction of $F^\sHN_\beta M$ to $M[\hbm]$~is
\[
F^\sHN_\beta(M[\hbm]):=(F^\sHN_\beta M)[\hbm]=\CC[\hb,\hbm]\cdot(V_\beta\cap M).
\]
As usual for filtrations discretely indexed by $\RR$, we set $\gr_\beta^{F^\sHN}=F^\sHN_\beta/F^\sHN_{<\beta}$.

We denote by $H$ the fibre $M_{|\hb=1}$ of $M$ (or $M[\hbm]$) at $\hb=1$, \ie $H=M/(\hb-1)M=M[\hbm]/(\hb-1)M[\hbm]$. This is a finite dimensional $\CC$-vector space.

\begin{definition}[Irregular Hodge filtration]\label{def:irregHF}
The \index{irregular Hodge filtration}\emph{irregular Hodge filtration} on $H$ is the increasing filtration indexed by $\RR$ naturally induced by $V_\beta\cap M$ on $H=M/(\hb-\nobreak1)M$:\index{$Fzzirr$@$F^\irr_\bbullet$}
\[
F_\beta^\irr H:=(V_\beta\cap M)\big/(V_\beta\cap M)\cap(\hb-1)M.
\]
\end{definition}

On the other hand, we can endow $H$ with the filtration naturally induced by $F_\beta^\sHN M$:
\[
F^\sHN_\beta H:=F^\sHN_\beta M\big/(F^\sHN_\beta M)\cap (\hb-1)M.
\]

\begin{lemme}\label{lem:irrHN}
The filtrations $F_\beta^\irr H$ and $F^\sHN_\beta H$ coincide.
\end{lemme}

\begin{proof}
The vector space $F^\sHN_\beta H$ is the image of $\CC[\hb]\cdot(V_\beta\cap M)$ in $M/(\hb-1)M$, which is also equal to the image of $V_\beta\cap M$ in $M/(\hb-1)M$.
\end{proof}

We now analyze the filtration induced on the nearby cycles at $v=0$. For $\gamma\in\RR$ fixed, we consider the increasing filtration indexed by $\ZZ$:
\begin{equation}\label{eq:FgralphaV}
F_p\gr_\gamma^V(M[\hbm]):=(V_\gamma\cap \hb^{-p}M)/(V_{<\gamma}\cap \hb^{-p}M),
\end{equation}
and the increasing filtration discretely indexed by $\RR$:
\begin{equation}\label{eq:FHNgralphaV}
F^\sHN_\beta\gr_\gamma^V(M[\hbm]):=F^\sHN_\beta(M[\hbm])\cap V_\gamma/F^\sHN_\beta(M[\hbm])\cap V_{<\gamma}.
\end{equation}

\begin{lemme}\label{lem:FFHN}
For every $\beta,\gamma\in\RR$ we have
\[
\gr_\beta^{F^\sHN}\gr_\gamma^V(M[\hbm])=0\quad\text{if }\beta\notin \gamma+\ZZ,
\]
and for $p\in\ZZ$,
\[
F^\sHN_{\gamma+p}\gr_\gamma^V(M[\hbm])=F_p\gr_\gamma^V(M[\hbm]).
\]
\end{lemme}

\begin{proof}
Let us fix $\gamma\in\RR$. The first point amounts to proving that, for any $p\in\ZZ$,
\[
V_\gamma\cap F^\sHN_{<\gamma+p}(M[\hbm])+V_{<\gamma}=V_\gamma\cap F^\sHN_{\gamma+p-1}(M[\hbm])+V_{<\gamma}.
\]
Note that
\[
F^\sHN_{<\gamma+p}(M[\hbm])=\CC[\hb,\hbm]\cdot(V_{<\gamma+p}\cap M)=\CC[\hb,\hbm]\cdot(V_{<\gamma}\cap \hb^{-p}M),
\]
and $\CC[\hbm]\cdot(V_{<\gamma}\cap \hb^{-p}M)\subset V_{<\gamma}$, so
\[
V_\gamma\cap F^\sHN_{<\gamma+p}(M[\hbm])+V_{<\gamma}=V_\gamma\cap\big(\hb\CC[\hb]\cdot(V_{<\gamma}\cap \hb^{-p}M)\big)+V_{<\gamma}.
\]
It is thus enough to prove that
\[
V_\gamma\cap\big(\hb\CC[\hb]\cdot(V_{<\gamma}\cap \hb^{-p}M)\big)\subset V_{\gamma-1}\cap \hb^{-p}M.
\]
We prove by induction on $\ell\geq1$ that, if $m_j\in V_{<\gamma}\cap \hb^{-p}M$ for $j=1,\dots,\ell$ and $\sum_j\hb^jm_j\in V_\gamma$, then $\sum_j\hb^jm_j\in V_{\gamma-1}\cap \hb^{-p}M$.

If $\ell=1$, we have $\hb m_1\in V_\gamma$, so $m_1\in V_{\gamma-1}\cap(V_{<\gamma}\cap \hb^{-p}M)=V_{\gamma-1}\cap \hb^{-p}M$. If $\ell\geq2$, we have, since $\gamma<\gamma+\ell-1$,
\[
\hb^\ell m_\ell\in V_\gamma+\sum_{j=1}^{\ell-1}\hb^j(V_{<\gamma}\cap \hb^{-p}M)\subset V_{<\gamma+\ell-1},
\]
hence
\[
\hb m_\ell\in V_{<\gamma}\cap(V_{<\gamma+1}\cap \hb^{-p+1}M)=V_{<\gamma}\cap \hb^{-p+1}M\subset V_{<\gamma}\cap \hb^{-p}M.
\]
We thus rewrite the sum as
\[
\sum_{j=1}^\ell \hb^jm_j=\sum_{j=1}^{\ell-2}\hb^jm_j+\hb^{\ell-1}(m_{\ell-1}+\hb m_\ell),
\]
which now satisfies the induction hypothesis, hence the result.

For the second point, by induction on $p$, it is a matter of proving that
\[
F^\sHN_{\gamma+p}(M[\hbm])\cap V_\gamma=(V_\gamma\cap \hb^{-p}M)\mod V_{<\gamma}.
\]
Since $\hb^{-1}\CC[\hbm](V_\gamma\cap \hb^{-p}M)\subset V_{\gamma-1}$, it is enough to prove that
\[
\Big(\sum_{j\geq0}\hb^j(V_\gamma\cap \hb^{-p}M)\Big)\cap V_\gamma\subset (V_\gamma\cap \hb^{-p}M)+V_{<\gamma}.
\]
We argue by induction on the degree $\ell$ of an element $\sum_{j=0}^\ell \hb^jm_j$ of the left-hand term. For such an element, we have $\hb^\ell m_\ell\in V_{\gamma+\ell-1}$, hence $\hb m_\ell\in V_\gamma\cap \hb^{-p+1}M\subset V_\gamma\cap \hb^{-p}M$, so we can rewrite the sum as $\sum_{j=0}^{\ell-2} \hb^jm_j+\hb^{\ell-1}(m_{\ell-1}+\hb m_\ell)$, and we conclude by induction.
\end{proof}

\subsection{$V$-adapted trivializing lattices}\label{subsec:Birkhoff}
A \index{trivializing lattice}\emph{trivializing lattice} for $M$ is a free $\CC[v]$-lattice $M^o$ of $M[\hbm]$ (\ie $\CC[v]$-submodule of maximal rank) such that the bundle~$\wt M$ on $\PP^1$ obtained by gluing $M$ and $M^o$ is trivializable. Such lattices do exist (any basis of $M$ as a $\CC[\hb]$-module generates such a $\CC[v]$-module). Then $\Gamma(\PP^1,\wt M)=M^o\cap M$ is a finite-dimensional $\CC$-vector space of dimension $\rk M$ and we have
\begin{equation}\label{eq:MM'}
M=\bigoplus_{j\geq0}\hb^j(M^o\cap M),\quad M^o=\bigoplus_{j\geq0}v^j(M^o\cap M),\quad M[\hbm]=\bigoplus_{j\in\ZZ}\hb^j(M^o\cap M).
\end{equation}

We say that $M^o$ is a \index{trivializing lattice!$V$-adapted --}\emph{$V$-adapted trivializing lattice for $M$} if it moreover satisfies the following property. For every $\gamma\in\RR$,
\begin{equation}\label{eq:Vsol}
V_\gamma\cap M=\bigoplus_{j\geq0} \hb^j(V_{\gamma-j}\cap M^o\cap M)=\bigoplus_{j\geq0} v^{-j}(V_{\gamma-j}\cap M^o\cap M).
\end{equation}

\begin{proposition}\label{prop:Vadapted}
There exists a $V$-adapted trivializing lattice for $M$.
\end{proposition}

\begin{lemme}\label{lem:Vadapted}
There exists a finite family of finite-dimensional subspaces $(E_\beta)_{\beta\in\RR}$ such that, denoting by $[\cbbullet]$ the projection $V_\gamma\to\gr_\gamma^V(M[\hbm])$ ($\gamma\in\RR$),
\begin{enumerate}
\item\label{lem:Vadapted1}
for every $\beta\in\RR$, $E_\beta\subset V_\beta\cap M$ and $E_\beta\cap V_{<\beta}=0$,
\item\label{lem:Vadapted2}
for every $\gamma\in\RR$, $\gr_\gamma^V(M[\hbm])=\bigoplus_{i\in\ZZ}[\hb^iE_{\gamma-i}]$,
\item\label{lem:Vadapted3}
for every $\gamma\in\RR$, $V_\gamma\cap M=\bigoplus_{\substack{j\geq0\\\beta\leq \gamma-j}}\hb^jE_\beta$.
\end{enumerate}
\end{lemme}

\begin{proof}[Proof of Lemma \ref{lem:Vadapted}]
Let us set
\begin{align*}
m&=\min\{\beta\mid V_\beta\cap M\neq0\}\\
\mu&=\min\{\beta\mid M=\CC[\hb]\cdot(V_\beta\cap M)\}=\min\{\beta\mid M=F_\beta^\sHN M\}.
\end{align*}
We will construct inductively a family $E_\beta$ so that $E_\beta=0$ for $\beta\notin[m,\mu]$. Assume we have constructed $E_\beta$ for $\beta<\beta_o$ satisfying
\begin{enumerate}
\item[\eqref{lem:Vadapted1}$_{<\beta_o}$]
for every $\beta<\beta_o$, $E_\beta\subset V_\beta\cap M$ and $E_\beta\cap V_{<\beta}=0$,
\item[\eqref{lem:Vadapted2}$_{<\beta_o}$]
for every $\gamma<\beta_o$, the sum $\sum_{i\geq0}[\hb^iE_{\gamma-i}]\subset\gr_\gamma^V(M[\hbm])$ is direct,
\item[\eqref{lem:Vadapted3}$_{<\beta_o}$]
for every $\gamma<\beta_o$, $V_\gamma\cap M=\bigoplus_{\substack{j\geq0\\\beta\leq \gamma-j}}\hb^jE_\beta$.
\end{enumerate}
We will construct $E_{\beta_o}$ such that \eqref{lem:Vadapted1}$_{\leq\beta_o}$, \eqref{lem:Vadapted2}$_{\leq\beta_o}$ and \eqref{lem:Vadapted3}$_{\leq\beta_o}$ hold. We note that these properties obviously hold if $\beta_o=m$ if we set $E_m=V_m\cap M$. Let $\beta_o>m$. We first claim that the sum
\[\tag*{$(*)_{\beta_o}$}\label{eq:sumlambdao}
\sum_{\substack{j\geq0\\\beta<\beta_o-j}}\hb^jE_\beta+\hb \sum_{i\geq0}\hb^iE_{\beta_o-1-i},
\]
which is contained in $V_{\beta_o}\cap M$, is direct. The first component equals $V_{<\beta_o}\cap M$ and is a direct sum. The second sum (forgetting the factor $\hb$) is written $E_{\beta_o-1}+\hb E_{\beta_o-2}+\cdots$, hence is part of the direct sum decomposition of $V_{\beta_o-1}\cap M$, so is also direct. This remains true after multiplying by~$\hb$. Let us consider its image in $\gr_{\beta_o-1}^V(M[\hbm])$. Each $\hb^iE_{\beta_o-i-1}$ ($i\geq0$) injects in $\gr_{\beta_o-1}^V(M[\hbm])$, according to \eqref{lem:Vadapted1}$_{<\beta_o}$. Moreover, by \eqref{lem:Vadapted2}$_{<\beta_o}$ applied with $a=\beta_o-1$, the sum of the classes $[\hb^iE_{\beta_o-i-1}]$ ($i\geq0$) in $\gr_{\beta_o-1}^V(M[\hbm])$ is direct. It follows that $E_{\beta_o-1}\oplus \hb E_{\beta_o-2}\oplus\cdots$ maps isomorphically onto its image in $\gr_{\beta_o-1}^V(M[\hbm])$, and thus has intersection zero with $V_{<\beta_o-1}$. This concludes the proof of the claim. Moreover, we see that
\begin{align*}
V_{<\beta_o}\cap\ref{eq:sumlambdao}&=V_{<\beta_o}\cap\Big((V_{<\beta_o}\cap M)\oplus\bigoplus_{i\geq1}\hb^iE_{\beta_o-i}\Big)\\
&=(V_{<\beta_o}\cap M)\oplus V_{<\beta_o}\cap\Big(\bigoplus_{i\geq1}\hb^iE_{\beta_o-i}\Big)=V_{<\beta_o}\cap M.
\end{align*}

We now choose a supplementary subspace $E_{\beta_o}$ of \ref{eq:sumlambdao} in $V_{\beta_o}\cap M$. Then $E_{\beta_o}\cap\nobreak V_{<\beta_o}=E_{\beta_o}\cap V_{<\beta_o}\cap M=0$, so \eqref{lem:Vadapted1}$_{\beta_o}$ is satisfied. By construction, \eqref{lem:Vadapted3}$_{\beta_o}$ is also satisfied, and we have seen that \eqref{lem:Vadapted2}$_{\beta_o}$ also holds for $\gamma=\beta_o$.

For $M$ as defined above, let us set $E=\bigoplus_{\beta\leq \mu}E_\beta$. We find, by induction on $k\geq0$, that
\[
(V_\mu\cap M)+\hb (V_\mu\cap M)+\cdots+\hb^k(V_\mu\cap M)=E\oplus \hb E\oplus\cdots \oplus\hb^{k-1}E\oplus \Big(\hb^k\hspace*{-2mm}\bigoplus_{\substack{j\geq0\\\beta\leq \mu-j}}\hspace*{-2mm}\hb^jE_\beta\Big).
\]
Letting $k$ tend to $\infty$, we obtain that $M=\CC[\hb]\otimes_\CC E$, and thus $\dim E=\rk M$ and $E_\beta=0$ for $\beta>\mu$.
\end{proof}

\begin{proof}[Proof of Proposition \ref{prop:Vadapted}]
Let $E$ be as in the proof of Lemma \ref{lem:Vadapted}. Then $M=\CC[\hb]\cdot E=\bigoplus_{j\geq0}\hb^jE$. We set $M^o=\CC[v]\cdot E=\bigoplus_{j\geq0}v^jE\subset M[\hbm]$. Then $E=M^o\cap M$. Moreover, by \ref{lem:Vadapted}\eqref{lem:Vadapted3}, we have $V_\gamma\cap M^o\cap M=\bigoplus_{\beta\leq \gamma}E_\beta$, hence
\[
V_\gamma\cap M=\bigoplus_{j\geq0}\hb^j\Big(\bigoplus_{\beta\leq \gamma-j}E_\beta\Big)=\bigoplus_{j\geq0}\hb^j(V_{\gamma-j}\cap M^o\cap M).\qedhere
\]
\end{proof}

We can also develop $V_\gamma$ with respect to the decomposition \eqref{eq:Vsol}. We have $V_\gamma=\bigcup_{q\geq0}V_\gamma\cap \hb^{-q}M$ and $V_\gamma\cap \hb^{-q} M=\hb^{-q}(V_{\gamma+q}\cap M)$, so, by \eqref{eq:Vsol},
\begin{equation}\label{eq:Vbetaq}
V_\gamma\cap \hb^{-q} M=\bigoplus_{j\geq0} \hb^{j-q}(V_{\gamma+q-j}\cap M^o\cap M)=\bigoplus_{k\geq-q} \hb^k(V_{\gamma-k}\cap M^o\cap M),
\end{equation}
and therefore, taking the limit $q\to\infty$,
\begin{equation}\label{eq:Vbeta}
V_\gamma=\bigoplus_{k\in\ZZ} \hb^k(V_{\gamma-k}\cap M^o\cap M)=\bigoplus_{k\in\ZZ} v^k(V_{\gamma+k}\cap M^o\cap M).
\end{equation}

Let us now consider the irregular Hodge filtration.

\begin{lemme}\label{lem:HNVM'M}
Let $M^o$ be a $V$-adapted trivializing lattice. By means of the identification $M^o\cap\nobreak M\isom M/(\hb-1)M$, $F_\beta^\irr H$ is identified with $V_\beta\cap M^o\cap M$.
\end{lemme}

\begin{proof}
The decompositions \eqref{eq:MM'} and \eqref{eq:Vsol} show that the image of $V_\beta\cap M$ in $H=M/(\hb-1)M$, that is, $F_\beta^\irr H$, is equal to $\sum_{j\geq0}(V_{\beta-j}\cap M^o\cap M)$, which is nothing but $V_\beta\cap M^o\cap M$.
\end{proof}

\begin{remarque}
We also have
\[
F^\sHN_\beta M=\CC[\hb]\otimes_\CC(V_\beta\cap M^o\cap M).
\]
Indeed, by the previous identifications, we have
\begin{align*}
F_\beta^\sHN M&=\CC[\hb]\cdot(V_\beta\cap M)=\CC[\hb]\cdot \bigoplus_{j\geq0} \hb^j(V_{\beta-j}\cap M^o\cap M)\\
&=\bigoplus_{k\geq0} \hb^k(V_\beta\cap M^o\cap M)=\CC[\hb]\otimes_\CC(V_\beta\cap M^o\cap M),
\end{align*}
hence the assertion.
\end{remarque}

\Subsection{Application to the irregular Hodge filtration of some confluent hypergeometric systems}\label{subsec:hypergeom}

Let us consider a sequence of real numbers $0\leq\alpha_0\leq\alpha_1\leq\cdots\leq\alpha_{\mu-1}<1$ ($\mu\geq1$) and let us consider the confluent \index{hypergeometric}hypergeometric differential operator
\[
\prod_{k=0}^{\mu-1}(\tau\partial_\tau-\alpha_k)-\tau.
\]
It has a regular singularity at $\tau=0$ and an irregular singularity of slope $1/\mu$ at $\tau=\infty$, and no other singularity. We know that the associated meromorphic flat bundle $\cH(\alphag,0)$ on $\PP^1$ is irreducible (\cf\cite[Cor.\,3.2.1]{Katz90}), that its index of rigidity is equal to $2$ (\cf\cite[Th.\,3.7.1\,\&\,Th.\,3.7.3]{Katz90}), and that it is rigid (\cf\cite[Th.\,4.7\,\&\,Th.\,4.10]{B-E04}). Since the $\alpha_k$'s are real, $\cH(\alphag,0)$ is locally unitary at $\tau=0$. The computation of the formal monodromy at $\tau=\infty$ is not difficult and shows the local unitarity at $\tau=\infty$ (\cf below the computation of $B_\infty^\mathrm{diag}$ after ramification, which is enough). By Theorem \ref{th:rigidP1}, its minimal extension $\cH^{\min}(\alphag,0)$ at $\tau=0$ underlies a unique object $\cT^{\min}(\alphag,0)$ of $\IrrMHM(\PP^1_\tau)$, and it comes equipped with a unique (up to a shift) irregular Hodge filtration. We aim at determining the ranks of the irregular Hodge bundles. We denote by $\cT^{\min}(\alphag,0)$ the associated pure polarizable twistor $\cD$-module on $\PP^1_\tau$ and by $\cT(\alphag,0)$ the localized object in $\MTM^\intt(\PP^1_\tau,(*0))$.

The following result is due to A.\,Castaño\,Domínguez and C.\,Sevenheck \cite{CD-S17}, that they proved in a different way.\footnote{Added in proofs: For a general confluent hypergeometric differential system $\cH(\alphag,\betag)$ defined by the hypergeometric differential equation $\prod_{i=1}^n(\tau\partial_\tau-\alpha_i)-\tau\prod_{j=1}^m(\tau\partial_\tau-\beta_j)$ with $n>m$, $\alpha_i,\beta_j\in[0,1)$ and $\alpha_i\neq\beta_j$ for all $i,j$, we recently generalized Theorem \ref{th:hypergeom} by proving that the jumps of the irregular Hodge filtration $F^\irr_\bbullet$ occur at $p=\rho(i):=\mu\alpha_i-i+\#\{j\mid\beta_j<\alpha_i\}\in\RR$, $\mu=n-m$, and the rank of the jump at $p$ is $\#\{i\mid\rho(i)=p\}$.}

\begin{theoreme}\label{th:hypergeom}
For any $\beta\in\RR$, we have $\rk F^\irr_\beta\cH(\alphag,0)=\#\{k\mid-\alpha_k\leq(\beta-k)/\mu\}$.
\end{theoreme}

\skpt
\begin{remarques}\label{rem:hypergeom}
\begin{enumerate}
\item\label{rem:hypergeom1}
Recall that the irregular Hodge filtration is unique up to a shift by a real number, so the formula above is understood up to a shift of the filtration.
\item\label{rem:hypergeom2}
In \cite[Prop.\,4]{Mochizuki14b}, T.\,Mochizuki computes the $\cR$-module part of the localized object $\cT(\alphag,0)$ (the case of the matrix $\cK(r,1)$ with the notation of \loccit). In such a way, a direct computation of the rescaled module would be possible, leading possibly to a computation of the irregular Hodge filtration by using Definition \ref{def:irrHF}. We did not pursue in this direction.
\item\label{rem:hypergeom3}
In \cite[\S6.1]{F-G09}, Frenkel and Gross consider the case where \hbox{$\alpha_0\!=\!\cdots\!=\!\alpha_{\mu-1}\!=\!0$}. In such a case, the formula of Theorem \ref{th:hypergeom} shows that each graded piece of the irregular Hodge filtration has rank one and the jumps occur at $k/\mu$ ($k=0,\dots,\mu-1$) up to a global shift.
\end{enumerate}
\end{remarques}

\pagebreak[2]
\skpt
\begin{proof}[Proof of Theorem \ref{th:hypergeom}]
\oldsubsubsection*{Ramification}
We consider the pullback $\cH_\mu(\alphag,0)$ of $\cH(\alphag,0)$ by the ramification $\finite:v\mto\tau=v^\mu$. It corresponds to the differential operator
\[
\mu^{-\mu}\prod_{k=0}^{\mu-1}(v\partial_v-s_k)-v^\mu,\quad s_k:=\mu\alpha_k,
\]
and we have a basis $\omegag:=(\omega_0,\dots,\omega_{\mu-1})$ of the free $\CC[v,v^{-1}]$-module $\cH_\mu(\alphag,0)$ in which the connection is given by
\[
\mu^{-1}(v\partial_v+\sigma_k)\omega_k=v\omega_{k+1}\quad k=0,\dots,\mu-1,
\]
where we have set $\omega_\mu:=\omega_0$ and $\sigma_k=k-s_k$. In other words, we have
\[
v\partial_v\omegag=\omegag\cdot(vA_0-A_\infty),
\]
where\vspace*{-.5\baselineskip}
\[
A_\infty=\mathrm{diag}(\sigma_0,\dots,\sigma_{\mu-1})\quad\text{and}\quad
\arraycolsep=5pt
A_0=\mu\begin{pmatrix}
0&&&&1\\
1&0&&&0\\
0&1&0&&0\\
\vdots&\ddots&\ddots&\ddots&\vdots\\
0&\cdots&0&1&0
\end{pmatrix}.
\]
It is easy to check that $\cH_\mu(\alphag,0)$ is a meromorphic flat bundle on $\PP^1_v$ with a regular singularity at $v=0$, an irregular singularity of pure slope $1$ at $v=\infty$, and no other singularity. Moreover, the pullback $\cT_\mu(\alphag,0)$ of $\cT(\alphag,0)$ endows $\cH_\mu(\alphag,0)$ with the structure of an integrable mixed twistor $\cD$-module on $\PP^1_v$ localized at $v=0$, and its minimal extension $\cT_\mu^{\min}(\alphag,0)$ at $v=0$ is a polarized twistor $\cD$-module which is pure of weight $0$, with associated $\CC[v]\langle\partial_v\rangle$-module $\cH_\mu^{\min}(\alphag,0)$.

At this point, we do not know that $\cT_\mu^{\min}(\alphag,0)$ belongs to $\IrrMHM(\PP^1_v)$. On the other hand, if this property holds, then we obtain $\cT^{\min}(\alphag,0)$ as a direct summand of the pushforward $\finite_\dag^0\cT^{\min}_\mu(\alphag,0)$, so the irregular Hodge filtration of $\cH(\alphag,0)_{|\CC^*_\tau}$ can be computed as the invariant part of the pushfoward of that of $\cH_\mu(\alphag,0)_{|\CC^*_v}$, and conversely, that of $\cH_\mu(\alphag,0)_{|\CC^*_v}$ is the pullback of that of $\cH(\alphag,0)_{|\CC^*_\tau}$. It follows in particular that the ranks and jumping indices are the same. We are thus reduced to proving that $\cT_\mu^{\min}(\alphag,0)$ belongs to $\IrrMHM(\PP^1_v)$ and to computing the ranks of its irregular Hodge bundles.

\oldsubsubsection*{Computation for $\cH_\mu(\alphag,0)$}

One checks that $\cH_\mu^{\min}(\alphag,0)$ is the Laplace transform of a regular holonomic $\CC[t]\langle\partial_t\rangle$-module $N$, \ie $\cH_\mu^{\min}(\alphag,0)=\Fou N$. We denote by $\ccN$ its minimal extension at $t=\infty$. Similarly, $\cT^{\min}_\mu(\alphag,0)$ is the Fourier-Laplace transform of an integrable polarizable regular twistor $\cD_{\PP^1_t}$-module $\cT$ which is pure of weight $0$ and with associated holonomic $\cD_{\PP^1_t}$-module $\ccN$ (\cf\cite{Bibi05b}). By the rigidity theorem \cite[Th.\,6.2]{H-S08}, it is associated to a variation of polarized complex Hodge structure of weight $0$. By the argument before Proposition \ref{prop:imts} we conclude that $\cT_\mu^{\min}(\alphag,0)$ belongs to $\IrrMHM(\PP^1_v)$.

\begin{lemme}\label{lem:Firrflat}
The irregular Hodge filtration $F_\cbbullet^\irr\cH_\mu^{\min}(\alphag,0)$ satisfies the following property:\vspace*{-.5\baselineskip}
\begin{align*}
&\forall\beta\in\RR,\quad F_\beta^\irr\cH_\mu^{\min}(\alphag,0)_{|\CC^*_v}\text{ is $\cO_{\CC^*_v}$-locally free, and}\\
&i^*_{v=1}F_\beta^\irr\cH_\mu^{\min}(\alphag,0)=F_\beta^\irr\bigl[i^*_{v=1}\cH_\mu^{\min}(\alphag,0)\bigr], 
\end{align*}
where the latter term is given by Proposition \ref{prop:imts}.
\end{lemme}

Let us take this lemma for granted. Recall also that the irregular Hodge filtration on the fiber $\cH_\mu^{\min}(\alphag,0)_{|v=1}$ can be computed in terms of the Deligne filtration $F^\Del_\bbullet$, as follows from \cite[Prop.\,3.1.2]{E-S-Y13}, and then we can apply \cite[Th.\,6.1(2)]{Bibi08} to compute it in terms of the Brieskorn lattice $G_0(\ccN)$ of the Hodge filtration of $\ccN$. If we set $u=1/v$, then $G_0(\ccN)$ is a free $\CC[u]$-module such that $\CC[u,u^{-1}]\otimes_{\CC[u]}G_0(\ccN)=\cH_\mu(\alphag,0)$.

The latter result gives a formula similar to that of Definition \ref{def:irregHF}, namely,
\[
\dim F^\irr_\beta(\cH_\mu^{\min}(\alphag,0)_{|v=1})=\dim\big[(V_\beta\cap G_0(\ccN))/(V_\beta\cap G_0(\ccN))\cap(u-1)G_0(\ccN)\big].
\]

On the other hand, let $G_0(\omegag)$ denote the $\CC[u]$-submodule of $\cH_\mu(\alphag,0)$ generated by $\omega_0,\dots,\omega_{\mu-1}$. Set $\omega'_k=v^{[\sigma_k]}\omega_k$ and $G_p(\omegag)=u^{-p}G_0(\omegag)$. Then \cite[(3.9)]{D-S02b} applies in the present context, and gives, for $\beta\in\RR$ and $p\in\ZZ$ such that $\beta-p\in[0,1)$,
\[
V_\beta\cap G_0(\omegag)=u^p\bigl[V_{\beta-p}\cap G_p(\omegag)\bigr]=u^p\biggl[\sum_{k\mid\sigma_k=\beta}\CC\cdot\omega'_k\biggr]+V_{<\beta}\cap G_0(\omegag)+V_\beta\cap G_{-1}(\omegag).
\]
We deduce that
\begin{align*}
\dim\big[(V_\beta\cap G_0(\omegag))/(V_\beta\cap G_0(\omegag))\cap(u-1)G_0(\omegag)\big]&=\#\{k\mid\sigma_k\leq\beta\}\\
&=\#\{k\mid-\alpha_k\leq(\beta-k)/\mu\}.
\end{align*}

\oldsubsubsection*{Comparison of Brieskorn lattices}
The statement of Theorem \ref{th:hypergeom} would now follow from the identification
\[
G_p(\ccN)=G_0(\omegag)
\]
for some $p$, that we now prove (recall that the Hodge filtration on $\ccN$, hence its Brieskorn lattice, is defined up to a shift by an integer). Let us denote by $\wh G_0$ the formalized module $\CC\lcr u\rcr\otimes_{\CC[u]}G_0$. Then, according to \cite[Prop.\,1.2]{Malgrange95}, it is enough to prove the equality $\wh G_0(\ccN)=\wh G_0(\omegag)$.

Since the connection takes the form
\[
u^2\partial_u\omegag=\omegag\cdot(-A_0+uA_\infty),
\]
there is a constant base change after which the matrix is $(-B_0+uB_\infty)$ with $B_0$ diagonal with distinct eigenvalues. Then, one can find a formal base change after which the matrix is $(-B_0+uB_\infty^\mathrm{diag})$, where $B_\infty^\mathrm{diag}$ is the diagonal part of $B_\infty$ (\cf\eg\cite[\S VI.3.f]{Bibi00}). It is not difficult to compute that $B_\infty^\mathrm{diag}=c\id$ for some constant $c\in\RR$ that we do not need to make explicit. The conclusion is that $\wh G_0(\omegag)$ splits into rank-one $\CC\lcr u\rcr$-modules with connection, which have pairwise distinct exponential part and \emph{the same} rank-one regular part.

On the other hand, if $F_\cbbullet\ccN$ is the Hodge filtration of $\ccN$, then we choose a generating index $p_o$ of $F_\cbbullet\ccN$, \ie such that, for all $\ell\geq0$, $F_{p_o+\ell}\ccN=\sum_{j=0}^\ell\partial_t^jF_{p_o}\ccN$, and we have (\cf\cite[\S1.d]{Bibi05} and \cite[App.]{S-Y14})
\[
G_0(\ccN)=u^{p_o}\sum_{j\geq0}u^j\wh\loc(F_{p_o}\ccN).
\]
Moreover, the microlocalized lattice $(F_{p_o}\ccN)^\mu$ (\cf\eg\cite[\S V.3]{Bibi00}) is supported at the singularities $t_i$ of $\ccN$ and is written as $(F_{p_o}\ccN)^\mu=\bigoplus_{i=0}^{\mu-1}(F_{p_o}\ccN)^\mu_{t_i}$. These are rank-one $\CC\lcr u\rcr$-modules with a connection having a regular singularity, hence a logarithmic singularity, and $\wh G_0(\ccN)$ can be identified with $\bigoplus_{i=0}^{\mu-1}\ccE^{t_i/u}\otimes(F_{p_o}\ccN)^\mu_{t_i}$ as a $\CC\lcr u\rcr$-module with connection (\cf\cite[Prop.\,V.3.4\,\&\,Prop.\,V.3.6]{Bibi00}). Our assertion is now reduced to proving that these modules have \emph{the same} rank-one regular part.

By construction, $\cH_\mu(\alphag,0)$ has an action of $\ZZ/\mu\ZZ$ over $v\mto\zeta v$ ($\zeta^\mu=1$) and its wild harmonic metric is invariant by this action, since it is the pullback of the harmonic metric for $\cH(\alphag,0)$. This action therefore lifts to $\cT_\mu^{\min}(\alphag,0)$. We also have $\partial_v\mto\zeta^{-1}\partial_v$, so we get a similar action on the inverse Fourier-Laplace transform of $\cT_\mu^{\min}(\alphag,0)$ induced by $t\mto\zeta^{-1}t$. In particular, the Hodge filtration is equivariant with respect to this action. Since this action permutes the $t_i$'s, we conclude that the germs $(F^{p_o}\ccN)^\mu_{t_i}$ are exchanged, hence are isomorphic, as wanted.
\end{proof}

\begin{proof}[Sketch of proof of Lemma \ref{lem:Firrflat}]
We work algebraically with respect to $v$. Let $(\ccN,F_\bbullet\ccN)$ be a filtered $\cD_{\PP^1_t}$-module underlying a complex mixed Hodge module, and let $R_F\ccN$ be the associated Rees module. We consider the exponentially twisted module $\cM=\cE^{tv/\hb}[v^{-1}]\otimes R_F\ccN$ as the localization along $v=0$ of the $\cR$-module part of an object of $\IrrMHM(\PP_t\times\PP^1_v)$ whose direct image to $\PP^1_v$ gives rise to the localized Laplace transform $\Fou(R_F\ccN)[v^{-1}]$. The rescaling corresponds to adding the variable~$\tau$ and considering $\taucM=\cE^{\tau tv/\hb}[v^{-1}]\otimes R_F\ccN$. The $\tauV$-filtration is computed as in \cite[Prop.4.1]{Bibi05b}, and we find
\[
\tauV_\beta(\cE^{\tau tv/\hb}[v^{-1}]\otimes R_F\ccN)=\bigl(\cE^{\tau t/\hb}\otimes R_F\ccN\bigr)\otimes_\CC\CC[v,v^{-1}].
\]
The same property holds thus after pushforward onto $\PP^1_v$, and then a similar property holds for $\tauV_\beta/(\tau-\hb)\tauV_\beta$, showing eventually that
\[
F_\beta^\irr(\cE^{tv/\hb}[v^{-1}]\otimes R_F\ccN)=F_\beta^\irr(\cE^{t/\hb}\otimes R_F\ccN)\otimes_\CC\CC[v,v^{-1}],
\]
hence the assertion.
\end{proof}

\subsection{Irregular Hodge filtration and tensor product}
We keep the notation of Section \ref{subsec:HNconnection}. The main result of this section is following proposition.

\begin{proposition}[Tensor product formula for the irregular Hodge filtration]\label{prop:tensorirregHodge}\index{irregular Hodge filtration!tensor product formula for the --}
Let $M_1,M_2$ be two $\CC[\hb]\langle\hb^2\partial_\hb\rangle$-modules as in Section \ref{subsec:HNconnection}. Then we have, on $H=H_1\otimes_\CC H_2$ (see Section \ref{subsec:HNconnection}):
\[
F^\irr_\beta H=\sum_{\beta_1+\beta_2=\beta}F^\irr_{\beta_1}H_1\otimes_\CC F^\irr_{\beta_2}H_2.
\]
\end{proposition}

As a consequence, we find, by taking bases of $H_i$ adapted to $F^\irr_\bbullet H_i$ ($i=1,2$):
\[
\gr^{F^\irr}_\beta H\simeq\bigoplus_{\beta_1+\beta_2=\beta}\gr^{F^\irr}_{\beta_1}H_1\otimes_\CC \gr^{F^\irr}_{\beta_2}H_2.
\]
Note that the equality of dimensions was obtained in \cite[Prop.\,3.7]{Bibi96bb} with a supplementary assumption, and was inspired from \cite{Varchenko83} (\cf also \cite[Ex.\,III.2.9]{Bibi00}).

\begin{remarque}[Tensor product formula for the Harder-Narasimhan filtration]\label{rem:tensorHN}\index{Harder-Narasimhan filtration!tensor product formula for the --}
The same proof also gives
\[
F^\sHN_\beta M=\sum_{\beta_1+\beta_2=\beta}F^\sHN_{\beta_1}M_1\otimes_{\CC[\hb]} F^\sHN_{\beta_2}M_2.
\]
\end{remarque}

\begin{lemme}
Assume that $M^o_i$ is a trivializing lattice for $M_i$ ($i=1,2$). Then $M^o:=M^o_1\otimes_{\CC[v]}M^o_2$ is a a trivializing lattice for $M:=M_1\otimes_{\CC[\hb]}M_2$.
\end{lemme}

\begin{proof}
By definition \eqref{eq:MM'} of $M^o_1,M^o_2$, we have
\begin{align*}
M^o&=(\CC[v]\otimes_\CC(M^o_1\cap M_1))\otimes_{\CC[v]}(\CC[v]\otimes_\CC(M^o_2\cap M_2))\\
&=\CC[v]\otimes_\CC \bigl[(M^o_1\cap M_1)\otimes_\CC(M^o_2\cap M_2)\bigl],
\end{align*}
and we have similarly
\[
M=\CC[\hb]\otimes_\CC \bigl[(M^o_1\cap M_1)\otimes_\CC(M^o_2\cap M_2)\bigr].
\]
Then $M^o\cap M=(M^o_1\cap M_1)\otimes_\CC(M^o_2\cap M_2)$
\end{proof}

We now consider the behaviour of $V$-adapted trivializing lattices by tensor product.

\begin{proposition}\label{prop:Vadaptedtensor}
Assume that $M^o_1$ (\resp $M^o_2$) is a $V$-adapted trivializing lattice for~$M_1$ (\resp $M_2$). Then $M^o:=M^o_1\otimes_{\CC[v]}M^o_2$ is a $V$-adapted trivializing lattice for $M:=M_1\otimes_{\CC[\hb]}M_2$.
\end{proposition}

\begin{proof}
For $\gamma_i\in\RR$ we set $V_{\gamma_i}^{(i)}:=V_{\gamma_i}(M_i[\hbm])$ ($i=1,2$). For $\gamma\in\RR$, let us set
\begin{align*}
U_\gamma:&=\sum_{\gamma_1+\gamma_2=\gamma}V_{\gamma_1}(M_1[\hbm])\otimes_{\CC[v]}V_{\gamma_2}(M_2[\hbm])\\
&=\bigoplus_{k\in\ZZ}v^k\biggl(\sum_{\substack{\gamma_1,\gamma_2\in\RR\\\gamma_1+\gamma_2=\gamma+k}}(V_{\gamma_1}^{(1)}\cap M^o_1\cap M_1)\otimes(V_{\gamma_2}^{(2)}\cap M^o_2\cap M_2)\biggr).\end{align*}

\begin{claim}
$U_\gamma=V_\gamma(M[\hbm])$.
\end{claim}

Note that the coefficient of $v^k$ in $U_\gamma$ is $U_{\gamma+k}\cap M^o\cap M$. If this claim is proved, we conclude both that $M^o$ is a $V$-adapted trivializing lattice and that
\[
V_\gamma\cap M^o\cap M=\sum_{\substack{\gamma_1,\gamma_2\in\RR\\\gamma_1+\gamma_2=\gamma}}(V_{\gamma_1}^{(1)}\cap M^o_1\cap M_1)\otimes(V_{\gamma_2}^{(2)}\cap M^o_2\cap M_2),
\]
which ends the proof of Proposition \ref{prop:Vadaptedtensor}.
\end{proof}

\begin{proof}[Proof of Proposition \ref{prop:tensorirregHodge}]
Lemma \ref{lem:HNVM'M} together with the previous formula give the desired assertion.
\end{proof}

\begin{proof}[Proof of the claim]
For $\gamma$ fixed, the filtration \hbox{$U_{\gamma+k}\cap M^o\cap M$}, which is equal to the coefficient of $v^k$ in the above expression for $U_\gamma$, is an exhaustive increasing filtration of the finite-dimensional vector space $M^o\cap M$, \ie \hbox{$U_{\gamma+k}\cap M^o\cap M=0$} for $k\ll0$ and $U_{\gamma+k}\cap M^o\cap M=M^o\cap M$ for $k\gg0$. It follows that $U_\gamma$ has finite type over $\CC[v]$, and since it is contained in $M[\hbm]$, it is $\CC[v]$-free. Moreover, it is easily checked that it is stable by $v\partial_v$ and that, if $N_i\geq0$ is such that $(v\partial_v+\gamma_i)^{N_i}V_{\gamma_i}^{(i)}\subset V_{<\gamma_i}^{(i)}$ for every $\gamma_i\in\RR$ ($i=1,2$), then for every $\gamma\in\RR$ one has $(v\partial_v+\gamma)^{N_1+N_2}U_\gamma\subset U_{<\gamma}$.
\end{proof}

\subsection{Irregular Hodge filtration and duality}
Let $M$ be as in Section \ref{subsec:HNconnection}. Then $M^\vee:=\Hom_{\CC[\hb]}(M,\CC[\hb])$ is a free $\CC[\hb]$-module endowed with a natural connection with a pole of order two at the origin and a regular singularity at infinity. We have a natural isomorphism $M^\vee/(\hb-1)M^\vee\simeq \Hom_\CC(H,\CC)=:H^\vee$.

\begin{proposition}[Duality for the irregular Hodge filtration]\label{prop:dualirregHodge}\index{irregular Hodge filtration!duality for the --}
With the previous assumption, we have
\[
F^\irr_\beta(H^\vee)=(F^\irr_{<-\beta}H)^\perp,
\]
where the orthogonality is taken with respect with the tautological pairing \hbox{$H\!\otimes_\CC\!H^\vee\!\to\!\CC$}.
\end{proposition}

As a consequence, we find an isomorphism
\[
\gr_\beta^{F^\irr}(H^\vee)\simeq(\gr_{-\beta}^{F^\irr}H)^\vee.
\]

\begin{proof}
Let $M^o$ be a $V$-adapted trivializing lattice for $M$. Then the $\CC$-vector space $(M^o\cap M)^\vee:=\Hom_\CC((M^o\cap M),\CC)$ generates $M^\vee$ as a $\CC[\hb]$-module, and we can set $(M^\vee)^o:=\bigoplus_{j\geq0}v^j(M^o\cap M)^\vee$ to get a trivializing lattice for $M^\vee$. We thus have
\[
M^\vee=\bigoplus_{j\geq0}\hb^j(M^o\cap M)^\vee,\quad (M^\vee)^o=\bigoplus_{j\geq0}v^j(M^o\cap M)^\vee,\quad M^\vee[\hbm]=\bigoplus_{j\in\ZZ}\hb^j(M^o\cap M)^\vee,
\]
with
\[
M^\vee[\hbm]=M^\vee\otimes_{\CC[\hb]}\CC[\hb,\hbm]=\Hom_{\CC[\hb,\hbm]}(M[\hbm],\CC[\hb,\hbm])=:M[\hbm]^\vee.
\]
Recall (\cf\cite[Lem.\,3.2]{Bibi96bb}) that we have
\begin{align*}
V_\gamma(M^\vee[\hbm])&=\Hom_{\CC[v]}(V_{<-\gamma+1}M[\hbm],\CC[v])\\
&=\{\varphi\in M[\hbm]^\vee\mid\varphi(V_{<-\gamma+1}M[\hbm])\subset\CC[v]\}.
\end{align*}
We will prove that
$(M^o)^\vee$ is a $V$-adapted trivialization lattice for $M^\vee$ and that
\[
V_{\gamma-k}(M^\vee[\hbm])\cap(M^o)^\vee\cap M^\vee=(V_{<-\gamma+k}\cap M^o\cap M)^\perp
\]
by showing the decomposition
\[
V_\gamma(M^\vee[\hbm])=\bigoplus_k\hb^k(V_{<-\gamma+k}\cap M^o\cap M)^\perp.
\]
Consider the decomposition of $M^\vee[\hbm]$ obtained above. Let $\varphi\!=\!\sum_{k=j}^{\ell}\hb^k\varphi_k$ be an element of $V_\gamma(M^\vee[\hbm])$ (with $\varphi_k\!\in\!(M^o\cap M)^\vee$). It satisfies \hbox{$\varphi(V_{<-\gamma+1}M[\hbm])\!\subset\!\CC[v]$}, and one checks that $\varphi_\ell\in(V_{<-\gamma+\ell}\cap M^o\cap M)^\perp$. By decreasing induction one obtains that $\varphi_k\in(V_{<-\gamma+k}\cap M^o\cap M)^\perp$ for all $k=j,\dots,\ell$, as was to be proved.
\end{proof}

\begin{remarque}[Duality of the spectrum]\index{spectrum}
The formula of Proposition \ref{prop:dualirregHodge} is reminiscent of \cite[\S3]{Bibi96bb} (\cf also \cite[Prop.\,III.2.7]{Bibi00}) and occurs in many places in Singularity theory under the name of ``duality of exponents''.
\end{remarque}

\subsection{Tannakian structure}\label{subsec:irregHNC}
Would the category consisting of free $\CC[\hb]$-modules of finite rank $M$ equipped with a connection $\nabla$ having a pole of order $\leq2$ at the origin, a regular singularity at infinity, and no other pole, be abelian, it would form a neutral Tannakian category with respect to tensor product and fibre functor given by the fiber at $\hb=1$, that is, $M\mto H=M/(\hb-1)M$. The identity object $\bun$ would be $(\CC[\hb],\rd)$, and the criterion of \cite[Prop.\,1.20]{D-M82} would obviously be fulfilled. Moreover, it would be endowed with a duality functor.

The previous results show that the functor $M\mto(M,F^\sHN_\bbullet M)$ into the corresponding filtered category would be compatible with the Tannakian structures, if both would be abelian. Similarly, the functor $M\mto(H,F^\irr_\bbullet H)$ would be compatible with the Tannakian structure, if the latter were abelian.

In order to remedy to the defect of abelianity at the source category, we can replace it with the category $\iMTS^\intt(\CC)$ (\cf Proposition \ref{prop:MTSTannakienne} and Section \ref{subsec:itriplesint}). Taking into account the weight filtration, the functors above send a $W$-filtered triple $((M',M'',\iC),W_\bbullet)$ to the bi-filtered module $(M'',F^\sHN_\bbullet M,W_\bbullet M)$ or the bi-filtered vector space $(H,F^\irr_\cbbullet H,W_\bbullet H)$. By Propositions \ref{prop:tensorirregHodge} and \ref{prop:dualirregHodge}, they are compatible with tensor product and with taking dual.

The categories at the target are not abelian however. A replacement of the preservation of the Tannakian structures would be that morphisms at the source are sent to strictly bi-filtered morphism at the target. Unfortunately, \emph{we cannot assert} that morphisms in $\iMTS^\intt(\CC)$ give rise to \emph{strictly bi-filtered morphisms} between objects $(H,F^\irr_\cbbullet,W_\bbullet)$. This drawback will be overcome by restricting the category $\iMTS^\intt(\CC)$ and considering the category $\IrrMHS$ of \emph{irregular mixed Hodge structures} defined \hbox{below}.

\section{Irregular mixed Hodge structures}\label{sec:irregHS}

\subsection{Rescaling an integrable $\cR$-triple}\label{subsec:rescRtriple}
We review the construction in the simpler algebraic setting, which is enough when $\dim X=0$. We consider the affine torus $\GG_{\rmm,\tau}$ with coordinate ring $\CC[\tau,\taum]$ and the morphism $\mu:\GG_{\rmm,\tau}\times\Afu_\hb\to\Afu_\hb$ defined by $\mu(\tau,\hb)=\hb\taum$. Equivalently, we consider the injective morphism of rings $\mu^*:\CC[\hb]\to\CC[\tau,\taum,\hb]$ defined by $(\mu^*f)(\hb,\tau)=f(\hb\taum)$, making $\CC[\tau,\taum,\hb]$ a $\CC[\hb]$-module.

\subsubsection*{Rescaling a $\CC[\hb]\langle\hb^2\partial_\hb\rangle$-module}
\index{rescaling!of a $\CC[\hb]\langle\hb^2\partial_\hb\rangle$-module}Given a $\CC[\hb]\langle\hb^2\partial_\hb\rangle$-module $M$, we defined the associated rescaled module \index{$Mtau$@$\tauM$}$\tauM$ as the pullback $\mu^+M$ by $\mu$. In other words, $\tauM=\CC[\tau,\taum,\hb]\otimes_{\CC[\hb]}M$ as a $\CC[\tau,\taum,\hb]$-module, and the integrable connection on $\tauM$ is the pullback by $\mu$ of that of $M$. The analytic version $\tauM$ is equal to $\mu^*\cM$ as an $\cO_{\CC_\tau\times\CC_\hb}(*\taucX_0)$-module and is equipped with the pullback connection. We still denote by $\taucM^\circ$ the restriction of $\taucM$ to $\CC_\tau\times\CC^*_\hb$.

More explicitly, $\tauM=\CC[\tau,\taum]\otimes_\CC M$ as a $\CC[\tau,\taum]$-module, and we obtain the $\CC[\tau,\taum,\hb]\langle\partiall_\tau,\hb^2\partial_\hb\rangle$-module structure by the following formulas:
\begin{equation}\label{eq:partiall0}
\begin{aligned}
\hb(1\otimes m)=\hb\otimes m&=\tau\otimes\hb m=\tau(1\otimes\hb m),\\
\partiall_\tau(1\otimes m)&=-1\otimes\hb^2\partial_\hb m,\\
\hb^2\partial_\hb(1\otimes m)&=\tau(1\otimes\hb^2\partial_\hb m)=\tau\otimes\hb^2\partial_\hb m.
\end{aligned}
\end{equation}
Then Proposition \ref{prop:Xpoint} shows that, after analytification, the corresponding object $\taucM$ is graded well-rescalable.

\subsubsection*{Rescaling the pairing}
We denote by $\thetacM^\nabla$ the following local system on $\CC^*_\tau\times\CC^*_\hb$:
\[
\ker\big[\nabla:\thetacM_{|\CC^*_\tau\times\CC^*_\hb}\to\Omega^1_{\CC^*_\tau\times\CC^*_\hb}\otimes\thetacM_{|\CC^*_\tau\times\CC^*_\hb}\big].
\]
We then have $\thetacM^\nabla\simeq\mu^{-1}M^\nabla$. Let us note that $\iota\circ\mu=\mu\circ\iota$. A $\iota$-sesquilinear pairing
\[
\iC^\nabla:M^{\prime\nabla}\otimes_\CC\iota^{-1}\ov{M^{\prime\prime\nabla}}\to\CC_{\CC_\hb^*}
\]
can thus be pulled back as a sesquilinear pairing
\[
\thetaiC^\nabla:=\mu^{-1}\iC^\nabla:\thetacM^{\prime\nabla}\otimes_\CC\iota^{-1}\,\ov{\thetacM^{\prime\prime\nabla}}\to\CC_{\CC^*_\tau\times\CC_\hb^*}.
\]
By tensoring with analytic functions, it defines a sesquilinear pairing
\[
\thetaiC:\thetacM^{\prime\circ}\otimes_\CC\iota^{-1}\,\ov{\thetacM^{\prime\prime\circ}}\to\cC^\infty_{\CC^*_\tau\times\CC_\hb^*}.
\]
Let us now notice that, on $\CC_\tau\times\CC_\hb^*$, $\taucM^\circ:=\taucM_{|\CC_\tau\times\CC_\hb^*}$ is an $\cO_{\CC_\tau\times\CC^*_\hb}[1/\tau]$-flat bundle with a meromorphic connection having a regular singularity, whose restriction to $\CC_\tau^*\times\CC_\hb^*$ is $\thetacM^\circ$. In other words, it is equal to the Deligne meromorphic extension of~$\thetacM^\circ$ along $\tau=0$. It follows (by expressing $\thetaiC$ on local multivalued flat sections) that $\thetaiC$ extends in a unique way as a sesquilinear pairing
\[
\tauiC:\taucM^{\prime\circ}\otimes_\CC\iota^{-1}\,\ov{\taucM^{\prime\prime\circ}}\to\Db^\rmod_{\CC_\tau\times\CC_\hb^*/\CC_\hb^*},
\]
where the superscript ``mod'' means ``moderate growth along $\{\tau=0\}$''.

\subsubsection*{Rescaling the integrable triple}
\index{rescaling!of an integrable triple}We can now conclude that the rescaling of an object $(M',M'',\iC^\nabla)$ of $\RdiTriples(\pt)$ is defined as the object $(\tauM',\tauM'',\tauiC)$ of $\RdiTriples(\Afu_\tau,(*0))$. The rescaling of a morphism $\lambda=(\lambda',\lambda'')$ is defined as $(\taulambda',\taulambda'')$, so that we get the (composed) rescaling functor
\[
\resc:\RdiTriples(\pt)\mto \RgscTriples(\Afu_\tau)\mto \RgscTriples(\CC_\tau),
\]
where the second functor is the analytification. This functor extends in a natural way to the category of objects with finite filtration
\[
\resc:\WRdiTriples(\pt)\mto\WRgscTriples(\CC_\tau).
\]
We will also consider the composition of the functor with the natural functor (\cf\eqref{eq:functorresc})
\[
\WRgscTriples(\CC_\tau)\mto\WRdiTriples(\CC_\tau,(*0)).
\]

\subsection{Irregular mixed Hodge structures}\label{subsec:irregHS}

Recall that the ``stupid'' localization functor \eqref{eq:stupidloctau} (with $X=\pt$) identifies the category $\iMTM^\intt(\CC_\tau,[*0])$ with the full subcategory $\iMTM^\intt(\CC_\tau,(*0))$ of the category $\WRdiTriples(\CC_\tau,(*0))$.

\skpt
\begin{definition}\label{def:IrrMHS}
\begin{enumerate}
\item\label{def:IrrMHS1}
The category \index{mixed Hodge structure!irregular --}\index{$IrrMHSC$@$\IrrMHS(\CC)$}$\IrrMHS(\CC)$ is the full subcategory of $\WRdiTriples(\pt)$ whose rescaled objects belong to $\iMTM^\intt(\CC_\tau,(*0))$. In particular, $\IrrMHS(\CC)$ is a full subcategory of $\iMTS^\intt(\CC)$.

\item\label{def:IrrMHS2}
For a subfield $\kk$ of $\RR$, the category \index{$IrrMHSk$@$\IrrMHS(\kk)$}$\IrrMHS(\kk)$ is the full subcategory of $\iMTS^\intt_\good(\kk)$ whose objects belong to $\IrrMHS(\CC)$. In other words, the $\kk$-rationality condition only depends on the underlying object in $\iMTS^\intt(\CC)$.
\end{enumerate}
\end{definition}

\begin{remarque}
For an object $(M',M'',\iC^\nabla)$ of $\iMTS^\intt(\CC)$, \ie for an integrable mixed twistor structure, the condition that the rescaled object belongs to $\iMTM^\intt(\CC^*_\tau)$ implies that $\thetaiC^\nabla$ is nondegenerate, giving rise to a vector bundle on $\PP^1_\hb\times\CC^*_\tau$ whose graded piece $\gr_\ell^W$ is the pullback of a vector bundle on $\CC^*_\tau$ twisted by $\cO_{\PP^1_\hb}(\ell)$. Recall that this property does not however characterize $\iMTM^\intt(\CC_\tau,[*0])$: an admissibility condition at $\tau=0$ is added for the latter category.
\end{remarque}

\begin{lemme}\label{lem:realeigenvalues}
Let $(M',M'',\iC^\nabla)$ be an object of $\IrrMHS(\CC)$. Then the eigenvalues of the monodromy of $M^{\prime\nabla},M^{\prime\prime\nabla}$ have absolute value equal to one (equivalently, the jumping indices of the Deligne lattices at infinity are real).
\end{lemme}

\begin{proof}
It follows from \cite[Lem.\,7.3.7]{Bibi01c} applied to each graded piece of the $W$\nobreakdash-filtration that, for any object of $\MTM^\intt(\CC_\tau,[*0])$, the eigenvalues of the monodromy of the corresponding local systems (prime and double prime) around $\tau=0$ have absolute value equal to one (this is the \emph{local unitarity} property in \loccit). For a local system of the form $\thetacM^\nabla=\mu^{-1}M^\nabla$, this is equivalent to saying that the eigenvalues of the monodromy of $M^\nabla$ have absolute value equal to one, which is the desired assertion.
\end{proof}

\skpt
\begin{theoreme}\label{th:IrrMHS}
\begin{enumerate}
\item\label{th:IrrMHS1}
The subcategories $\IrrMHM(\pt)$ and $\IrrMHS(\CC)$ coincide, as well as the subcategories $\IrrMHM(\pt,\kk)$ and $\IrrMHS(\kk)$.
\item\label{th:IrrMHS2}
The category $\IrrMHS(\CC)$ is abelian and naturally equipped with the structure of a neutral Tannakian category, with fibre functor $(M',M'',\iC^\nabla)\mto H:=M''/(1-\hb)M''$.
\item\label{th:IrrMHS3}
There is a fully faithful functor $\MHS(\CC)\mto\IrrMHS(\CC)$ which identifies $\MHS(\CC)$ to a full subcategory of $\IrrMHS(\CC)$.
\item\label{th:IrrMHS4}
There is a natural functor \index{irregular Hodge filtration}\emph{``irregular Hodge filtration''} from $\IrrMHS(\CC)$ to the category of bi-filtered vector spaces $(H,F^\irr_\cbbullet,W_\bbullet)$ (where $F^\irr_\cbbullet$ is a filtration indexed by~$\RR$ and jumps at most at $A+\ZZ$ for some finite $A\subset[0,1)$) which is compatible with tensor product and with taking dual. Any morphism in $\IrrMHS(\CC)$ gives rise to a strictly bi-filtered morphism between the corresponding bi-filtered vector spaces.
\end{enumerate}
\end{theoreme}

\begin{proof}
Let us first notice that, if \ref{th:IrrMHS}\eqref{th:IrrMHS1} is proved, then Theorems \ref{th:mtmgresc} and \ref{th:mainRF}, proved in Chapter \ref{part:2} can be applied in the special case $X=\pt$, and give \ref{th:IrrMHS}\eqref{th:IrrMHS3}, as well as the functor ``irregular Hodge filtration'' of \ref{th:IrrMHS}\eqref{th:IrrMHS4} together with the bi-strictness of morphisms.

\ref{th:IrrMHS}\eqref{th:IrrMHS1} follows from Proposition \ref{prop:Xpoint} which ensures that $\tauM$ is graded well-rescalable.

For \ref{th:IrrMHS}\eqref{th:IrrMHS2}, the abelianity of $\IrrMHS\!=\!\IrrMHM(\pt)$ is a particular case of Proposition \ref{prop:MTMresc}\eqref{prop:MTMresc3}. Let us first define the tensor product and the duality functor. Let $M,M_1,M_2$ be free $\CC[\hb]$-modules of finite rank endowed with a connection $\nabla$ with a pole of order two at the origin, a regular singularity at infinity, and no other pole.

By using Lemma \ref{lem:dualtauM}, it is easy to check that the pullback by $\mu$ of the dual of an object of $\IrrMHS(\CC)$, \resp of the tensor product of objects of $\IrrMHS(\CC)$, is an object of $\MTM^\intt(\CC_\tau,(*0))$. In order to show that both operations (duality and tensor product) preserve $\IrrMHS(\CC)$, it remains to prove that the admissibility property is preserved. This is a particular case of Lemmas 9.1.12 and 9.1.14 in \cite{Mochizuki11}.

Last, for \ref{th:IrrMHS}\eqref{th:IrrMHS4}, It remains to prove the compatibility of the irregular Hodge filtration with tensor product and duality. We first notice that, according to Proposition \ref{prop:Xpoint}, the irregular Hodge filtration as defined in \loccit\ for an object of $\IrrMHM(\pt)$ is nothing but the irregular Hodge filtration attached to $M''$ as in Definition \ref{def:irregHF}. Then Propositions \ref{prop:tensorirregHodge} and \ref{prop:dualirregHodge} give the desired compatibility, as in Section~\ref{subsec:irregHNC}.
\end{proof}

\begin{remarque}[$\kk$-structure]
By construction, $\IrrMHS(\CC)$ is a full neutral Tannakian subcategory of $\iMTS^\intt(\CC)$. As a consequence of Definition \ref{def:IrrMHS}\eqref{def:IrrMHS2}, and according to Proposition \ref{prop:MTSTannakienne}, $\IrrMHS(\kk)$ is a full neutral Tannakian subcategory of $\iMTM^\intt(\pt,\kk)$.
\end{remarque}

\subsubsection*{Strictness for the Harder-Narasimhan filtration}

The statement \ref{th:IrrMHS}\eqref{th:IrrMHS4} can be extended to the Harder-Narasimhan filtration. Let $(M',M'',\iC^\nabla)$ be an object of $\IrrMHS(\CC)$ and set $M=M''$. We denote by $MV_\beta$ the vector bundle on $\PP^1$ obtained by gluing~$M$ with $V_\beta(M[\hbm])$ with the identification $M[\hbm]=V_\beta(M[\hbm])[\hbm]$. A~morphism of integrable mixed twistor structures also induces a strictly $W$-filtered bundle morphism $(M_1V_\beta,W_\bbullet)\to(M_2V_\beta,W_\bbullet)$ since each morphism $M_1[\hbm]\to M_2[\hbm]$ is strictly compatible with the $V$-filtration.

For $\alpha\in[0,1)$ and $k\in\ZZ$, we have $MV_{\alpha+k}=MV_\alpha\otimes\cO_{\PP^1}(k)$. Let us denote by $\HN^\cbbullet$ the Harder-Narasimhan filtration. We then have $\HN^0MV_{\alpha+k}=\HN^{-k}MV_\alpha(k)$. Note that $\HN^0MV_\beta$ is the $\cO_{\PP^1}$-sub-bundle of $MV_\beta$ generated by $M\cap V_\beta(M[\hbm])=\Gamma(\PP^1,MV_\beta)$ (\cf \cite[\S1.a]{Bibi08}). Any morphism $\lambda:\cM_1\to\cM_2$ compatible with the connections extends in a unique way as a morphism $M_1V_\beta\to M_2V_\beta$ that we still denote~$\lambda$, and it is compatible with the Harder-Narasimhan filtration.

\begin{proposition}
If $(\icT_1,W_\bbullet),(\icT_2,W_\bbullet)$ are objects of $\IrrMHS(\CC)$ and if $\lambda$ is a morphism in $\IrrMHM(\CC)$, then for each $\alpha\in[0,1)$, the induced morphism $M_1V_\alpha\to M_2V_\alpha$, when restricted on $\CC^*_\hb$, is strictly compatible with the induced Harder-Narasimhan filtrations.
\end{proposition}

\begin{proof}
We remark that $F_\alpha^{\irr,\bbullet}\ccM$ is the fibre at $\hb=1$ of the Harder-Narasimhan filtration of $MV_\alpha$, according to Lemma \ref{lem:irrHN} and we know (Theorem \ref{th:IrrMHS}\eqref{th:IrrMHS4}) that~$\lambda$ induces a strict morphism $(\ccM_1,F_\bbullet^\irr\ccM_1)\to(\ccM_2,F_\bbullet^\irr\ccM_2)$. We can argue similarly with the fibre at any $\hb_o\neq0$ by applying Theorem \ref{th:IrrMHS}\eqref{th:IrrMHS4} to $\lambda_{|\hb=\hb_o}$.
\end{proof}

\section{Application to Künneth and Thom-Sebastiani formulas}\label{sec:KunnethTS}

\subsection{Künneth formula for $\IrrMHM$}
Let $X$ be a smooth projective variety that we endow with its analytic structure, set $\cX=X\times\CC_\hb$, and let $\cR_\cX^\intt$ be the subsheaf $\cD_{\cX/\CC_\hb}\langle\hb^2\partial_\hb\rangle$ of the sheaf $\cD_\cX$ of holomorphic differential operators on $\cX$. Let $(\cT,W)$ be an object of $\WRdTriples(X)$, with $\cT=(\cM',\cM'',C_{\bS})$. We will often set $\cM=\cM''$. Let $\ccM$ be the $\cD_X$-module $\Xi_{\DR}(\cT):=\cM/(\hb-1)\cM$, which is filtered by the filtration induced by $W_\bbullet\cM$. This is a finite filtration by $\cD_X$-submodules. If~$(\cT,W)$ is an object of $\IrrMHM(X)$, then~$\ccM$ is endowed with a coherent filtration $F_\bbullet^\irr\ccM$ discretely indexed by $\RR$ (\cf Section \ref{sec:MTrM}) and $\Xi_{\DR}(\cT)$ denotes the corresponding bi-filtered $\cD_X$-module (Notation \ref{nota:XiDR}).

We denote by $a_X:X\to\pt$ the constant map. Let $\cM$ be an $\cR_\cX^\intt$-module. Its relative de~Rham complex $\DR\cM$, with terms $\Omega^\cbbullet_{\cX/\CC_\hb}\otimes_{\cO_\cX}\cM$, gives rise to the de~Rham cohomology modules $H^j_\dR(X,\cM)$. These are $\cO_{\CC_\hb}$-modules endowed with a compatible action of $\hb^2\partial_\hb$, that is, $\cR^\intt_\pt$-modules, and we can equivalently regard them as $\CC[\hb]\langle\hb^2\partial_\hb\rangle$-modules by means of the Deligne meromorphic extension at $\hb=\infty$. In terms of $\cR^\intt$-module operations, we have\index{$HDRMc$@$H^\cbbullet_\dR(X,\cM)$}
\[
H^{n+j}_\dR(X,\cM)=a_{X\dag}^j\cM,\quad n=\dim X.
\]

For an object $(\cT,W)$ of $\WRdTriples(X)$, the pushforward $a_{X\dag}^j(\cT,W)$ has components $H^{n-j}_\dR(X,\cM')$ and $H^{n+j}_\dR(X,\cM'')$, with $W$-filtrations naturally induced by $W_\cbbullet\cT$.

If $(\cT,W)$ is an object of $\MTM^\intt(X)$, then (\cf\cite{Mochizuki11}) $a_{X\dag}^j(\cT,W)$ is an object of $\MTS^\intt(\CC)$ (\cf Section \ref{subsec:twtriple}). Due to the equivalence of categories $\iMTS^\intt(\CC)\simeq\MTS^\intt(\CC)$, we can as well consider it as an object of $\iMTS^\intt(\CC)$.

\begin{remarque}[Strictness]\label{rem:strictness}
If $\cM=\cM''$ is the second component of $\cT$ as above, we note that $a_{X\dag}\cM$ has $\cO_{\CC_\hb}$-flat cohomology, since it is strict, being part of an object of $\MTS^\intt(\CC)$.
\end{remarque}

By the previous remark, the $\cD$-module pushforward $a_{X\dag}^j\ccM=:H^{n+j}_\dR(X,\ccM)$ is isomorphic to \hbox{$H^{n+j}_\dR(X,\cM)/(\hb-1)H^{n+j}_\dR(X,\cM)$}. The filtered vector space $(H^{n+j}_\dR(X,\ccM),W_\bbullet H^{n+j}_\dR(X,\ccM))$ will be denoted by $H^{n+j}_\dR(X,\Xi_{\DR}(\cT,W_\bbullet))$.

According to Section \ref{subsec:irregHNC}, this vector space is also bi-filtered by $F^\irr_\bbullet,W_\bbullet$. We will consider the bi-filtration only when $(\cT,W)$ is an object of $\IrrMHM(X)$.

\begin{corollaire}[of Theorem \ref{th:pushforwardresc}]\label{th:imdirIrrMHM}
For $(\cT,W)$ in $\IrrMHM(X)$ with $X$ smooth projective, the natural morphism\index{$HDRMcc$@$H^\cbbullet_\dR(X,\ccM)$}
\[
\bH^{n+j}(X,F^\irr_\bbullet\DR\ccM)\to \bH^{n+j}(X,\DR\ccM)=:H^{n+j}_\dR(X,\ccM)
\]
is an injection for all $j$, and its image defines the filtration $F^\irr_\bbullet H^{n+j}_\dR(X,\ccM)$.\qed
\end{corollaire}

For $(\cT,W)$ in $\IrrMHM(X)$, we then denote by \index{$HDRTcW$@$H^\cbbullet_\dR(X,\Xi_{\DR}(\cT,W_\bbullet))$}$H^{n+j}_\dR(X,\Xi_{\DR}(\cT,W_\bbullet))$ the bi-filtered vector space $(H^{n+j}_\dR(X,\ccM),F^\irr_\bbullet H^{n+j}_\dR(X,\ccM),W_\bbullet H^{n+j}_\dR(X,\ccM))$.

\begin{corollaire}[Künneth formula for $\IrrMHM$]\label{cor:imdirIrrMHM}
\index{Künneth formula}Assume that $(\cT_i,W)$ are objects of $\IrrMHM(X_i)$ ($i=1,2$), and that $(\cT_1\boxtimes\cT_2,W)$ is an object of $\IrrMHM(X_1\times X_2)$. Then
\[
H_\dR^{n_1+n_2+k}(X_1\times X_2,\cM_1\hbboxtimes\cM_2)\simeq \bigoplus_{k_1+k_2=k}\hspace*{-3mm}H_\dR^{n_1+k_1}(X_1,\cM_1)\otimes_{\cO_\hb} H_\dR^{n_2+k_2}(X_2,\cM_2),
\]
and
\[
F^\irr_\beta H^{n_1+n_2+k}_\dR(\ccM_1\boxtimes\ccM_2)
\simeq \bigoplus_{k_1+k_2=k}\sum_{\beta_1+\beta_2=\beta} F^\irr_{\beta_1}H^{n_1+k_1}_\dR(\ccM_1)\otimes_\CC F^\irr_{\beta_2}H^{n_2+k_2}_\dR(\ccM_2).
\]
\end{corollaire}

\begin{proof}
Assume that, for $i=1,2$, $(\cT_i,W)$ is an object of $\MTM^\intt(X_i)$. There is an external product functor $\boxtimes:\MTM^\intt(X_1)\times\MTM^\intt(X_2)\mto\MTM^\intt(X_1\times X_2)$ and the projective pushforward functor is compatible with it (\cf\cite[Prop.\,11.4.6 \&~Lem.\,11.4.14]{Mochizuki11}), so that in particular
\[
a_{X_1\times X_2\dag}^k(\cT_1\boxtimes\cT_2)\simeq\bigoplus_{k_1+k_2=k}a_{X_1\dag}^{k_1}(\cT_1)\boxtimes a_{X_2\dag}^{k_2}(\cT_2),
\]
in a way compatible with $W$. Taking the ``double prime component'', the first isomorphism of the corollary holds between the corresponding $\CC[\hb]\langle\hb^2\partial_\hb\rangle$-modules, in a way compatible with $W_\bbullet$. Proposition \ref{prop:tensorirregHodge}, gives the Künneth formula for the corresponding irregular Hodge filtrations. If $\cT_1,\cT_2,\cT_1\boxtimes\cT_2$ are objects of $\IrrMHM$, then Corollary \ref{th:imdirIrrMHM} identifies the irregular Hodge filtrations of the corresponding $\CC[\hb]\langle\hb^2\partial_\hb\rangle$-modules with those appearing in the second isomorphism of the corollary.
\end{proof}

\subsection{Complex mixed Hodge modules and their Brieskorn lattice}\label{subsec:Cmhm}

We will consider complex mixed Hodge modules as in Definition \ref{def:MHMC}, but we will only need a less precise notion. We will say that a $W$-filtered $R_F\cD_X$-module $(\cM,W_\bbullet\cM)$ underlies a mixed twistor $\cD_X$-module if $\cR_\cX\otimes_{R_F\cD_X}W_k\cM$ is equal to $W_k\cM''$ of a complex mixed Hodge module.

We will also say that $(\ccM,F_\bbullet\ccM,W_\bbullet\ccM)$ \emph{underlies a complex mixed Hodge module} if $(R_F\ccM,W_\bbullet R_F\ccM)$ does so. We will say that a filtered $\cD_X$-module $(\ccM,F_\bbullet\ccM)$ \emph{underlies a complex mixed Hodge module} if it can be completed as $(\ccM,F_\bbullet\ccM,W_\bbullet\ccM)$ underlying a complex mixed Hodge module.

Let $\mero$ be a meromorphic function on $X$ with pole divisor $P$, and let $(\ccM,F_\bbullet\ccM)$ be a filtered holonomic $\cD_X$-module underlying a complex mixed Hodge module. We set $\cM=R_F\ccM$.

\begin{definition}[Brieskorn lattice]\label{def:Brieskornlattice}
The $\CC[\hb]$-module with action of $\hb^2\partial_\hb$\index{$G0XMmero$@$G_0^\cbbullet(X,\cM,\mero)$}
\[
G_0^k(X,\cM,\mero):=H^k_\dR\bigl(X,\cE^{\mero/\hb}\otimes\cM\bigr)
\]
is called the \index{Brieskorn lattice}\emph{Brieskorn lattice of $\cM$ with respect to $\mero$}.
\end{definition}

If $\mero$ induces a morphism $X\to\PP^1$, then $\cE^{\mero/\hb}=\cE_*^{\mero/\hb}$ and we have
\[
G_0^k(X,\cM,\mero)=\bH^k\bigl(X,(\Omega_X^\cbbullet(*P_\red)\otimes R_F\ccM,\hb\nabla+\rd \mero)\bigr).
\]

\begin{proposition}\label{prop:Brieskornfree}
If $\cM$ underlies a complex mixed Hodge module, then the Brieskorn lattice satisfies the following properties.
\begin{enumerate}
\item\label{prop:Brieskornfree1}
$G_0^k(X,\cM,\mero)$ is $\CC[\hb]$-free of finite rank.
\item\label{prop:Brieskornfree2}
If $\finite:X\to Y$ is a finite morphism between projective smooth varieties of dimension $n$ and $m$ respectively, and if $\merob\in\cO_Y(Q)$ is a meromorphic function on $Y$ with pole divisor $Q$ then, setting $\mero=\merob\circ\finite$, we have for all $k$:
\[
G_0^{n+k}(X,\cM,\mero)=G_0^{m+k}(Y,\finite_\dag^0\cM,\merob).
\]
\end{enumerate}
\end{proposition}

\pagebreak[2]
\skpt
\begin{proof}
\begin{enumerate}
\item
This follows from Remark \ref{rem:strictness}, since $\cE^{\mero/\hb}\otimes\cM$ underlies an object of $\MTM^\intt(X)$, as recalled in Section\ref{subsec:exptwist}. See also a different proof on Page \pageref{proof:Brieskornfree1}.
\item
We note that $\finite_\dag^j\cM=0$ if $j\neq0$ since $\finite$ is finite. We can apply Proposition \ref{prop:expMHM}, and in particular the analogue of \ref{subsub:commutation}\eqref{enum:rel4} (third line) to obtain $\finite_\dag^0(\cE^{\mero/\hb}\otimes\cM)\simeq\cE^{\merob/\hb}\otimes\finite_\dag^0\cM$. We then apply $a^k_{Y\dag}$ to both terms.\qedhere
\end{enumerate}
\end{proof}

\begin{proposition}[Thom-Sebastiani formula for the Brieskorn lattice]\label{prop:KunnethBrieskorn}
\index{Thom-Sebastiani formula}Let $\mero_i$ ($i=1,2$) be meromorphic functions and suppose $\cM_i$ underlie complex mixed Hodge modules on the smooth projective varieties $X_i$. Set $X = X_1 \times X_2$, $\mero = \mero_1\sboxplus \mero_2$ and $\cM =\cM_1\hbboxtimes\cM_2$. Then the Brieskorn lattice of $\cM$ with respect to $\mero$ satisfies
\[
G_0^k(X,\cM,\mero)\simeq\bigoplus_{i+j=k}G_0^i(X_1,\cM_1,\mero_1)\otimes_{\CC[\hb]}G_0^j(X_2,\cM_2,\mero_2).
\]
\end{proposition}

\begin{proof}
This is nothing but the first isomorphism in Corollary \ref{cor:imdirIrrMHM}, according to \eqref{eq:tensorTf}. We will however give another proof, in the spirit of the latter approach, on Page \pageref{proof:KunnethBrieskorn}.
\end{proof}

\subsection{Statement of the Thom-Sebastiani formula}
Let $X$ be a smooth complex projective variety, and let~$\mero$ be a rational function on~$X$ with pole divisor $P$. We set $\ccE^\mero:=(\cO_X(*P_\red),\rd+\rd \mero)$. Let $(\ccM,F_\bbullet\ccM)$ be a filtered holonomic left $\cD_X$-module underlying a complex mixed Hodge module, and such that $\ccM=\ccM(*P_\red)$. We regard~$\ccM$ as an $\cO_X$-module with integrable connection~$\nabla$. The twisted de~Rham cohomology
\begin{align*}
H^k_{\dR}(X,\ccE^\mero\otimes\ccM):&=\bH^k\bigl(X,(\Omega_X^\cbbullet(*P_\red)\otimes\ccM,\nabla+\rd \mero)\bigr)\\
&=\bH^k\bigl(X,(\Omega_X^\cbbullet\otimes\ccM,\nabla+\rd \mero)\bigr)
\end{align*}
is naturally equipped with a decreasing filtration $F_\irr^\cbbullet H^k_{\dR}(X,\ccE^\mero\otimes\ccM)$ indexed by~$\RR$ called the \emph{irregular Hodge filtration}: According to \cite[Th.\,1.3]{S-Y14}, $\ccE^\mero\otimes\ccM$ is \hbox{endowed} with a filtration $F_\irr^\cbbullet$, which induces a filtration $F_\irr^\cbbullet$ on the twisted de~Rham complex $\DR(\ccE^\mero\otimes\ccM)$, and induces the desired filtration on the hypercohomology of~$\DR(\ccE^\mero\otimes\nobreak\ccM)$. Moreover, \cite[Cor.\,1.4]{S-Y14} (\cf also \cite{Mochizuki15a}) shows a good beha\-viour by taking hypercohomology (degeneration at $E_1$ of the spectral sequence).

Given two such sets of data $(X_i,\mero_i,(\ccM_i,F_\bbullet\ccM_i))_{i=1,2}$, we consider their external product:
\begin{equation}\label{eq:extprod}
X=X_1\times X_2,\quad \mero=\mero_1\sboxplus \mero_2,\quad (\ccM,F_\bbullet\ccM)=(\ccM_1,F_\bbullet\ccM_1)\boxtimes(\ccM_2,F_\bbullet\ccM_2),
\end{equation}
where \index{$Boxplus$@$\sboxplus$}$\mero_1\sboxplus \mero_2$ is the \emph{Thom-Sebastiani sum} defined by $(\mero_1\sboxplus \mero_2)(x_1,x_2)\!:=\!\mero_1(x_1)+\mero_2(x_2)$, having pole divisor $P:=(P_1\times X_2)\cup(X_1\times P_2)$.

\begin{theoreme}[Thom-Sebastiani formula for the irregular Hodge filtration]\label{th:main}
\index{Thom-Sebastiani formula}Under these assumptions, we have for each $\beta\in\RR$:
\begin{multline*}
F^\irr_\beta H^k(X,\ccE^\mero\otimes\ccM)\\[-5pt]
\simeq\bigoplus_{k_1+k_2=k}\biggl(\sum_{\beta_1+\beta_2=\beta}F^\irr_{\beta_1} H^{k_1}(X_1,\ccE^{\mero_1}\otimes\ccM_1)\otimes_\CC F^\irr_{\beta_2} H^{k_2}(X_2,\ccE^{\mero_2}\otimes\ccM_2)\biggr).
\end{multline*}
\end{theoreme}

\begin{remarque}
With this theorem together with \cite[Cor.\,8.19]{S-Y14}, one obtains another proof of the behaviour of the spectrum of tame functions on smooth affine varieties with respect to the Thom-Sebastiani sum (\cf\cite[\S2]{N-S97}).
\end{remarque}

The Thom-Sebastiani formula of Theorem \ref{th:main} generalizes to arbitrary complex mixed Hodge modules the formula previously obtained by K.-C.\,Chen and the second author in \cite{C-Y16}. The case considered in \loccit\ is that of a regular function~$\mero_i$ on a smooth quasi-projective variety $U_i$, with an open embedding $j_i:U_i\hto X_i$, with~$X_i$ smooth projective such that $X_i\moins U_i$ is a divisor with normal crossings and~$\mero_i$ extends as a morphism $X_i\to\PP^1$, and \hbox{$\ccM_i=j_{i+}\cO_{U_i}$}. While the proof of \cite{C-Y16} relies on a detailed analysis near the divisors at infinity, the proof given here uses the full strength of the category $\MTM^\intt(X)$ proved by Mochizuki in \cite{Mochizuki11}. We will also sketch another proof, which emphasizes the role of the Harder-Narasimhan filtration \eqref{eq:FM*}. It takes advantage of $\cD$-module theory to reduce the question to the case where $X_i=\Afu$ and $\mero_i$ is a global coordinate on $\Afu$ ($i=1,2$). In such a case, the pushforward of $\ccM$ by~$\mero$ is then identified with the additive convolution of $\ccM_1$ with $\ccM_2$. In order to reduce to this case, we will use the construction explained in \cite[\S2]{S-Y14} to reduce the question to the case where $\mero_i$ are morphisms $X_i\to\PP^1$, and then we will apply the one-dimensional result to the various direct images of $\ccM_i$ by $\mero_i$.

\begin{proof}[First proof of Theorem \ref{th:main}]
If $\cT_i$ are objects of $\MTM^\intt(X_i)$ corresponding to complex mixed Hodge modules, and if $\mero_i$ are meromorphic functions on $X_i$ ($i=1,2$), then $\cT^{\mero_i/\hb}\otimes\cT_i$ are objects of $\IrrMHM(X_i)$ (\cf Theorem \ref{th:mtmgresc}\eqref{th:mtmgresc3}). One first checks easily that
\[
(\cT_*^{\mero_1/\hb}\otimes\cT_1)\boxtimes(\cT_*^{\mero_2/\hb}\otimes\cT_2)=\cT_*^{(\mero_1\sboxplus \mero_2)/\hb}\otimes(\cT_1\boxtimes\cT_2).
\]
Then we observe that
\[
\Gamma_{[*P]}\big[(\cT_*^{\mero_1/\hb}\otimes\cT_1)\boxtimes(\cT_*^{\mero_2/\hb}\otimes\cT_2)\big]=\Gamma_{[*P_1]}(\cT_*^{\mero_1/\hb}\otimes\cT_1)\boxtimes\Gamma_{[*P_2]}(\cT_*^{\mero_2/\hb}\otimes\cT_2).
\]
It follows that
\begin{equation}\label{eq:tensorTf}
(\cT^{\mero_1/\hb}\otimes\cT_1)\boxtimes(\cT^{\mero_2/\hb}\otimes\cT_2)=\cT^{(\mero_1\sboxplus \mero_2)/\hb}\otimes(\cT_1\boxtimes\cT_2),
\end{equation}
and thus $(\cT^{\mero_1/\hb}\otimes\cT_1)\boxtimes(\cT^{\mero_2/\hb}\otimes\cT_2)$ also belongs to $\IrrMHM(X_1\times X_2)$. We can thus apply Corollary \ref{cor:imdirIrrMHM}.
\end{proof}

\oldsubsubsection*{Second proof of Theorem \ref{th:main}}
For $(\ccM,F_\bbullet\ccM)$ underlying a complex mixed Hodge module on the smooth complex projective variety $X$, set $\cM:=R_F\ccM=\bigoplus_pF_p\ccM \hb^p$. We first give another proof of Proposition \ref{prop:Brieskornfree}\eqref{prop:Brieskornfree1}.

\begin{proof}[Second proof of Proposition \ref{prop:Brieskornfree}\eqref{prop:Brieskornfree1}]\label{proof:Brieskornfree1}
Let us set $U=X\moins P_\red$. The computation of the $\CC[\hb]$-module $\bH^k\bigl(X,(\Omega_X^\cbbullet\otimes R_F\ccM,\hb\nabla+\rd\mero)\bigr)$ can be done by working in the algebraic context with respect to $X$, and in such a context, setting $(M,F_\bbullet M)=(\ccM,F_\bbullet\ccM)_{|U}$, we have
\[
G_0^k(X,\cM,\mero)\simeq\bH^k\bigl(U,(\Omega_U^\cbbullet\otimes R_FM,\hb\nabla+\rd \mero)\bigr).
\]
If $\mero$ is a morphism $X\to\PP^1$, then \cite[Prop.\,in \S1]{Bibi97b} is the desired statement.

Let us now consider the general case, that is, let $\mero$ be a global section of $\cO_X(P)$. By the construction of \cite[Rem.\,2.3]{S-Y14}, there exists a filtered holonomic $\cD_{X\times\PP^1}$\nobreakdash-module $(\ccN,F_\bbullet\ccN)$ which underlies a complex mixed Hodge module on a product $X\times\PP^1$ such that, denoting by $p$ the projection $X\times\PP^1\to X$, we have $(\ccM,F_\bbullet\ccM)=p_+(\ccN,F_\bbullet\ccN)=\cH^0p_+(\ccN,F_\bbullet\ccN)$ (note that the notation $\ccM$, $\ccN$ in \loccit\ is switched). We denote by~$t$ the affine coordinate on $\PP^1$ such that $\mero$ is the composition of~$t$ and the inclusion of the graph of~$\mero$, on the open set $U\subset X$ where $\mero$ is holomorphic. Arguing like in \cite[Lem.\,2.4]{S-Y14}, we have
\begin{equation}\label{eq:KMN}
G_0^k(X,\cM,\mero)\simeq G_0^{k+1}(X\times\PP^1,\cN,t).
\end{equation}
The first part of the proof can be applied to $\cN$.
\end{proof}

\begin{theoreme}\label{th:irregNS}
If $(\ccM,F_\bbullet\ccM)$ underlies a complex mixed Hodge module, the Harder-Narasimhan filtration \eqref{eq:FM*} of the Brieskorn lattice $G_0^k(X,\cM,\mero)$ induces the irregular Hodge filtration $F^\cbbullet_\irr\bH^k\bigl(X,(\Omega_X^\cbbullet\otimes\ccM,\nabla+\rd \mero)\bigr)$ by restriction to $\hb=1$.
\end{theoreme}

\begin{proof}[\proofname\ in the case where $X=\PP^1$, $\mero=\id:\PP^1\to\PP^1$]
In this case, we denote by~$t$ the coordinate on the affine chart \hbox{$\Afu=\PP^1\moins\{\infty\}$}. Then $\ccE^\mero=(\cO_{\PP^1}(*\infty),\rd+\rd t)$. In this case, \hbox{$G_0^k(X,\cM,\mero)=0$} for $k\neq1$. Then we set $G_0\!\!=G_0^1(X,\cM,\mero)$. By Lemma~\ref{lem:irrHN}, the image of the Harder-Narasimhan filtration $F^\sHN_\beta$ in $M/(\hb-1)M$ is that of $V_\beta\cap G_0$. The identification asserted by the theorem is exactly that obtained in \cite[Prop.\,6.10]{Bibi08}, since the filtration $F_\irr^\cbbullet$ is also identified with the Deligne filtration~$F^\cbbullet_\Del$ (\cf \cite[(5.2)]{S-Y14}).
\end{proof}

\begin{proof}[\proofname\ in the case where $\mero:X\to\PP^1$ is a morphism]
In this case, $P_\red=\mero^{-1}(\infty)$. We set $n=\dim X$. We first claim that
\begin{equation}\label{eq:brieskf}
G_0^k(X,\cM,\mero)\simeq G_0^1(\PP^1,\mero_\dag^{k-n}\cM,t).
\end{equation}
This follows from the property that the Hodge-to-de\,Rham spectral sequence for the filtered module $(\ccM,F_\bbullet\ccM)$ degenerates at $E_1$, since it underlies a complex mixed Hodge module. This implies that $\mero_\dag^jR_F\ccM=R_F\mero_\dag^j\ccM$. It is then easy to check that (recall that $\cE^{\mero/\hb}=\cE_*^{\mero/\hb}$)
\[
\mero_\dag^j(\cE^{\mero/\hb}\otimes R_F\ccM)\simeq \cE^{t/\hb}\otimes \mero_\dag^j(R_F\ccM).
\]
Taking hypercohomology on $\PP^1$ gives the result, since the hypercohomology on $\PP^1$ only exists in degree one, as remarked above.

On the other hand, by the $E_1$-degeneration property for the irregular Hodge filtration by a projective morphism (here equal to $\mero$), as stated in \cite[Th.\,1.3(4)]{S-Y14}, we have
\[
F^\cbbullet_\irr\bH^k\bigl(X,(\Omega_X^\cbbullet\otimes\ccM,\nabla+\rd \mero)\bigr)=F^\cbbullet_\irr\bH^1\bigl(\PP^1,(\Omega_{\PP^1}^\cbbullet\otimes \mero_\dag^{k-n}\ccM,\nabla+\rd t)\bigr).
\]
(\Cf a similar reasoning in the proof of \cite[Th.\,3.3.1]{E-S-Y13}.)

In this case, the theorem follows from the first case applied to the various $\mero_\dag^j\ccM$ with their Hodge filtration.
\end{proof}

\begin{proof}[\proofname\ in the general case]
Let $\mero$ be a global section of $\cO_X(P)$, and let $(\ccN,F_\bbullet\ccN)$ be attached to $(\ccM,F_\bbullet\ccM)$ as in the second proof of Proposition \ref{prop:Brieskornfree}\eqref{prop:Brieskornfree1} on Page \pageref{proof:Brieskornfree1}. According to the definition of $F^\irr_\cbbullet(\ccE^\mero\otimes\nobreak\ccM)$ from $F^\irr_\cbbullet(\ccE^t\otimes\ccN)$ (\cf \cite[Def.\,5.1]{S-Y14}), and due to the degeneration at~$E_1$ of the spectral sequences attached to the irregular Hodge filtration for the maps $a_X:X\to\mathrm{pt}$ and $a_{X\times\PP^1}:X\times\PP^1\to\mathrm{pt}$, we obtain that
\[
F^\irr_\cbbullet\bH^k\bigl(X,(\Omega_X^\cbbullet\otimes\ccM,\nabla+\rd \mero)\bigr)=F^\irr_\cbbullet\bH^{k+1}\bigl(X\times\PP^1,(\Omega_{X\times\PP^1}^\cbbullet\otimes\ccN,\nabla+\rd t)\bigr).
\]
We can now apply to $(\ccN,F_\bbullet\ccN)$ the case where the function is a morphism to $\PP^1$ and obtain the result for $(\ccM,F_\bbullet\ccM)$, according to \eqref{eq:KMN}.
\end{proof}

\begin{proof}[End of the second proof of Theorem \ref{th:main}]
We can apply Remark~\ref{rem:tensorHN} to each summand occurring in Proposition \ref{prop:KunnethBrieskorn} to obtain the Thom-Sebastiani formula for the filtration induced by the Harder-Narasimhan filtration. Lastly, we use Theorem~\ref{th:irregNS} to identify the latter with the irregular Hodge filtration.
\end{proof}

\begin{proof}[Second proof of Proposition \ref{prop:KunnethBrieskorn}]\label{proof:KunnethBrieskorn}
Since $X_1,X_2$ are projective, we can work in the algebraic category of complex mixed Hodge modules. We denote by $p_i$ ($i=1,2$) the projections \hbox{$X:=X_1\times X_2\to\nobreak X_i$}. Recall that a left $R_F\cD_X$-module is an $R_F\cO_X$-module equipped with an integrable $\hb$\nobreakdash-connection, and that $R_F\cO_X=\cO_X[\hb]$. The pullback $R_F\cD_X$-module \hbox{$p_1^+\cN_1$ is $\cO_{X_1\times X_2}[\hb]\otimes_{p_1^{-1}\cO_{X_1}[\hb]}p_1^{-1}\cN_1$} equipped with the pullback $\hb$-connection. We then set
\[
\cN_1\boxtimes_{R_F\cD}\cN_2:=p_1^+\cN_1\otimes_{\cO_{X_1\times X_2}[\hb]}p_2^+\cN_2.
\]
The natural morphism
\[
p_1^{-1}\DR(\cN_1)\otimes_{\CC[\hb]}p_2^{-1}\DR(\cN_2)\to\DR(\cN_1\boxtimes_{R_F\cD}\cN_2)
\]
is an isomorphism if $\cN_i$ is $R_F\cD_{X_i}$-coherent. Indeed, the question is local, and we can reduce to the case where $\cN_i\simeq R_F\cD_{X_i}$ by taking a local resolution by free $R_F\cD_{X_i}$\nobreakdash-modules. The question reduces to proving that the natural morphism
\[
p_1^{-1}\cO_{X_1}[\hb]\otimes_{\CC[\hb]}p_2^{-1}\cO_{X_2}[\hb]\to\cO_{X_1\times X_2}[\hb]
\]
is an isomorphism, which is clear in the present algebraic setting. Still in the algebraic setting, we conclude by using an argument similar to that of \cite[Lem.\,1.5.31]{H-T-T08} that the natural morphism
\[
\bR\Gamma\bigl(X_1,\DR(\cN_1)\bigr)\otimes_{\CC[\hb]}\bR\Gamma\bigl(X_2,\DR(\cN_2)\bigr)\to\bR\Gamma\bigl(X,\DR(\cN_1\boxtimes_{R_F\cD}\cN_2)\bigr)
\]
is an isomorphism. We notice that, as in \eqref{eq:tensorTf},
\[
(\cE^{\mero_1/\hb}\otimes\cM_1)\boxtimes_{R_F\cD}(\cE^{\mero_2/\hb}\otimes\cM_2)\simeq\cE^{\mero/\hb}\otimes\cM.
\]

To conclude, it is enough to use the property that, given bounded complexes $L_1^\cbbullet$ and~$L_2^\cbbullet$ of flat $\CC[\hb]$-modules with free $\CC[\hb]$-cohomology, then the Künneth formula for the cohomology of $L_1^\cbbullet\otimes_{\CC[\hb]}L_2^\cbbullet$ holds, that is, we have an isomorphism of free $\CC[\hb]$-modules
\[
H^k(L_1^\cbbullet\otimes_{\CC[\hb]}L_2^\cbbullet)\simeq\bigoplus_{i+j=k}H^i(L_1^\cbbullet)\otimes_{\CC[\hb]}H^j(L_2^\cbbullet).
\]

We apply the latter result to $L_i^\cbbullet=\bR\Gamma\bigl(X_i,\DR(\cE^{\mero_i/\hb}\otimes\cM_i)\bigr)$ ($i=1,2$). The $\CC[\hb]$-freeness of the cohomology is provided by Proposition \ref{prop:Brieskornfree}.
\end{proof}

\section{Finite group actions and alternating products}\label{sec:alternating}

\subsection{General framework}
Let $X$ be a smooth projective variety of dimension $n$ and let \index{$Sgot$@$\gotS$}$\gotS$ be a finite group acting on it by algebraic automorphisms. We fix an embedding $X/\gotS\hto Y$ into a smooth projective variety $Y$ of dimension $m$, and we denote by $\finite:X\to Y$ the composition $X\to X/\gotS\hto Y$.

Let $(\cT,W_\bbullet)$ be an object of $\iMTM^\intt(X)$. Assume that $(\cT,W_\bbullet)$ is endowed with automorphisms $\lambda_\sigma:(\cT,W_\bbullet)\isom\sigma_\dag(\cT,W_\bbullet)$ such that $\sigma'_\dag(\lambda_\sigma)\circ\lambda_{\sigma'}=\lambda_{\sigma\circ\sigma'}$ for all $\sigma,\sigma'\in\gotS$ (we then say that $(\cT,W_\bbullet)$ is $\gotS$-\emph{equivariant}). Note that $\finite_\dag^k\cT=0$ for $k\neq0$ since $\finite$ is finite, and $\gotS$ acts on $(\finite_\dag^0\cT,W_\bbullet\finite_\dag^0\cT)$ by automorphisms in $\iMTM^\intt(Y)$ since $\finite\circ\sigma=\finite$ for all $\sigma\in\gotS$. As a consequence, $(\finite_\dag^0\cT,W_\bbullet\finite_\dag^0\cT)$ decomposes in $\iMTM^\intt(Y)$ with respect to the characters of $\gotS$:
\begin{equation}\label{eq:deccharactersMTMint}
(\finite_\dag^0\cT,W_\bbullet\finite_\dag^0\cT)=\bigoplus_{\chi\in\Hom(\gotS,\CC^*)}(\finite_\dag^0\cT,W_\bbullet\finite_\dag^0\cT)^\chi.
\end{equation}

\begin{proposition}\label{prop:IrrMHMgotS}
If $(\cT,W_\bbullet)$ is an object of the full subcategory $\IrrMHM(X)$, then the decomposition \eqref{eq:deccharactersMTMint} holds in $\IrrMHM(Y)$. Moreover, using, Notation~\ref{nota:XiDR}, the bi-filtered vector space $\Xi_{\DR}\finite_\dag^0(\cT,W_\bbullet)$ has a similar decomposition compatible with the decomposition $\finite_\dag^0\ccM=\bigoplus_\chi(\finite_\dag^0\ccM)^\chi$, \ie the filtrations $F_\bbullet^\irr\finite_\dag^0\ccM$ and $W_\bbullet\finite_\dag^0\ccM$ both decompose with respect to the action of $\gotS$.
\end{proposition}

\begin{proof}
For the first part, we use that $(\finite_\dag^0\cT,W_\bbullet\finite_\dag^0\cT)$ is an object of $\IrrMHM(Y)$ (\cf Theorem \ref{th:pushforwardresc}) and that the category $\IrrMHM(Y)$ is stable by direct summand in $\iMTM^\intt(Y)$ (Proposition \ref{prop:MTMresc}\eqref{prop:MTMresc3}). The second part is then clear, because the construction of $F_\bbullet^\irr$ is compatible with direct sums in $\WRdiTriples(\tauY,(*\tauY_0))$.
\end{proof}

\begin{corollaire}\label{cor:IrrMHMgotS}
With the assumption of Proposition \ref{prop:IrrMHMgotS}, $\gotS$ acts on every vector space $H^k_{\dR}(X,\ccM)$ by linear automorphisms which strictly preserve $F_\bbullet^\irr H^k_{\dR}(X,\ccM)$ and $W_\bbullet H^k_{\dR}(X,\ccM)$. The $\chi$-component
\[
\big(H^k_{\dR}(X,\ccM)^\chi,(F_\bbullet^\irr H^k_{\dR}(X,\ccM))^\chi,(W_\bbullet H^k_{\dR}(X,\ccM))^\chi\big)
\]
is naturally identified with $H^{k+m-n}_{\dR}\big(Y,\Xi_{\DR}((\finite_\dag^0(\cT,W_\bbullet))^\chi)\big)$.
\end{corollaire}

\begin{proof}
Denoting by $a_Y:Y\to\pt$ the constant map, we apply the functor $a_{Y\dag}^k$ to the objects in Proposition \ref{prop:IrrMHMgotS}.
\end{proof}

\begin{corollaire}\label{cor:BrieskorngotS}
Let $\merob\in\cO_Y(Q)$ be a meromorphic function with pole divisor $Q$ and let $\cM$ be an $\cR_\cX$-module underlying a complex mixed Hodge module on $X$ which is $\gotS$-equivariant. Set $\mero=\merob\circ\finite$. Then the Brieskorn lattice $G_0^k(X,\cM,\mero)$ decomposes as $\bigoplus_{\chi\in\Hom(\gotS,\CC^*)}G_0^k(X,\cM,\mero)^\chi$ and \index{$G0XMmerochi$@$G_0^\cbbullet(X,\cM,\mero)^\chi$}$G_0^k(X,\cM,\mero)^\chi\simeq G_0^{k+m-n}(Y,(\finite^0_\dag\cM)^\chi,\merob)$.
\end{corollaire}

\begin{proof}
This is a consequence of Proposition \ref{prop:Brieskornfree}\eqref{prop:Brieskornfree2}.
\end{proof}

\subsection{Alternating products}
Let $\gotS_r$ denote the symmetric group on $r$ letters. Let $(X,\cM,\mero)$ be the data of a smooth projective variety, an $\cR_X$-module underlying a mixed Hodge module $(\cT,W_\bbullet)$ on $X$ and a meromorphic function $\mero$ with pole divisor $P$. The symmetric group acts on $X^r$ and the Thom-Sebastiani sum\index{$Fzoplusr$@$\mero^{\oplus r}$}
\[
\mero^{\oplus r}(x_1,\dots,x_r):=\mero(x_1)+\cdots+\mero(x_r)
\]
is a meromorphic function on $X^r$ which descends as a meromorphic function \index{$Fzoplusrpar$@$\mero^{(\oplus r)}$}$\mero^{(\oplus r)}$ on the quotient space $X^r/\gotS_r$. One can choose an embedding $X^r/\gotS_r\hto\nobreak Y$ into a smooth projective variety $Y$ such that $\mero^{(\oplus r)}$ is the restriction to $X^r/\gotS_r$ of a meromorphic function $\merob$ on $Y$. As above, we denote by $\finite:X^r\to Y$ the corresponding finite morphism and by $n$ \resp $m$ the dimension of $X$ \resp $Y$.

Let us denote by $\cM^{\,\hbboxtimes r}$ the $r$-fold external product $\cM\hbboxtimes\cdots\hbboxtimes\cM$, underlying the mixed Hodge module $(\cT,W_\bbullet)^{\boxtimes r}$ on $X^r$. It is obviously $\gotS_r$-equivariant, by the action $\sigma\cdot(m_1,\dots,m_r)=(m_{\sigma(1)},\dots,m_{\sigma(r)})$.

Let \index{$Sgn$@$\sgn$}$\sgn:\gotS_r\to\{\pm1\}$ be the signature. As a particular case of Corollary \ref{cor:BrieskorngotS}, we find
\[
G_0^k(X^r,\cM^{\,\hbboxtimes r},\mero^{\oplus r})^\sgn\simeq G_0^{k+(m-n)r}(Y,(\finite^0_\dag\cM^{\,\hbboxtimes r})^\sgn,\merob).
\]

\begin{proposition}\label{prop:alternatingBrieskornFirr}
\index{alternating product}Assume $G_0^k(X,\cM,\mero)=0$ except for $k=\ell$. Then the $\CC[\hb]\langle\hb^2\partial_\hb\rangle$-linear morphism
\begin{starequation}\label{eq:alternatingBrieskorn}
\bigwedge^r_{\CC[\hb]}\!\!G_0^\ell(X,\cM,\mero)\!\ra\!G_0^{\ell r}(X^r,\cM^{\,\hbboxtimes r}\!,\mero^{\oplus r})^\sgn\!\simeq\!G_0^{(\ell+m-n)r}(Y,(\finite^0_\dag\cM^{\,\hbboxtimes r})^\sgn,\merob)
\end{starequation}%
is an isomorphism, where the alternating product on the left-hand side is the antisymmetric component $\big[G_0^\ell(X,\cM,\mero)^{\otimes r}]^\sgn$ of the $r$-fold tensor product of the Brieskorn lattice by itself. Also, the morphism of bi-filtered vector spaces
\begin{starstarequation}\label{eq:alternatingFirr}
\bigwedge^r\big(H^\ell_{\dR}(X,\ccM),F_\bbullet^\irr,W_\bbullet \big)\to H^{(\ell+m-n)r}_{\dR}\big(Y,\Xi_{\DR}((\finite_\dag^0(\cT,W_\bbullet)^{\boxtimes r})^\sgn)\big)
\end{starstarequation}%
is an isomorphism.
\end{proposition}

\begin{proof}
The assumption implies that the Thom-Sebastiani formula of Proposition \ref{prop:KunnethBrieskorn} reads\index{Thom-Sebastiani formula}
\[
G_0^k(X^r,\cM^{\,\hbboxtimes r},\mero^{\oplus r})\simeq
\begin{cases}
\bigotimes^r_{\CC[\hb]}G_0^\ell(X,\cM,\mero)&\text{if }k=\ell r,\\0&\text{otherwise},
\end{cases}
\]
and the first assertion follows by taking the $\sgn$-isotypical part. For the second assertion, one can adapt the proof of Proposition \ref{prop:Vadaptedtensor} to the case of the exterior product, and get the second assertion from the first one.
\end{proof}

\subsection{Irregular Hodge filtration for the mirror of the Grassmannian}
We now consider a geometric example in order to make more concrete Proposition~\ref{prop:alternatingBrieskornFirr}. Let $U$ be a smooth quasi-projective variety and let $\mero:U\to\CC$ be a regular function on $U$. We consider the mixed Hodge module \index{$QHU$@$\QQ_U^\rH$}$\QQ_U^\rH$ on it (\cf\cite{MSaito87}). Its underlying filtered right $\cD_U$-module is $\Omega^n_U$ with the filtration such that $\gr_p^F\Omega^n_U=\nobreak0$ if $p\neq-n$. We use the left $\cD_U$-module $\cO_U$ with the filtration $F^\rH_\bbullet$ satisfying the same jumping property. Let us choose an open embedding $j:U\hto X$ of $U$ into some projective variety $X$ such that $X\moins U$ is a divisor $D$. The pushforward $\Hj_*\QQ_U^\rH$ has the underlying filtered $\cD_X$-module $(\cO_X(*D),F^\rH_\bbullet\cO_X(*D))$. We denote by $R_F(\cO_X(*D))$ the associated $R_F\cD_X$-module. In this algebraic context, the Brieskorn lattices $G_0^k(X,R_F(\cO_X(*D)),\mero)$ do not depend on the choice of $X$ and we then simply denote them by \index{$G0Umero$@$G_0^k(U,\mero)$}$G_0^\cbbullet(U,\mero)$.

We assume that $\mero$ is \index{cohomologically tame function}\emph{cohomologically tame}, in the sense that there exists a projective morphism $\mero_Z:Z\to \CC$ extending $\mero$ (so $k:U\hto Z$ is a partial compactification) such that the complex $\bR k_*\CC_U$ has no vanishing cycles with respect to $\mero_Z$ anywhere on~$Z\moins U$. This implies in particular that the critical points of $\mero$ on $U$ are isolated. Moreover, this also implies, by the proposition in \cite[\S1]{Bibi97b}, that $G_0^k(U,\mero)=0$ for $k\neq n=\dim U$ and $G_0^n(U,\mero)$ is $\CC[\hb]$-free of rank equal to the sum of the Milnor numbers of $\mero$ on $U$.

\begin{exemple}\label{exem:Pn}
The mirror of $\PP^n$ is the function
\[
\mero(x_1,\dots,x_n)=x_1+\cdots+x_n+\frac1{x_1\cdots x_n}
\]
on the torus $U=(\CC^*)^n$. It is a convenient nondegenerate Laurent polynomial in the sense of Kouchnirenko \cite{Kouchnirenko76}, hence satisfies the cohomological tameness property above (\cf\eg\cite{Bibi96bb}). The $\CC[\hb]$-freeness in this case has been obtained in \hbox{\cite[Th.\,5.1]{A-S97}}. We have $\rk G_0^n(U,\mero)=n+1$. We claim that the irregular (decreasing) Hodge filtration $F_\irr^\cbbullet H^n(U,\rd+\rd \mero)$ (recall that $F_\irr^\beta:=F^\irr_{-\beta}$) satisfies
\[
\dim\gr^\beta_{F_\irr}H^n(U,\rd+\rd \mero)=
\begin{cases}
1&\text{if }\beta=0,1,\dots,n,\\
0&\text{otherwise}.
\end{cases}
\]
Indeed, this is a consequence of \cite[Prop.\,3.2]{D-S02b} with the following changes. Our $G_0^n(U,\mero)$ is denoted there by $G_n$ and, setting $\eta_0=\hb^{-n}\omega_0=\hb^{-n}\rd x_1/x_1\wedge\cdots\wedge\rd x_n/x_n$, we define $\eta_k=\hb^{-n}\omega_k$ ($k=0,\dots,n$) as in \loccit, and $\eta_k\in V_{-n+k}\cap G_n$ induces a basis of $\gr^{n-k}_{F_\irr}H^n(U,\rd+\rd \mero)$. Notice that this is the Example in Remark \ref{rem:hypergeom}\eqref{rem:hypergeom3}.

For $p\in\NN$, let us set
\[
d_p=\#\{(i_1,\dots,i_r)\in\ZZ^r\mid i_1+\cdots+i_r=p,\;n\geq i_1>\cdots> i_r\geq0\}.
\]
Then
\[
d_p=\dim \gr^p_{F_\irr}\bigwedge^rH^n(U,\rd+\rd \mero).
\]
Recall also (\cf\eg\cite[p.\,196]{G-H78}) that the cohomology $H^*(G(r,n+1),\QQ)$ of the \index{Grassmannian}Grassmannian \index{$G1rn$@$G(r,n+1)$}$G(r,n+1)$ is Hodge-Tate with
\[
\dim H^{2p}(G(r,n+1),\CC)=\dim\gr^p_FH^*(G(r,n+1),\CC)=d_p.
\]
\end{exemple}

Let us come back to the general setting. We have $\QQ_U^{\rH\boxtimes r}\simeq\QQ^\rH_{U^r}$ in $\MHM(U)$ (\cf\cite[(4.2.13) to (4.2.15)]{MSaito87}) and in the present algebraic setting, \eqref{eq:alternatingBrieskorn} reads
\[
\bigwedge_{\CC[\hb]}^rG_0^n(U,\mero)\simeq G_0^{mr}(Y,(\finite^0_\dag R_F\cO_X[*P_\red])^\sgn,\merob).
\]

We will simplify the right-hand side. Let $\Delta_{r-1}(U)\subset U^r$ denote the Zariski closed subset where at least two components coincide and set $U^r_{r-1}:=U^r\moins\Delta_{r-1}$. The morphism $\finite:U^r_{r-1}\to U^r_{r-1}/\gotS^r=:V^r_{r-1}$ is finite étale and $V^r_{r-1}$ is smooth quasi-projective. We then consider the algebraic mixed Hodge module $(\Hrho_*\QQ^\rH_{U^r_{r-1}})^\sgn$, whose associated perverse sheaf is the anti-invariant shifted local system $(\finite_*\QQ_{U^r_{r-1}})^\sgn[rn]$. We denote by $G_0^k(V^r_{r-1},\sgn,\mero^{(\oplus r)})$ the corresponding Brieskorn lattices. In general, we a priori do not know whether $G_0^k(V^r_{r-1},\sgn,\mero^{(\oplus r)})$ is zero for $k\neq nr$. Let us however notice that, when $\mero$ (which is cohomologically tame) has only simple critical points, like in Example \ref{exem:Pn}, then $\mero^{(\oplus r)}$ is cohomologically tame on $V^r_{r-1}$ with respect to the local system $(\finite_*\QQ_{U^r_{r-1}})^\sgn$ (\cf \cite[Rem.\,3.13]{K-S06}), so in this case we can say that $G_0^k(V^r_{r-1},\sgn,\mero^{(\oplus r)})=0$ for $k\neq nr$.

\begin{proposition}\label{prop:Fgeom}
If $\mero:U\to\CC$ is cohomologically tame on $U$ then, as $\CC[\hb]\langle\hb^2\partial_\hb\rangle$-modules,
\[
G_0^\ell(V^r_{r-1},\sgn,\mero^{(\oplus r)})\simeq
\begin{cases}
\bigwedge_{\CC[\hb]}^rG_0^n(U,\mero)&\text{if }\ell=nr,\\
0&\text{otherwise}.
\end{cases}
\]
\end{proposition}

\begin{proof}
Since we have $\QQ_U^{\rH\boxtimes r}\simeq\QQ^\rH_{U^r}$, we are led to identifying $G_0^\ell(V^r_{r-1},\sgn,\mero^{(\oplus r)})$ with $G_0^\ell(U^r,\mero^{(\oplus r)})^\sgn$.

For $r\geq2$, and for $s\in[1,r-1]$, let $\Delta_s=\Delta_s(U)\subset U^r$ denote the union of the diagonals of $U^r$ defined by the equality of $r-s+1$ components, and let us set $\Delta_0=\emptyset$. The pure diagonals $\Delta_s\moins\Delta_{s-1}$ are denoted by $\Delta_{(s)}$. We also denote by~$U^r_s$ the open subset $U^r\moins\Delta_s$. They form sequences
\[
\emptyset=\Delta_0\subset\Delta_1\subset\cdots\subset\Delta_{r-1},\quad U^r=U^r_0\supset U^r_1\supset\cdots\supset U^r_{r-1}.
\]
We have the following properties:
\begin{itemize}
\item
$\Delta_s$ is of pure dimension $sn$ with irreducible components all identified with~$U^s$, and the restriction of $\mero^{\oplus r}$ to any such component is of the form \hbox{$\mero_{\bma}^{\oplus s}:=a_1\mero\!\sboxplus\!\cdots\!\sboxplus\!a_s \mero$} with~$a_i$ (\hbox{$i=1,\dots, s$}) satisfying $a_1+\cdots+a_s=r$.
\item
$\Delta_{(s)}$ is a disjoint union of copies of $U^s_{s-1}$ and we have a commutative diagram
\[
\xymatrix{
\bigsqcup U_{s-1}^s\ar@{=}[r]\ar@{^{ (}->}[d]&\Delta_{(s)}\ar@{^{ (}->}[d]\\
\bigsqcup U^s\ar[r]&\Delta_s
}
\]
where the inclusions are open.
\end{itemize}

Let us fix $s,r$ with $1\leq s<r$ and consider the open-closed decomposition
\[
\bigsqcup U_{s-1}^s=\Delta_{(s)}\Hto{i}U^r_{s-1}\Hfrom{j}U^r_s.
\]
It induces the distinguished triangle in $\catD^\rb(\MHM(U^r_{s-1}))$:
\[
\Hi_*\QQ_{\Delta_{(s)}}^\rH((s-r)n)\to\QQ_{U^r_{s-1}}^\rH\to\Hj_*\QQ_{U^r_s}^\rH\To{+1}
\]
and we recall that $\QQ_{\Delta_{(s)}}^\rH$ is isomorphic to a direct sum of copies of $\QQ_{U^s_{s-1}}^\rH$. According to Proposition \ref{prop:Brieskornfree}\eqref{prop:Brieskornfree2}, we obtain a long exact sequence of $\CC[\hb]\langle\hb^2\partial_\hb\rangle$-modules
\begin{multline}\label{eq:Mrm}
\cdots\to\hb^{s-r}\underbrace{\bigoplus G_0^{\ell+2(s-r)n}(U^s_{s-1},\mero^{\oplus r}_{|U^s_{s-1}})}_{(*)}\\[-5pt]
\to G_0^\ell(U^r_{s-1},\mero^{\oplus r})\to G_0^\ell(U^r_s,\mero^{\oplus r})\to\cdots
\end{multline}
For any component in $(*)$, there exists a transposition in $\gotS_r$ which acts as the identity on the corresponding $U^s$, hence it acts by the identity on the corresponding term $G_0^{\ell+2(s-r)n}(U^s,\mero^{\oplus r}_{|U^s_{s-1}})$, so that $(*)^\sgn=0$. As a consequence,
\[
G_0^\ell(U^r_{s-1},\mero^{\oplus r})^\sgn\to G_0^\ell(U^r_s,\mero^{\oplus r})^\sgn
\]
is an isomorphism for every $\ell$. Composing these isomorphisms from $s=1$ to $s=r-1$ entails
\[
G_0^\ell(U^r,\mero^{\oplus r})^\sgn\isom G_0^\ell(U^r_{r-1},\mero^{\oplus r})^\sgn.
\]
This gives the desired result since the left-hand side is zero unless $\ell=rn$, and the right-hand side is isomorphic to $G_0^\ell(U^r_{r-1},\sgn,\mero^{\oplus r})$ by an analogue of Corollary \ref{cor:BrieskorngotS}.
\end{proof}

\begin{corollaire}
With an obvious notation we have\vspace*{-5pt}
\begin{multline*}
\bigwedge^n\big(H^r(U,\rd+\rd \mero),F_\irr^\bbullet H^n(U,\rd+\rd \mero)\big)\\\simeq\big(H^{nr}(V^r_{r-1},\sgn,\rd+\rd \mero^{(\oplus r)}),F_\irr^\bbullet H^{nr}(V^r_{r-1},\sgn,\rd+\rd \mero^{(\oplus r)})\big).\quad\qed
\end{multline*}
\end{corollaire}

\begin{exemplePncont}
Keeping the notation of Example \ref{exem:Pn}, we conclude that, for all $p\in\ZZ$ we have
\[
d_p=\dim\gr^p_FH^*(G(r,n+1),\CC)=\dim\gr^p_{F^\irr}H^{nr}(V^r_{r-1},\sgn,\rd+\rd \mero^{(\oplus r)}).
\]
\end{exemplePncont}

\section[Exponential mixed Hodge structures]{Exponential mixed Hodge structures and irregular Hodge structures}

Let us recall the main properties of the category \index{mixed Hodge structure!exponential --}\index{$EMSQ$@$\EMHS(\QQ)$}$\EMHS(\QQ)$ of \emph{exponential mixed $\QQ$-Hodge structures}, as defined in \cite{K-S10}. It is the full sub-category of the category $\MHM(\Afu,\QQ)$ of algebraic mixed Hodge modules (\cf \cite{MSaito87}) consisting of objects whose underlying perverse sheaf $\ccF_\QQ$ has global hyper\-cohomology equal to zero. It is endowed with the structure of a neutral Tannakian category with the additive convolution functor and fibre functor defined as follows: let $t$ be the coordinate on~$\Afu$ and consider the compactification of $\Afuan$ as a disc $\ov{\Afuan}$ with boundary $S_\infty$; let $I\subset S_\infty$ be the open interval defined by the condition $\arg t\in(\pi/2,3\pi/2)$ and let $\alpha:\Afuan\hto\Afuan\cup I$ be the open inclusion; then the vector space $H^0_\rc(\Afuan\cup I,\bR\alpha_*\ccF_\QQ)$ is the desired fibre. Moreover, if $j_0:\Afu\moins\{0\}\hto\Afu$ denotes the inclusion, then the unit object $\bun$ in this category is the mixed Hodge module with underlying perverse sheaf $j_{0!}\QQ_{\Afu\moins\{0\}}$.

\begin{proposition}\label{prop:EMHSIrrMHS}
There is a natural faithful functor \index{$EMSR$@$\EMHS(\RR)$}$\EMHS(\RR)\mto\IrrMHS(\RR)$ of neutral Tannakian categories, sending $\EMHS(\QQ)$ to $\IrrMHS(\QQ)$.
\end{proposition}

In this way, we can consider $\EMHS(\QQ)$ as a subcategory of $\IrrMHS(\QQ)$.

\skpt
\begin{proof}
\begin{enumerate}
\item\label{enum:EMHSIrrMHS1}
We first show that there is a natural functor $\EMHS(\RR)\mto\iMTM^\intt(\pt)$ which takes values in $\IrrMHS(\CC)$. In order to work in the analytic category, we identify $\MHM(\Afu,\RR)$ as a subcategory of $\MHM(\PP^1,\RR)$ by the functor $j_*$. In \cite[\S13.5]{Mochizuki11}, T.\,Mochizuki constructs a functor $\MHM(\PP^1,\RR)\mto\MTM^\intt(\PP^1)$. Its variant with values in $\iMTM^\intt(\PP^1)$ is described in Section \ref{subsec:MHM}. To a real mixed Hodge module $(\ccN,F_\bbullet\ccN,\ccF_\RR,W_\bbullet)$ (where~$\ccF_\RR$ is a real perverse sheaf on $\PP^1$), one associates the object $(\cM',\cM'',\iC^\nabla,W)$ of $\WRdTriples(\PP^1)$ in the following way:
\begin{itemize}
\item
$\cM''=R_F\ccN$ is the Rees module of the filtration $F_\bbullet\ccN$, and $\cM'=R_F\ccN^\vee$ is similarly defined from the dual filtered $\cD_{\PP^1}$-module $(\ccN^\vee,F_\bbullet\ccN^\vee)$.
\item
The pairing $\iC^\nabla$ is obtained from the distribution-valued pairing
\[
\ccC:\ccN^\vee\otimes_\CC\nobreak\ov\ccN\to\Db_{\PP^1}
\]
that the real structure defines.
\item
The filtration $W$ naturally extends to $\cM',\cM''$ and $\iC^\nabla$ is naturally compatible with it, so that $W$ defines a filtration in $\RdTriples(\PP^1)$.
\end{itemize}

Let $\cT^{t/\hb}$ be the object of $\iMTM^\intt(\PP^1)$ attached to the function $t$. By Theorem \ref{th:mtmgresc}\eqref{th:mtmgresc3}, the functor $\cT^{t/\hb}\otimes\cbbullet$ sends $\MHM(\PP^1,\RR)$ to $\IrrMHM(\PP^1)$. The desired functor is obtained by composing with pushforward by the constant map $\IrrMHM(\PP^1)\to\IrrMHM(\pt)=\IrrMHS(\CC)$. Putting all together, we have constructed a natural functor $\MHM(\Afu,\RR)\mto\iMTM^\intt(\pt)$ which takes values in $\IrrMHS(\CC)$.

The fibre of the integrable mixed twistor object corresponding to $(\ccN,F_\bbullet\ccN,\ccF_\RR,W_\bbullet)$ at $\hb=1$ is $H^0_\dR(\PP^1,\ccE^t\otimes\ccN)$, which is known to be isomorphic to the Betti space \hbox{$H^0_\rc(\Afuan\cup I,\bR\alpha_*\ccF_\CC)$}. This implies the faithfulness of the restriction to $\EMHS(\RR)$ of the functor $\MHM(\Afu,\RR)\mto\iMTM^\intt(\pt)$.

\item\label{enum:EMHSIrrMHS2}
We will now prove that the previous functor $\EMHS(\RR)\mto\iMTM^\intt(\pt)$ is compatible with the Tannakian structure, \ie transforms additive convolution in $\EMHS(\RR)$ to tensor product in $\iMTM^\intt(\pt)$. Since the Tannakian structure on $\IrrMHS(\CC)$ is that induced by that of $\iMTM^\intt(\pt)$, the above functor $\EMHS(\RR)\to\IrrMHS(\CC)$ is also compatible with the Tannakian structures. It is enough to consider the functor $\EMHS(\RR)\mto\MTM^\intt(\pt)$.

We already notice that the assertion, only considered for $\cM',\cM''$, follows from the Thom-Sebastiani formula of Proposition \ref{prop:KunnethBrieskorn} applied to $\mero_1=\mero_2=\id_{\PP^1}$. However, we will give a different proof, in analogy with the property that Fourier transformation changes the convolution of $L^1$ functions with the product of their Fourier transforms.

It will be easier to work in the algebraic category, and use the algebraically defined functors in \cite[Chap.\,14]{Mochizuki11}. We denote by $s:\Afu_{t_1}\times\Afu_{t_2}\to\Afu_t$ the sum map $(t_1,t_2)\mapsto t=t_1+t_2$ and by $a:\Afu_t\to\pt$ the constant map. Let $\cT_1,\cT_2$ be objects of $\EMHS(\RR)$, that we regard as objects of $\MTM^\intt(\Afu)$. The additive convolution is then defined by the formula $\cT_1\star\cT_2=\Ts_*^0(\cT_1\boxtimes\cT_2)$. Moreover, we know that \hbox{$\Ts_*^j(\cT_1\boxtimes\cT_2)=0$} for $j\neq0$ according to the faithfulness of the restriction functor to $\hb=1$, since this holds for the underlying $\cD$-modules. The object in $\IrrMHS(\CC)$ corresponding to $\cT_1\star\cT_2$ is (we use Notation \eqref{eq:notationTvarphiotimes})
\[
\Ta_*^0\big[\cT^{t/\hb}\otimes(\cT_1\star\cT_2)\big]=\Ta_*^0\big[\cT^{t/\hb}\otimes\Ts_*^0(\cT_1\boxtimes\cT_2)\big],
\]
and we know similarly that $\Ta_*^j\big[\cT^{t/\hb}\otimes(\cT_1\star\cT_2)\big]=0$ for $j\neq0$. By adapting \cite[Lem.\,11.3.4]{Mochizuki11} to the present algebraic setting, with $\cT_0$ (there) corresponding to $\cT^{t/\hb}$ (here), we obtain
\begin{align*}
\cT^{t/\hb}\otimes\Ts_*^0(\cT_1\boxtimes\cT_2)&\simeq \Ts_*^0\big[s^*\cT^{t/\hb}\otimes(\cT_1\boxtimes\cT_2)\big]\\
&\simeq \Ts_*^0\big[(\cT^{t_1/\hb}\otimes\cT_1)\boxtimes(\cT^{t_2/\hb}\otimes\cT_2)\big]\\
&\simeq \Ts_*\big[(\cT^{t_1/\hb}\otimes\cT_1)\boxtimes(\cT^{t_2/\hb}\otimes\cT_2)\big].
\end{align*}
We conclude that, working in $\catD^\rb(\MTM^\intt(\Afu))$,
\begin{align*}
\Ta_*^0\big[\cT^{t/\hb}\otimes(\cT_1\star\cT_2)\big]&\simeq\Ta_*\big[\cT^{t/\hb}\otimes(\cT_1\star\cT_2)\big]\\
&\simeq \Ta_*\Ts_*\big[(\cT^{t_1/\hb}\otimes\cT_1)\boxtimes(\cT^{t_2/\hb}\otimes\cT_2)\big]\\
&\simeq \Ta_{1*}(\cT^{t_1/\hb}\otimes\cT_1)\boxtimes\Ta_{2*}(\cT^{t_2/\hb}\otimes\cT_2),
\end{align*}
where the last line uses the compatibility of the pushforward with composition (\cf\cite[Prop.\,14.3.18]{Mochizuki11}) and the compatibility of the pushforward with the external product, which follows from \cite[Lem.\,11.4.14]{Mochizuki11}.

\item\label{enum:EMHSIrrMHS3}
We end the proof by considering the rationality question. Here also, it is enough to work with $\MTM^\intt_\good(\pt)$ instead of $\iMTM^\intt_\good(\pt)$.
\begin{enumerate}
\item
If $\cT$ is the image of a real mixed Hodge module on $\PP^1$, then it comes equipped with a real structure $\kappa$ which makes it an object of $\MTM^\intt_\good(\PP^1,\RR)$. Indeed, this follows from the argument in the bottom of Page 406 in \cite{Mochizuki11}. Note that the $\kk$-perverse sheaf coming in the definition of a $\kk$-mixed Hodge module gives rise to a $\kk$-Betti structure for $\cT$, as seen by noticing that the condition on compatibility with the Stokes structure is empty.
\item
Let us now consider $\cT^{t/\hb}$ as defined in Section \ref{subsec:cTvarphihb} with the function $\mero=\id:\PP^1\to\PP^1$. It is a pure object in $\MTM^\intt_\good(\PP^1,\QQ)$.

\item
We conclude from \cite[Prop.\,13.4.6]{Mochizuki11} that~\hbox{$\cT^{t/\hb}\!\otimes\!\cT$} (with the good $\kk$-structures above) is an object of $\MTM^\intt_\good(\PP^1,[*\infty],\kk)$ (\cf Remark \ref{rem:exptwistRgood}). By the pushforward theorem of \cite[Prop.\,13.4.25]{Mochizuki11}, we conclude that the pushforward of $\cT^{t/\hb}\otimes\cT$ by the constant map $\PP^1\to\pt$ is an object of $\MTM^\intt_\good(\pt,\kk)$. Together with the first part of the proof, we obtain that $\EMHS(\kk)$ is sent to $\IrrMHS(\kk)$.

\item
For the tannakian property, the argument in Item \eqref{enum:EMHSIrrMHS2} of the proof above can be adapted to $\MTM^\intt_\good(\Afu,\kk)$, by using \cite[Prop.\,13.4.26]{Mochizuki11}.\qedhere
\end{enumerate}
\end{enumerate}
\end{proof}

\appendix
\chapterspace{-4}
\chapter*{Appendix}
\renewcommand{\thechapter}{A}
\section{Base change theorems in the complex analytic setting}\label{app:basechange}
In Appendix \ref{app:basechange}, we denote by $\map:X\to Y$ a proper holomorphic map between complex manifolds. We fix a complex manifold $S$ and we denote by $\map_S:X\times S\to Y\times S$ the map $\map\times\id_S$. We denote by $p_X:X\times S\to X$ the projection, and similarly for $Y$, or simply $p$ if the context is clear.

\subsection{Base change for coherent \texorpdfstring{$\cO$}{O}-modules}

\begin{proposition}\label{prop:basechangeO}
Let $\cF$ be a coherent $\cO_X$-module. Then for each $k\in\ZZ$, the canonical morphism
\[
p_Y^*R^k\map_*\cF\to R^k\map_{S,*}p_X^*\cF
\]
is an isomorphism, and the same holds if $\cF$ is an inductive limit of coherent $\cO_X$\nobreakdash-modules.
\end{proposition}

\begin{proof}
Since $\map$ is proper, we have natural isomorphisms
\begin{align*}
p_Y^*R^k\map_*\cF&:=\cO_{Y\times S}\otimes_{p_Y^{-1}\cO_Y}p_Y^{-1}R^k\map_*\cF\\
&\isom\cO_{Y\times S}\otimes_{p_Y^{-1}\cO_Y}R^k\map_{S,*}p_X^{-1}\cF\quad\text{\cite[Prop.\,2.5.11]{K-S90}}\\
&\isom R^k\map_{S,*}\bigl(\map_S^{-1}\cO_{Y\times S}\otimes_{\map_S^{-1}p_Y^{-1}\cO_Y}p_X^{-1}\cF\bigr)\quad\text{\cite[Prop.\,2.6.6]{K-S90}}
\end{align*}
and we use the morphisms $\map_S^{-1}p_Y^{-1}\cO_Y=p_X^{-1}\map^{-1}\cO_Y\to p_X^{-1}\cO_X$ and $\map_S^{-1}\cO_{Y\times S}\to\cO_{X\times S}$, to define the canonical morphism.

Since the morphism is well-defined, the question is local on $S$, so we can assume that~$S$ is an open set in $\CC^{\dim S}$. We denote by $\PP=\PP^{\dim S}$ the projective space and by~$\infty$ its divisor at infinity. We now denote by $q_X:X\times\PP\to X$ the projection and by $q_X^\an $ the partial analytification functor $\Mod(\cO_X[t])\mto\Mod(\cO_{X\times\PP}(*\infty))$, where~$t$ is a chosen coordinate system on $\CC^{\dim S}$. We know that it is quasi-inverse to the sheaf-theoretic pushforward functor $q_{X,*}$ restricted to the localized objects, which has no higher direct images by Grauert's theorem (\cf\eg\cite[App.\,A]{D-S02a}). Let us set $\cF[t]=\cO_X[t]\otimes_{\cO_X}\cF$.

Since $\map$ is proper, it is compatible with inductive limits and we have
\[
(R^k\map_*\cF)[t]\isom R^k\map_*(\cF[t]).
\]
Therefore,
\[
(R^k\map_*\cF)[t]\isom R^k\map_*(q_{X,*}q_X^\an \cF[t])=q_{Y,*}R^k\map_{\PP,*}(q_X^\an \cF[t]),
\]
and thus
\[
q_Y^*\bigl((R^k\map_*\cF)[t]\bigr)\isom R^k\map_{\PP,*}(q_X^\an \cF[t]).
\]
Restricting to $S$ gives the desired isomorphism.
\end{proof}

\subsection{Base change for good \texorpdfstring{$\cR_\cX$}{RcX}-modules}

\begin{proposition}\label{prop:basechangeR}
Let $\cM$ be a good $\cR_\cX$-module. Then for each $k$, there is a functorial morphism
\[
p_Y^+\cH^k\map_+\cM\to\cH^k\map_{S,+}p_X^+\cM
\]
which is an isomorphism.
\end{proposition}

\begin{proof}
We decompose $\map$ as the composition of the graph inclusion (closed immersion) and the second projection. It is then enough to construct the morphism in both cases.

\oldsubsubsection*{Proof in the case of a projection}
We assume that $\map$ is the projection \hbox{$X\!=\!Y\!{\times}Z\!\to\!Y$}. We have a natural morphism
\[
p_Y^+\bR \map_*\pDR_{\cX/\cY}\cM\to\bR \map_{S,*}\pDR_{\cX\times S/\cY\times S}p_X^+\cM
\]
(recall that $p^+\cM=p^*\cM$ equipped with the pullback connection). The morphism induced at the level of $E_2$ terms of the spectral sequences is an isomorphism, according to Proposition \ref{prop:basechangeO}. By the goodness property, the spectral sequences degenerate at a finite stage, hence the result.

\oldsubsubsection*{Proof in the case of an immersion}
Due to the possible existence of $\hb$-torsion, Kashiwara's equivalence does not hold without any supplementary assumption for $\cR$\nobreakdash-modules. Nevertheless, many results concerning pushforward and pullback of a $\cD$-module by a closed immersion can be extended. Let $\map:X\hto Y$ be the immersion of a closed submanifold $X$ of $Y$ and let us set $\delta=\dim X-\dim Y$. We have $\cH^k\map_+=0$ for $k\neq0$, and a natural isomorphism $\id\isom \map^+[\delta]\map_+$ on $\Mod(\cR_\cX)$. On the other hand, given any $\cR_\cY$-module~$\cM$ supported on $X$, the natural morphism $\map_+\map^+[\delta]\cM\to\cM$ is injective but its cokernel is a priori only a $\hb$-torsion $\cR_\cY$-module supported on $X$. We have
\[
\map_S^+[\delta]p_Y^+\cH^0\map_+\cM\simeq p_X^+\map^+[\delta]\cH^0\map_+\cM\simeq p_X^+\cM.
\]
Applying $\cH^0\map_{S,+}$ on the left, we get
\[
\cH^0\map_{S,+}\map_S^+[\delta]p_Y^+\cH^0\map_+\cM\simeq\cH^0\map_{S,+}p_X^+\cM
\]
and it remains to check that the natural morphism $\cH^0\map_{S,+}\map_S^+[\delta]p_Y^+\cH^0\map_+\cM\to p_Y^+\cH^0\map_+\cM$ is an isomorphism. This is a local question on $X,Y,S$ and can be checked by a local computation.
\end{proof}

\section{Compatibility between duality and external product}\label{app:exernaldual}
We give the proof of the compatibility as asserted in the third line of \ref{subsub:commutation}\eqref{enum:rel6}. We consider left $\cD$-modules, and right $\cD$-modules are regarded as left $\cD^\op$-modules (\ie $\cD$~with its opposed structure). Let us consider the diagram
\[
\xymatrix{
&X\times X'\ar[dl]_{p}\ar[dr]^{p'}&\\
X&&X'
}
\]
Note that $\cD_{X\times X'}$ is $\cD_X\boxtimes\cD_{X'}$-flat since $\cO_{X\times X'}$ is $\cO_X\boxtimes\nobreak\cO_{X'}$-flat. Let us now assume that $\ccM,\ccM'$ are $\cD$-coherent (or in $\catD^\rb_\coh(\cD)$).

\begin{enumerate}
\item\label{enum:1}
On the one hand, we have
\[
\bR\cHom_{\cD_X\boxtimes\cD_{X'}}(\ccM\boxtimes\ccM',\cD_X\boxtimes\cD_{X'})
\simeq\bR\cHom_{\cD_X}(\ccM,\cD_X)\boxtimes\bR\cHom_{\cD_{X'}}(\ccM',\cD_{X'}).
\]

As a matter of fact, let $I^\cbbullet$ be a $\cD_X\otimes\cD_X^\op$-injective resolution of $\cD_X$, and similarly with $I^{\prime\cbbullet}$, and let $J^\cbbullet$ be a $(\cD_X\boxtimes\cD_{X'})\otimes(\cD_X\boxtimes\cD_{X'})^\op$-injective resolution of $I^\cbbullet\boxtimes I^{\prime\cbbullet}$. We clearly have a $(\cD_X\boxtimes\cD_{X'})^\op$-linear morphism
\begin{multline*}
\cHom_{\cD_X}(\ccM,I^\cbbullet)\boxtimes\cHom_{\cD_{X'}}(\ccM',I^{\prime\cbbullet})=\cHom_{\cD_X\boxtimes\cD_{X'}}(\ccM\boxtimes\ccM',I^\cbbullet\boxtimes I^{\prime\cbbullet})\\
\to \cHom_{\cD_X\boxtimes\cD_{X'}}(\ccM\boxtimes\ccM',J^\cbbullet).
\end{multline*}
The desired isomorphism is now a local question, and since $\ccM,\ccM'$ are coherent, it is enough, by taking a locally free resolution of them, to check it for $\cD$, for which the assertion is clear.
\item\label{enum:2}
Next, there is a natural morphism (in $\catD^\rb(\cD_{X\times X'}^\op)$):
\begin{multline*}
\bR\cHom_{\cD_X\boxtimes\cD_{X'}}(\ccM\boxtimes\ccM',\cD_X\boxtimes\cD_{X'})\otimes_{\cD_X\boxtimes\cD_{X'}}\cD_{X\times X'}\\
\to\bR\cHom_{\cD_X\boxtimes\cD_{X'}}(\ccM\boxtimes\ccM',\cD_{X\times X'}),
\end{multline*}
where $\cD_{X\times X'}$ is regarded as a $\cD_X\boxtimes\cD_{X'}$-module and as a $\cD_{X\times X'}^\op$-module. Indeed, one considers an injective resolution $J^\cbbullet$ of $\cD_X\boxtimes\cD_{X'}$ as a $(\cD_X\boxtimes\cD_{X'})\otimes(\cD_X\boxtimes\cD_{X'})^\op$-module. There is thus a natural morphism of $\cD_{X\times X'}^\op$-modules
\begin{multline*}
\cHom_{\cD_X\boxtimes\cD_{X'}}(\ccM\boxtimes\ccM',J^\cbbullet)\otimes_{\cD_X\boxtimes\cD_{X'}}\cD_{X\times X'}\\
\to\cHom_{\cD_X\boxtimes\cD_{X'}}(\ccM\boxtimes\ccM',J^\cbbullet\otimes_{\cD_X\boxtimes\cD_{X'}}\cD_{X\times X'}).
\end{multline*}
One then chooses an injective resolution $K^\cbbullet$ of $J^\cbbullet\otimes_{\cD_X\boxtimes\cD_{X'}}\cD_{X\times X'}$ as a $(\cD_X\boxtimes\nobreak\cD_{X'})\otimes\cD_{X\times X'}^\op$-module, and one obtained the desired morphism.

If moreover $\ccM,\ccM'$ are $\cD$-coherent, then this morphism is an isomorphism in $\catD^\rb(\cD_{X\times X'}^\op)$. Indeed, the assertion is local, and by taking local free resolutions of $\ccM,\ccM'$, one can reduce to the case $\ccM=\cD_X$ and $\ccM'=\cD_{X'}$, in which case we clearly have an isomorphism.

\item\label{enum:3}
Lastly, one has a natural morphism
\[
\bR\cHom_{\cD_X\boxtimes\cD_{X'}}(\ccM\boxtimes\ccM',\cD_{X\times X'})\to\bR\cHom_{\cD_{X\times X'}}(\ccM\boxtimes_\cD\ccM',\cD_{X\times X'}).
\]
Indeed, since $\cD_{X\times X'}$ is $\cD_X\boxtimes\cD_{X'}$-flat, an injective $\cD_{X\times X'}$-module is also an injective $\cD_X\boxtimes\nobreak\cD_{X'}$-module. By taking an injective resolution $I^\cbbullet$ of $\cD_{X\times X'}$ as a $\cD_{X\times X'}\otimes\nobreak\cD_{X\times X'}^\op$-module, we get a morphism of $\cD_{X\times X'}^\op$-complexes
\[
\cHom_{\cD_X\boxtimes\cD_{X'}}(\ccM\boxtimes\ccM',I^\cbbullet)\to\cHom_{\cD_{X\times X'}}(\ccM\boxtimes_\cD\ccM',I^\cbbullet),
\]
hence the desired morphism. It is an isomorphism if $\ccM,\ccM'$ are $\cD$-coherent, which is seen as above.
\end{enumerate}

Adding the usual shifts, one obtains for coherent $\ccM,\ccM'$:
\[
\bD\ccM\boxtimes_\cD\bD\ccM'\simeq\bD(\ccM\boxtimes_\cD\ccM').
\]
As a matter of fact, the left-hand side is the term considered in \eqref{enum:2}, after \eqref{enum:1}, and the right-hand side is the second term in \eqref{enum:3}.\qed

\backmatter
\chapterspace{-2}
\let\oldog\og\let\oldfg\fg
\providecommand{\SortNoop}[1]{}\providecommand{\eprint}[1]{\href{http://arxiv.org/abs/#1}{\texttt{arXiv\string:\allowbreak#1}}}\providecommand{\hal}[1]{\href{https://hal.archives-ouvertes.fr/hal-#1}{\texttt{hal-#1}}}\providecommand{\doi}[1]{\href{http://dx.doi.org/#1}{\texttt{doi\string:\allowbreak#1}}}
\providecommand{\bysame}{\leavevmode ---\ }
\providecommand{\og}{``}
\providecommand{\fg}{''}
\providecommand{\smfandname}{\&}
\providecommand{\smfedsname}{\'eds.}
\providecommand{\smfedname}{\'ed.}
\providecommand{\smfmastersthesisname}{M\'emoire}
\providecommand{\smfphdthesisname}{Th\`ese}

\renewcommand{\indexname}{Index of notation}
\begin{theindex}

  \item $\sboxplus$, 101
  \item $\boxtimes_\cD$, 43
  \item $\hbboxtimes$, 29
  \item $C_{\bS}$, 14, 25
  \item $C_{\bS}^\nabla$, 16
  \item $\gC^{d_X,d_X}_{X\times\Omega/\Omega}$, 24
  \item $\gC^{\infty,d_X,d_X}_{X\times\Omega/\Omega}$, 24
  \item $\iC$, 16, 27
  \item $\iC^\nabla$, 16
  \item $\thetaiC$, 55
  \item $\CC_\hb$, 8
  \item $\Db_{X\times\Omega}$, 23
  \item $\Db_{X\times\Omega/\Omega}$, 23
  \item $\Db^\infty_{X\times\Omega/\Omega}$, 23
  \item $\Db_{X\times\bS/\bS}$, 17, 23
  \item $\Db_{X\times\bS/\bS}^\modD$, 33
  \item $\EMHS(\QQ)$, 110
  \item $\EMHS(\RR)$, 110
  \item $\cE^{\mero/\hb}$, 40
  \item $\cE_*^{\mero/\hb}$, 40
  \item $\ste^0_\dag$, 43
  \item $\ste^+$, 43
  \item $\map$, 8
  \item $\taumap$, 56
  \item $\mero^{\oplus r}$, 106
  \item $\mero^{(\oplus r)}$, 106
  \item $F^\sHN_\bbullet$, 84
  \item $F^\irr_\bbullet$, 61, 84
  \item $\mapsm$, 8
  \item $G_0^k(U,\mero)$, 107
  \item $G_0^\cbbullet(X,\cM,\mero)$, 100
  \item $G_0^\cbbullet(X,\cM,\mero)^\chi$, 106
  \item $G(r,n+1)$, 108
  \item $\gamma$, 8
  \item $\Gamma_{[*H]}$, 43
  \item $\fun$, 8
  \item $H$, 8
  \item $H_\fun$, 8
  \item $H^\cbbullet_\dR(X,\cM)$, 99
  \item $H^\cbbullet_\dR(X,\ccM)$, 99
  \item $H^\cbbullet_\dR(X,\Xi_{\DR}(\cT,W_\bbullet))$, 99
  \item $\HN^p$, 83
  \item $\iota$, 8
  \item $i_\fun$, 8
  \item $i_{\tau=1}$, 52
  \item $i_{\tau=\hb}$, 52
  \item $i_{\hb=1}$, 52
  \item $\iMTM^\intt(X,\RR)$, 38
  \item $\iMTM^\intt(X)$, 20
  \item $\iMTM^\intt(X,[\star H])$, 31
  \item $\iMTM^\intt(X,(*H))$, 31
  \item $\iMTM^\intt_\good(X,\kk)$, 39
  \item $\iMTM^\resc(X)$, 70
  \item $\iMTS^\intt(\CC)$, 17
  \item $\iMTS^\intt(\RR)$, 17
  \item $\iMTS_\good^\intt(\kk)$, 17
  \item $\IrrMHM(X)$, 70
  \item $\IrrMHM(X,\kk)$, 70
  \item $\IrrMHS(\CC)$, 96
  \item $\IrrMHS(\kk)$, 96
  \item $\kappa$, 9
  \item $\MHM(X)$, 1
  \item $\MHM(X,\CC)$, 48
  \item $\MHM(X,\RR)$, 47
  \item $\MHM(X,\kk)$, 47
  \item $\MTM(X)$, 1, 19
  \item $\MTM(X,\RR)$, 38
  \item $\MTM^\intt(X,\RR)$, 38
  \item $\MTM^\intt(X)$, 19
  \item $\MTM^\intt(X,[\star H])$, 31
  \item $\MTM^\intt(X,(*H))$, 31
  \item $\MTM^\intt_\good(X,\kk)$, 39
  \item $\MTS(\CC)$, 15
  \item $\MTS^\intt(\CC)$, 15
  \item $\MTS(\RR)$, 15
  \item $\MTS^\intt(\RR)$, 15
  \item $\MTS_\good^\intt(\kk)$, 15
  \item $\MTW(X)$, 19
  \item $\MTW^\intt(X)$, 19
  \item $\cM[\star H]$, 18
  \item $\cM^\circ$, 17
  \item $\taucM$, 53
  \item $\cM_\fun$, 8
  \item $\ccM_\fun$, 8
  \item $\tauM$, 95
  \item $\mu$, 52
  \item $\cO_{\bS}$, 8
  \item $\pi$, 8
  \item $\pi^\circ$, 8
  \item $\QQ_U^\rH$, 107
  \item $\cQ$, 13
  \item $R_F\ccM$, 47
  \item $R_{F_{\alpha+\bbullet}^\irr}\ccM$, 61
  \item $\RTriples(X)$, 8
  \item $\RdTriples(X)$, 17, 26
  \item $\RdiTriples(X)$, 20, 28
  \item $\RscTriples(X)$, 68
  \item $\RgscTriples(X)$, 68
  \item $\RTriples(\pt)$, 14
  \item $\RdTriples(\pt)$, 14
  \item $\RdiTriples(\pt)$, 16
  \item $\cR_\cX$, 17
  \item $\cR_\cX^\intt$, 17
  \item $\bS$, 8
  \item $\cS$, 10
  \item $\sgn$, 106
  \item $\gotS$, 105
  \item $\sigma$, 8
  \item $\icT^{\mero/\hb}$, 39
  \item $\icT_*^{\mero/\hb}$, 40
  \item $\thetacT $, 56
  \item $\cT^{\mero/\hb}$, 40
  \item $\cT_*^{\mero/\hb}$, 40
  \item $\bT(\ell)$, 10, 15
  \item $\ibT(\ell)$, 16
  \item $\bU(p,q)$, 11, 15
  \item $\WRTriples(X)$, 19
  \item $\WRdTriples(X)$, 18
  \item $X_\fun$, 8
  \item $\tauX$, 52
  \item $\tauX_0$, 52
  \item $\thetaX$, 52
  \item $\cX$, 8
  \item $\cX^\circ$, 8
  \item $\taucX$, 52
  \item $\thetacX$, 52
  \item $\Xi_\fun$, 58
  \item $\Xi_{\DR}$, \finindexnotations{8}

  \indexspace

  \item adjunction, 28
  \item alternating product, 106

  \indexspace

  \item Beilinson's functor, 58
  \item Brieskorn lattice, 100

  \indexspace

  \item cohomologically tame function, 107

  \indexspace

  \item Deligne's meromorphic extension, 65

  \indexspace

  \item Grassmannian, 108

  \indexspace

  \item Harder-Narasimhan filtration, 83
    \subitem tensor product formula for the --, 92
  \item Hermitian duality, 10
  \item Hodge structure
    \subitem non-commutative --, 13, 81
    \subitem semi-infinite --, 81
  \item holonomic $\cD$-module of exponential-regular origin, 47
  \item hypergeometric, 7, 88

  \indexspace

  \item irregular Hodge filtration, 61, 84, 97
    \subitem duality for the --, 93
    \subitem tensor product formula for the --, 92

  \indexspace

  \item jumping index, 65

  \indexspace

  \item K\IeC {\"u}nneth formula, 99

  \indexspace

  \item localization, 18, 30, 39
    \subitem ``stupid'', 18, 30, 39, 41
    \subitem dual, 18, 30

  \indexspace

  \item maximalization, 58
  \item mixed Hodge module, 1
    \subitem complex --, 48
    \subitem exponentially twisted --, 5, 76
    \subitem Fourier-Laplace transform of a --, 73
    \subitem irregular --, 1, 70
    \subitem irregular -- of exponential origin, 4
    \subitem real --, 47
  \item mixed Hodge structure
    \subitem exponential --, 81, 110
    \subitem irregular --, 96
  \item mixed twistor $\cD$-module, 1
    \subitem integrable --, 19
  \item mixed twistor-rescaled $\cD$-module, 70

  \indexspace

  \item polarization, 11

  \indexspace

  \item Rees module, 47
  \item rescaling
    \subitem of a $\CC[\hb]\langle\hb^2\partial_\hb\rangle$-module, 95
    \subitem of a coherent $\cO_X$-module, 53
    \subitem of an $\cR^\intt_\cX$-module, 53
    \subitem of an integrable $\iota$-sesquilinear pairing, 55
    \subitem of an integrable triple, 96
  \item rescaling functor, 56
  \item $\cR_\cX$-module
    \subitem integrable --, 18
    \subitem integrable holonomic --, 21
    \subitem localizable along $H$ --, 18
    \subitem strictly $\RR$-specializable --, 21
      \subsubitem regular --, 54
    \subitem strictly specializable --, 18, 21

  \indexspace

  \item $\iota$-sesquilinear pairing, 27
  \item $\sigma$-sesquilinear pairing, 25
  \item spectrum, 65, 94
  \item Stokes filtration, 12
  \item Stokes-filtered local system, 12
  \item strict, 8

  \indexspace

  \item Tate object, 10, 15, 16
  \item Tate twist, 28
  \item TERP structure, 13, 81
  \item Thom-Sebastiani formula, 101, 102, 107
  \item trivializing lattice, 85
    \subitem $V$-adapted --, 86
  \item twistor structure, 9
    \subitem $\kk$-good integrable mixed --, 12
    \subitem integrable --, 10
    \subitem integrable mixed --, 11, 81
    \subitem integrable polarizable pure --, 11
    \subitem mixed --, 11
    \subitem polarizable pure --, 11
    \subitem pure --, 11
    \subitem real --, 9
    \subitem real integrable mixed --, 11
    \subitem real mixed --, 11

  \indexspace

  \item well-rescalable $\cR^\intt_\cX$-module
    \subitem graded --, 62

\end{theindex}
\let\mathcal\oldmathcal
\end{document}